\documentclass[12pt]{amsart}
\usepackage{amsmath,amsthm,amssymb,amscd,cite,array}
\usepackage[dvips]{graphicx}

\newtheorem{pr}{Proposition}
\newtheorem{lem}[pr]{Lemma}
\newtheorem{thm}[pr]{Theorem}
\newtheorem{s}[pr]{Corollary}
\theoremstyle{remark}
\newtheorem{zam}{Remark}
\newtheorem*{obozn}{{\rm\bf Denotation}}
\newtheorem*{obozns}{{\rm\bf Denotations}}

\renewcommand\div{\text{ }\vdots\text{ }}
\newcommand\ndiv{\not\vdots\text{ }}
\newcommand{\myChar}{\mathrm{char\,}K}
\newcommand{\myNod}{\text{{\rm gcd}}}
\newcommand{\Hom}{\mathrm{Hom}}
\renewcommand{\Im}{\mathrm{Im}}
\newcommand{\Ker}{\mathrm{Ker}}
\newcommand{\HH}{\mathrm{HH}}
\newcommand{\N}{\mathbb{N}}
\newcommand{\Z}{\mathbb{Z}}
\newcommand{\cl}{\mathrm{cl}}

\def\a{\alpha}
\def\b{\beta}
\def\g{\gamma}
\def\le{\leqslant}
\def\ge{\geqslant}
\def\ra{\rightarrow}

\begin{document}

\title{Hochschild cohomology ring for self-injective algebras of tree class $E_6$}
\author{Mariya Pustovykh}
\email{masha@lsi.ru}

\begin{abstract}
We describe the Hochschild cohomology ring for one of the two self-injective
algebras of tree class $E_6$ in terms of generators and relations.
\end{abstract}
\maketitle

\tableofcontents

\section{Introduction}

Consider a self-injective basic algebra of finite representation
type over an algebraically closed field. According to Riedtmann's
classification, the stable $AR$-quiver of such an algebra can be
described with the help of an associated tree, which must be
congruent with one of the Dynkin diagrams $A_n, D_n, E_6, E_7$, or
$E_8$ (see \cite{Riedt}). The complete description of the Hochschild
cohomology ring was obtained for an algebras of the types $A_n$ and
$D_n$, see \cite{Erd,Gen&Ka,Ka,Pu} (type $A_n$) and
\cite{Volkov1,Volkov2,Volkov3,Volkov4,Volkov5,Volkov6} (type $D_n$).
Consider algebras of tree class $E_6$. Any algebra of the class
$E_6$ is derived equivalent to the path algebra for some quiver with
relations. Namely, let $\mathcal Q_s$ ($s\in\N$) is the following
quiver:

\begin{figure}[h]
\includegraphics[width=10cm, scale=1]{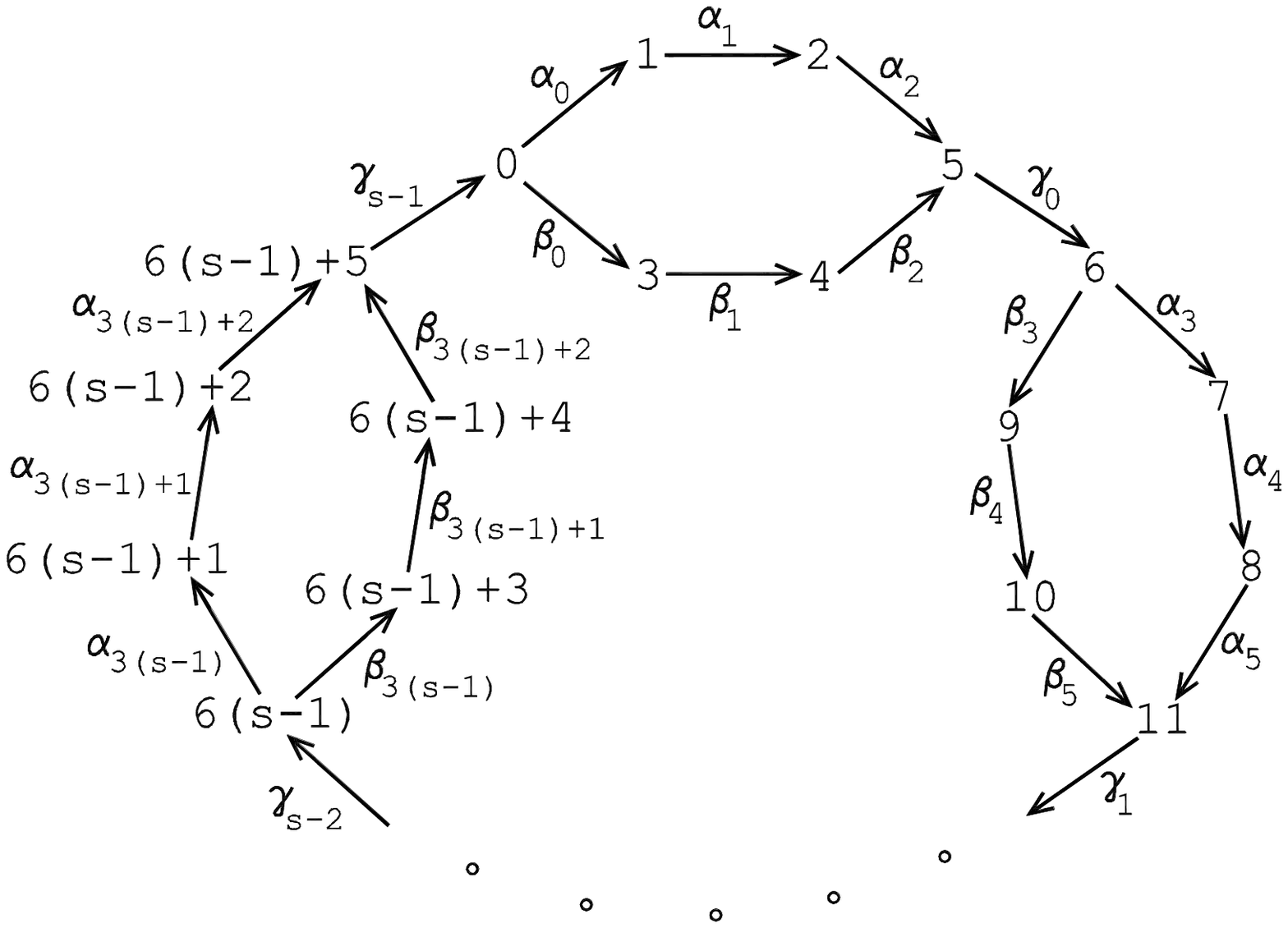}
\end{figure}

Then any algebra of the class $E_6$ is derived equivalent to one of the two following algebras:

1) $R_s=K\left[\mathcal Q_s\right]/I$, where $K$ is a field, and $I$ is the ideal in the path
algebra $K\left[\mathcal Q_s\right]$ of the quiver $\mathcal Q_s$, generated by

a) all the paths of length $5$;

b) the expressions of the form $\a^3-\b^3$, $\a\g\b$, $\b\g\a$.

2) $R_s^\prime=K\left[\mathcal Q_s\right]/I^\prime$, where $K$ is a field, and $I^\prime$ is the
ideal in the path algebra $K\left[\mathcal Q_s\right]$ of the quiver $\mathcal Q$, generated by

a) all the paths of length $5$;

b) the expressions of the form $\a^3-\b^3$, $\a_{3t}\g_{t-1}\b_{3t-1}$, $\b_{3t}\g_{t-1}\a_{3t-1}$
($1\le t\le s-1$), $\a_0\g_{s-1}\a_{3(s-1)+2}$, $\b_0\g_{s-1}\b_{3(s-1)+2}$.

Henceforth we will often omit indexes in arrows $\a_i$, $\b_i$ and $\g_i$ as long as subscripts are
clear from the context.

The present paper is dedicated to the study of Hochschild cohomology ring structure for algebra
$R_s$. For this algebra we obtain the description of Hochschild cohomology ring structure in terms
of generators and relations. In studies of the structure of cohomology ring we will construct the
bimodule resolution of $R_s$, which could  be seen as a whole result.

\section{Statement of the main results}

In what follows, we assume $n=6$.

Let $\HH^t(R)$ is the $t$th group of the Hochschild cohomology ring of $R$ with coefficients in
$R$. Let $\ell$ be the aliquot, and $r$ be the residue of division of $t$ by $11$, $m$ be the
aliquot of division of $r$ by $2$.

Consider the case of $s>1$. To describe Hochschild cohomology ring of algebra $R_s$ we must
introduce the following conditions on an arbitrary degree $t$:

$($1$)$ $r=0$, $\ell\div 2$, $\ell n+m\equiv 0(s)$ or $s=1$;\label{degs}

$($2$)$ $r=0$, $\ell\div 2$, $\myChar=3$, $\ell n+m\equiv 1(s)$ or $s=1$;

$($3$)$ $r=1$, $\ell\div 2$, $\ell n+m\equiv 0(s)$ or $s=1$;

$($4$)$ $r=1$, $\ell\ndiv 2$, $\ell n+m\equiv 0(s)$ or $s=1$;

$($5$)$ $r=2$, $\ell\ndiv 2$, $\ell n+m\equiv 1(s)$ or $s=1$;

$($6$)$ $r=3$, $\ell\div 2$, $\myChar=2$, $\ell n+m\equiv 0(s)$ or $s=1$;

$($7$)$ $r=3$, $\ell\ndiv 2$, $\ell n+m\equiv 0(s)$ or $s=1$;

$($8$)$ $r=4$, $\ell\ndiv 2$, $\ell n+m\equiv 1(s)$ or $s=1$;

$($9$)$ $r=4$, $\ell\ndiv 2$, $\myChar=3$, $\ell n+m\equiv 0(s)$ or $s=1$;

$($10$)$ $r=4$, $\ell\div 2$, $\myChar=2$, $\ell n+m\equiv 1(s)$ or $s=1$;

$($11$)$ $r=5$, $\ell\div 2$, $\myChar=3$, $\ell n+m\equiv 0(s)$ or $s=1$;

$($12$)$ $r=5$, $\ell\ndiv 2$, $\myChar=3$, $\ell n+m\equiv 0(s)$ or $s=1$;

$($13$)$ $r=6$, $\ell\div 2$, $\ell n+m\equiv 0(s)$ or $s=1$;

$($14$)$ $r=6$, $\ell\div 2$, $\myChar=3$, $\ell n+m\equiv 1(s)$ or $s=1$;

$($15$)$ $r=6$, $\ell\ndiv 2$, $\myChar=2$, $\ell n+m\equiv 0(s)$ or $s=1$;

$($16$)$ $r=7$, $\ell\div 2$, $\ell n+m\equiv 0(s)$ or $s=1$;

$($17$)$ $r=7$, $\ell\ndiv 2$, $\myChar=2$, $\ell n+m\equiv 0(s)$ or $s=1$;

$($18$)$ $r=8$, $\ell\div 2$, $\ell n+m\equiv 0(s)$ or $s=1$;

$($19$)$ $r=9$, $\ell\div 2$, $\ell n+m\equiv 0(s)$ or $s=1$;

$($20$)$ $r=9$, $\ell\ndiv 2$, $\ell n+m\equiv 0(s)$ or $s=1$;

$($21$)$ $r=10$, $\ell\ndiv 2$, $\ell n+m\equiv 1(s)$ or $s=1$;

$($22$)$ $r=10$, $\ell\ndiv 2$, $\myChar=3$, $\ell n+m\equiv 0(s)$ or $s=1$.

Let $M=11\frac{2s}{\myNod(2n,s)}$.

\begin{zam}
We will prove in paragraph \ref{sect_res} that the minimal period of bimodule resolution of $R_s$
is $M$.
\end{zam}

Let $\{t_{1, i},\dots,t_{\alpha_i, i} \}$ be a set of all degrees $t$, that satisfy the conditions
of item $i$ from the above list, and such that $0\le t_{j, i}<M$ $(j=1,\dots,\alpha_i)$. Consider
the set $$\mathcal
X=\bigcup_{i=1}^{22}\left\{X^{(i)}_{t_{j,i}}\right\}_{j=1}^{\alpha_i}\cup\{T\},$$ and define a
graduation of polynomial ring $K[\mathcal X]$ such that
\begin{align*}\label{degs2}
&\deg X^{(i)}_{t_{j,i}}=t_{j, i} \:\text{for all} \: i=1,\dots,22 \:\text{and}\: j=1,\dots,\alpha_i;\tag{$\circ$}\\
&\deg T=M.
\end{align*}

\begin{zam}\label{brief_notation}
Hereafter we shall use simplified denotation $X^{(i)}$ instead of $X^{(i)}_{t_{j,i}}$, since lower
indexes are clear from context.
\end{zam}

\begin{obozn}
$$\widetilde X^{(i)}= \begin{cases}X^{(i)},\quad \deg\widetilde
X^{(i)}<\deg T\\TX^{(i)},\quad\text{otherwise.}\end{cases}$$
\end{obozn}

Define a graduate $K$-algebra $\mathcal A=K[\mathcal X]/I$, where $I$ is the ideal generated by
homogeneous elements corresponding to the following relations.
\begin{align*}
&X^{(3)}X^{(1)}=X^{(3)}X^{(2)}=X^{(3)}X^{(3)}=X^{(3)}X^{(5)}=X^{(3)}X^{(8)}=0;\\
&X^{(3)}X^{(9)}=X^{(3)}X^{(10)}=X^{(3)}X^{(11)}=X^{(3)}X^{(12)}=X^{(3)}X^{(14)}=0;\\
&X^{(3)}X^{(16)}=X^{(3)}X^{(17)}=X^{(3)}X^{(19)}=X^{(3)}X^{(21)}=X^{(3)}X^{(22)}=0;\\
&X^{(3)}X^{(4)}=\widetilde X^{(5)};\quad X^{(3)}X^{(6)}=\widetilde
X^{(10)};\\
&X^{(3)}X^{(7)}=2\widetilde X^{(8)};\quad
X^{(3)}X^{(13)}=2\widetilde
X^{(16)};\\
&X^{(3)}X^{(15)}=\widetilde X^{(17)};\quad
X^{(3)}X^{(18)}=\widetilde X^{(19)};\quad X^{(3)}X^{(20)}=\widetilde
X^{(21)}.
\end{align*}
\begin{align*}
X^{(4)}X^{(7)}&=\begin{cases}\widetilde X^{(10)},\quad\myChar
=2,\\0,\quad\text{otherwise};\end{cases}&\text{(r1)}\\
X^{(7)}X^{(7)}&=\begin{cases}s\widetilde X^{(14)},\quad\myChar
=3,\\0,\quad\text{otherwise};\end{cases}&\text{(r2)}\\
X^{(4)}X^{(13)}&=\begin{cases}\widetilde X^{(17)},\quad\myChar
=2,\\0,\quad\text{otherwise};\end{cases}&\text{(r3)}\\
X^{(7)}X^{(18)}&=\begin{cases}-s\widetilde X^{(2)},\quad\myChar
=3,\\0,\quad\text{otherwise};\end{cases}&\text{(r4)}\\
X^{(13)}X^{(20)}&=\begin{cases}\widetilde X^{(10)},\quad\myChar
=2,\\0,\quad\text{otherwise};\end{cases}&\text{(r5)}\\
X^{(18)}X^{(18)}&=\begin{cases}-\widetilde X^{(12)},\quad\myChar
=3,\\0,\quad\text{otherwise};\end{cases}&\text{(r6)}\\
X^{(18)}X^{(20)}&=\begin{cases}-s\widetilde X^{(14)},\quad\myChar
=3,\\0,\quad\text{otherwise}.\end{cases}&\text{(r7)}
\end{align*}

Describe the rest relations as a tables (numbers (r1)--(r7) in tables cells are the number of
relation that defines a multiplication of the following elements).

\setlength{\extrarowheight}{1mm}
\begin{tabular}{c|c|c|c|c|c|c|c|c|c}
&$X^{(1)}$&$X^{(2)}$&$X^{(4)}$&$X^{(6)}$&$X^{(7)}$&$X^{(8)}$&$X^{(9)}$&$X^{(11)}$&$X^{(12)}$\\
\hline
$X^{(1)}$&$\widetilde X^{(1)}$&$\widetilde
X^{(2)}$&$\widetilde X^{(4)}$&$\widetilde X^{(6)}$
&$\widetilde X^{(7)}$&$\widetilde X^{(8)}$&$\widetilde X^{(9)}$&$\widetilde X^{(11)}$&$\widetilde X^{(12)}$ \\
\hline
$X^{(2)}$& &0&0&0&0&0&$\widetilde X^{(8)}$&0&0 \\
\hline
$X^{(4)}$& & & 0&$\widetilde X^{(8)}$&(1)&0&$\widetilde X^{(11)}$&0&$s\widetilde X^{(14)}$ \\
\hline
$X^{(6)}$& & & &0&0&0&0&0&0\\
\hline
$X^{(7)}$& & & & &(2)&0&$s\widetilde X^{(16)}$&0&0 \\
\hline
$X^{(8)}$& & & & & &0&0&0&0\\
\hline
$X^{(9)}$& & & & & & &0&0&$-s\widetilde X^{(19)}$\\
\hline
$X^{(11)}$& & & & & &  & &0&$-s\widetilde X^{(21)}$ \\
\end{tabular}

$\quad$

$\quad$

\begin{tabular}{c|c|c|c|c|c|c|c}
&$X^{(13)}$&$X^{(14)}$&$X^{(15)}$&$X^{(16)}$&$X^{(18)}$&$X^{(20)}$&$X^{(22)}$\\
\hline
$X^{(1)}$&$\widetilde X^{(13)}$&$\widetilde X^{(14)}$&$\widetilde X^{(15)}$&$\widetilde X^{(16)}$&$\widetilde X^{(18)}$&$\widetilde X^{(20)}$&$\widetilde X^{(22)}$ \\
\hline
$X^{(2)}$&$\widetilde X^{(14)}$&0&0&0&0&0&$\widetilde X^{(21)}$ \\
\hline
$X^{(4)}$&(3)&0&$\widetilde X^{(16)}$&0&$\widetilde X^{(20)}$&0&0 \\
\hline
$X^{(6)}$&$\widetilde X^{(19)}$&0&$\widetilde X^{(20)}$&0&0&0&0\\
\hline
$X^{(7)}$&$-2\widetilde X^{(20)}$&0&$\widetilde X^{(19)}$&$-\widetilde X^{(21)}$&(4)&0&$-s\widetilde X^{(5)}$ \\
\hline
$X^{(8)}$&$-\widetilde X^{(21)}$&0&0&0&0&0&0\\
\hline
$X^{(9)}$&$-\widetilde X^{(22)}$&$-\widetilde X^{(21)}$&0&0&$s\widetilde X^{(3)}$&$s\widetilde X^{(5)}$&0\\
\hline
$X^{(11)}$&0&0&0&0&$s\widetilde X^{(5)}$&0&0 \\
\end{tabular}

$\quad$

$\quad$

\begin{tabular}{c|c|c|c|c|c|c|c|c}
&$X^{(12)}$&$X^{(13)}$&$X^{(14)}$&$X^{(15)}$&$X^{(16)}$&$X^{(18)}$&$X^{(20)}$&$X^{(22)}$\\
\hline
$X^{(12)}$&0&$-s\widetilde X^{(2)}$&0&0&0&0&0&$s\widetilde X^{(8)}$ \\
\hline
$X^{(13)}$& &$2\widetilde X^{(4)}$&0&$\widetilde X^{(3)}$&$\widetilde X^{(5)}$&$-\widetilde X^{(7)}$&(5)&$\widetilde X^{(11)}$ \\
\hline
$X^{(14)}$& & & 0&0&0&0&0&0 \\
\hline
$X^{(15)}$& & & &$\widetilde X^{(4)}$&0&$\widetilde X^{(6)}$&$\widetilde X^{(8)}$&0\\
\hline
$X^{(16)}$& & & & &0&$-\widetilde X^{(8)}$&0&0 \\
\hline
$X^{(18)}$& & & & & &(6)&(7)&$s\widetilde X^{(16)}$\\
\hline
$X^{(20)}$& & & & & & &0&0\\
\hline
$X^{(22)}$& & & & & &  & &0 \\
\end{tabular}

\begin{thm}\label{main_thm}
Let $s>1$, $R=R_s$ is algebra of the type $E_6$. Then the Hochschild cohomology ring $\HH^*(R)$ is
isomorphic to $\mathcal A$ as a graded $K$-algebra.
\end{thm}

Consider the case of $s=1$.

Let us introduce the set $$\mathcal X^\prime=\begin{cases}\mathcal X\cup \left\{X^{(23)}_0,
X^{(24)}_0,
X^{(25)}_0, X^{(26)}_0, X^{(27)}_0, X^{(28)}_0\right\},\quad\myChar\ne 3;\\
\mathcal X\cup \left\{X^{(24)}_0, X^{(25)}_0, X^{(26)}_0, X^{(27)}_0, X^{(28)}_0\right\}, \quad
\myChar=3;\end{cases}$$ and define a graduation of polynomial ring $K[\mathcal X^\prime]$ such that
\begin{align*}
&\deg X^{(i)}_{t_{j,i}}=t_{j, i} \:\text{for all} \: i=1,\dots,22 \:\text{and}\: j=1,\dots,\alpha_i;\\
&\deg T=M \text{ (similar to (\ref{degs2}))};\\&\deg X^{(23)}_0=\deg X^{(24)}_0=\deg
X^{(25)}_0=\deg X^{(26)}_0=\deg X^{(27)}_0=\deg X^{(28)}_0=0.
\end{align*}

Define a graduate $K$-algebra $\mathcal A^\prime=K[\mathcal X^\prime]/I^\prime$, where $I^\prime$
is the ideal generated by homogeneous elements corresponding to the relations described in the case
of $s>1$, and by the following relations:
\begin{align*}
X^{(1)}X^{(i)}=&\begin{cases}\widetilde X^{(i)},\quad
t_1=0;\\0,\quad\text{otherwise},\end{cases}\quad i\in\{23,24,28\};\\
X^{(1)}X^{(i)}=&\begin{cases}\widetilde X^{(i)},\quad
t_1=0;\\\widetilde X^{(2)},\quad t_1>0\text{ and
}\myChar=3;\\0,\quad\text{otherwise},\end{cases}
i\in\{25,26,27\};\\
X^{(9)}X^{(i)}=&\widetilde X^{(8)},\quad i\in\{25,26,27\};\\
X^{(13)}X^{(i)}=&\begin{cases}\widetilde X^{(14)},\quad
\myChar=3;\\0,\quad\text{otherwise},\end{cases} i\in\{25,26,27\};\\
X^{(22)}X^{(i)}=&\widetilde X^{(21)},\quad i\in\{25,26,27\};\\
X^{(j)}X^{(i)}=&0,\quad j\in[2, 28]\setminus\{9,13,22\},\quad
i\in[23, 28],
\end{align*}
where $t_1$ denotes a degree of the element $X^{(1)}$.

\begin{thm}\label{main_thm2}
Let $s=1$, $R=R_1$ is algebra of the type $E_6$. Then the Hochschild cohomology ring $\HH^*(R)$ is
isomorphic to $\mathcal A^\prime$ as a graded $K$-algebra.
\end{thm}

\begin{zam}
From the descriptions of rings $\HH^*(R)$ given in theorems
\ref{main_thm} and \ref{main_thm2} it implies, in particular, that
they are commutative.
\end{zam}

\section{Bimodule resolution}\label{sect_res}

We will construct the minimal projective bimodule resolution of the $R$ in the following form: $$
\dots\longrightarrow Q_3\stackrel{d_2}\longrightarrow Q_2\stackrel{d_1}\longrightarrow
Q_1\stackrel{d_0}\longrightarrow Q_0\stackrel\varepsilon\longrightarrow R\longrightarrow 0
$$

Let $\Lambda$ be an enveloping algebra of algebra $R$. Then $R$--$R$-bimodules can be considered as
left $\Lambda$-modules.

\begin{obozns}$\quad$

(1) Let $e_i,\text{ }i\in \Z_{ns}=\{0, 1,\dots, ns-1\},$ be the idempotents of the algebra
$K\left[\mathcal Q_s\right]$, that correspond to the vertices of the quiver $\mathcal Q_s$.

(2) Denote by $P_{i,j}=R(e_i\otimes e_j)R=\Lambda(e_i\otimes e_j)$, $i,j\in \Z_{ns}$. Note that the
modules $P_{i,j}$, forms the full set of the (pairwise non-isomorphic
 by) indecomposable projective $\Lambda$-modules.

(3) For $a\in\Z$, $t\in\N$ we denote the smallest nonnegative deduction of $a$ modulo $t$ with
$(a)_t$ (in particular, $0\le(a)_t\le t-1$).

\end{obozns}

Let $R=R_s$. We introduce an automorphism $\sigma\text{: }R\rightarrow R$, which is mapping as
follows:
$$\sigma(e_i)=\begin{cases}
 e_{i+n^2},\quad i\equiv 0,5(n);\\
 e_{i+n^2+2},\quad i\equiv 1,2(n);\\
 e_{i+n^2-2},\quad i\equiv 3,4(n),
 \end{cases}$$
 $$\sigma(\g_i)=-\g_{i+n},$$
$$\sigma(\a_i)=\begin{cases}
 -\b_{i+3n},\quad i\equiv 0,1(3);\\
 \b_{i+3n},\quad i\equiv 2(3),
 \end{cases} \sigma(\b_i)=\begin{cases}
 -\a_{i+3n},\quad i\equiv 1,2(3);\\
 \a_{i+3n},\quad i\equiv 0(3).
 \end{cases}$$

Define the helper functions $f\text{: }\Z\times\Z\rightarrow\Z$, $h\text{:
}\Z\times\Z\rightarrow\Z$ and $\lambda\text{: }\Z\rightarrow\Z$, which act in the following way:
$$f(x,y)=\begin{cases}1,\quad x=y;\\0,\quad x\ne y,\end{cases}\quad
 h(x,y)=\begin{cases}
 1,\quad x\div 2,\text{ }x<y;\\
 0,\quad x\ndiv 2,\text{ }x<y;\\
 1,\quad x\ndiv 2,\text{ }x\ge y;\\
 0,\quad x\div 2,\text{ }x\ge y,
\end{cases}$$
$$\lambda(i)=\begin{cases}i,\quad i\equiv 0,5(\mathrm{mod}\text{ } 6);\\
 i+2,\quad i\equiv 1,2(\mathrm{mod}\text{ } 6);\\
 i-2,\quad i\equiv 3,4(\mathrm{mod}\text{ } 6).\end{cases}$$

Introduce $Q_r\text{ }(r\le 10)$. Let $m$ be the aliquot of division of $r$ by $2$ for considered
degree $r$. We have
\begin{align*}
Q_{2m}&=\bigoplus_{r=0}^{s-1} Q_{2m,r}^\prime,\quad 0\le m\le n-1,\\
Q_{2m+1}&=\bigoplus_{r=0}^{s-1} Q_{2m+1,r}^\prime,\quad 0\le m\le
n-2,
\end{align*}
where
\begin{multline*}
Q_{2m,r}^\prime=\left(\bigoplus_{i=0}^{f(m,2)}P_{(r+m)n-1+h(m,2)+i,rn}\right)\\
\oplus\bigoplus_{i=0}^{f(m,3)}\left(P_{(r+m)n+1+(m)_3+3i,rn+1}
\oplus P_{\lambda((r+m)n+1+(m)_3+3i),rn+3}\right)\\
\oplus\bigoplus_{i=0}^{f(m,2)}\big(P_{(r+m)n+1+(m+1)_3+3f(m,5)+3i,rn+2}
\oplus P_{\lambda((r+m)n+1+(m+1)_3+3f(m,5)+3i),rn+4}\big)\\
\oplus\left(\bigoplus_{i=0}^{f(m,3)}P_{(r+m+1)n-h(m,3)+i,rn+5}\right),
\end{multline*}
\begin{multline*}
Q_{2m+1,r}^\prime=\left(\bigoplus_{i=0}^{1-f(m,4)}P_{(r+m)n+1+h(m,0)+4f(m,4)+2i,rn}\right)\\
\oplus P_{(r+m+1)n-1+h(m,5)-4f(m,0),rn+1}\oplus P_{\lambda((r+m+1)n-1+h(m,5)-4f(m,0)),rn+3}\\
\oplus P_{(r+m+1)n-1+h(m,0)+4f(m,4),rn+2}\oplus
P_{\lambda((r+m+1)n-1+h(m,0)+4f(m,4),rn+2),rn+4}\\
\oplus\left(\bigoplus_{i=0}^{1-f(m,0)}P_{(r+m+1)n+1+h(m,5)-2f(m,0)+2i,rn+5}\right).\\
\end{multline*}

Now we shall describe differentials $d_r$ for $r\le 10$. Since $Q_i$ are direct sums, their
elements can be concerned as column vectors, hence differentials can be described as matrixes
(which are being multiplied by column vectors from the right). Now let us describe the matrixes of
differentials componentwisely.

\begin{zam}
\textup{Numeration of lines and columns always starts with zero.}
\end{zam}

\begin{obozns}$\quad$

(1) Denote by $w_{i\ra j}$ the way that starts in $i$th vertex and ends in $j$th.

(2) Denote by $w_{i\ra}^{(m)}$ the way that starts in $i$th vertex and has length $m$.

(3) Denote by $w_{\ra i}^{(m)}$ the way that ends in $i$th vertex and has length $m$.
\end{obozns}

Define the helper functions $g\text{: }\Z\rightarrow\Z$ $f_0\text{:
}\Z\times\Z\rightarrow\Z$ and $f_1\text{: }\Z\times\Z\rightarrow\Z$,
which act in the following way:
$$
g(j)=\begin{cases}1,\quad (j)_{2s}<s;\\3\quad\text{otherwise},\end{cases}\quad
f_0(x,y)=\begin{cases}1,\quad x<y;\\0\quad x\ge y,\end{cases}\quad f_1(x,y)=\begin{cases}1,\quad
x<y;\\-1\quad x\ge y.\end{cases}
$$
\centerline{\bf Description of the $d_0$} \centerline{$d_0:Q_1\rightarrow Q_0\text{ -- is an }
(7s\times 6s)\text{ matrix}.$}

If $0\le j<2s$, then
$$(d_0)_{ij}=\begin{cases}
w_{(j+m)n\ra (j+m)n+g(j)}\otimes e_{jn},\quad i=(j)_s;\\
-e_{(j+m)n+g(j)}\otimes w_{jn\ra jn+g(j)},\quad i=j+s;\\
0\quad\text{otherwise.}\end{cases}$$

If $2s\le j<4s$, then $$(d_0)_{ij}=\begin{cases}
w_{(j+m)n+g(j)\ra (j+m)n+g(j)+1}\otimes e_{jn+g(j)},\quad i=j-s;\\
-e_{(j+m)n+g(j)+1}\otimes w_{jn+g(j)\ra jn+g(j)+1},\quad i=j+s;\\
0\quad\text{otherwise.}\end{cases}$$

If $4s\le j<6s$, then $$(d_0)_{ij}=\begin{cases}
w_{(j+m)n+g(j)+1\ra (j+m)n+5}\otimes e_{jn+g(j)+1},\quad i=j-s;\\
-e_{(j+m)n+5}\otimes w_{jn+g(j)+1\ra jn+5},\quad i=5s+(j)_s;\\
0\quad\text{otherwise.}\end{cases}$$

If $6s\le j<7s$, then $$(d_0)_{ij}=\begin{cases}
w_{(j+m)n+5\ra (j+m)n+6}\otimes e_{jn+5},\quad i=j-s;\\
-e_{(j+m)n+6}\otimes w_{jn+5\ra (j+1)n},\quad i=(j+1)_s;\\
0\quad\text{otherwise.}\end{cases}$$

\centerline{\bf Description of the $d_1$} \centerline{$d_1:Q_2\rightarrow Q_1\text{ -- is an }
(6s\times 7s)\text{ matrix}.$}

If $0\le j<s$, then
$$(d_1)_{ij}=\begin{cases}
w_{jn+1+j_1+2f(j_1,2)\ra jn+5}\otimes w_{jn\ra jn+j_1},\quad  i=j+2sj_1,\text{ }0\le j_1< 3;\\
-w_{jn+3+j_1\ra jn+5}\otimes w_{jn\ra jn+j_1+2(1-f(j_1,0))},\quad  i=j+2sj_1+s,\text{ }0\le j_1< 3;\\
0\quad\text{otherwise.}\end{cases}$$

If $s\le j<3s$, then $$(d_1)_{ij}=\begin{cases}
w_{\ra (j+1)n+1+g(j+s)}^{(4-j_1)}\otimes w_{jn+g(j+s)\ra}^{(j_1)},\\\quad\quad\quad  i=j+2sj_1+s-s(1-f_0(j,2s))f(j_1,2),\text{ }0\le j_1< 3;\\
w_{\ra (j+1)n+1+g(j+s)}^{(1-j_1)}\otimes w_{jn+g(j+s)\ra}^{(j_1+3)},\\\quad\quad\quad  i=(j+1)_s+2sj_1+s(1-f_0(j,2s)),\text{ }0\le j_1< 2;\\
0\quad\text{otherwise.}\end{cases}$$

If $3s\le j<5s$, then $$(d_1)_{ij}=\begin{cases}
w_{\ra (j+1)n+g(j)}^{(2)}\otimes e_{jn+g(j+s)+1},\quad i=j+s;\\
w_{\ra (j+1)n+g(j)}^{(1)}\otimes w_{jn+g(j+s)+1\ra}^{(1)},\quad i=(j)_s+6s;\\
e_{(j+1)n+g(j)}\otimes w_{jn+g(j+s)+1\ra}^{(2)},\quad i=(j+1)_s+sf_0(j,4s);\\
0\quad\text{otherwise.}\end{cases}$$

If $5s\le j<6s$, then $$(d_1)_{ij}=\begin{cases}
w_{(j+1)n\ra (j+2)n}\otimes e_{jn+5},\quad i=j+s;\\
w_{(j+1)n+1+j_1+2f(j_1,3)+2f(j_1,2)\ra (j+2)n}\otimes w_{jn+5\ra (j+1)n+j_1+2f(j_1,3)},\\\quad\quad\quad  i=(j+1)_s+2sj_1,\text{ }0\le j_1< 4;\\
0\quad\text{otherwise.}\end{cases}$$

\centerline{\bf Description of the $d_2$} \centerline{$d_2:Q_3\rightarrow Q_2\text{ -- is an }
(8s\times 6s)\text{ matrix}.$}

If $0\le j<2s$, then
$$(d_2)_{ij}=\begin{cases}
w_{(j+m-1)n+5\ra (j+m)n+g(j)+1}\otimes e_{jn},\quad i=(j)_s;\\
-f_1(j,s)e_{(j+m)n+g(j)+1}\otimes w_{jn\ra jn+g(j)},\quad i=j+s;\\
f_1(j,s)w_{(j+m)n+g(j)\ra (j+m)n+g(j)+1}\otimes w_{jn\ra jn+1+g(j+s)},\quad  i=(j+s)_{2s}+3s;\\
0\quad\text{otherwise.}\end{cases}$$

If $2s\le j<4s$, then $$(d_2)_{ij}=\begin{cases}
w_{(j+m)n+g(j)+1\ra (j+m)n+5}\otimes e_{jn+g(j)},\quad i=j-s;\\
-w_{(j+m)n+g(j+s)\ra (j+m)n+5}\otimes w_{jn+g(j)\ra jn+g(j)+1},\quad i=j+s;\\
-f_1(j,3s)e_{(j+m)n+5}\otimes w_{jn+g(j)\ra (j+1)n},\quad i=(j+1)_s;\\
0\quad\text{otherwise.}\end{cases}$$

If $4s\le j<6s$, then $$(d_2)_{ij}=\begin{cases}
w_{(j+m)n+g(j+s)\ra (j+m+1)n}\otimes e_{jn+g(j)+1},\quad i=j-s;\\
-e_{(j+m+1)n}\otimes w_{jn+g(j)+1\ra jn+5},\quad i=5s+(j)_s;\\
w_{(j+m)n+5\ra (j+m+1)n}\otimes w_{jn+g(j)+1\ra (j+1)n},\quad i=(j+1)_s,\text{ }j<5s;\\
0\quad\text{otherwise.}\end{cases}$$

If $6s\le j<8s$, then $$(d_2)_{ij}=\begin{cases}
w_{(j+m+1)n\ra (j+m+1)n+g(j)}\otimes e_{jn+5},\quad i=5s+(j)_s;\\
-w_{(j+m+1)n-1\ra (j+m+1)n+g(j)}\otimes w_{jn+5\ra (j+1)n},\quad i=(j+1)_s,\text{ }j<7s;\\
-e_{(j+m+1)n+g(j)}\otimes w_{jn+5\ra (j+1)n+g(j+s)+1},\quad i=3s+(j+1)_s+sf_0(j,7s);\\
0\quad\text{otherwise.}\end{cases}$$

\centerline{\bf Description of the $d_3$} \centerline{$d_3:Q_4\rightarrow Q_3\text{ -- is an }
(9s\times 8s)\text{ matrix}.$}

If $0\le j<s$, then
$$(d_3)_{ij}=\begin{cases}
w_{(j+m)n+2\ra (j+m)n+5}\otimes e_{jn},\quad i=j;\\
-w_{(j+m)n+4\ra (j+m)n+5}\otimes e_{jn},\quad i=j+s;\\
e_{(j+m)n+5}\otimes w_{jn\ra jn+1},\quad i=j+2s;\\
e_{(j+m)n+5}\otimes w_{jn\ra jn+3},\quad i=j+3s;\\
0\quad\text{otherwise.}\end{cases}$$

If $s\le j<2s$, then $$(d_3)_{ij}=\begin{cases}
w_{(j+m)n+2+j_1+2(1-f(j_1,0))\ra (j+m+1)n}\otimes w_{jn\ra jn+j_1},\quad  i=j-s+2sj_1,\text{ }0\le j_1< 3;\\
-e_{(j+m+1)n}\otimes w_{jn\ra jn+4},\quad i=j+4s;\\
0\quad\text{otherwise.}\end{cases}$$

If $2s\le j<4s$, then $$(d_3)_{ij}=\begin{cases}
w_{\ra (j+m+1)n+g(j+s)}^{(2-j_1)}\otimes w_{jn+g(j)\ra}^{(j_1)},\quad  i=2j_1s+j,\text{ }0\le j_1< 2;\\
e_{(j+m+1)n+g(j+s)}\otimes w_{jn+g(j)\ra}^{(2)},\quad i=6s+(j+s)_{2s};\\
0\quad\text{otherwise.}\end{cases}$$

If $4s\le j<6s$, then $$(d_3)_{ij}=\begin{cases}
w_{\ra (j+m+1)n+g(j)}^{(1-j_1)}\otimes w_{jn+g(j)+1\ra}^{(j_1)},\quad  i=j+2sj_1,\text{ }0\le j_1< 2;\\
0\quad\text{otherwise.}\end{cases}$$

If $6s\le j<8s$, then $$(d_3)_{ij}=\begin{cases}
w_{\ra (j+m+1)n+g(j+s)+1}^{(2)}\otimes e_{jn+g(j)+1},\quad i=j-2s;\\
w_{\ra (j+m+1)n+g(j+s)+1}^{(1)}\otimes w_{jn+g(j)+1\ra}^{(1)},\quad i=(j+s)_{2s}+6s;\\
-f_1(j,7s)e_{(j+m+1)n+g(j+s)+1}\otimes w_{jn+g(j)+1\ra}^{(2)},\quad  i=(j+1)_s+sf_0(j,7s);\\
0\quad\text{otherwise.}\end{cases}$$

If $8s\le j<9s$, then $$(d_3)_{ij}=\begin{cases}
w_{(j+m+1)n+2\ra (j+m+1)n+5}\otimes w_{jn+5\ra (j+1)n},\quad i=(j+1)_s;\\
e_{(j+m+1)n+5}\otimes w_{jn+5\ra (j+1)n+1},\quad i=2s+(j+1)_s;\\
w_{(j+m+1)n+1\ra (j+m+1)n+5}\otimes e_{jn+5},\quad i=j-2s;\\
-w_{(j+m+1)n+3\ra (j+m+1)n+5}\otimes e_{jn+5},\quad i=j-s;\\
0\quad\text{otherwise.}\end{cases}$$

\centerline{\bf Description of the $d_4$} \centerline{$d_4:Q_5\rightarrow Q_4\text{ -- is an }
(8s\times 9s)\text{ matrix}.$}

If $0\le j<2s$, then
$$(d_4)_{ij}=\begin{cases}
w_{(j+m-1)n+5\ra (j+m)n+1}\otimes e_{jn},\quad i=j,\text{ }j<s;\\
-f_1(j,s)w_{(j+m)n\ra (j+m)n+g(j)}\otimes e_{jn},\quad i=(j)_s+s;\\
-e_{(j+m)n+g(j)}\otimes w_{jn\ra jn+g(j+s)},\quad i=(j+s)_{2s}+2s;\\
e_{(j+m)n+g(j)}\otimes w_{jn\ra jn+1+g(j)},\quad i=j+4s;\\
0\quad\text{otherwise.}\end{cases}$$

If $2s\le j<4s$, then $$(d_4)_{ij}=\begin{cases}
w_{(j+m)n+5\ra (j+m+1)n}\otimes w_{jn+1\ra (j+1)n},\quad i=(j+1)_s,\text{ }j<3s;\\
w_{(j+m)n+g(j+s)\ra (j+m+1)n}\otimes e_{jn+g(j)},\quad i=j;\\
-w_{(j+m)n+g(j+s)+1\ra (j+m+1)n}\otimes w_{jn+g(j)\ra jn+g(j)+1},\quad i=4s+j;\\
-f_1(j,3s)e_{(j+m+1)n}\otimes w_{jn+g(j)\ra (j+1)n},\quad i=(j+1)_s+s;\\
0\quad\text{otherwise.}\end{cases}$$

If $4s\le j<6s$, then $$(d_4)_{ij}=\begin{cases}
e_{(j+m)n+5}\otimes w_{jn+g(j)+1\ra (j+1)n},\quad i=(j+1)_s,\text{ }j<5s;\\
w_{(j+m)n+g(j)\ra (j+m)n+5}\otimes e_{jn+g(j)+1},\quad i=j;\\
-w_{(j+m)n+g(j+s)+1\ra (j+m)n+5}\otimes e_{jn+g(j)+1},\quad i=j+2s;\\
-f_1(j,5s)e_{(j+m)n+5}\otimes w_{jn+g(j)+1\ra jn+5},\quad i=8s+(j)_s;\\
0\quad\text{otherwise.}\end{cases}$$

If $6s\le j<8s$, then $$(d_4)_{ij}=\begin{cases}
-w_{(j+m+1)n-1\ra (j+m+1)n+g(j)+1}\otimes w_{jn+5\ra (j+1)n},\quad i=(j+1)_s,\text{ }j<7s;\\
w_{(j+m+1)n-1\ra (j+m+1)n+g(j)+1}\otimes e_{jn+5},\quad i=8s+(j)_s;\\
f_1(j,7s)w_{(j+m+1)n+g(j)\ra (j+m+1)n+g(j)+1}\otimes w_{jn+5\ra (j+1)n+g(j+s)},\\\quad\quad\quad  i=2s+(j+1)_s+sf_0(j,7s);\\
-f_1(j,7s)e_{(j+m+1)n+g(j)+1}\otimes w_{jn+5\ra (j+1)n+g(j+s)+1},\\\quad\quad\quad  i=6s+(j+1)_s+sf_0(j,7s);\\
0\quad\text{otherwise.}\end{cases}$$

\centerline{\bf Description of the $d_5$} \centerline{$d_5:Q_6\rightarrow Q_5\text{ -- is an }
(9s\times 8s)\text{ matrix}.$}

If $0\le j<s$, then
$$(d_5)_{ij}=\begin{cases}
w_{(j+m)n+1\ra (j+m+1)n}\otimes e_{jn},\quad i=j;\\
w_{(j+m)n+3\ra (j+m+1)n}\otimes e_{jn},\quad i=j+s;\\
(2f(j_1,0)-1)w_{(j+m+1)n-j_1\ra (j+m+1)n}\otimes w_{jn\ra jn+1+j_1},\\\quad\quad\quad  i=j+2s(1+j_1),\text{ }0\le j_1< 2;\\
(2f(j_1,0)-1)w_{(j+m+1)n-j_1\ra (j+m+1)n}\otimes w_{jn\ra jn+3+j_1},\\\quad\quad\quad  i=j+2s(1+j_1)+s,\text{ }0\le j_1< 2;\\
0\quad\text{otherwise.}\end{cases}$$

If $s\le j<3s$, then $$(d_5)_{ij}=\begin{cases}
w_{\ra (j+m+1)n+g(j+s)}^{(1)}\otimes e_{jn+g(j+s)},\quad i=j+s;\\
-e_{(j+m+1)n+g(j+s)}\otimes w_{jn+g(j+s)\ra}^{(3)},\quad i=(j+1)_s+s(1-f_0(j,2s));\\
0\quad\text{otherwise.}\end{cases}$$

If $3s\le j<5s$, then $$(d_5)_{ij}=\begin{cases}
-w_{\ra (j+m+1)n+g(j)+1}^{(3)}\otimes w_{jn+g(j+s)\ra}^{(1)},\quad i=j+s;\\
w_{\ra (j+m+1)n+g(j)+1}^{(2)}\otimes e_{jn+g(j+s)},\quad i=j-s;\\
w_{\ra (j+m+1)n+g(j)+1}^{(1)}\otimes w_{jn+g(j+s)\ra}^{(3)},\quad i=(j+1)_s+sf_0(j,4s);\\
-f_1(j,4s)e_{(j+m+1)n+g(j)+1}\otimes w_{jn+g(j+s)\ra}^{(2)},\quad i=(j)_{2s}+6s;\\
0\quad\text{otherwise.}\end{cases}$$

If $5s\le j<7s$, then $$(d_5)_{ij}=\begin{cases}
f_1(j,6s)e_{(j+m+1)n+g(j+s)+1}\otimes w_{jn+g(j+s)+1\ra}^{(1)},\quad i=j+s;\\
w_{\ra (j+m+1)n+g(j+s)+1}^{(3)}\otimes e_{jn+g(j+s)+1},\quad i=j-s;\\
0\quad\text{otherwise.}\end{cases}$$

If $7s\le j<8s$, then $$(d_5)_{ij}=\begin{cases}
w_{(j+m+1)n+1\ra (j+m+1)n+5}\otimes w_{jn+5\ra (j+1)n},\quad i=(j+1)_s;\\
w_{(j+m+1)n+3\ra (j+m+1)n+5}\otimes w_{jn+5\ra (j+1)n},\quad i=s+(j+1)_s;\\
-e_{(j+m+1)n+5}\otimes w_{jn+5\ra (j+1)n+2},\quad i=4s+(j+1)_s;\\
-e_{(j+m+1)n+5}\otimes w_{jn+5\ra (j+1)n+4},\quad i=5s+(j+1)_s;\\
w_{(j+m+1)n+2\ra (j+m+1)n+5}\otimes e_{jn+5},\quad i=j-s;\\
-w_{(j+m+1)n+4\ra (j+m+1)n+5}\otimes e_{jn+5},\quad i=j;\\
0\quad\text{otherwise.}\end{cases}$$

If $8s\le j<9s$, then $$(d_5)_{ij}=\begin{cases}
w_{(j+m+1)n+2\ra (j+m+2)n}\otimes e_{jn+5},\quad i=j-2s;\\
-e_{(j+m+2)n}\otimes w_{jn+5\ra (j+1)n+3},\quad i=3s+(j+1)_s;\\
0\quad\text{otherwise.}\end{cases}$$

\centerline{\bf Description of the $d_6$} \centerline{$d_6:Q_7\rightarrow Q_6\text{ -- is an }
(8s\times 9s)\text{ matrix}.$}

If $0\le j<2s$, then
$$(d_6)_{ij}=\begin{cases}
w_{(j+m)n\ra (j+m)n+g(j)+1}\otimes e_{jn},\quad i=(j)_s;\\
-w_{(j+m)n+g(j)\ra (j+m)n+g(j)+1}\otimes w_{jn\ra jn+g(j)},\quad i=j+s;\\
-e_{(j+m)n+g(j)+1}\otimes w_{jn\ra jn+g(j+s)},\quad i=(j+s)_{2s}+3s;\\
e_{(j+m)n+g(j)+1}\otimes w_{jn\ra jn+g(j)+1},\quad i=j+5s;\\
0\quad\text{otherwise.}\end{cases}$$

If $2s\le j<4s$, then $$(d_6)_{ij}=\begin{cases}
w_{(j+m)n+g(j)\ra (j+m)n+5}\otimes e_{jn+g(j)},\quad i=j-s;\\
-w_{(j+m)n+g(j+s)+1\ra (j+m)n+5}\otimes e_{jn+g(j)},\quad i=j+s;\\
-w_{(j+m)n+g(j)+1\ra (j+m)n+5}\otimes w_{jn+g(j)\ra jn+g(j)+1},\quad i=j+3s;\\
e_{(j+m)n+5}\otimes w_{jn+g(j)\ra jn+5},\quad i=(j)_s+7s;\\
0\quad\text{otherwise.}\end{cases}$$

If $4s\le j<6s$, then $$(d_6)_{ij}=\begin{cases}
w_{(j+m)n+g(j)+1\ra (j+m+1)n}\otimes e_{jn+g(j)+1},\quad i=j+s;\\
-f_1(j,5s)e_{(j+m+1)n}\otimes w_{jn+g(j)+1\ra jn+5},\quad i=(j)_s+8s;\\
e_{(j+m+1)n}\otimes w_{jn+4\ra (j+1)n},\quad i=(j+1)_s,\text{ }j\ge  5s;\\
-w_{(j+m)n+5\ra (j+m+1)n}\otimes w_{jn+4\ra jn+5},\quad i=j+2s,\text{ }j\ge  5s;\\
0\quad\text{otherwise.}\end{cases}$$

If $6s\le j<8s$, then $$(d_6)_{ij}=\begin{cases}
w_{(j+m)n+5\ra (j+m+1)n+1}\otimes e_{jn+5},\quad i=j+s,\text{ }j<7s;\\
-w_{(j+m+1)n\ra (j+m+1)n+1}\otimes w_{jn+5\ra (j+1)n},\quad i=(j+1)_s,\text{ }j<7s;\\
-f_1(j,7s)w_{(j+m+1)n\ra (j+m+1)n+g(j)}\otimes e_{jn+5},\quad i=(j)_s+8s;\\
e_{(j+m+1)n+g(j)}\otimes w_{jn+5\ra (j+1)n+g(j)},\quad i=(j+1)_s+s+s(1-f_0(j,7s));\\
0\quad\text{otherwise.}\end{cases}$$

\centerline{\bf Description of the $d_7$} \centerline{$d_7:Q_8\rightarrow Q_7\text{ -- is an }
(6s\times 8s)\text{ matrix}.$}

If $0\le j<s$, then
$$(d_7)_{ij}=\begin{cases}
w_{(j+m)n+2\ra (j+m)n+5}\otimes e_{jn},\quad i=j;\\
-w_{(j+m)n+4\ra (j+m)n+5}\otimes e_{jn},\quad i=j+s;\\
e_{(j+m)n+5}\otimes w_{jn\ra jn+1},\quad i=j+2s;\\
-e_{(j+m)n+5}\otimes w_{jn\ra jn+3},\quad i=j+3s;\\
0\quad\text{otherwise.}\end{cases}$$

If $s\le j<3s$, then $$(d_7)_{ij}=\begin{cases}
-e_{(j+m+1)n+g(j+s)+1}\otimes w_{jn+g(j+s)\ra}^{(3)},\quad i=(j+1)_s+s(1-f_0(j,2s));\\
w_{\ra (j+m+1)n+g(j+s)+1}^{(3)}\otimes e_{jn+g(j+s)},\quad i=j+s;\\
w_{\ra (j+m+1)n+g(j+s)+1}^{(2)}\otimes w_{jn+g(j+s)\ra}^{(1)},\quad i=j+3s;\\
-w_{\ra (j+m+1)n+g(j+s)+1}^{(1)}\otimes w_{jn+g(j+s)\ra}^{(2)},\quad i=j+5s;\\
0\quad\text{otherwise.}\end{cases}$$

If $3s\le j<5s$, then $$(d_7)_{ij}=\begin{cases}
w_{\ra (j+m+1)n+g(j)}^{(1)}\otimes e_{jn+g(j+s)+1},\quad i=j+s;\\
e_{(j+m+1)n+g(j)}\otimes w_{jn+g(j+s)+1\ra}^{(1)},\quad i=(j)_{2s}+6s;\\
0\quad\text{otherwise.}\end{cases}$$

If $5s\le j<6s$, then $$(d_7)_{ij}=\begin{cases}
w_{(j+m+1)n+4\ra (j+m+2)n}\otimes w_{jn+5\ra (j+1)n},\quad i=(j+1)_s+s;\\
-w_{(j+m+1)n+5\ra (j+m+2)n}\otimes w_{jn+5\ra (j+1)n+1},\quad i=(j+1)_s+2s;\\
-e_{(j+m+2)n}\otimes w_{jn+5\ra (j+1)n+2},\quad i=(j+1)_s+4s;\\
-e_{(j+m+2)n}\otimes w_{jn+5\ra (j+1)n+4},\quad i=(j+1)_s+5s;\\
w_{(j+m+1)n+1\ra (j+m+2)n}\otimes e_{jn+5},\quad i=j+s;\\
w_{(j+m+1)n+3\ra (j+m+2)n}\otimes e_{jn+5},\quad i=j+2s;\\
0\quad\text{otherwise.}\end{cases}$$

\centerline{\bf Description of the $d_8$} \centerline{$d_8:Q_9\rightarrow Q_8\text{ -- is an }
(7s\times 6s)\text{ matrix}.$}

If $0\le j<s$, then
$$(d_8)_{ij}=\begin{cases}
w_{(j+m-1)n+5\ra (j+m)n+5}\otimes e_{jn},\quad i=j;\\
-e_{(j+m)n+5}\otimes w_{jn\ra (j+1)n},\quad i=(j+1)_s;\\
-w_{(j+m)n+2\ra (j+m)n+5}\otimes w_{jn\ra jn+1},\quad i=j+s;\\
w_{(j+m)n+4\ra (j+m)n+5}\otimes w_{jn\ra jn+3},\quad i=j+2s;\\
w_{(j+m)n+3\ra (j+m)n+5}\otimes w_{jn\ra jn+2},\quad i=j+3s;\\
-w_{(j+m)n+1\ra (j+m)n+5}\otimes w_{jn\ra jn+4},\quad i=j+4s;\\
0\quad\text{otherwise.}\end{cases}$$

If $s\le j<3s$, then $$(d_8)_{ij}=\begin{cases}
w_{(j+m)n+5\ra (j+m+1)n}\otimes w_{jn+g(j+s)\ra (j+1)n},\quad i=(j+1)_s,\text{ }j<2s;\\
w_{(j+m)n+g(j+s)+1\ra (j+m+1)n}\otimes e_{jn+g(j+s)},\quad i=j;\\
-w_{(j+m)n+g(j)\ra (j+m+1)n}\otimes w_{jn+g(j+s)\ra jn+g(j+s)+1},\quad i=j+2s;\\
e_{(j+m+1)n}\otimes w_{jn+g(j+s)\ra jn+5},\quad i=(j)_s+5s;\\
0\quad\text{otherwise.}\end{cases}$$

If $3s\le j<5s$, then $$(d_8)_{ij}=\begin{cases}
-w_{(j+m)n+5\ra (j+m+1)n+g(j)}\otimes w_{jn+g(j+s)+1\ra (j+1)n},\quad i=(j+1)_s,\text{ }j<4s;\\
w_{(j+m)n+g(j)\ra (j+m+1)n+g(j)}\otimes e_{jn+g(j+s)+1},\quad i=j;\\
-e_{(j+m+1)n+g(j)}\otimes w_{jn+g(j+s)+1\ra (j+1)n+g(j+s)+1},\\\quad\quad\quad  i=(j+1)_s+3s+s(1-f_0(j,4s));\\
-w_{(j+m+1)n\ra (j+m+1)n+g(j)}\otimes w_{jn+g(j+s)+1\ra jn+5},\\\quad\quad\quad i=(j)_s+5s;\\
0\quad\text{otherwise.}\end{cases}$$

If $5s\le j<7s$, then $$(d_8)_{ij}=\begin{cases}
w_{(j+m+1)n\ra (j+m+1)n+g(j+s)+1}\otimes e_{jn+5},\quad i=(j)_s+5s;\\
w_{(j+m)n+5\ra (j+m+1)n+g(j+s)+1}\otimes w_{jn+5\ra (j+1)n},\quad  i=(j+1)_s,\text{ }j\ge 6s;\\
e_{(j+m+1)n+g(j+s)+1}\otimes w_{jn+5\ra (j+1)n+g(j+s)},\\\quad\quad\quad i=(j+1)_s+s+s(1-f_0(j,6s));\\
w_{(j+m+1)n+g(j+s)\ra (j+m+1)n+g(j+s)+1}\otimes w_{jn+5\ra (j+1)n+g(j)+1},\\\quad\quad\quad  i=(j+1)_s+3s+sf_0(j,6s);\\
0\quad\text{otherwise.}\end{cases}$$

\centerline{\bf Description of the $d_9$} \centerline{$d_9:Q_{10}\rightarrow Q_9\text{ -- is an }
(6s\times 7s)\text{ matrix}.$}

If $0\le j<s$, then
$$(d_9)_{ij}=\begin{cases}
w_{(j+m)n+5\ra (j+m+1)n}\otimes e_{jn},\quad i=j;\\
e_{(j+m+1)n}\otimes w_{jn\ra jn+1},\quad i=j+s;\\
-e_{(j+m+1)n}\otimes w_{jn\ra jn+3},\quad i=j+2s;\\
0\quad\text{otherwise.}\end{cases}$$

If $s\le j<3s$, then $$(d_9)_{ij}=\begin{cases}
w_{\ra (j+m+1)n+g(j)}^{(1)}\otimes e_{jn+g(j+s)},\quad i=j;\\
e_{(j+m+1)n+g(j)}\otimes w_{jn+g(j+s)\ra}^{(1)},\quad i=j+2s;\\
0\quad\text{otherwise.}\end{cases}$$

If $3s\le j<5s$, then $$(d_9)_{ij}=\begin{cases}
w_{\ra (j+m+1)n+g(j)+1}^{(1)}\otimes e_{jn+g(j+s)+1},\quad i=j;\\
e_{(j+m+1)n+g(j)+1}\otimes w_{jn+g(j+s)+1\ra}^{(1)},\quad i=(j)_{2s}+5s;\\
0\quad\text{otherwise.}\end{cases}$$

If $5s\le j<6s$, then $$(d_9)_{ij}=\begin{cases}
e_{(j+m+1)n+5}\otimes w_{jn+5\ra (j+1)n},\quad i=(j+1)_s;\\
w_{(j+m+1)n+2\ra (j+m+1)n+5}\otimes e_{jn+5},\quad i=j;\\
-w_{(j+m+1)n+4\ra (j+m+1)n+5}\otimes e_{jn+5},\quad i=j+s;\\
0\quad\text{otherwise.}\end{cases}$$

\centerline{\bf Description of the $d_{10}$} \centerline{$d_{10}:Q_{11}\rightarrow Q_{10}\text{ --
is an } (6s\times 6s)\text{ matrix}.$}

If $0\le j<s$, then
$$(d_{10})_{ij}=\begin{cases}
w_{(j+m)n\ra (j+m+1)n}\otimes e_{jn},\quad i=j;\\
-e_{(j+m+1)n}\otimes w_{jn\ra (j+1)n},\quad i=(j+1)_s;\\
-w_{(j+m)n+3\ra (j+m+1)n}\otimes w_{jn\ra jn+1},\quad i=j+s;\\
w_{(j+m)n+1\ra (j+m+1)n}\otimes w_{jn\ra jn+3},\quad i=j+2s;\\
w_{(j+m)n+4\ra (j+m+1)n}\otimes w_{jn\ra jn+2},\quad i=j+3s;\\
-w_{(j+m)n+2\ra (j+m+1)n}\otimes w_{jn\ra jn+4},\quad i=j+4s;\\
w_{(j+m)n+5\ra (j+m+1)n}\otimes w_{jn\ra jn+5},\quad i=j+5s;\\
0\quad\text{otherwise.}\end{cases}$$

If $s\le j<3s$, then $$(d_{10})_{ij}=\begin{cases}
f_1(j,2s)w_{(j+m+1)n\ra (j+m+1)n+g(j)}\otimes w_{jn+g(j+s)\ra (j+1)n},\quad i=(j+1)_s;\\
w_{(j+m)n+g(j)\ra (j+m+1)n+g(j)}\otimes e_{jn+g(j+s)},\quad i=j;\\
-e_{(j+m+1)n+g(j)}\otimes w_{jn+g(j+s)\ra (j+1)n+g(j+s)},\quad i=(j+1)_s+s+s(1-f_0(j,2s));\\
-w_{(j+m)n+g(j)+1\ra (j+m+1)n+g(j)}\otimes w_{jn+g(j+s)\ra jn+g(j+s)+1},\quad i=j+2s;\\
-f_1(j,2s)w_{(j+m)n+5\ra (j+m+1)n+g(j)}\otimes w_{jn+g(j+s)\ra jn+5},\quad i=(j)_s+5s;\\
0\quad\text{otherwise.}\end{cases}$$

If $3s\le j<5s$, then $$(d_{10})_{ij}=\begin{cases}
-f_1(j,4s)w_{(j+m+1)n\ra (j+m+1)n+g(j)+1}\otimes w_{jn+g(j+s)+1\ra (j+1)n},\quad i=(j+1)_s;\\
w_{(j+m)n+g(j)+1\ra (j+m+1)n+g(j)+1}\otimes e_{jn+g(j+s)+1},\quad i=j;\\
w_{(j+m+1)n+g(j)\ra (j+m+1)n+g(j)+1}\otimes w_{jn+g(j+s)+1\ra (j+1)n+g(j+s)},\\\quad\quad\quad  i=(j+1)_s+s+s(1-f_0(j,4s));\\
-e_{(j+m+1)n+g(j)+1}\otimes w_{jn+g(j+s)+1\ra (j+1)n+g(j+s)+1},\\\quad\quad\quad  i=(j+1)_s+3s+s(1-f_0(j,4s));\\
f_1(j,4s)w_{(j+m)n+5\ra (j+m+1)n+g(j)+1}\otimes w_{jn+g(j+s)+1\ra jn+5},\quad i=(j)_s+5s;\\
0\quad\text{otherwise.}\end{cases}$$

If $5s\le j<6s$, then $$(d_{10})_{ij}=\begin{cases}
-w_{(j+m+1)n\ra (j+m+1)n+5}\otimes w_{jn+5\ra (j+1)n},\quad i=(j+1)_s;\\
w_{(j+m)n+5\ra (j+m+1)n+5}\otimes e_{jn+5},\quad i=j;\\
w_{(j+m+1)n+3\ra (j+m+1)n+5}\otimes w_{jn+5\ra (j+1)n+1},\quad i=(j+1)_s+s;\\
-w_{(j+m+1)n+1\ra (j+m+1)n+5}\otimes w_{jn+5\ra (j+1)n+3},\quad i=(j+1)_s+2s;\\
-w_{(j+m+1)n+4\ra (j+m+1)n+5}\otimes w_{jn+5\ra (j+1)n+2},\quad i=(j+1)_s+3s;\\
w_{(j+m+1)n+2\ra (j+m+1)n+5}\otimes w_{jn+5\ra (j+1)n+4},\quad i=(j+1)_s+4s;\\
-e_{(j+m+1)n+5}\otimes w_{jn+5\ra (j+1)n+5},\quad i=(j+1)_s+5s;\\
0\quad\text{otherwise.}\end{cases}$$

\begin{thm}\label{resol_thm}
Let $R=R_s$ is algebra of the type $E_6$. Then the minimal projective resolution of the
$\Lambda$-module $R$ is of the form:
\begin{equation}\label{resolv}\tag{$+$} \dots\longrightarrow
Q_3\stackrel{d_2}\longrightarrow Q_2\stackrel{d_1}\longrightarrow
Q_1\stackrel{d_0}\longrightarrow
Q_0\stackrel\varepsilon\longrightarrow R\longrightarrow
0,
\end{equation}
where $\varepsilon$ is the multiplication map $(\varepsilon(a\otimes b)=ab)$; $Q_r\text{ }(r\le
10)$ and $d_r\text{ }(r\le 10)$ were described before; further $Q_{11\ell+r}$, where $\ell\in \N$
and $0\le r\le 10$, is obtained from $Q_r$ by replacing every direct summand $P_{i,j}$ to
$P_{\sigma^\ell(i),j}$ correspondingly $($here $\sigma(i)=j$, if $\sigma(e_i)=e_j)$, and the
differential $d_{11\ell+r}$ is obtained from $d_r$ by act of $\sigma^\ell$ by all left tensor
components of the corresponding matrix.
\end{thm}

To prove that the terms $Q_i$ are of this form we introduce $P_i=Re_i$ is the projective cover of
the simple $R$-modules $S_i$, corresponding to the vertices of the quiver $\mathcal Q_s$. We will
find projective resolutions of the simple $R$-modules $S_i$.

\begin{obozn}
For $R$-module $M$ its $m$th syzygy is denoted by $\Omega^m(M)$.
\end{obozn}

\begin{zam}\label{note_brev}
From here we denote the multiplication homomorphism from the right by an element $w$ by $w$.
\end{zam}

\begin{lem}\label{lem_s0}
The begin of the minimal projective resolution of $S_{rn}$ is of the form
\begin{multline*}
\dots\longrightarrow P_{(r+3)n+5}
\stackrel{\binom{\a}{-\b}}\longrightarrow P_{(r+3)n+2}\oplus
P_{(r+3)n+4}
\stackrel{(\a^2\text{ }\b^2)}\longrightarrow\\
\longrightarrow P_{(r+3)n}
\stackrel{\binom{\g\a^2}{\g\b^2}}\longrightarrow P_{(r+2)n+1}\oplus
P_{(r+2)n+3} \stackrel{\binom{\a\g\phantom{-}0}{-\a\phantom{-}\b}}
\longrightarrow\\\longrightarrow P_{(r+1)n+5}\oplus P_{(r+2)n}
\stackrel{\binom{\phantom{-}\a\phantom{-}\g\a}{-\b\phantom{-}0}}\longrightarrow
P_{(r+1)n+2}\oplus P_{(r+1)n+4} \stackrel{(\a^2\g\text{
}\b^2\g)}\longrightarrow\\\longrightarrow
P_{rn+5}\stackrel{\binom{\a^2}{-\b^2}}\longrightarrow P_{rn+1}\oplus
P_{rn+3} \stackrel{(\a\text{ }\b)}\longrightarrow
P_{rn}\longrightarrow S_{rn}\longrightarrow 0.
\end{multline*}
At that $\Omega^{9}(S_{rn})\simeq S_{(r+4)n+5}$.
\end{lem}

\begin{lem}
The begin of the minimal projective resolution of $S_{rn+1}$ is of the form $$\dots\longrightarrow
P_{rn+2} \stackrel{\a}\longrightarrow P_{rn+1}\longrightarrow S_{rn+1}\longrightarrow 0.$$ At that
$\Omega^{2}(S_{rn+1})\simeq S_{(r+1)n+2}$.
\end{lem}

\begin{lem}
The begin of the minimal projective resolution of $S_{rn+2}$ is of the form
\begin{multline*}
\dots\longrightarrow P_{(r+4)n+3} \stackrel{\b}\longrightarrow
P_{(r+4)n}\stackrel{\g\a}\longrightarrow
P_{(r+3)n+2}\stackrel{\a^2\g}\longrightarrow
P_{(r+2)n+5}\stackrel{\binom{\a^2}{-\b}}\longrightarrow\\\longrightarrow
P_{(r+2)n+1}\oplus P_{(r+2)n+4}\stackrel{(\a\text{
}\b^2)}\longrightarrow
P_{(r+2)n}\stackrel{\g\b^2}\longrightarrow\\
\longrightarrow
P_{(r+1)n+3}\stackrel{\b\g}\longrightarrow
P_{rn+5}\stackrel{\a}\longrightarrow P_{rn+2}\longrightarrow
S_{rn+2}\longrightarrow 0.
\end{multline*}
At that $\Omega^{9}(S_{rn+2})\simeq S_{(r+5)n+3}$.
\end{lem}

\begin{lem}
The begin of the minimal projective resolution of $S_{rn+3}$ is of the form $$\dots\longrightarrow
P_{rn+4} \stackrel{\b}\longrightarrow P_{rn+3}\longrightarrow S_{rn+3}\longrightarrow 0.$$ At that
$\Omega^{2}(S_{rn+3})\simeq S_{(r+1)n+4}$.
\end{lem}

\begin{lem}
The begin of the minimal projective resolution of $S_{rn+4}$ is of the form
\begin{multline*}
\dots\longrightarrow P_{(r+4)n+1} \stackrel{\a}\longrightarrow
P_{(r+4)n}\stackrel{\g\b}\longrightarrow
P_{(r+3)n+4}\stackrel{\b^2\g}\longrightarrow
P_{(r+2)n+5}\stackrel{\binom{\a}{-\b^2}}\longrightarrow\\\longrightarrow
P_{(r+2)n+2}\oplus P_{(r+2)n+3}\stackrel{(\a^2\text{
}\b)}\longrightarrow
P_{(r+2)n}\stackrel{\g\a^2}\longrightarrow\\
\longrightarrow P_{(r+1)n+1}\stackrel{\a\g}\longrightarrow
P_{rn+5}\stackrel{\b}\longrightarrow P_{rn+4}\longrightarrow
S_{rn+4}\longrightarrow 0.
\end{multline*}
At that $\Omega^{9}(S_{rn+4})\simeq S_{(r+5)n+1}$.
\end{lem}

\begin{lem}\label{lem_s5}
The begin of the minimal projective resolution of $S_{rn+5}$ is of the form $$\dots\longrightarrow
P_{(r+1)n} \stackrel{\g}\longrightarrow P_{rn+5}\longrightarrow S_{rn+5}\longrightarrow 0.$$ At
that $\Omega^{2}(S_{rn+5})\simeq S_{(r+2)n}$.
\end{lem}
\begin{proof}
Proofs of the lemmas consist of direct check that given sequences are exact, and it is immediate.
\end{proof}
We shall need the Happel's lemma (see \cite{Ha}), as revised in \cite{Gen&Ka}:

\begin{lem}[Happel]\label{lem_Ha}
Let
$$\dots\rightarrow Q_m\rightarrow Q_{m-1}\rightarrow\dots\rightarrow
Q_1\rightarrow Q_0\rightarrow R\rightarrow 0$$ be the minimal projective resolution of $R$. Then
$$Q_m\cong\bigoplus_{i,j}P_{i,j}^{\dim\mathrm{Ext}^m_R(S_j,S_i)}.$$
\end{lem}

\begin{proof}[Proof of the theorem \ref{resol_thm}]
Descriptions for $Q_i$  immediately follows from lemmas \ref{lem_s0} -- \ref{lem_s5} and Happel's
lemma.

As proved in \cite{VGI}, to prove that sequence \eqref{resolv} is exact in $Q_m$ ($m\le 11$)
it will be sufficient to show that $d_md_{m+1}=0$. It is easy to verify this relation by
a straightforward calculation of matrixes products.

Since the sequence is exact in $Q_{11}$, it follows that
$\Omega^{11}({}_\Lambda R)\simeq {}_1R_{\sigma}$, where
$\Omega^{11}({}_\Lambda R)=\Im d_{10}$ is the 11th syzygy of the
module $R$, and ${}_1R_{\sigma}$ is a twisted bimodule. Hence, an
exactness in $Q_t$ ($t>11$) holds.

\end{proof}

We recall that for $R$-bimodule $M$ the {\it twisted bimodule} is a linear
space $M$, on which left act right acts of the algebra $R$ (denoted by
asterisk) are assigned by the following way:
$$r*m*s = \lambda(r)\cdot m\cdot\mu(s) \text{ for } r,s\in R
\text{ and } m\in M,$$ where $\lambda,\mu$ are some automorphisms of algebra $R$.
Such twisted bimodule we shall denote by ${}_\lambda M_\mu$.

\begin{s}
We have isomorphism $\Omega^{11}({}_\Lambda R)\simeq {}_1R_{\sigma}$.
\end{s}

\begin{pr}
Automorphism $\sigma$ has a finite order, and

$(1)$ if $\myChar=2$, then order of $\sigma$ is equal to $\frac{2s}{\myNod(2n,s)}$;

$(2)$ if $\myChar\ne 2$, then order of $\sigma$ is equal to $\frac{2s}{\myNod(2n,s)}$, if
$\frac{s}{\myNod(n,s)}$ is divisible by $4$, and to $\frac{4s}{\myNod(2n,s)}$ otherwise.
\end{pr}

\begin{pr}
The minimal period of bimodule resolution of $R$ is $11\frac{2s}{\myNod(2n,s)}$.
\end{pr}

\begin{proof}
Since $\Omega^{11}({}_\Lambda R)\simeq {}_1R_{\sigma}$, period is equal to $11a$. There is an
isomorphism $\sigma^a(Q_0)\simeq Q_0$, i. e. $\sigma^a$ must act identically on idempotents. Hence,
if $\myChar=2$, then $a=\deg\sigma$. For the case $\myChar\ne 2$ we just need to show that
$a=\frac{\deg\sigma}{2}$, if $\frac{s}{\myNod(n,s)}$ is not divisible by $4$. Consider the case
$\myChar\ne 2$, $\frac{s}{\myNod(n,s)}$ is not divisible by $4$. Denote by
$b=\frac{\deg\sigma}{2}$. We see that $\sigma^b(e_i)=e_i,$ $\sigma^b(\g_i)=\g_i$,
$$\sigma^b(\a_i)=\begin{cases}
 -\a_i,\quad i\equiv 0,2(3);\\
 \a_i,\quad i\equiv 1(3),
 \end{cases} \sigma^b(\b_i)=\begin{cases}
 -\b_i,\quad i\equiv 0,2(3);\\
 \b_i,\quad i\equiv 1(3).
 \end{cases}$$
Let $x=\sum\limits_{i=0}^{s-1} e_{6i+1}+e_{6i+2}+e_{6i+3}+e_{6i+4}-e_{6i}-e_{6i+5}$. We have:
$x^{-1}=x$, $\sigma^b(w)=xwx^{-1}$, for any path $w$, so $\sigma^b$ is an inner automorphism of
$R$, and $11b$ is a period of bimodule resolution.
\end{proof}

\section{The additive structure of $\HH^*(R)$}

\begin{pr}[Dimensions of homomorphism groups, $s>1$]\label{dim_hom}
Let $s>1$ and $R=R_s$ is algebra of the type $E_6$. Next, $t\in\N\cup\{0\}$,
$\ell$ be the aliquot, and $r$ be the residue of division of $t$ by $11$.

$(1)$ If $r=0$, then $$\dim_K\Hom_\Lambda(Q_t, R)=\begin{cases}
 6s,\quad \ell n+m\equiv 0(s)\text{ or }\ell n+m\equiv 1(s),\text{ }\ell\div 2;\\
 2s,\quad \ell n+m\equiv 0(s)\text{ or }\ell n+m\equiv 1(s),\text{ }\ell\ndiv 2;\\
 0,\quad\text{otherwise.}
 \end{cases}$$

$(2)$ If $r=1$, then $$\dim_K\Hom_\Lambda(Q_t, R)=\begin{cases}
 7s,\quad \ell n+m\equiv 0(s),\text{ }\ell\div 2;\\
 5s,\quad \ell n+m\equiv 0(s),\text{ }\ell\ndiv 2;\\
 0,\quad\text{otherwise.}
 \end{cases}$$

$(3)$ If $r\in\{2,8\}$, then $$\dim_K\Hom_\Lambda(Q_t, R)=\begin{cases}
 3s,\quad \ell n+m\equiv 0(s),\text{ }\ell\div 2;\\
 s,\quad \ell n+m\equiv 0(s),\text{ }\ell\ndiv 2;\\
 s,\quad \ell n+m\equiv 1(s),\text{ }\ell\div 2;\\
 3s,\quad \ell n+m\equiv 1(s),\text{ }\ell\ndiv 2;\\
 0,\quad\text{otherwise.}
 \end{cases}$$

$(4)$ If $r\in\{3,5,7\}$, then $$\dim_K\Hom_\Lambda(Q_t, R)=\begin{cases}
 8s,\quad \ell n+m\equiv 0(s);\\
 0,\quad\text{otherwise.}
 \end{cases}$$

$(5)$ If $r=4$, then $$\dim_K\Hom_\Lambda(Q_t, R)=\begin{cases}
 2s,\quad \ell n+m\equiv 0(s),\text{ }\ell\div 2;\\
 6s,\quad \ell n+m\equiv 0(s),\text{ }\ell\ndiv 2;\\
 5s,\quad \ell n+m\equiv 1(s),\text{ }\ell\div 2;\\
 7s,\quad \ell n+m\equiv 1(s),\text{ }\ell\ndiv 2;\\
 0,\quad\text{otherwise.}
 \end{cases}$$

$(6)$ If $r=6$, then $$\dim_K\Hom_\Lambda(Q_t, R)=\begin{cases}
 7s,\quad \ell n+m\equiv 0(s),\text{ }\ell\div 2;\\
 5s,\quad \ell n+m\equiv 0(s),\text{ }\ell\ndiv 2;\\
 6s,\quad \ell n+m\equiv 1(s),\text{ }\ell\div 2;\\
 2s,\quad \ell n+m\equiv 1(s),\text{ }\ell\ndiv 2;\\
 0,\quad\text{otherwise.}
 \end{cases}$$

$(7)$ If $r=9$, then $$\dim_K\Hom_\Lambda(Q_t, R)=\begin{cases}
 5s,\quad \ell n+m\equiv 0(s),\text{ }\ell\div 2;\\
 7s,\quad \ell n+m\equiv 0(s),\text{ }\ell\ndiv 2;\\
 0,\quad\text{otherwise.}
 \end{cases}$$

$(8)$ If $r=10$, then $$\dim_K\Hom_\Lambda(Q_t, R)=\begin{cases}
 2s,\quad \ell n+m\equiv 0(s)\text{ or }\ell n+m\equiv 1(s),\text{ }\ell\div 2;\\
 6s,\quad \ell n+m\equiv 0(s)\text{ or }\ell n+m\equiv 1(s),\text{ }\ell\ndiv 2;\\
 0,\quad\text{otherwise.}
 \end{cases}$$

\end{pr}

\begin{proof}
The dimension $\dim_K\Hom_\Lambda(P_{i,j}, R)$ is equal to the number of linear independent nonzero
paths of the quiver $\mathcal{Q}_s$, leading from $j$th vertex to $i$th, and the proof is to
consider cases $r=0$, $r=1$ etc.
\end{proof}

\begin{pr}[Dimensions of homomorphism groups, $s=1$]

Let $R=R_1$ is algebra of the type $E_6$. Next, $t\in\N\cup\{0\}$,
$\ell$ be the aliquot, and $r$ be the residue of division of $t$ by $11$.

$(1)$ If $r=0$, then $$\dim_K\Hom_\Lambda(Q_t, R)=\begin{cases}
 12,\quad \ell\div 2;\\4,\quad \ell\ndiv 2.\end{cases}$$

$(2)$ If $r=1$, then $$\dim_K\Hom_\Lambda(Q_t, R)=\begin{cases}
 7,\quad \ell\div 2;\\5,\quad \ell\ndiv 2.\end{cases}$$

$(3)$ If $r\in\{2,8\}$, then $\dim_K\Hom_\Lambda(Q_t, R)=4$.

$(4)$ If $r\in\{3,5,7\}$, then $\dim_K\Hom_\Lambda(Q_t, R)=8$.

$(5)$ If $r=4$, then $$\dim_K\Hom_\Lambda(Q_t, R)=\begin{cases}
 7,\quad \ell\div 2;\\13,\quad \ell\ndiv 2.\end{cases}$$

$(6)$ If $r=6$, then $$\dim_K\Hom_\Lambda(Q_t, R)=\begin{cases}
 13,\quad \ell\div 2;\\7,\quad \ell\ndiv 2.\end{cases}$$

$(7)$ If $r=9$, then $$\dim_K\Hom_\Lambda(Q_t, R)=\begin{cases}
 5,\quad \ell\div 2;\\7,\quad \ell\ndiv 2.\end{cases}$$

$(8)$ If $r=10$, then $$\dim_K\Hom_\Lambda(Q_t, R)=\begin{cases}
 4,\quad \ell\div 2;\\12,\quad \ell\ndiv 2.\end{cases}$$

\end{pr}

\begin{proof}
The proof is basically the same as proof of proposition \ref{dim_hom}.
\end{proof}

\begin{pr}[Dimensions of coboundaries groups]\label{dim_im}
Let $R=R_s$ is algebra of the type $E_6$, and let
\begin{equation}\tag{$\times$}\label{ind_resolv} 0\longrightarrow
\Hom_\Lambda(Q_0, R)\stackrel{\delta^0}\longrightarrow \Hom_\Lambda(Q_1,
R)\stackrel{\delta^1}\longrightarrow \Hom_\Lambda(Q_2, R)\stackrel{\delta^2}\longrightarrow\dots
\end{equation} be a complex, obtained from minimal projective
resolution \eqref{resolv} of algebra $R$, by applying functor
$\Hom_\Lambda(-,R)$.

Consider coboundaries groups $\Im\delta^s$ of the complex
\eqref{ind_resolv}. Let $\ell$ be the aliquot, and $r$ be the
residue of division of $t$ by $11$, $m$ be the aliquot of
division of $r$ by $2$. Then$:$

$(1)$ If $r=0$, then $$\dim_K\Im\delta^t=\begin{cases}
 6s-1,\quad \ell n+m\equiv 0(s),\text{ }\ell\div 2;\\
 2s,\quad \ell n+m\equiv 0(s),\text{ }\ell\ndiv 2;\\
 0,\quad\text{otherwise.}
 \end{cases}$$

$(2)$ If $r=1$, then $$\dim_K\Im\delta^t=\begin{cases}
 s,\quad \ell n+m\equiv 0(s),\text{ }\ell\div 2;\\
 3s-1,\quad \ell n+m\equiv 0(s),\text{ }\ell\ndiv 2;\\
 0,\quad\text{otherwise.}
 \end{cases}$$

$(3)$ If $r=2$, then $$\dim_K\Im\delta^t=\begin{cases}
 3s,\quad \ell n+m\equiv 0(s),\text{ }\ell\div 2;\\
 s,\quad \ell n+m\equiv 0(s),\text{ }\ell\ndiv 2;\\
 0,\quad\text{otherwise.}
 \end{cases}$$

$(4)$ If $r=3$, then $$\dim_K\Im\delta^t=\begin{cases}
 5s-1,\quad \ell n+m\equiv 0(s),\text{ }\ell\div 2,\text{ }\myChar=2;\\
 5s,\quad \ell n+m\equiv 0(s),\text{ }\ell\div 2,\text{ }\myChar\ne 2;\\
 7s-1,\quad \ell n+m\equiv 0(s),\text{ }\ell\ndiv 2;\\
 0,\quad\text{otherwise.}
 \end{cases}$$

$(5)$ If $r=4$, then $$\dim_K\Im\delta^t=\begin{cases}
 2s,\quad \ell n+m\equiv 0(s),\text{ }\ell\div 2;\\
 6s-1,\quad \ell n+m\equiv 0(s),\text{ }\ell\ndiv 2,\text{ }\myChar=3;\\
 6s,\quad \ell n+m\equiv 0(s),\text{ }\ell\ndiv 2,\text{ }\myChar\ne 3;\\
 0,\quad\text{otherwise.}
 \end{cases}$$

$(6)$ If $r=5$, then $$\dim_K\Im\delta^t=\begin{cases}
 6s-1,\quad \ell n+m\equiv 0(s),\text{ }\ell\div 2,\text{ }\myChar=3;\\
 6s,\quad \ell n+m\equiv 0(s),\text{ }\ell\div 2,\text{ }\myChar\ne 3;\\
 2s,\quad \ell n+m\equiv 0(s),\text{ }\ell\ndiv 2;\\
  0,\quad\text{otherwise.}
 \end{cases}$$

$(7)$ If $r=6$, then $$\dim_K\Im\delta^t=\begin{cases}
 7s-1,\quad \ell n+m\equiv 0(s),\text{ }\ell\div 2;\\
 5s-1,\quad \ell n+m\equiv 0(s),\text{ }\ell\ndiv 2,\text{ }\myChar=2;\\
 5s,\quad \ell n+m\equiv 0(s),\text{ }\ell\ndiv 2,\text{ }\myChar\ne 2;\\
 0,\quad\text{otherwise.}
 \end{cases}$$

$(8)$ If $r=7$, then $$\dim_K\Im\delta^t=\begin{cases}
 s,\quad \ell n+m\equiv 0(s),\text{ }\ell\div 2;\\
 3s,\quad \ell n+m\equiv 0(s),\text{ }\ell\ndiv 2;\\
 0,\quad\text{otherwise.}
 \end{cases}$$

$(9)$ If $r=8$, then $$\dim_K\Im\delta^t=\begin{cases}
 3s-1,\quad \ell n+m\equiv 0(s),\text{ }\ell\div 2;\\
 s,\quad \ell n+m\equiv 0(s),\text{ }\ell\ndiv 2;\\
 0,\quad\text{otherwise.}
 \end{cases}$$

$(10)$ If $r=9$, then $$\dim_K\Im\delta^t=\begin{cases}
 2s,\quad \ell n+m\equiv 0(s),\text{ }\ell\div 2;\\
 6s-1,\quad \ell n+m\equiv 0(s),\text{ }\ell\ndiv 2;\\
 0,\quad\text{otherwise.}
 \end{cases}$$

$(11)$ If $r=10$, then $$\dim_K\Im\delta^t=\begin{cases}
 2s,\quad \ell n+m\equiv 0(s),\text{ }\ell\div 2;\\
 6s-1,\quad \ell n+m\equiv 0(s),\text{ }\ell\ndiv 2,\text{ }\myChar=3;\\
 6s,\quad \ell n+m\equiv 0(s),\text{ }\ell\ndiv 2,\text{ }\myChar\ne 3;\\
 0,\quad\text{otherwise.}
 \end{cases}$$

\end{pr}

\begin{proof}
The proof is technical and consists in constructing the image matrixes from the description of
differential matrixes and the subsequent computations of the ranks of image matrixes.
\end{proof}

\begin{thm}[Additive structure, $s>1$]
Let $s>1$ and $R=R_s$ is algebra of the type $E_6$. Next, $t\in\N\cup\{0\}$,
$\ell$ be the aliquot, and $r$ be the residue of division of $t$ by $11$,
$m$ be the aliquot of division of $r$ by $2$. Then
$\dim_K\HH^t(R)=1$, if one of the following conditions takes place$:$

$(1)$ $r\in \{0,6,7,8\}$, $\ell n+m\equiv 0(s)$, $\ell\div 2$;

$(2)$ $r\in \{0,6\}$, $\ell n+m\equiv 1(s)$, $\ell\div 2$,
$\myChar=3$;

$(3)$ $r\in \{1,9\}$, $\ell n+m\equiv 0(s)$;

$(4)$ $r\in \{2,4,10\}$, $\ell n+m\equiv 1(s)$, $\ell\ndiv 2$;

$(5)$ $r=3$, $\ell n+m\equiv 0(s)$, $\ell\div 2$, $\myChar=2$;

$(6)$ $r=3$, $\ell n+m\equiv 0(s)$, $\ell\ndiv 2$;

$(7)$ $r\in \{4,10\}$, $\ell n+m\equiv 0(s)$, $\ell\ndiv 2$,
$\myChar=3$;

$(8)$ $r=4$, $\ell n+m\equiv 1(s)$, $\ell\div 2$, $\myChar=2$;

$(9)$ $r=5$, $\ell n+m\equiv 0(s)$, $\myChar=3$;

$(10)$ $r\in \{6,7\}$, $\ell n+m\equiv 0(s)$, $\ell\ndiv 2$,
$\myChar=2$.

In other cases $\dim_K\HH^t(R)=0$.
\end{thm}

\begin{proof}
As $\dim_K\HH^t(R)=\dim_K\Ker\delta^t-\dim_K\Im\delta^{t-1}$, and
$\dim_K\Ker\delta^t=\dim_K\Hom_\Lambda(Q_t,R)-\dim_K\Im\delta^t$,
the assertions of theorem easily follows from propositions
\ref{dim_hom} -- \ref{dim_im}.
\end{proof}

\begin{thm}[Additive structure, $s=1$]
Let $R=R_1$ is algebra of the type $E_6$. Next, $t\in\N\cup\{0\}$,
$\ell$ be the aliquot, and $r$ be the residue of division of $t$ by $11$.\\
$($a$)$ $\dim_K\HH^t(R)=7$, if $t=0$.\\
$($b$)$ $\dim_K\HH^t(R)=2$, if one of the following conditions takes place$:$

$(1)$ $r\in \{0,6\}$, $t>0$, $\ell\div 2$, $\myChar=3$;

$(2)$ $r\in \{4,10\}$, $\ell\ndiv 2$, $\myChar=3$.\\ $($c$)$
$\dim_K\HH^t(R)=1$, if one of the following conditions takes
place$:$

$(1)$ $r\in \{0,6\}$, $t>0$, $\ell\div 2$, $\myChar\ne 3$;

$(2)$ $r\in \{1,9\}$;

$(3)$ $r\in \{2,3\}$, $\ell\ndiv 2$;

$(4)$ $r\in \{3,4\}$, $\ell\div 2$, $\myChar=2$;

$(5)$ $r\in \{4,10\}$, $\ell\ndiv 2$, $\myChar\ne 3$;

$(6)$ $r=5$, $\myChar=3$;

$(7)$ $r\in \{6,7\}$, $\ell\ndiv 2$, $\myChar=2$;

$(8)$ $r\in \{7,8\}$, $\ell\div 2$.\\ $($d$)$ In other cases
$\dim_K\HH^t(R)=0$.

\end{thm}

\section{Generators of $\HH^*(R)$}

For $s>1$ introduce the set of generators $Y^{(1)}_t$, $Y^{(2)}_t$,
\dots $Y^{(22)}_t$, such that $\deg Y_t^{(i)}=t$, $0\le t <
11\deg\sigma$ and $t$ satisfies conditions of (i)th item from the
list on page \pageref{degs}. For $s=1$ introduce the set of
generators $Y^{(1)}_t$, $Y^{(2)}_t$, \dots $Y^{(28)}_t$, such that
$\deg Y_t^{(i)}=t$, $0\le t < 11\deg\sigma$ and $t$ satisfies
conditions of (i)th item from the list on page \pageref{degs} for
$i\ge 22$ and $t=0$ if $i>22$. Now let us describe the matrixes of
$Y^{(i)}_t$ componentwisely.

\begin{obozn}
Let us represent the degree $t$ of the generator element in the form
$t=11\ell+r$ ($0\le r\le 10$) and denote by
$\kappa=(-1)^{\lfloor\frac{\ell}{2}\rfloor}$.
\end{obozn}

(1) $Y^{(1)}_t$ is an $(6s\times 6s)$ matrix, whose elements
$y_{ij}$ have the following form:

If $0\le j<s$, then $$y_{ij}=
\begin{cases}
\kappa e_{jn}\otimes e_{jn},\quad i=j;\\
0,\quad\text{otherwise.}\end{cases}$$

If $s\le j<3s$, then $$y_{ij}=
\begin{cases}
e_{jn+g(j+s)}\otimes e_{jn+g(j+s)},\quad i=j;\\
0,\quad\text{otherwise.}\end{cases}$$

If $3s\le j<5s$, then $$y_{ij}=
\begin{cases}
e_{jn+g(j+s)+1}\otimes e_{jn+g(j+s)+1},\quad i=j;\\
0,\quad\text{otherwise.}\end{cases}$$

If $5s\le j<6s$, then $$y_{ij}=
\begin{cases}
\kappa e_{jn+5}\otimes e_{jn+5},\quad i=j;\\
0,\quad\text{otherwise.}\end{cases}$$

(2) $Y^{(2)}_t$ is an $(6s\times 6s)$ matrix with a single nonzero
element:
$$y_{3s,3s}=w_{jn+g(j+s)+1\ra (j+1)n+g(j+s)+1}\otimes e_{jn+g(j+s)+1}.$$

(3) $Y^{(3)}_t$ is an $(7s\times 6s)$ matrix with a single nonzero
element:
$$y_{5s,6s}=\kappa w_{jn+5\ra (j+1)n}\otimes e_{jn+5}.$$

(4) $Y^{(4)}_t$ is an $(7s\times 6s)$ matrix, whose elements
$y_{ij}$ have the following form:

If $0\le j<s$, then $$y_{ij}=
\begin{cases}
-w_{jn\ra jn+g(j+s)}\otimes e_{jn},\quad i=j;\\
0,\quad\text{otherwise.}\end{cases}$$

If $s\le j<5s$, then $y_{ij}=0$.

If $5s\le j<6s$, then $$y_{ij}=
\begin{cases}
\kappa w_{jn+g(j)+1\ra jn+5}\otimes e_{jn+g(j)+1},\quad i=j-s;\\
0,\quad\text{otherwise.}\end{cases}$$

If $6s\le j<7s$, then $y_{ij}=0$.

(5) $Y^{(5)}_t$ is an $(6s\times 6s)$ matrix with a single nonzero
element:
$$y_{3s,3s}=w_{jn+g(j+s)+1\ra (j+1)n+g(j+s)}\otimes e_{jn+g(j+s)+1}.$$

(6) $Y^{(6)}_t$ is an $(8s\times 6s)$ matrix, whose elements
$y_{ij}$ have the following form:

If $0\le j<s$, then $$y_{ij}=
\begin{cases}
w_{jn\ra jn+g(j)+1}\otimes e_{jn},\quad i=j;\\
0,\quad\text{otherwise.}\end{cases}$$

If $s\le j<2s$, then $$y_{ij}=
\begin{cases}
w_{jn\ra jn+g(j)+1}\otimes e_{jn},\quad i=j-s;\\
0,\quad\text{otherwise.}\end{cases}$$

If $2s\le j<4s$, then $y_{ij}=0$.

If $4s\le j<5s$, then $$y_{ij}=
\begin{cases}
w_{jn+g(j)+1\ra (j+1)n}\otimes e_{jn+g(j)+1},\quad i=j-s;\\
0,\quad\text{otherwise.}\end{cases}$$

If $5s\le j<6s$, then $y_{ij}=0$.

If $6s\le j<7s$, then $$y_{ij}=
\begin{cases}
w_{jn+5\ra (j+1)n+g(j)}\otimes e_{jn+5},\quad i=j-s;\\
0,\quad\text{otherwise.}\end{cases}$$

If $7s\le j<8s$, then $y_{ij}=0$.

(7) $Y^{(7)}_t$ is an $(8s\times 6s)$ matrix, whose elements
$y_{ij}$ have the following form:

If $0\le j<2s$, then $$y_{ij}=
\begin{cases}
w_{jn\ra jn+g(j+s)+1}\otimes e_{jn},\quad i=(j)_s;\\
0,\quad\text{otherwise.}\end{cases}$$

If $2s\le j<4s$, then $$y_{ij}=
\begin{cases}
-\kappa w_{jn+g(j)\ra jn+5}\otimes e_{jn+g(j)},\quad i=j-s;\\
0,\quad\text{otherwise.}\end{cases}$$

If $4s\le j<6s$, then $$y_{ij}=
\begin{cases}
-\kappa w_{jn+g(j)+1\ra (j+1)n}\otimes e_{jn+g(j)+1},\quad i=j-s;\\
0,\quad\text{otherwise.}\end{cases}$$

If $6s\le j<8s$, then $y_{ij}=0$.

(8) $Y^{(8)}_t$ is an $(9s\times 6s)$ matrix with a single nonzero
element:
$$y_{3s,6s}=w_{jn+g(j)+1\ra (j+1)n+g(j)+1}\otimes e_{jn+g(j)+1}.$$

(9) $Y^{(9)}_t$ is an $(9s\times 6s)$ matrix, whose elements
$y_{ij}$ have the following form:

If $0\le j<s$, then $y_{ij}=0$.

If $s\le j<2s$, then $$y_{ij}=
\begin{cases}
\kappa e_{jn}\otimes e_{jn},\quad i=j-s;\\
0,\quad\text{otherwise.}\end{cases}$$

If $2s\le j<4s$, then $$y_{ij}=
\begin{cases}
e_{jn+g(j)}\otimes e_{jn+g(j)},\quad i=j-s;\\
0,\quad\text{otherwise.}\end{cases}$$

If $4s\le j<6s$, then $y_{ij}=0$.

If $6s\le j<8s$, then $$y_{ij}=
\begin{cases}
e_{jn+g(j)+1}\otimes e_{jn+g(j)+1},\quad i=j-3s;\\
0,\quad\text{otherwise.}\end{cases}$$

If $8s\le j<9s$, then $$y_{ij}=
\begin{cases}
\kappa e_{jn+5}\otimes e_{jn+5},\quad i=j-3s;\\
0,\quad\text{otherwise.}\end{cases}$$

(10) $Y^{(10)}_t$ is an $(9s\times 6s)$ matrix with a single nonzero
element:
$$y_{5s,8s}=w_{jn+5\ra (j+1)n+5}\otimes e_{jn+5}.$$

(11) $Y^{(11)}_t$ is an $(8s\times 6s)$ matrix, whose elements
$y_{ij}$ have the following form:

If $0\le j<s$, then $$y_{ij}=
\begin{cases}
w_{jn\ra jn+g(j)}\otimes e_{jn},\quad i=j;\\
0,\quad\text{otherwise.}\end{cases}$$

If $s\le j<2s$, then $y_{ij}=0$.

If $2s\le j<3s$, then $$y_{ij}=
\begin{cases}
\kappa w_{jn+g(j)\ra (j+1)n}\otimes e_{jn+g(j)},\quad i=j-s;\\
0,\quad\text{otherwise.}\end{cases}$$

If $3s\le j<4s$, then $y_{ij}=0$.

If $4s\le j<5s$, then $$y_{ij}=
\begin{cases}
-\kappa w_{jn+g(j)+1\ra jn+5}\otimes e_{jn+g(j)+1},\quad i=j-s;\\
0,\quad\text{otherwise.}\end{cases}$$

If $5s\le j<6s$, then $y_{ij}=0$.

If $6s\le j<7s$, then $$y_{ij}=
\begin{cases}
w_{jn+5\ra (j+1)n+g(j)+1}\otimes e_{jn+5},\quad i=j-s;\\
0,\quad\text{otherwise.}\end{cases}$$

If $7s\le j<8s$, then $y_{ij}=0$.

(12) $Y^{(12)}_t$ is an $(8s\times 6s)$ matrix, whose elements
$y_{ij}$ have the following form:

If $0\le j<s$, then $$y_{ij}=
\begin{cases}
-w_{jn\ra jn+g(j+s)}\otimes e_{jn},\quad i=j;\\
0,\quad\text{otherwise.}\end{cases}$$

If $s\le j<4s$, then $y_{ij}=0$.

If $4s\le j<5s$, then $$y_{ij}=
\begin{cases}
\kappa w_{jn+g(j)+1\ra jn+5}\otimes e_{jn+g(j)+1},\quad i=j-s;\\
0,\quad\text{otherwise.}\end{cases}$$

If $5s\le j<8s$, then $y_{ij}=0$.

(13) $Y^{(13)}_t$ is an $(9s\times 6s)$ matrix, whose elements
$y_{ij}$ have the following form:

If $0\le j<s$, then $y_{ij}=0$.

If $s\le j<3s$, then $$y_{ij}=
\begin{cases}
-f_1(j,2s)e_{jn+g(j+s)}\otimes e_{jn+g(j+s)},\quad i=j;\\
0,\quad\text{otherwise.}\end{cases}$$

If $3s\le j<5s$, then $y_{ij}=0$.

If $5s\le j<7s$, then $$y_{ij}=
\begin{cases}
-f_1(j,6s)e_{jn+g(j+s)+1}\otimes e_{jn+g(j+s)+1},\quad i=j-2s;\\
0,\quad\text{otherwise.}\end{cases}$$

If $7s\le j<8s$, then $y_{ij}=0$.

If $8s\le j<9s$, then $$y_{ij}=
\begin{cases}
-\kappa w_{jn+5\ra (j+1)n}\otimes e_{jn+5},\quad i=j-3s;\\
0,\quad\text{otherwise.}\end{cases}$$

(14) $Y^{(14)}_t$ is an $(9s\times 6s)$ matrix with a single nonzero
element:
$$y_{0,0}=\kappa w_{jn\ra (j+1)n}\otimes e_{jn}.$$

(15) $Y^{(15)}_t$ is an $(9s\times 6s)$ matrix, whose elements
$y_{ij}$ have the following form:

If $0\le j<s$, then $$y_{ij}=
\begin{cases}
e_{jn}\otimes e_{jn},\quad i=j;\\
0,\quad\text{otherwise.}\end{cases}$$

If $s\le j<3s$, then $y_{ij}=0$.

If $3s\le j<5s$, then $$y_{ij}=
\begin{cases}
w_{jn+g(j+s)\ra jn+g(j+s)+1}\otimes e_{jn+g(j+s)},\quad i=j-2s;\\
0,\quad\text{otherwise.}\end{cases}$$

If $5s\le j<7s$, then $y_{ij}=0$.

If $7s\le j<8s$, then $$y_{ij}=
\begin{cases}
e_{jn+5}\otimes e_{jn+5},\quad i=j-2s;\\
0,\quad\text{otherwise.}\end{cases}$$

If $8s\le j<9s$, then $y_{ij}=0$.

(16) $Y^{(16)}_t$ is an $(8s\times 6s)$ matrix, whose elements
$y_{ij}$ have the following form:

If $j=0$, then $$y_{ij}=
\begin{cases}
w_{jn\ra jn+g(j)+1}\otimes e_{jn},\quad i=j;\\
0,\quad\text{otherwise.}\end{cases}$$

If $0<j<2s$, then $y_{ij}=0$.

If $j=2s$, then $$y_{ij}=
\begin{cases}
-\kappa w_{jn+g(j)\ra jn+5}\otimes e_{jn+g(j)},\quad i=j-s;\\
0,\quad\text{otherwise.}\end{cases}$$

If $2s<j<8s$, then $y_{ij}=0$.

(17) $Y^{(17)}_t$ is an $(8s\times 6s)$ matrix, whose elements
$y_{ij}$ have the following form:

If $j=0$, then $$y_{ij}=
\begin{cases}
w_{jn\ra jn+g(j+s)+1}\otimes e_{jn},\quad i=j;\\
0,\quad\text{otherwise.}\end{cases}$$

If $0<j<2s$, then $y_{ij}=0$.

If $j=2s$, then $$y_{ij}=
\begin{cases}
w_{jn+g(j)\ra jn+5}\otimes e_{jn+g(j)},\quad i=j-s;\\
0,\quad\text{otherwise.}\end{cases}$$

If $2s<j<8s$, then $y_{ij}=0$.

(18) $Y^{(18)}_t$ is an $(6s\times 6s)$ matrix, whose elements
$y_{ij}$ have the following form:

If $0\le j<s$, then $y_{ij}=0$.

If $s\le j<3s$, then $$y_{ij}=
\begin{cases}
w_{jn+g(j+s)\ra jn+g(j+s)+1}\otimes e_{jn+g(j+s)},\quad i=j;\\
0,\quad\text{otherwise.}\end{cases}$$

If $3s\le j<5s$, then $y_{ij}=0$.

If $5s\le j<6s$, then $$y_{ij}=
\begin{cases}
-\kappa w_{jn+5\ra (j+1)n}\otimes e_{jn+5},\quad i=j;\\
0,\quad\text{otherwise.}\end{cases}$$

(19) $Y^{(19)}_t$ is an $(7s\times 6s)$ matrix, whose elements
$y_{ij}$ have the following form:

If $0\le j<s$, then $y_{ij}=0$.

If $j=s$, then $$y_{ij}=
\begin{cases}
\kappa w_{jn+g(j+s)\ra (j+1)n}\otimes e_{jn+g(j+s)},\quad i=j;\\
0,\quad\text{otherwise.}\end{cases}$$

If $s<j<2s$, then $y_{ij}=0$.

If $j=2s$, then $$y_{ij}=
\begin{cases}
\kappa w_{jn+g(j+s)\ra (j+1)n}\otimes e_{jn+g(j+s)},\quad i=j;\\
0,\quad\text{otherwise.}\end{cases}$$

If $2s<j<7s$, then $y_{ij}=0$.

(20) $Y^{(20)}_t$ is an $(7s\times 6s)$ matrix, whose elements
$y_{ij}$ have the following form:

If $0\le j<s$, then $$y_{ij}=
\begin{cases}
\kappa w_{jn\ra jn+5}\otimes e_{jn},\quad i=j;\\
0,\quad\text{otherwise.}\end{cases}$$

If $s\le j<2s$, then $$y_{ij}=
\begin{cases}
\kappa w_{jn+g(j+s)\ra (j+1)n}\otimes e_{jn+g(j+s)},\quad i=j;\\
0,\quad\text{otherwise.}\end{cases}$$

If $2s\le j<3s$, then $y_{ij}=0$.

If $3s\le j<4s$, then $$y_{ij}=
\begin{cases}
-w_{jn+g(j+s)+1\ra (j+1)n+g(j+s)}\otimes e_{jn+g(j+s)+1},\quad i=j;\\
0,\quad\text{otherwise.}\end{cases}$$

If $4s\le j<6s$, then $y_{ij}=0$.

If $6s\le j<7s$, then $$y_{ij}=
\begin{cases}
-w_{jn+5\ra (j+1)n+g(j)+1}\otimes e_{jn+5},\quad i=j-s;\\
0,\quad\text{otherwise.}\end{cases}$$

(21) $Y^{(21)}_t$ is an $(6s\times 6s)$ matrix with a single nonzero
element:
$$y_{0,0}=-\kappa w_{jn\ra (j+1)n}\otimes e_{jn}.$$

(22) $Y^{(22)}_t$ is an $(6s\times 6s)$ matrix, whose elements
$y_{ij}$ have the following form:

If $0\le j<s$, then $y_{ij}=0$.

If $s\le j<3s$, then $$y_{ij}=
\begin{cases}
f_1(j,2s)e_{jn+g(j+s)}\otimes e_{jn+g(j+s)},\quad i=j;\\
0,\quad\text{otherwise.}\end{cases}$$

If $3s\le j<5s$, then $$y_{ij}=
\begin{cases}
f_1(j,4s)e_{jn+g(j+s)+1}\otimes e_{jn+g(j+s)+1},\quad i=j;\\
0,\quad\text{otherwise.}\end{cases}$$

If $5s\le j<6s$, then $y_{ij}=0$.

(23) $Y^{(23)}_t$ is an $(6s\times 6s)$ matrix with a single nonzero
element:
$$y_{3s,3s}=w_{jn+g(j+s)+1\ra (j+1)n+g(j+s)+1}\otimes e_{jn+g(j+s)+1}.$$

(24) $Y^{(24)}_t$ is an $(6s\times 6s)$ matrix with a single nonzero
element:
$$y_{0,0}=w_{jn\ra (j+1)n}\otimes e_{jn}.$$

(25) $Y^{(25)}_t$ is an $(6s\times 6s)$ matrix with a single nonzero
element:
$$y_{s,s}=w_{jn+g(j+s)\ra (j+1)n+g(j+s)}\otimes e_{jn+g(j+s)}.$$

(26) $Y^{(26)}_t$ is an $(6s\times 6s)$ matrix with a single nonzero
element:
$$y_{2s,2s}=-w_{jn+g(j+s)\ra (j+1)n+g(j+s)}\otimes e_{jn+g(j+s)}.$$

(27) $Y^{(27)}_t$ is an $(6s\times 6s)$  matrix with a single
nonzero element:
$$y_{4s,4s}=-w_{jn+g(j+s)+1\ra (j+1)n+g(j+s)+1}\otimes e_{jn+g(j+s)+1}.$$

(28) $Y^{(28)}_t$ is an $(6s\times 6s)$ matrix with a single nonzero
element:
$$y_{5s,5s}=w_{jn+5\ra (j+1)n+5}\otimes e_{jn+5}.$$

\section{$\Omega$-shifts of generators of the algebra $\HH^*(R)$}

Let $Q_\bullet\rightarrow R$ be the minimal projective bimodule
resolution of the algebra $R$, constructed in paragraph
\ref{sect_res}. Any $t$-cocycle $f\in\Ker\delta^t$ is lifted
(uniquely up to homotopy) to a chain map of complexes $\{\varphi_i:
Q_{t+i}\rightarrow Q_i\}_{i\ge 0}$. The homomorphism $\varphi_i$ is
called the {\it $i$th translate} of the cocycle $f$ and will be
denoted by $\Omega^i(f)$. For cocycles $f_1\in\Ker\delta^{t_1}$ and
$f_2\in\Ker\delta^{t_2}$ we have
\begin{equation}\tag{$*$}\label{mult_formula}
\cl f_2\cdot \cl f_1=\cl(\Omega^0(f_2)\Omega^{t_2}(f_1)).
\end{equation}

We shall now describe $\Omega$-translates for generators of the
algebra $\HH^*(R)$ and then find multiplications of the generators
using the formula \eqref{mult_formula}.

\begin{obozns}$\quad$

(1) For generator degree $t$ represent it in the form $t=11\ell+r$
($0\le r\le 10$) and denote by
$\kappa=(-1)^{\lfloor\frac{\ell}{2}\rfloor}$.

(2) For translate $\Omega^{t_0}$ represent $t_0$ in the form
$t_0=11\ell_0+r_0$ ($0\le r_0\le 10$) and denote by
$\kappa_0=(-1)^{\ell_0}$.
\end{obozns}

Define the helper functions $z_0\text{:
}\Z\times\Z\times\Z\rightarrow\Z$ and $z_1\text{:
}\Z\times\Z\times\Z\rightarrow\Z$, which act in the following way:
$$z_0(x,\ell,k)=(x-\ell k)_s+(\ell s)_{2s},\quad z_1(x,\ell,k)=(x-\ell k)_s+((\ell+1)s)_{2s}$$

\begin{pr}[Translates for the case 1]
$({\rm I})$ Let $r_0\in\N$, $r_0<11$. $r_0$-translates of the
elements $Y^{(1)}_t$ are described by the following way.

$(1)$ If $r_0=0$, then $\Omega^{0}(Y_t^{(1)})$ is described with
$(6s\times 6s)$-matrix with the following elements $b_{ij}${\rm:}

If $0\le j<s$, then $$b_{ij}=
\begin{cases}
\kappa e_{(j+m)n}\otimes e_{jn},\quad i=j;\\
0,\quad\text{otherwise.}\end{cases}$$

If $s\le j<3s$, then $$b_{ij}=
\begin{cases}
e_{(j+m)n+g(j+s)}\otimes e_{jn+g(j+s)},\quad i=j;\\
0,\quad\text{otherwise.}\end{cases}$$

If $3s\le j<5s$, then $$b_{ij}=
\begin{cases}
e_{(j+m)n+g(j+s)+1}\otimes e_{jn+g(j+s)+1},\quad i=j;\\
0,\quad\text{otherwise.}\end{cases}$$

If $5s\le j<6s$, then $$b_{ij}=
\begin{cases}
\kappa e_{(j+m)n+5}\otimes e_{jn+5},\quad i=j;\\
0,\quad\text{otherwise.}\end{cases}$$

$(2)$ If $r_0=1$, then $\Omega^{1}(Y_t^{(1)})$ is described with
$(7s\times 7s)$-matrix with the following elements $b_{ij}${\rm:}

If $0\le j<2s$, then $$b_{ij}=
\begin{cases}
e_{(j+m)n+g(j)}\otimes e_{jn},\quad i=j;\\
0,\quad\text{otherwise.}\end{cases}$$

If $2s\le j<4s$, then $$b_{ij}=
\begin{cases}
e_{(j+m)n+g(j)+1}\otimes e_{jn+g(j)},\quad i=j;\\
0,\quad\text{otherwise.}\end{cases}$$

If $4s\le j<6s$, then $$b_{ij}=
\begin{cases}
\kappa e_{(j+m)n+5}\otimes e_{jn+g(j)+1},\quad i=j;\\
0,\quad\text{otherwise.}\end{cases}$$

If $6s\le j<7s$, then $$b_{ij}=
\begin{cases}
\kappa e_{(j+m+1)n}\otimes e_{jn+5},\quad i=j;\\
0,\quad\text{otherwise.}\end{cases}$$

$(3)$ If $r_0=2$, then $\Omega^{2}(Y_t^{(1)})$ is described with
$(6s\times 6s)$-matrix with the following elements $b_{ij}${\rm:}

If $0\le j<s$, then $$b_{ij}=
\begin{cases}
\kappa e_{(j+m-1)n+5}\otimes e_{jn},\quad i=j;\\
0,\quad\text{otherwise.}\end{cases}$$

If $s\le j<3s$, then $$b_{ij}=
\begin{cases}
e_{(j+m)n+g(j+s)+1}\otimes e_{jn+g(j+s)},\quad i=j;\\
0,\quad\text{otherwise.}\end{cases}$$

If $3s\le j<5s$, then $$b_{ij}=
\begin{cases}
e_{(j+m)n+g(j)}\otimes e_{jn+g(j+s)+1},\quad i=j;\\
0,\quad\text{otherwise.}\end{cases}$$

If $5s\le j<6s$, then $$b_{ij}=
\begin{cases}
\kappa e_{(j+m+1)n}\otimes e_{jn+5},\quad i=j;\\
0,\quad\text{otherwise.}\end{cases}$$

$(4)$ If $r_0=3$, then $\Omega^{3}(Y_t^{(1)})$ is described with
$(8s\times 8s)$-matrix with the following elements $b_{ij}${\rm:}

If $0\le j<2s$, then $$b_{ij}=
\begin{cases}
e_{(j+m)n+g(j)+1}\otimes e_{jn},\quad i=j;\\
0,\quad\text{otherwise.}\end{cases}$$

If $2s\le j<4s$, then $$b_{ij}=
\begin{cases}
\kappa e_{(j+m)n+5}\otimes e_{jn+g(j)},\quad i=j;\\
0,\quad\text{otherwise.}\end{cases}$$

If $4s\le j<6s$, then $$b_{ij}=
\begin{cases}
\kappa e_{(j+m+1)n}\otimes e_{jn+g(j)+1},\quad i=j;\\
0,\quad\text{otherwise.}\end{cases}$$

If $6s\le j<8s$, then $$b_{ij}=
\begin{cases}
e_{(j+m+1)n+g(j)}\otimes e_{jn+5},\quad i=j;\\
0,\quad\text{otherwise.}\end{cases}$$

$(5)$ If $r_0=4$, then $\Omega^{4}(Y_t^{(1)})$ is described with
$(9s\times 9s)$-matrix with the following elements $b_{ij}${\rm:}

If $0\le j<s$, then $$b_{ij}=
\begin{cases}
\kappa e_{(j+m-1)n+5}\otimes e_{jn},\quad i=j;\\
0,\quad\text{otherwise.}\end{cases}$$

If $s\le j<2s$, then $$b_{ij}=
\begin{cases}
\kappa e_{(j+m)n}\otimes e_{jn},\quad i=j;\\
0,\quad\text{otherwise.}\end{cases}$$

If $2s\le j<4s$, then $$b_{ij}=
\begin{cases}
e_{(j+m)n+g(j+s)}\otimes e_{jn+g(j)},\quad i=j;\\
0,\quad\text{otherwise.}\end{cases}$$

If $4s\le j<6s$, then $$b_{ij}=
\begin{cases}
e_{(j+m)n+g(j)}\otimes e_{jn+g(j)+1},\quad i=j;\\
0,\quad\text{otherwise.}\end{cases}$$

If $6s\le j<8s$, then $$b_{ij}=
\begin{cases}
e_{(j+m)n+g(j+s)+1}\otimes e_{jn+g(j)+1},\quad i=j;\\
0,\quad\text{otherwise.}\end{cases}$$

If $8s\le j<9s$, then $$b_{ij}=
\begin{cases}
\kappa e_{(j+m)n+5}\otimes e_{jn+5},\quad i=j;\\
0,\quad\text{otherwise.}\end{cases}$$

$(6)$ If $r_0=5$, then $\Omega^{5}(Y_t^{(1)})$ is described with
$(8s\times 8s)$-matrix with the following elements $b_{ij}${\rm:}

If $0\le j<2s$, then $$b_{ij}=
\begin{cases}
e_{(j+m)n+g(j)}\otimes e_{jn},\quad i=j;\\
0,\quad\text{otherwise.}\end{cases}$$

If $2s\le j<4s$, then $$b_{ij}=
\begin{cases}
\kappa e_{(j+m+1)n}\otimes e_{jn+g(j)},\quad i=j;\\
0,\quad\text{otherwise.}\end{cases}$$

If $4s\le j<6s$, then $$b_{ij}=
\begin{cases}
\kappa e_{(j+m)n+5}\otimes e_{jn+g(j)+1},\quad i=j;\\
0,\quad\text{otherwise.}\end{cases}$$

If $6s\le j<8s$, then $$b_{ij}=
\begin{cases}
e_{(j+m+1)n+g(j)+1}\otimes e_{jn+5},\quad i=j;\\
0,\quad\text{otherwise.}\end{cases}$$

$(7)$ If $r_0=6$, then $\Omega^{6}(Y_t^{(1)})$ is described with
$(9s\times 9s)$-matrix with the following elements $b_{ij}${\rm:}

If $0\le j<s$, then $$b_{ij}=
\begin{cases}
\kappa e_{(j+m)n}\otimes e_{jn},\quad i=j;\\
0,\quad\text{otherwise.}\end{cases}$$

If $s\le j<3s$, then $$b_{ij}=
\begin{cases}
e_{(j+m)n+g(j+s)}\otimes e_{jn+g(j+s)},\quad i=j;\\
0,\quad\text{otherwise.}\end{cases}$$

If $3s\le j<5s$, then $$b_{ij}=
\begin{cases}
e_{(j+m)n+g(j)+1}\otimes e_{jn+g(j+s)},\quad i=j;\\
0,\quad\text{otherwise.}\end{cases}$$

If $5s\le j<7s$, then $$b_{ij}=
\begin{cases}
e_{(j+m)n+g(j+s)+1}\otimes e_{jn+g(j+s)+1},\quad i=j;\\
0,\quad\text{otherwise.}\end{cases}$$

If $7s\le j<8s$, then $$b_{ij}=
\begin{cases}
\kappa e_{(j+m)n+5}\otimes e_{jn+5},\quad i=j;\\
0,\quad\text{otherwise.}\end{cases}$$

If $8s\le j<9s$, then $$b_{ij}=
\begin{cases}
\kappa e_{(j+m+1)n}\otimes e_{jn+5},\quad i=j;\\
0,\quad\text{otherwise.}\end{cases}$$

$(8)$ If $r_0=7$, then $\Omega^{7}(Y_t^{(1)})$ is described with
$(8s\times 8s)$-matrix with the following elements $b_{ij}${\rm:}

If $0\le j<2s$, then $$b_{ij}=
\begin{cases}
e_{(j+m)n+g(j)+1}\otimes e_{jn},\quad i=j;\\
0,\quad\text{otherwise.}\end{cases}$$

If $2s\le j<4s$, then $$b_{ij}=
\begin{cases}
\kappa e_{(j+m)n+5}\otimes e_{jn+g(j)},\quad i=j;\\
0,\quad\text{otherwise.}\end{cases}$$

If $4s\le j<6s$, then $$b_{ij}=
\begin{cases}
\kappa e_{(j+m+1)n}\otimes e_{jn+g(j)+1},\quad i=j;\\
0,\quad\text{otherwise.}\end{cases}$$

If $6s\le j<8s$, then $$b_{ij}=
\begin{cases}
e_{(j+m+1)n+g(j)}\otimes e_{jn+5},\quad i=j;\\
0,\quad\text{otherwise.}\end{cases}$$

$(9)$ If $r_0=8$, then $\Omega^{8}(Y_t^{(1)})$ is described with
$(6s\times 6s)$-matrix with the following elements $b_{ij}${\rm:}

If $0\le j<s$, then $$b_{ij}=
\begin{cases}
\kappa e_{(j+m-1)n+5}\otimes e_{jn},\quad i=j;\\
0,\quad\text{otherwise.}\end{cases}$$

If $s\le j<3s$, then $$b_{ij}=
\begin{cases}
e_{(j+m)n+g(j+s)+1}\otimes e_{jn+g(j+s)},\quad i=j;\\
0,\quad\text{otherwise.}\end{cases}$$

If $3s\le j<5s$, then $$b_{ij}=
\begin{cases}
e_{(j+m)n+g(j)}\otimes e_{jn+g(j+s)+1},\quad i=j;\\
0,\quad\text{otherwise.}\end{cases}$$

If $5s\le j<6s$, then $$b_{ij}=
\begin{cases}
\kappa e_{(j+m+1)n}\otimes e_{jn+5},\quad i=j;\\
0,\quad\text{otherwise.}\end{cases}$$

$(10)$ If $r_0=9$, then $\Omega^{9}(Y_t^{(1)})$ is described with
$(7s\times 7s)$-matrix with the following elements $b_{ij}${\rm:}

If $0\le j<s$, then $$b_{ij}=
\begin{cases}
\kappa e_{(j+m)n+5}\otimes e_{jn},\quad i=j;\\
0,\quad\text{otherwise.}\end{cases}$$

If $s\le j<3s$, then $$b_{ij}=
\begin{cases}
\kappa e_{(j+m+1)n}\otimes e_{jn+g(j+s)},\quad i=j;\\
0,\quad\text{otherwise.}\end{cases}$$

If $3s\le j<5s$, then $$b_{ij}=
\begin{cases}
e_{(j+m+1)n+g(j)}\otimes e_{jn+g(j+s)+1},\quad i=j;\\
0,\quad\text{otherwise.}\end{cases}$$

If $5s\le j<7s$, then $$b_{ij}=
\begin{cases}
e_{(j+m+1)n+g(j+s)+1}\otimes e_{jn+5},\quad i=j;\\
0,\quad\text{otherwise.}\end{cases}$$

$(11)$ If $r_0=10$, then $\Omega^{10}(Y_t^{(1)})$ is described with
$(6s\times 6s)$-matrix with the following elements $b_{ij}${\rm:}

If $0\le j<s$, then $$b_{ij}=
\begin{cases}
\kappa e_{(j+m)n}\otimes e_{jn},\quad i=j;\\
0,\quad\text{otherwise.}\end{cases}$$

If $s\le j<3s$, then $$b_{ij}=
\begin{cases}
e_{(j+m)n+g(j)}\otimes e_{jn+g(j+s)},\quad i=j;\\
0,\quad\text{otherwise.}\end{cases}$$

If $3s\le j<5s$, then $$b_{ij}=
\begin{cases}
e_{(j+m)n+g(j)+1}\otimes e_{jn+g(j+s)+1},\quad i=j;\\
0,\quad\text{otherwise.}\end{cases}$$

If $5s\le j<6s$, then $$b_{ij}=
\begin{cases}
\kappa e_{(j+m)n+5}\otimes e_{jn+5},\quad i=j;\\
0,\quad\text{otherwise.}\end{cases}$$

\medskip
$({\rm II})$ Represent an arbitrary $t_0\in\N$ in the form
$t_0=11\ell_0+r_0$, where $0\le r_0\le 10.$ Then
$\Omega^{t_0}(Y_t^{(1)})$ is a $\Omega^{r_0}(Y_t^{(1)})$, whose left
components twisted by $\sigma^{\ell_0}$.
\end{pr}
\begin{pr}[Translates for the case 2]
$({\rm I})$ Let $r_0\in\N$, $r_0<11$. $r_0$-translates of the
elements $Y^{(2)}_t$ are described by the following way.

$(1)$ If $r_0=0$, then $\Omega^{0}(Y_t^{(2)})$ is described with
$(6s\times 6s)$-matrix with one nonzero element that is of the
following form{\rm:}
$$b_{3s+z_0(0,\ell_0,n),3s+z_0(0,\ell_0,n)}=w_{(j+m)n+g(j+s)+1\ra (j+m+1)n+g(j+s)+1}\otimes e_{jn+g(j+s)+1}.$$

$(2)$ If $r_0=1$, then $\Omega^{1}(Y_t^{(2)})$ is described with
$(7s\times 7s)$-matrix with one nonzero element that is of the
following form{\rm:}
$$b_{2s+z_0(0,\ell_0,n),2s+z_0(0,\ell_0,n)}=w_{(j+m)n+g(j)+1\ra (j+m+1)n+g(j)+1}\otimes e_{jn+g(j)}.$$

$(3)$ If $r_0=2$, then $\Omega^{2}(Y_t^{(2)})$ is described with
$(6s\times 6s)$-matrix with one nonzero element that is of the
following form{\rm:}
$$b_{s+z_0(-1,\ell_0,n),s+z_0(-1,\ell_0,n)}=w_{(j+m)n+g(j+s)+1\ra (j+m+1)n+g(j+s)+1}\otimes e_{jn+g(j+s)}.$$

$(4)$ If $r_0=3$, then $\Omega^{3}(Y_t^{(2)})$ is described with
$(8s\times 8s)$-matrix with one nonzero element that is of the
following form{\rm:}
$$b_{z_0(-1,\ell_0,n),z_0(-1,\ell_0,n)}=w_{(j+m)n+g(j)+1\ra (j+m+1)n+g(j)+1}\otimes e_{jn}.$$

$(5)$ If $r_0=4$, then $\Omega^{4}(Y_t^{(2)})$ is described with
$(9s\times 9s)$-matrix with one nonzero element that is of the
following form{\rm:}
$$b_{6s+z_1(-2,\ell_0,n),6s+z_1(-2,\ell_0,n)}=w_{(j+m)n+g(j+s)+1\ra (j+m+1)n+g(j+s)+1}\otimes e_{jn+g(j)+1}.$$

$(6)$ If $r_0=5$, then $\Omega^{5}(Y_t^{(2)})$ is described with
$(8s\times 8s)$-matrix with one nonzero element that is of the
following form{\rm:}
$$b_{6s+z_0(-3,\ell_0,n),6s+z_0(-3,\ell_0,n)}=w_{(j+m+1)n+g(j)+1\ra (j+m+2)n+g(j)+1}\otimes e_{jn+5}.$$

$(7)$ If $r_0=6$, then $\Omega^{6}(Y_t^{(2)})$ is described with
$(9s\times 9s)$-matrix with the following elements $b_{ij}${\rm:}

If $0\le j<3s+z_1(-3,\ell_0,n)$, then $b_{ij}=0$.

If $j=3s+z_1(-3,\ell_0,n)$, then $$b_{ij}=
\begin{cases}
w_{(j+m)n+g(j)+1\ra (j+m+1)n+g(j)+1}\otimes e_{jn+g(j+s)},\quad i=j;\\
0,\quad\text{otherwise.}\end{cases}$$

If $3s+z_1(-3,\ell_0,n)<j<5s+z_0(-3,\ell_0,n)$, then $b_{ij}=0$.

If $j=5s+z_0(-3,\ell_0,n)$, then $$b_{ij}=
\begin{cases}
w_{(j+m)n+g(j+s)+1\ra (j+m+1)n+g(j+s)+1}\otimes e_{jn+g(j+s)+1},\quad i=j;\\
0,\quad\text{otherwise.}\end{cases}$$

If $5s+z_0(-3,\ell_0,n)<j<9s$, then $b_{ij}=0$.

$(8)$ If $r_0=7$, then $\Omega^{7}(Y_t^{(2)})$ is described with
$(8s\times 8s)$-matrix with one nonzero element that is of the
following form{\rm:}
$$b_{z_0(-3,\ell_0,n),z_0(-3,\ell_0,n)}=w_{(j+m)n+g(j)+1\ra (j+m+1)n+g(j)+1}\otimes e_{jn}.$$

$(9)$ If $r_0=8$, then $\Omega^{8}(Y_t^{(2)})$ is described with
$(6s\times 6s)$-matrix with one nonzero element that is of the
following form{\rm:}
$$b_{s+z_0(-4,\ell_0,n),s+z_0(-4,\ell_0,n)}=w_{(j+m)n+g(j+s)+1\ra (j+m+1)n+g(j+s)+1}\otimes e_{jn+g(j+s)}.$$

$(10)$ If $r_0=9$, then $\Omega^{9}(Y_t^{(2)})$ is described with
$(7s\times 7s)$-matrix with one nonzero element that is of the
following form{\rm:}
$$b_{5s+z_0(-5,\ell_0,n),5s+z_0(-5,\ell_0,n)}=w_{(j+m+1)n+g(j+s)+1\ra (j+m+2)n+g(j+s)+1}\otimes e_{jn+5}.$$

$(11)$ If $r_0=10$, then $\Omega^{10}(Y_t^{(2)})$ is described with
$(6s\times 6s)$-matrix with one nonzero element that is of the
following form{\rm:}
$$b_{3s+z_1(-5,\ell_0,n),3s+z_1(-5,\ell_0,n)}=w_{(j+m)n+g(j)+1\ra (j+m+1)n+g(j)+1}\otimes e_{jn+g(j+s)+1}.$$

\medskip
$({\rm II})$ Represent an arbitrary $t_0\in\N$ in the form
$t_0=11\ell_0+r_0$, where $0\le r_0\le 10.$ Then
$\Omega^{t_0}(Y_t^{(2)})$ is a $\Omega^{r_0}(Y_t^{(2)})$, whose left
components twisted by $\sigma^{\ell_0}$, and coefficients multiplied
by $(-1)^{\ell_0}$.
\end{pr}
\begin{pr}[Translates for the case 3]\label{shifts_3}
$({\rm I})$ Let $r_0\in\N$, $r_0<11$. $r_0$-translates of the
elements $Y^{(3)}_t$ are described by the following way.

$(1)$ If $r_0=0$, then $\Omega^{0}(Y_t^{(3)})$ is described with
$(7s\times 6s)$-matrix with one nonzero element that is of the
following form{\rm:}
$$b_{5s+(-3\ell_0)_s,6s+(-3\ell_0)_s}=\kappa w_{(j+m)n+5\ra (j+m+1)n}\otimes e_{jn+5}.$$

$(2)$ If $r_0=1$, then $\Omega^{1}(Y_t^{(3)})$ is described with
$(6s\times 7s)$-matrix with the following elements $b_{ij}${\rm:}

If $0\le j<s+(-3\ell_0)_s$, then $b_{ij}=0$.

If $j=s+(-3\ell_0)_s$, then $$b_{ij}=
\begin{cases}
-w_{(j+m)n+g(j+s)+1\ra (j+m+1)n+g(j+s)+1}\otimes e_{jn+g(j+s)},\quad i=j+s;\\
-w_{(j+m)n+5\ra (j+m+1)n+g(j+s)+1}\otimes w_{jn+g(j+s)\ra jn+g(j+s)+1},\quad i=j+3s;\\
0,\quad\text{otherwise.}\end{cases}$$

If $s+(-3\ell_0)_s<j<2s+(-3\ell_0)_s$, then $b_{ij}=0$.

If $j=2s+(-3\ell_0)_s$, then $$b_{ij}=
\begin{cases}
-w_{(j+m)n+g(j+s)+1\ra (j+m+1)n+g(j+s)+1}\otimes e_{jn+g(j+s)},\quad i=j+s;\\
-w_{(j+m)n+5\ra (j+m+1)n+g(j+s)+1}\otimes w_{jn+g(j+s)\ra jn+g(j+s)+1},\quad i=j+3s;\\
0,\quad\text{otherwise.}\end{cases}$$

If $2s+(-3\ell_0)_s<j<3s+(-3\ell_0)_s$, then $b_{ij}=0$.

If $j=3s+(-3\ell_0)_s$, then $$b_{ij}=
\begin{cases}
-w_{(j+m)n+5\ra (j+m+1)n+g(j)}\otimes e_{jn+g(j+s)+1},\quad i=j+s;\\
0,\quad\text{otherwise.}\end{cases}$$

If $3s+(-3\ell_0)_s<j<4s+(-3\ell_0)_s$, then $b_{ij}=0$.

If $j=4s+(-3\ell_0)_s$, then $$b_{ij}=
\begin{cases}
-w_{(j+m)n+5\ra (j+m+1)n+g(j)}\otimes e_{jn+g(j+s)+1},\quad i=j+s;\\
0,\quad\text{otherwise.}\end{cases}$$

If $4s+(-3\ell_0)_s<j<5s+(-1-3\ell_0)_s$, then $b_{ij}=0$.

If $j=5s+(-1-3\ell_0)_s$, then $$b_{ij}=
\begin{cases}
\kappa e_{(j+m+2)n}\otimes w_{jn+5\ra (j+1)n+5},\quad i=(j+1)_s+6s;\\
0,\quad\text{otherwise.}\end{cases}$$

If $5s+(-1-3\ell_0)_s<j<6s$, then $b_{ij}=0$.

$(3)$ If $r_0=2$, then $\Omega^{2}(Y_t^{(3)})$ is described with
$(8s\times 6s)$-matrix with the following elements $b_{ij}${\rm:}

If $0\le j<(-3\ell_0)_s$, then $b_{ij}=0$.

If $j=(-3\ell_0)_s$, then $$b_{ij}=
\begin{cases}
w_{(j+m-1)n+5\ra (j+m)n+g(j)+1}\otimes e_{jn},\quad i=j;\\
0,\quad\text{otherwise.}\end{cases}$$

If $(-3\ell_0)_s<j<s+(-3\ell_0)_s$, then $b_{ij}=0$.

If $j=s+(-3\ell_0)_s$, then $$b_{ij}=
\begin{cases}
w_{(j+m-1)n+5\ra (j+m)n+g(j)+1}\otimes e_{jn},\quad i=j-s;\\
0,\quad\text{otherwise.}\end{cases}$$

If $s+(-3\ell_0)_s<j<6s+(-1-3\ell_0)_s$, then $b_{ij}=0$.

If $j=6s+(-1-3\ell_0)_s$, then $$b_{ij}=
\begin{cases}
e_{(j+m+1)n+g(j)}\otimes w_{jn+5\ra (j+1)n+g(j+s)+1},\quad i=(j+1)_s+4s;\\
0,\quad\text{otherwise.}\end{cases}$$

If $6s+(-1-3\ell_0)_s<j<7s+(-1-3\ell_0)_s$, then $b_{ij}=0$.

If $j=7s+(-1-3\ell_0)_s$, then $$b_{ij}=
\begin{cases}
e_{(j+m+1)n+g(j)}\otimes w_{jn+5\ra (j+1)n+g(j+s)+1},\quad i=(j+1)_s+3s;\\
0,\quad\text{otherwise.}\end{cases}$$

If $7s+(-1-3\ell_0)_s<j<8s$, then $b_{ij}=0$.

$(4)$ If $r_0=3$, then $\Omega^{3}(Y_t^{(3)})$ is described with
$(9s\times 8s)$-matrix with the following elements $b_{ij}${\rm:}

If $0\le j<6s+(-1-3\ell_0)_s$, then $b_{ij}=0$.

If $j=6s+(-1-3\ell_0)_s$, then $$b_{ij}=
\begin{cases}
-e_{(j+m+1)n+g(j+s)+1}\otimes w_{jn+g(j)+1\ra (j+1)n},\quad i=(j+1)_s+s;\\
0,\quad\text{otherwise.}\end{cases}$$

If $6s+(-1-3\ell_0)_s<j<7s+(-1-3\ell_0)_s$, then $b_{ij}=0$.

If $j=7s+(-1-3\ell_0)_s$, then $$b_{ij}=
\begin{cases}
e_{(j+m+1)n+g(j+s)+1}\otimes w_{jn+g(j)+1\ra (j+1)n},\quad i=(j+1)_s;\\
0,\quad\text{otherwise.}\end{cases}$$

If $7s+(-1-3\ell_0)_s<j<8s+(-1-3\ell_0)_s$, then $b_{ij}=0$.

If $j=8s+(-1-3\ell_0)_s$, then $$b_{ij}=
\begin{cases}
-\kappa w_{(j+m+1)n+g(j)\ra (j+m+1)n+5}\otimes e_{jn+5},\quad i=j-2s;\\
\kappa w_{(j+m+1)n+g(j+s)\ra (j+m+1)n+5}\otimes e_{jn+5},\quad i=j-s;\\
0,\quad\text{otherwise.}\end{cases}$$

If $8s+(-1-3\ell_0)_s<j<9s$, then $b_{ij}=0$.

$(5)$ If $r_0=4$, then $\Omega^{4}(Y_t^{(3)})$ is described with
$(8s\times 9s)$-matrix with the following elements $b_{ij}${\rm:}

If $0\le j<2s+(-1-3\ell_0)_s$, then $b_{ij}=0$.

If $j=2s+(-1-3\ell_0)_s$, then $$b_{ij}=
\begin{cases}
-\kappa w_{(j+m)n+5\ra (j+m+1)n}\otimes w_{jn+g(j)\ra (j+1)n},\quad i=(j+1)_s;\\
\kappa e_{(j+m+1)n}\otimes w_{jn+g(j)\ra (j+1)n},\quad i=(j+1)_s+s;\\
0,\quad\text{otherwise.}\end{cases}$$

If $2s+(-1-3\ell_0)_s<j<3s+(-1-3\ell_0)_s$, then $b_{ij}=0$.

If $j=3s+(-1-3\ell_0)_s$, then $$b_{ij}=
\begin{cases}
-\kappa e_{(j+m+1)n}\otimes w_{jn+g(j)\ra (j+1)n},\quad i=(j+1)_s+s;\\
0,\quad\text{otherwise.}\end{cases}$$

If $3s+(-1-3\ell_0)_s<j<4s+(-1-3\ell_0)_s$, then $b_{ij}=0$.

If $j=4s+(-1-3\ell_0)_s$, then $$b_{ij}=
\begin{cases}
-\kappa e_{(j+m)n+5}\otimes w_{jn+g(j)+1\ra (j+1)n},\quad i=(j+1)_s;\\
\kappa e_{(j+m)n+5}\otimes w_{jn+g(j)+1\ra jn+5},\quad i=j+4s;\\
0,\quad\text{otherwise.}\end{cases}$$

If $4s+(-1-3\ell_0)_s<j<5s+(-1-3\ell_0)_s$, then $b_{ij}=0$.

If $j=5s+(-1-3\ell_0)_s$, then $$b_{ij}=
\begin{cases}
-\kappa e_{(j+m)n+5}\otimes w_{jn+g(j)+1\ra jn+5},\quad i=j+3s;\\
0,\quad\text{otherwise.}\end{cases}$$

If $5s+(-1-3\ell_0)_s<j<8s$, then $b_{ij}=0$.

$(6)$ If $r_0=5$, then $\Omega^{5}(Y_t^{(3)})$ is described with
$(9s\times 8s)$-matrix with the following elements $b_{ij}${\rm:}

If $0\le j<s+(-1-3\ell_0)_s$, then $b_{ij}=0$.

If $j=s+(-1-3\ell_0)_s$, then $$b_{ij}=
\begin{cases}
-e_{(j+m+1)n+g(j+s)}\otimes w_{jn+g(j+s)\ra (j+1)n},\quad i=(j+1)_s;\\
0,\quad\text{otherwise.}\end{cases}$$

If $s+(-1-3\ell_0)_s<j<2s+(-1-3\ell_0)_s$, then $b_{ij}=0$.

If $j=2s+(-1-3\ell_0)_s$, then $$b_{ij}=
\begin{cases}
-e_{(j+m+1)n+g(j+s)}\otimes w_{jn+g(j+s)\ra (j+1)n},\quad i=(j+1)_s+s;\\
0,\quad\text{otherwise.}\end{cases}$$

If $2s+(-1-3\ell_0)_s<j<3s+(-1-3\ell_0)_s$, then $b_{ij}=0$.

If $j=3s+(-1-3\ell_0)_s$, then $$b_{ij}=
\begin{cases}
-w_{(j+m+1)n\ra (j+m+1)n+g(j)+1}\otimes e_{jn+g(j+s)},\quad i=j-s;\\
w_{(j+m)n+5\ra (j+m+1)n+g(j)+1}\otimes w_{jn+g(j+s)\ra jn+g(j+s)+1},\quad i=j+s;\\
0,\quad\text{otherwise.}\end{cases}$$

If $3s+(-1-3\ell_0)_s<j<4s+(-1-3\ell_0)_s$, then $b_{ij}=0$.

If $j=4s+(-1-3\ell_0)_s$, then $$b_{ij}=
\begin{cases}
-w_{(j+m+1)n\ra (j+m+1)n+g(j)+1}\otimes e_{jn+g(j+s)},\quad i=j-s;\\
w_{(j+m)n+5\ra (j+m+1)n+g(j)+1}\otimes w_{jn+g(j+s)\ra jn+g(j+s)+1},\quad i=j+s;\\
0,\quad\text{otherwise.}\end{cases}$$

If $4s+(-1-3\ell_0)_s<j<5s+(-1-3\ell_0)_s$, then $b_{ij}=0$.

If $j=5s+(-1-3\ell_0)_s$, then $$b_{ij}=
\begin{cases}
-w_{(j+m)n+5\ra (j+m+1)n+g(j+s)+1}\otimes e_{jn+g(j+s)+1},\quad i=j-s;\\
0,\quad\text{otherwise.}\end{cases}$$

If $5s+(-1-3\ell_0)_s<j<6s+(-1-3\ell_0)_s$, then $b_{ij}=0$.

If $j=6s+(-1-3\ell_0)_s$, then $$b_{ij}=
\begin{cases}
-w_{(j+m)n+5\ra (j+m+1)n+g(j+s)+1}\otimes e_{jn+g(j+s)+1},\quad i=j-s;\\
0,\quad\text{otherwise.}\end{cases}$$

If $6s+(-1-3\ell_0)_s<j<9s$, then $b_{ij}=0$.

$(7)$ If $r_0=6$, then $\Omega^{6}(Y_t^{(3)})$ is described with
$(8s\times 9s)$-matrix with the following elements $b_{ij}${\rm:}

If $0\le j<(-1-3\ell_0)_s$, then $b_{ij}=0$.

If $j=(-1-3\ell_0)_s$, then $$b_{ij}=
\begin{cases}
w_{(j+m)n\ra (j+m)n+g(j)+1}\otimes e_{jn},\quad i=j;\\
-w_{(j+m)n+g(j)\ra (j+m)n+g(j)+1}\otimes w_{jn\ra jn+g(j)},\quad i=j+s;\\
0,\quad\text{otherwise.}\end{cases}$$

If $(-1-3\ell_0)_s<j<s+(-1-3\ell_0)_s$, then $b_{ij}=0$.

If $j=s+(-1-3\ell_0)_s$, then $$b_{ij}=
\begin{cases}
w_{(j+m)n\ra (j+m)n+g(j)+1}\otimes e_{jn},\quad i=j-s;\\
-w_{(j+m)n+g(j)\ra (j+m)n+g(j)+1}\otimes w_{jn\ra jn+g(j)},\quad i=j+s;\\
0,\quad\text{otherwise.}\end{cases}$$

If $s+(-1-3\ell_0)_s<j<2s+(-1-3\ell_0)_s$, then $b_{ij}=0$.

If $j=2s+(-1-3\ell_0)_s$, then $$b_{ij}=
\begin{cases}
\kappa w_{(j+m)n+g(j)\ra (j+m)n+5}\otimes e_{jn+g(j)},\quad i=j-s;\\
0,\quad\text{otherwise.}\end{cases}$$

If $2s+(-1-3\ell_0)_s<j<3s+(-1-3\ell_0)_s$, then $b_{ij}=0$.

If $j=3s+(-1-3\ell_0)_s$, then $$b_{ij}=
\begin{cases}
\kappa w_{(j+m)n+g(j)\ra (j+m)n+5}\otimes e_{jn+g(j)},\quad i=j-s;\\
0,\quad\text{otherwise.}\end{cases}$$

If $3s+(-1-3\ell_0)_s<j<8s$, then $b_{ij}=0$.

$(8)$ If $r_0=7$, then $\Omega^{7}(Y_t^{(3)})$ is described with
$(6s\times 8s)$-matrix with the following elements $b_{ij}${\rm:}

If $0\le j<s+(-2-3\ell_0)_s$, then $b_{ij}=0$.

If $j=s+(-2-3\ell_0)_s$, then $$b_{ij}=
\begin{cases}
-e_{(j+m+1)n+g(j+s)+1}\otimes w_{jn+g(j+s)\ra (j+1)n},\quad i=(j+1)_s;\\
0,\quad\text{otherwise.}\end{cases}$$

If $s+(-2-3\ell_0)_s<j<2s+(-2-3\ell_0)_s$, then $b_{ij}=0$.

If $j=2s+(-2-3\ell_0)_s$, then $$b_{ij}=
\begin{cases}
-e_{(j+m+1)n+g(j+s)+1}\otimes w_{jn+g(j+s)\ra (j+1)n},\quad i=(j+1)_s+s;\\
0,\quad\text{otherwise.}\end{cases}$$

If $2s+(-2-3\ell_0)_s<j<5s+(-2-3\ell_0)_s$, then $b_{ij}=0$.

If $j=5s+(-2-3\ell_0)_s$, then $$b_{ij}=
\begin{cases}
-\kappa w_{(j+m+1)n+g(j+s)\ra (j+m+2)n}\otimes e_{jn+5},\quad i=j+s;\\
-\kappa w_{(j+m+1)n+g(j)\ra (j+m+2)n}\otimes e_{jn+5},\quad i=j+2s;\\
0,\quad\text{otherwise.}\end{cases}$$

If $5s+(-2-3\ell_0)_s<j<6s$, then $b_{ij}=0$.

$(9)$ If $r_0=8$, then $\Omega^{8}(Y_t^{(3)})$ is described with
$(7s\times 6s)$-matrix with the following elements $b_{ij}${\rm:}

If $0\le j<(-2-3\ell_0)_s$, then $b_{ij}=0$.

If $j=(-2-3\ell_0)_s$, then $$b_{ij}=
\begin{cases}
\kappa e_{(j+m)n+5}\otimes w_{jn\ra (j+1)n},\quad i=(j+1)_s;\\
0,\quad\text{otherwise.}\end{cases}$$

If $(-2-3\ell_0)_s<j<s+(-2-3\ell_0)_s$, then $b_{ij}=0$.

If $j=s+(-2-3\ell_0)_s$, then $$b_{ij}=
\begin{cases}
\kappa w_{(j+m)n+g(j+s)+1\ra (j+m+1)n}\otimes e_{jn+g(j+s)},\quad i=j;\\
-\kappa w_{(j+m)n+g(j)\ra (j+m+1)n}\otimes w_{jn+g(j+s)\ra jn+g(j+s)+1},\quad i=j+2s;\\
0,\quad\text{otherwise.}\end{cases}$$

If $s+(-2-3\ell_0)_s<j<2s+(-2-3\ell_0)_s$, then $b_{ij}=0$.

If $j=2s+(-2-3\ell_0)_s$, then $$b_{ij}=
\begin{cases}
-\kappa e_{(j+m+1)n}\otimes w_{jn+g(j+s)\ra jn+5},\quad i=j+3s;\\
0,\quad\text{otherwise.}\end{cases}$$

If $2s+(-2-3\ell_0)_s<j<3s+(-2-3\ell_0)_s$, then $b_{ij}=0$.

If $j=3s+(-2-3\ell_0)_s$, then $$b_{ij}=
\begin{cases}
w_{(j+m)n+g(j)\ra (j+m+1)n+g(j)}\otimes e_{jn+g(j+s)+1},\quad i=j;\\
0,\quad\text{otherwise.}\end{cases}$$

If $3s+(-2-3\ell_0)_s<j<4s+(-2-3\ell_0)_s$, then $b_{ij}=0$.

If $j=4s+(-2-3\ell_0)_s$, then $$b_{ij}=
\begin{cases}
w_{(j+m)n+g(j)\ra (j+m+1)n+g(j)}\otimes e_{jn+g(j+s)+1},\quad i=j;\\
0,\quad\text{otherwise.}\end{cases}$$

If $4s+(-2-3\ell_0)_s<j<7s$, then $b_{ij}=0$.

$(10)$ If $r_0=9$, then $\Omega^{9}(Y_t^{(3)})$ is described with
$(6s\times 7s)$-matrix with the following elements $b_{ij}${\rm:}

If $0\le j<(-2-3\ell_0)_s$, then $b_{ij}=0$.

If $j=(-2-3\ell_0)_s$, then $$b_{ij}=
\begin{cases}
\kappa e_{(j+m+1)n}\otimes w_{jn\ra jn+g(j)},\quad i=j+s;\\
0,\quad\text{otherwise.}\end{cases}$$

If $(-2-3\ell_0)_s<j<2s+(-2-3\ell_0)_s$, then $b_{ij}=0$.

If $j=2s+(-2-3\ell_0)_s$, then $$b_{ij}=
\begin{cases}
-w_{(j+m+1)n\ra (j+m+1)n+g(j)}\otimes e_{jn+g(j+s)},\quad i=j;\\
0,\quad\text{otherwise.}\end{cases}$$

If $2s+(-2-3\ell_0)_s<j<6s$, then $b_{ij}=0$.

$(11)$ If $r_0=10$, then $\Omega^{10}(Y_t^{(3)})$ is described with
$(6s\times 6s)$-matrix with the following elements $b_{ij}${\rm:}

If $0\le j<(-2-3\ell_0)_s$, then $b_{ij}=0$.

If $j=(-2-3\ell_0)_s$, then $$b_{ij}=
\begin{cases}
\kappa w_{(j+m)n\ra (j+m+1)n}\otimes e_{jn},\quad i=j;\\
0,\quad\text{otherwise.}\end{cases}$$

If $(-2-3\ell_0)_s<j<s+(-3-3\ell_0)_s$, then $b_{ij}=0$.

If $j=s+(-3-3\ell_0)_s$, then $$b_{ij}=
\begin{cases}
e_{(j+m+1)n+g(j)}\otimes w_{jn+g(j+s)\ra (j+1)n+g(j+s)},\quad i=(j+1)_s+s;\\
0,\quad\text{otherwise.}\end{cases}$$

If $s+(-3-3\ell_0)_s<j<2s+(-3-3\ell_0)_s$, then $b_{ij}=0$.

If $j=2s+(-3-3\ell_0)_s$, then $$b_{ij}=
\begin{cases}
e_{(j+m+1)n+g(j)}\otimes w_{jn+g(j+s)\ra (j+1)n+g(j+s)},\quad i=(j+1)_s+2s;\\
0,\quad\text{otherwise.}\end{cases}$$

If $2s+(-3-3\ell_0)_s<j<3s+(-3-3\ell_0)_s$, then $b_{ij}=0$.

If $j=3s+(-3-3\ell_0)_s$, then $$b_{ij}=
\begin{cases}
-w_{(j+m+1)n+g(j)\ra (j+m+1)n+g(j)+1}\otimes w_{jn+g(j+s)+1\ra (j+1)n+g(j+s)},\quad i=(j+1)_s+s;\\
e_{(j+m+1)n+g(j)+1}\otimes w_{jn+g(j+s)+1\ra (j+1)n+g(j+s)+1},\quad i=(j+1)_s+3s;\\
0,\quad\text{otherwise.}\end{cases}$$

If $3s+(-3-3\ell_0)_s<j<4s+(-3-3\ell_0)_s$, then $b_{ij}=0$.

If $j=4s+(-3-3\ell_0)_s$, then $$b_{ij}=
\begin{cases}
-w_{(j+m+1)n+g(j)\ra (j+m+1)n+g(j)+1}\otimes w_{jn+g(j+s)+1\ra (j+1)n+g(j+s)},\quad i=(j+1)_s+2s;\\
e_{(j+m+1)n+g(j)+1}\otimes w_{jn+g(j+s)+1\ra (j+1)n+g(j+s)+1},\quad i=(j+1)_s+4s;\\
0,\quad\text{otherwise.}\end{cases}$$

If $4s+(-3-3\ell_0)_s<j<5s+(-3-3\ell_0)_s$, then $b_{ij}=0$.

If $j=5s+(-3-3\ell_0)_s$, then $$b_{ij}=
\begin{cases}
-\kappa w_{(j+m+1)n+g(j)\ra (j+m+1)n+5}\otimes w_{jn+5\ra (j+1)n+g(j+s)},\quad i=(j+1)_s+s;\\
\kappa w_{(j+m+1)n+g(j+s)\ra (j+m+1)n+5}\otimes w_{jn+5\ra (j+1)n+g(j)},\quad i=(j+1)_s+2s;\\
\kappa w_{(j+m+1)n+g(j)+1\ra (j+m+1)n+5}\otimes w_{jn+5\ra (j+1)n+g(j+s)+1},\quad i=(j+1)_s+3s;\\
-\kappa w_{(j+m+1)n+g(j+s)+1\ra (j+m+1)n+5}\otimes w_{jn+5\ra (j+1)n+g(j)+1},\quad i=(j+1)_s+4s;\\
\kappa e_{(j+m+1)n+5}\otimes w_{jn+5\ra (j+1)n+5},\quad i=(j+1)_s+5s;\\
0,\quad\text{otherwise.}\end{cases}$$

If $5s+(-3-3\ell_0)_s<j<6s$, then $b_{ij}=0$.

\medskip
$({\rm II})$ Represent an arbitrary $t_0\in\N$ in the form
$t_0=11\ell_0+r_0$, where $0\le r_0\le 10.$ Then
$\Omega^{t_0}(Y_t^{(3)})$ is a $\Omega^{r_0}(Y_t^{(3)})$, whose left
components twisted by $\sigma^{\ell_0}$, and coefficients multiplied
by $(-1)^{\ell_0}$.
\end{pr}
\begin{pr}[Translates for the case 4]
$({\rm I})$ Let $r_0\in\N$, $r_0<11$. $r_0$-translates of the
elements $Y^{(4)}_t$ are described by the following way.

$(1)$ If $r_0=0$, then $\Omega^{0}(Y_t^{(4)})$ is described with
$(7s\times 6s)$-matrix with the following elements $b_{ij}${\rm:}

If $0\le j<s$, then $$b_{ij}=
\begin{cases}
-w_{(j+m)n\ra (j+m)n+g(j+s)}\otimes e_{jn},\quad i=j;\\
0,\quad\text{otherwise.}\end{cases}$$

If $s\le j<5s$, then $b_{ij}=0$.

If $5s\le j<6s$, then $$b_{ij}=
\begin{cases}
\kappa w_{(j+m)n+g(j)+1\ra (j+m)n+5}\otimes e_{jn+g(j)+1},\quad i=j-s;\\
0,\quad\text{otherwise.}\end{cases}$$

If $6s\le j<7s$, then $b_{ij}=0$.

$(2)$ If $r_0=1$, then $\Omega^{1}(Y_t^{(4)})$ is described with
$(6s\times 7s)$-matrix with the following elements $b_{ij}${\rm:}

If $0\le j<s$, then $$b_{ij}=
\begin{cases}
\kappa w_{(j+m)n+g(j+s)\ra (j+m)n+5}\otimes e_{jn},\quad i=j+s;\\
\kappa w_{(j+m)n+g(j+s)+1\ra (j+m)n+5}\otimes w_{jn\ra jn+g(j+s)},\quad i=j+3s;\\
0,\quad\text{otherwise.}\end{cases}$$

If $s\le j<2s$, then $$b_{ij}=
\begin{cases}
w_{(j+m+1)n+g(j)\ra (j+m+1)n+g(j)+1}\otimes w_{jn+g(j+s)\ra (j+1)n},\quad i=(j+1)_s+s;\\
0,\quad\text{otherwise.}\end{cases}$$

If $2s\le j<4s$, then $b_{ij}=0$.

If $4s\le j<5s$, then $$b_{ij}=
\begin{cases}
w_{(j+m)n+5\ra (j+m+1)n+g(j+s)}\otimes e_{jn+g(j+s)+1},\quad i=j+s;\\
w_{(j+m+1)n\ra (j+m+1)n+g(j+s)}\otimes w_{jn+g(j+s)+1\ra jn+5},\quad i=j+2s;\\
0,\quad\text{otherwise.}\end{cases}$$

If $5s\le j<6s$, then $$b_{ij}=
\begin{cases}
\kappa w_{(j+m+1)n\ra (j+m+2)n}\otimes e_{jn+5},\quad i=j+s;\\
0,\quad\text{otherwise.}\end{cases}$$

$(3)$ If $r_0=2$, then $\Omega^{2}(Y_t^{(4)})$ is described with
$(8s\times 6s)$-matrix with the following elements $b_{ij}${\rm:}

If $0\le j<s$, then $$b_{ij}=
\begin{cases}
-e_{(j+m)n+g(j+s)+1}\otimes w_{jn\ra jn+g(j+s)},\quad i=j+2s;\\
0,\quad\text{otherwise.}\end{cases}$$

If $s\le j<3s$, then $b_{ij}=0$.

If $3s\le j<4s$, then $$b_{ij}=
\begin{cases}
\kappa w_{(j+m)n+g(j)+1\ra (j+m)n+5}\otimes e_{jn+g(j)},\quad i=j-s;\\
0,\quad\text{otherwise.}\end{cases}$$

If $4s\le j<5s$, then $$b_{ij}=
\begin{cases}
-\kappa w_{(j+m)n+g(j+s)\ra (j+m+1)n}\otimes e_{jn+g(j)+1},\quad i=j-s;\\
0,\quad\text{otherwise.}\end{cases}$$

If $5s\le j<6s$, then $b_{ij}=0$.

If $6s\le j<7s$, then $$b_{ij}=
\begin{cases}
w_{(j+m)n+5\ra (j+m+1)n+g(j+s)}\otimes w_{jn+5\ra (j+1)n},\quad i=(j+1)_s;\\
-e_{(j+m+1)n+g(j+s)}\otimes w_{jn+5\ra (j+1)n+g(j)+1},\quad i=(j+1)_s+3s;\\
0,\quad\text{otherwise.}\end{cases}$$

If $7s\le j<8s$, then $b_{ij}=0$.

$(4)$ If $r_0=3$, then $\Omega^{3}(Y_t^{(4)})$ is described with
$(9s\times 8s)$-matrix with the following elements $b_{ij}${\rm:}

If $0\le j<s$, then $b_{ij}=0$.

If $s\le j<2s$, then $$b_{ij}=
\begin{cases}
\kappa w_{(j+m)n+g(j)+1\ra (j+m+1)n}\otimes e_{jn},\quad i=j;\\
0,\quad\text{otherwise.}\end{cases}$$

If $2s\le j<3s$, then $b_{ij}=0$.

If $3s\le j<4s$, then $$b_{ij}=
\begin{cases}
w_{(j+m)n+5\ra (j+m+1)n+g(j)}\otimes e_{jn+g(j)},\quad i=j;\\
0,\quad\text{otherwise.}\end{cases}$$

If $4s\le j<5s$, then $$b_{ij}=
\begin{cases}
w_{(j+m+1)n\ra (j+m+1)n+g(j+s)}\otimes e_{jn+g(j)+1},\quad i=j;\\
e_{(j+m+1)n+g(j+s)}\otimes w_{jn+g(j)+1\ra jn+5},\quad i=j+3s;\\
0,\quad\text{otherwise.}\end{cases}$$

If $5s\le j<7s$, then $b_{ij}=0$.

If $7s\le j<8s$, then $$b_{ij}=
\begin{cases}
-e_{(j+m+1)n+g(j)+1}\otimes w_{jn+g(j)+1\ra (j+1)n},\quad i=(j+1)_s+s;\\
0,\quad\text{otherwise.}\end{cases}$$

If $8s\le j<9s$, then $$b_{ij}=
\begin{cases}
-\kappa w_{(j+m+1)n+g(j+s)+1\ra (j+m+1)n+5}\otimes w_{jn+5\ra (j+1)n},\quad i=(j+1)_s+s;\\
0,\quad\text{otherwise.}\end{cases}$$

$(5)$ If $r_0=4$, then $\Omega^{4}(Y_t^{(4)})$ is described with
$(8s\times 9s)$-matrix with the following elements $b_{ij}${\rm:}

If $0\le j<s$, then $$b_{ij}=
\begin{cases}
-w_{(j+m-1)n+5\ra (j+m)n+g(j+s)}\otimes e_{jn},\quad i=j;\\
e_{(j+m)n+g(j+s)}\otimes w_{jn\ra jn+g(j)},\quad i=j+2s;\\
0,\quad\text{otherwise.}\end{cases}$$

If $s\le j<2s$, then $b_{ij}=0$.

If $2s\le j<3s$, then $$b_{ij}=
\begin{cases}
-\kappa w_{(j+m)n+g(j+s)\ra (j+m+1)n}\otimes e_{jn+g(j)},\quad i=j;\\
\kappa w_{(j+m)n+g(j+s)+1\ra (j+m+1)n}\otimes w_{jn+g(j)\ra jn+g(j)+1},\quad i=j+4s;\\
0,\quad\text{otherwise.}\end{cases}$$

If $3s\le j<4s$, then $b_{ij}=0$.

If $4s\le j<5s$, then $$b_{ij}=
\begin{cases}
-\kappa w_{(j+m)n+g(j+s)+1\ra (j+m)n+5}\otimes e_{jn+g(j)+1},\quad i=j+2s;\\
0,\quad\text{otherwise.}\end{cases}$$

If $5s\le j<6s$, then $b_{ij}=0$.

If $6s\le j<7s$, then $$b_{ij}=
\begin{cases}
-w_{(j+m)n+5\ra (j+m+1)n+g(j+s)+1}\otimes w_{jn+5\ra (j+1)n},\quad i=(j+1)_s;\\
w_{(j+m)n+5\ra (j+m+1)n+g(j+s)+1}\otimes e_{jn+5},\quad i=j+2s;\\
0,\quad\text{otherwise.}\end{cases}$$

If $7s\le j<8s$, then $b_{ij}=0$.

$(6)$ If $r_0=5$, then $\Omega^{5}(Y_t^{(4)})$ is described with
$(9s\times 8s)$-matrix with the following elements $b_{ij}${\rm:}

If $0\le j<s$, then $b_{ij}=0$.

If $s\le j<2s$, then $$b_{ij}=
\begin{cases}
e_{(j+m+1)n+g(j)}\otimes w_{jn+g(j+s)\ra (j+1)n},\quad i=(j+1)_s+s;\\
w_{(j+m+1)n\ra (j+m+1)n+g(j)}\otimes e_{jn+g(j+s)},\quad i=j+s;\\
0,\quad\text{otherwise.}\end{cases}$$

If $2s\le j<4s$, then $b_{ij}=0$.

If $4s\le j<5s$, then $$b_{ij}=
\begin{cases}
e_{(j+m+1)n+g(j+s)+1}\otimes w_{jn+g(j+s)\ra jn+5},\quad i=j+3s;\\
0,\quad\text{otherwise.}\end{cases}$$

If $5s\le j<6s$, then $$b_{ij}=
\begin{cases}
-w_{(j+m)n+5\ra (j+m+1)n+g(j)+1}\otimes e_{jn+g(j+s)+1},\quad i=j-s;\\
0,\quad\text{otherwise.}\end{cases}$$

If $6s\le j<7s$, then $b_{ij}=0$.

If $7s\le j<8s$, then $$b_{ij}=
\begin{cases}
\kappa w_{(j+m+1)n+g(j)+1\ra (j+m+1)n+5}\otimes e_{jn+5},\quad i=j;\\
0,\quad\text{otherwise.}\end{cases}$$

If $8s\le j<9s$, then $b_{ij}=0$.

$(7)$ If $r_0=6$, then $\Omega^{6}(Y_t^{(4)})$ is described with
$(8s\times 9s)$-matrix with the following elements $b_{ij}${\rm:}

If $0\le j<s$, then $$b_{ij}=
\begin{cases}
e_{(j+m)n+g(j+s)+1}\otimes w_{jn\ra jn+g(j)},\quad i=j+3s;\\
0,\quad\text{otherwise.}\end{cases}$$

If $s\le j<2s$, then $b_{ij}=0$.

If $2s\le j<3s$, then $$b_{ij}=
\begin{cases}
-\kappa w_{(j+m)n+g(j+s)+1\ra (j+m)n+5}\otimes e_{jn+g(j)},\quad i=j+s;\\
0,\quad\text{otherwise.}\end{cases}$$

If $3s\le j<5s$, then $b_{ij}=0$.

If $5s\le j<6s$, then $$b_{ij}=
\begin{cases}
-\kappa w_{(j+m)n+g(j)+1\ra (j+m+1)n}\otimes e_{jn+g(j)+1},\quad i=j+s;\\
0,\quad\text{otherwise.}\end{cases}$$

If $6s\le j<7s$, then $$b_{ij}=
\begin{cases}
w_{(j+m+1)n\ra (j+m+1)n+g(j+s)}\otimes w_{jn+5\ra (j+1)n},\quad i=(j+1)_s;\\
-e_{(j+m+1)n+g(j+s)}\otimes w_{jn+5\ra (j+1)n+g(j+s)},\quad i=(j+1)_s+2s;\\
-w_{(j+m)n+5\ra (j+m+1)n+g(j+s)}\otimes e_{jn+5},\quad i=j+s;\\
0,\quad\text{otherwise.}\end{cases}$$

If $7s\le j<8s$, then $b_{ij}=0$.

$(8)$ If $r_0=7$, then $\Omega^{7}(Y_t^{(4)})$ is described with
$(6s\times 8s)$-matrix with the following elements $b_{ij}${\rm:}

If $0\le j<s$, then $b_{ij}=0$.

If $s\le j<2s$, then $$b_{ij}=
\begin{cases}
e_{(j+m+1)n+g(j)+1}\otimes w_{jn+g(j+s)\ra (j+1)n},\quad i=(j+1)_s+s;\\
-w_{(j+m)n+5\ra (j+m+1)n+g(j)+1}\otimes e_{jn+g(j+s)},\quad i=j+s;\\
0,\quad\text{otherwise.}\end{cases}$$

If $2s\le j<4s$, then $b_{ij}=0$.

If $4s\le j<5s$, then $$b_{ij}=
\begin{cases}
w_{(j+m+1)n\ra (j+m+1)n+g(j+s)}\otimes e_{jn+g(j+s)+1},\quad i=j+s;\\
-e_{(j+m+1)n+g(j+s)}\otimes w_{jn+g(j+s)+1\ra jn+5},\quad i=j+3s;\\
0,\quad\text{otherwise.}\end{cases}$$

If $5s\le j<6s$, then $$b_{ij}=
\begin{cases}
\kappa w_{(j+m+1)n+g(j)+1\ra (j+m+2)n}\otimes w_{jn+5\ra (j+1)n},\quad i=(j+1)_s+s;\\
0,\quad\text{otherwise.}\end{cases}$$

$(9)$ If $r_0=8$, then $\Omega^{8}(Y_t^{(4)})$ is described with
$(7s\times 6s)$-matrix with the following elements $b_{ij}${\rm:}

If $0\le j<s$, then $$b_{ij}=
\begin{cases}
\kappa w_{(j+m-1)n+5\ra (j+m)n+5}\otimes e_{jn},\quad i=j;\\
\kappa w_{(j+m)n+g(j+s)+1\ra (j+m)n+5}\otimes w_{jn\ra jn+g(j+s)},\quad i=j+2s;\\
0,\quad\text{otherwise.}\end{cases}$$

If $s\le j<2s$, then $b_{ij}=0$.

If $2s\le j<3s$, then $$b_{ij}=
\begin{cases}
-\kappa w_{(j+m)n+g(j+s)+1\ra (j+m+1)n}\otimes e_{jn+g(j+s)},\quad i=j;\\
0,\quad\text{otherwise.}\end{cases}$$

If $3s\le j<4s$, then $b_{ij}=0$.

If $4s\le j<5s$, then $$b_{ij}=
\begin{cases}
w_{(j+m+1)n\ra (j+m+1)n+g(j+s)}\otimes w_{jn+g(j+s)+1\ra jn+5},\quad i=j+s;\\
0,\quad\text{otherwise.}\end{cases}$$

If $5s\le j<6s$, then $$b_{ij}=
\begin{cases}
-w_{(j+m)n+5\ra (j+m+1)n+g(j)+1}\otimes w_{jn+5\ra (j+1)n},\quad i=(j+1)_s;\\
-e_{(j+m+1)n+g(j)+1}\otimes w_{jn+5\ra (j+1)n+g(j)},\quad i=(j+1)_s+2s;\\
0,\quad\text{otherwise.}\end{cases}$$

If $6s\le j<7s$, then $b_{ij}=0$.

$(10)$ If $r_0=9$, then $\Omega^{9}(Y_t^{(4)})$ is described with
$(6s\times 7s)$-matrix with the following elements $b_{ij}${\rm:}

If $0\le j<2s$, then $b_{ij}=0$.

If $2s\le j<3s$, then $$b_{ij}=
\begin{cases}
w_{(j+m+1)n\ra (j+m+1)n+g(j+s)}\otimes e_{jn+g(j+s)},\quad i=j;\\
0,\quad\text{otherwise.}\end{cases}$$

If $3s\le j<4s$, then $b_{ij}=0$.

If $4s\le j<5s$, then $$b_{ij}=
\begin{cases}
-e_{(j+m+1)n+g(j+s)+1}\otimes w_{jn+g(j+s)+1\ra jn+5},\quad i=j+2s;\\
0,\quad\text{otherwise.}\end{cases}$$

If $5s\le j<6s$, then $b_{ij}=0$.

$(11)$ If $r_0=10$, then $\Omega^{10}(Y_t^{(4)})$ is described with
$(6s\times 6s)$-matrix with the following elements $b_{ij}${\rm:}

If $0\le j<s$, then $$b_{ij}=
\begin{cases}
-\kappa e_{(j+m+1)n}\otimes w_{jn\ra (j+1)n},\quad i=(j+1)_s;\\
\kappa w_{(j+m)n+g(j)\ra (j+m+1)n}\otimes w_{jn\ra jn+g(j+s)},\quad i=j+2s;\\
-\kappa w_{(j+m)n+g(j)+1\ra (j+m+1)n}\otimes w_{jn\ra jn+g(j+s)+1},\quad i=j+4s;\\
\kappa w_{(j+m)n+5\ra (j+m+1)n}\otimes w_{jn\ra jn+5},\quad i=j+5s;\\
0,\quad\text{otherwise.}\end{cases}$$

If $s\le j<2s$, then $b_{ij}=0$.

If $2s\le j<3s$, then $$b_{ij}=
\begin{cases}
w_{(j+m+1)n\ra (j+m+1)n+g(j+s)}\otimes w_{jn+g(j+s)\ra (j+1)n},\quad i=(j+1)_s;\\
-w_{(j+m)n+5\ra (j+m+1)n+g(j+s)}\otimes w_{jn+g(j+s)\ra jn+5},\quad i=j+3s;\\
0,\quad\text{otherwise.}\end{cases}$$

If $3s\le j<4s$, then $b_{ij}=0$.

If $4s\le j<5s$, then $$b_{ij}=
\begin{cases}
w_{(j+m+1)n\ra (j+m+1)n+g(j+s)+1}\otimes w_{jn+g(j+s)+1\ra (j+1)n},\quad i=(j+1)_s;\\
-w_{(j+m)n+5\ra (j+m+1)n+g(j+s)+1}\otimes w_{jn+g(j+s)+1\ra jn+5},\quad i=j+s;\\
0,\quad\text{otherwise.}\end{cases}$$

If $5s\le j<6s$, then $$b_{ij}=
\begin{cases}
\kappa w_{(j+m+1)n+g(j+s)\ra (j+m+1)n+5}\otimes w_{jn+5\ra (j+1)n+g(j)},\quad i=(j+1)_s+2s;\\
-\kappa w_{(j+m+1)n+g(j+s)+1\ra (j+m+1)n+5}\otimes w_{jn+5\ra (j+1)n+g(j)+1},\quad i=(j+1)_s+4s;\\
\kappa e_{(j+m+1)n+5}\otimes w_{jn+5\ra (j+1)n+5},\quad i=(j+1)_s+5s;\\
0,\quad\text{otherwise.}\end{cases}$$

\medskip
$({\rm II})$ Represent an arbitrary $t_0\in\N$ in the form
$t_0=11\ell_0+r_0$, where $0\le r_0\le 10.$ Then
$\Omega^{t_0}(Y_t^{(4)})$ is a $\Omega^{r_0}(Y_t^{(4)})$, whose left
components twisted by $\sigma^{\ell_0}$.
\end{pr}
\begin{pr}[Translates for the case 6]
$({\rm I})$ Let $r_0\in\N$, $r_0<11$. $r_0$-translates of the
elements $Y^{(6)}_t$ are described by the following way.

$(1)$ If $r_0=0$, then $\Omega^{0}(Y_t^{(6)})$ is described with
$(8s\times 6s)$-matrix with the following elements $b_{ij}${\rm:}

If $0\le j<s$, then $$b_{ij}=
\begin{cases}
w_{(j+m)n\ra (j+m)n+g(j)+1}\otimes e_{jn},\quad i=j;\\
0,\quad\text{otherwise.}\end{cases}$$

If $s\le j<2s$, then $$b_{ij}=
\begin{cases}
w_{(j+m)n\ra (j+m)n+g(j)+1}\otimes e_{jn},\quad i=j-s;\\
0,\quad\text{otherwise.}\end{cases}$$

If $2s\le j<4s$, then $b_{ij}=0$.

If $4s\le j<5s$, then $$b_{ij}=
\begin{cases}
w_{(j+m)n+g(j)+1\ra (j+m+1)n}\otimes e_{jn+g(j)+1},\quad i=j-s;\\
0,\quad\text{otherwise.}\end{cases}$$

If $5s\le j<6s$, then $b_{ij}=0$.

If $6s\le j<7s$, then $$b_{ij}=
\begin{cases}
w_{(j+m)n+5\ra (j+m+1)n+g(j)}\otimes e_{jn+5},\quad i=j-s;\\
0,\quad\text{otherwise.}\end{cases}$$

If $7s\le j<8s$, then $b_{ij}=0$.

$(2)$ If $r_0=1$, then $\Omega^{1}(Y_t^{(6)})$ is described with
$(9s\times 7s)$-matrix with the following elements $b_{ij}${\rm:}

If $0\le j<s$, then $b_{ij}=0$.

If $s\le j<2s$, then $$b_{ij}=
\begin{cases}
w_{(j+m)n+g(j+s)\ra (j+m+1)n}\otimes e_{jn},\quad i=j-s;\\
w_{(j+m)n+g(j+s)+1\ra (j+m+1)n}\otimes w_{jn\ra jn+g(j+s)},\quad i=j+s;\\
0,\quad\text{otherwise.}\end{cases}$$

If $2s\le j<3s$, then $b_{ij}=0$.

If $3s\le j<4s$, then $$b_{ij}=
\begin{cases}
w_{(j+m)n+5\ra (j+m+1)n+g(j+s)}\otimes w_{jn+g(j)\ra jn+g(j)+1},\quad i=j+2s;\\
0,\quad\text{otherwise.}\end{cases}$$

If $4s\le j<5s$, then $$b_{ij}=
\begin{cases}
w_{(j+m)n+5\ra (j+m+1)n+g(j)}\otimes e_{jn+g(j)+1},\quad i=j;\\
0,\quad\text{otherwise.}\end{cases}$$

If $5s\le j<6s$, then $b_{ij}=0$.

If $6s\le j<7s$, then $$b_{ij}=
\begin{cases}
w_{(j+m+1)n+g(j+s)\ra (j+m+1)n+g(j+s)+1}\otimes w_{jn+g(j)+1\ra (j+1)n},\quad i=(j+1)_s+s;\\
0,\quad\text{otherwise.}\end{cases}$$

If $7s\le j<8s$, then $$b_{ij}=
\begin{cases}
w_{(j+m+1)n\ra (j+m+1)n+g(j+s)+1}\otimes w_{jn+g(j)+1\ra jn+5},\quad i=j-s;\\
0,\quad\text{otherwise.}\end{cases}$$

If $8s\le j<9s$, then $$b_{ij}=
\begin{cases}
w_{(j+m+1)n\ra (j+m+1)n+5}\otimes e_{jn+5},\quad i=j-2s;\\
0,\quad\text{otherwise.}\end{cases}$$

$(3)$ If $r_0=2$, then $\Omega^{2}(Y_t^{(6)})$ is described with
$(8s\times 6s)$-matrix with the following elements $b_{ij}${\rm:}

If $0\le j<s$, then $$b_{ij}=
\begin{cases}
w_{(j+m-1)n+5\ra (j+m)n+g(j)}\otimes e_{jn},\quad i=j;\\
0,\quad\text{otherwise.}\end{cases}$$

If $s\le j<2s$, then $b_{ij}=0$.

If $2s\le j<3s$, then $$b_{ij}=
\begin{cases}
w_{(j+m)n+5\ra (j+m+1)n}\otimes w_{jn+g(j)\ra (j+1)n},\quad i=(j+1)_s;\\
0,\quad\text{otherwise.}\end{cases}$$

If $3s\le j<4s$, then $$b_{ij}=
\begin{cases}
w_{(j+m)n+g(j+s)\ra (j+m+1)n}\otimes w_{jn+g(j)\ra jn+g(j)+1},\quad i=j+s;\\
0,\quad\text{otherwise.}\end{cases}$$

If $4s\le j<5s$, then $$b_{ij}=
\begin{cases}
w_{(j+m)n+g(j+s)\ra (j+m)n+5}\otimes e_{jn+g(j)+1},\quad i=j-s;\\
0,\quad\text{otherwise.}\end{cases}$$

If $5s\le j<6s$, then $b_{ij}=0$.

If $6s\le j<7s$, then $$b_{ij}=
\begin{cases}
w_{(j+m+1)n+g(j)\ra (j+m+1)n+g(j)+1}\otimes w_{jn+5\ra (j+1)n+g(j+s)+1},\quad i=(j+1)_s+4s;\\
0,\quad\text{otherwise.}\end{cases}$$

If $7s\le j<8s$, then $b_{ij}=0$.

$(4)$ If $r_0=3$, then $\Omega^{3}(Y_t^{(6)})$ is described with
$(9s\times 8s)$-matrix with the following elements $b_{ij}${\rm:}

If $0\le j<s$, then $$b_{ij}=
\begin{cases}
w_{(j+m)n+g(j)+1\ra (j+m+1)n}\otimes e_{jn},\quad i=j;\\
w_{(j+m)n+5\ra (j+m+1)n}\otimes w_{jn\ra jn+g(j)},\quad i=j+2s;\\
0,\quad\text{otherwise.}\end{cases}$$

If $s\le j<3s$, then $b_{ij}=0$.

If $3s\le j<4s$, then $$b_{ij}=
\begin{cases}
e_{(j+m+1)n+g(j)+1}\otimes w_{jn+g(j+s)\ra (j+1)n},\quad i=(j+1)_s+s;\\
0,\quad\text{otherwise.}\end{cases}$$

If $4s\le j<5s$, then $$b_{ij}=
\begin{cases}
e_{(j+m+1)n+g(j)+1}\otimes w_{jn+g(j+s)\ra (j+1)n},\quad i=(j+1)_s;\\
0,\quad\text{otherwise.}\end{cases}$$

If $5s\le j<7s$, then $b_{ij}=0$.

If $7s\le j<8s$, then $$b_{ij}=
\begin{cases}
w_{(j+m+1)n+g(j+s)+1\ra (j+m+1)n+5}\otimes w_{jn+5\ra (j+1)n},\quad i=(j+1)_s;\\
e_{(j+m+1)n+5}\otimes w_{jn+5\ra (j+1)n+g(j+s)},\quad i=(j+1)_s+2s;\\
0,\quad\text{otherwise.}\end{cases}$$

If $8s\le j<9s$, then $b_{ij}=0$.

$(5)$ If $r_0=4$, then $\Omega^{4}(Y_t^{(6)})$ is described with
$(8s\times 9s)$-matrix with the following elements $b_{ij}${\rm:}

If $0\le j<s$, then $$b_{ij}=
\begin{cases}
w_{(j+m-1)n+5\ra (j+m)n+g(j)+1}\otimes e_{jn},\quad i=j;\\
w_{(j+m)n+g(j)\ra (j+m)n+g(j)+1}\otimes w_{jn\ra jn+g(j+s)},\quad i=j+3s;\\
e_{(j+m)n+g(j)+1}\otimes w_{jn\ra jn+g(j+s)+1},\quad i=j+7s;\\
0,\quad\text{otherwise.}\end{cases}$$

If $s\le j<2s$, then $$b_{ij}=
\begin{cases}
w_{(j+m)n+g(j)\ra (j+m)n+g(j)+1}\otimes w_{jn\ra jn+g(j+s)},\quad i=j+s;\\
e_{(j+m)n+g(j)+1}\otimes w_{jn\ra jn+g(j+s)+1},\quad i=j+5s;\\
0,\quad\text{otherwise.}\end{cases}$$

If $2s\le j<3s$, then $$b_{ij}=
\begin{cases}
e_{(j+m)n+5}\otimes w_{jn+g(j)\ra (j+1)n},\quad i=(j+1)_s;\\
0,\quad\text{otherwise.}\end{cases}$$

If $3s\le j<8s$, then $b_{ij}=0$.

$(6)$ If $r_0=5$, then $\Omega^{5}(Y_t^{(6)})$ is described with
$(6s\times 8s)$-matrix with the following elements $b_{ij}${\rm:}

If $0\le j<s$, then $$b_{ij}=
\begin{cases}
w_{(j+m)n+g(j)\ra (j+m)n+5}\otimes e_{jn},\quad i=j;\\
w_{(j+m)n+g(j+s)\ra (j+m)n+5}\otimes e_{jn},\quad i=j+s;\\
e_{(j+m)n+5}\otimes w_{jn\ra jn+g(j)+1},\quad i=j+4s;\\
e_{(j+m)n+5}\otimes w_{jn\ra jn+g(j+s)+1},\quad i=j+5s;\\
0,\quad\text{otherwise.}\end{cases}$$

If $s\le j<5s$, then $b_{ij}=0$.

If $5s\le j<6s$, then $$b_{ij}=
\begin{cases}
e_{(j+m+2)n}\otimes w_{jn+5\ra (j+1)n+g(j+s)},\quad i=(j+1)_s+2s;\\
0,\quad\text{otherwise.}\end{cases}$$

$(7)$ If $r_0=6$, then $\Omega^{6}(Y_t^{(6)})$ is described with
$(7s\times 9s)$-matrix with the following elements $b_{ij}${\rm:}

If $0\le j<s$, then $$b_{ij}=
\begin{cases}
w_{(j+m)n\ra (j+m)n+5}\otimes e_{jn},\quad i=j;\\
w_{(j+m)n+g(j)\ra (j+m)n+5}\otimes w_{jn\ra jn+g(j)},\quad i=j+s;\\
w_{(j+m)n+g(j+s)\ra (j+m)n+5}\otimes w_{jn\ra jn+g(j+s)},\quad i=j+2s;\\
0,\quad\text{otherwise.}\end{cases}$$

If $s\le j<2s$, then $$b_{ij}=
\begin{cases}
e_{(j+m+1)n}\otimes w_{jn+g(j+s)\ra (j+1)n},\quad i=(j+1)_s;\\
0,\quad\text{otherwise.}\end{cases}$$

If $2s\le j<3s$, then $b_{ij}=0$.

If $3s\le j<4s$, then $$b_{ij}=
\begin{cases}
w_{(j+m+1)n\ra (j+m+1)n+g(j)}\otimes w_{jn+g(j+s)+1\ra (j+1)n},\quad i=(j+1)_s;\\
0,\quad\text{otherwise.}\end{cases}$$

If $4s\le j<5s$, then $b_{ij}=0$.

If $5s\le j<6s$, then $$b_{ij}=
\begin{cases}
w_{(j+m+1)n+g(j+s)\ra (j+m+1)n+g(j+s)+1}\otimes w_{jn+5\ra (j+1)n+g(j+s)},\quad i=(j+1)_s+s;\\
0,\quad\text{otherwise.}\end{cases}$$

If $6s\le j<7s$, then $$b_{ij}=
\begin{cases}
w_{(j+m+1)n\ra (j+m+1)n+g(j+s)+1}\otimes w_{jn+5\ra (j+1)n},\quad i=(j+1)_s;\\
w_{(j+m+1)n+g(j+s)\ra (j+m+1)n+g(j+s)+1}\otimes w_{jn+5\ra (j+1)n+g(j+s)},\quad i=(j+1)_s+2s;\\
0,\quad\text{otherwise.}\end{cases}$$

$(8)$ If $r_0=7$, then $\Omega^{7}(Y_t^{(6)})$ is described with
$(6s\times 8s)$-matrix with the following elements $b_{ij}${\rm:}

If $0\le j<s$, then $$b_{ij}=
\begin{cases}
w_{(j+m)n+g(j)+1\ra (j+m+1)n}\otimes e_{jn},\quad i=j;\\
w_{(j+m)n+5\ra (j+m+1)n}\otimes w_{jn\ra jn+g(j+s)},\quad i=j+3s;\\
e_{(j+m+1)n}\otimes w_{jn\ra jn+g(j)+1},\quad i=j+4s;\\
e_{(j+m+1)n}\otimes w_{jn\ra jn+g(j+s)+1},\quad i=j+5s;\\
0,\quad\text{otherwise.}\end{cases}$$

If $s\le j<6s$, then $b_{ij}=0$.

$(9)$ If $r_0=8$, then $\Omega^{8}(Y_t^{(6)})$ is described with
$(6s\times 6s)$-matrix with the following elements $b_{ij}${\rm:}

If $0\le j<s$, then $$b_{ij}=
\begin{cases}
w_{(j+m)n+g(j)+1\ra (j+m+1)n}\otimes w_{jn\ra jn+g(j)},\quad i=j+s;\\
w_{(j+m)n+g(j)\ra (j+m+1)n}\otimes w_{jn\ra jn+g(j+s)+1},\quad i=j+4s;\\
0,\quad\text{otherwise.}\end{cases}$$

If $s\le j<2s$, then $b_{ij}=0$.

If $2s\le j<3s$, then $$b_{ij}=
\begin{cases}
w_{(j+m)n+5\ra (j+m+1)n+g(j)}\otimes w_{jn+g(j+s)\ra (j+1)n},\quad i=(j+1)_s;\\
0,\quad\text{otherwise.}\end{cases}$$

If $3s\le j<4s$, then $$b_{ij}=
\begin{cases}
w_{(j+m+1)n+g(j)\ra (j+m+1)n+g(j)+1}\otimes w_{jn+g(j+s)+1\ra (j+1)n+g(j+s)+1},\quad i=(j+1)_s+3s;\\
0,\quad\text{otherwise.}\end{cases}$$

If $4s\le j<5s$, then $$b_{ij}=
\begin{cases}
w_{(j+m)n+5\ra (j+m+1)n+g(j)+1}\otimes w_{jn+g(j+s)+1\ra (j+1)n},\quad i=(j+1)_s;\\
w_{(j+m+1)n+g(j)\ra (j+m+1)n+g(j)+1}\otimes w_{jn+g(j+s)+1\ra (j+1)n+g(j+s)+1},\quad i=(j+1)_s+4s;\\
0,\quad\text{otherwise.}\end{cases}$$

If $5s\le j<6s$, then $$b_{ij}=
\begin{cases}
w_{(j+m+1)n+g(j)+1\ra (j+m+1)n+5}\otimes w_{jn+5\ra (j+1)n+g(j)},\quad i=(j+1)_s+2s;\\
w_{(j+m+1)n+g(j)\ra (j+m+1)n+5}\otimes w_{jn+5\ra (j+1)n+g(j+s)+1},\quad i=(j+1)_s+3s;\\
0,\quad\text{otherwise.}\end{cases}$$

$(10)$ If $r_0=9$, then $\Omega^{9}(Y_t^{(6)})$ is described with
$(7s\times 7s)$-matrix with the following elements $b_{ij}${\rm:}

If $0\le j<s$, then $b_{ij}=0$.

If $s\le j<2s$, then $$b_{ij}=
\begin{cases}
w_{(j+m)n+5\ra (j+m+1)n+g(j+s)}\otimes e_{jn},\quad i=j-s;\\
0,\quad\text{otherwise.}\end{cases}$$

If $2s\le j<5s$, then $b_{ij}=0$.

If $5s\le j<6s$, then $$b_{ij}=
\begin{cases}
e_{(j+m+1)n+5}\otimes w_{jn+g(j)+1\ra (j+1)n},\quad i=(j+1)_s;\\
0,\quad\text{otherwise.}\end{cases}$$

If $6s\le j<7s$, then $$b_{ij}=
\begin{cases}
w_{(j+m+1)n+5\ra (j+m+2)n}\otimes w_{jn+5\ra (j+1)n},\quad i=(j+1)_s;\\
0,\quad\text{otherwise.}\end{cases}$$

$(11)$ If $r_0=10$, then $\Omega^{10}(Y_t^{(6)})$ is described with
$(6s\times 6s)$-matrix with the following elements $b_{ij}${\rm:}

If $0\le j<s$, then $$b_{ij}=
\begin{cases}
w_{(j+m)n\ra (j+m)n+5}\otimes e_{jn},\quad i=j;\\
w_{(j+m)n+g(j+s)\ra (j+m)n+5}\otimes w_{jn\ra jn+g(j)},\quad i=j+s;\\
w_{(j+m)n+g(j)\ra (j+m)n+5}\otimes w_{jn\ra jn+g(j+s)},\quad i=j+2s;\\
w_{(j+m)n+g(j+s)+1\ra (j+m)n+5}\otimes w_{jn\ra jn+g(j)+1},\quad i=j+3s;\\
w_{(j+m)n+g(j)+1\ra (j+m)n+5}\otimes w_{jn\ra jn+g(j+s)+1},\quad i=j+4s;\\
e_{(j+m)n+5}\otimes w_{jn\ra jn+5},\quad i=j+5s;\\
0,\quad\text{otherwise.}\end{cases}$$

If $s\le j<2s$, then $$b_{ij}=
\begin{cases}
w_{(j+m+1)n\ra (j+m+1)n+g(j)+1}\otimes w_{jn+g(j+s)\ra (j+1)n},\quad i=(j+1)_s;\\
w_{(j+m+1)n+g(j)\ra (j+m+1)n+g(j)+1}\otimes w_{jn+g(j+s)\ra (j+1)n+g(j+s)},\quad i=(j+1)_s+s;\\
0,\quad\text{otherwise.}\end{cases}$$

If $2s\le j<3s$, then $$b_{ij}=
\begin{cases}
w_{(j+m+1)n\ra (j+m+1)n+g(j)+1}\otimes w_{jn+g(j+s)\ra (j+1)n},\quad i=(j+1)_s;\\
w_{(j+m+1)n+g(j)\ra (j+m+1)n+g(j)+1}\otimes w_{jn+g(j+s)\ra (j+1)n+g(j+s)},\quad i=(j+1)_s+2s;\\
0,\quad\text{otherwise.}\end{cases}$$

If $3s\le j<6s$, then $b_{ij}=0$.

\medskip
$({\rm II})$ Represent an arbitrary $t_0\in\N$ in the form
$t_0=11\ell_0+r_0$, where $0\le r_0\le 10.$ Then
$\Omega^{t_0}(Y_t^{(6)})$ is a $\Omega^{r_0}(Y_t^{(6)})$, whose left
components twisted by $\sigma^{\ell_0}$.
\end{pr}
\begin{pr}[Translates for the case 7]
$({\rm I})$ Let $r_0\in\N$, $r_0<11$. $r_0$-translates of the
elements $Y^{(7)}_t$ are described by the following way.

$(1)$ If $r_0=0$, then $\Omega^{0}(Y_t^{(7)})$ is described with
$(8s\times 6s)$-matrix with the following elements $b_{ij}${\rm:}

If $0\le j<2s$, then $$b_{ij}=
\begin{cases}
w_{(j+m)n\ra (j+m)n+g(j+s)+1}\otimes e_{jn},\quad i=(j)_s;\\
0,\quad\text{otherwise.}\end{cases}$$

If $2s\le j<4s$, then $$b_{ij}=
\begin{cases}
-\kappa w_{(j+m)n+g(j)\ra (j+m)n+5}\otimes e_{jn+g(j)},\quad i=j-s;\\
0,\quad\text{otherwise.}\end{cases}$$

If $4s\le j<6s$, then $$b_{ij}=
\begin{cases}
-\kappa w_{(j+m)n+g(j)+1\ra (j+m+1)n}\otimes e_{jn+g(j)+1},\quad i=j-s;\\
0,\quad\text{otherwise.}\end{cases}$$

If $6s\le j<8s$, then $b_{ij}=0$.

$(2)$ If $r_0=1$, then $\Omega^{1}(Y_t^{(7)})$ is described with
$(9s\times 7s)$-matrix with the following elements $b_{ij}${\rm:}

If $0\le j<s$, then $$b_{ij}=
\begin{cases}
\kappa w_{(j+m)n+g(j)\ra (j+m)n+5}\otimes e_{jn},\quad i=j;\\
\kappa w_{(j+m)n+g(j+s)\ra (j+m)n+5}\otimes e_{jn},\quad i=j+s;\\
0,\quad\text{otherwise.}\end{cases}$$

If $s\le j<2s$, then $$b_{ij}=
\begin{cases}
-\kappa w_{(j+m)n+g(j)\ra (j+m+1)n}\otimes e_{jn},\quad i=j;\\
\kappa w_{(j+m)n+g(j+s)+1\ra (j+m+1)n}\otimes w_{jn\ra jn+g(j+s)},\quad i=j+s;\\
-\kappa w_{(j+m)n+g(j)+1\ra (j+m+1)n}\otimes w_{jn\ra jn+g(j)},\quad i=j+2s;\\
0,\quad\text{otherwise.}\end{cases}$$

If $2s\le j<4s$, then $$b_{ij}=
\begin{cases}
f_1(j,3s)w_{(j+m)n+g(j)+1\ra (j+m+1)n+g(j)}\otimes e_{jn+g(j)},\quad i=j;\\
0,\quad\text{otherwise.}\end{cases}$$

If $4s\le j<6s$, then $b_{ij}=0$.

If $6s\le j<8s$, then $$b_{ij}=
\begin{cases}
f_1(j,7s)w_{(j+m)n+5\ra (j+m+1)n+g(j)+1}\otimes e_{jn+g(j)+1},\quad i=j-2s;\\
f_1(j,7s)w_{(j+m+1)n\ra (j+m+1)n+g(j)+1}\otimes w_{jn+g(j)+1\ra jn+5},\quad i=j-s(j_2-6);\\
0,\quad\text{otherwise.}\end{cases}$$

If $8s\le j<9s$, then $$b_{ij}=
\begin{cases}
\kappa w_{(j+m+1)n+g(j)\ra (j+m+1)n+5}\otimes w_{jn+5\ra (j+1)n},\quad i=(j+1)_s;\\
0,\quad\text{otherwise.}\end{cases}$$

$(3)$ If $r_0=2$, then $\Omega^{2}(Y_t^{(7)})$ is described with
$(8s\times 6s)$-matrix with the following elements $b_{ij}${\rm:}

If $0\le j<2s$, then $b_{ij}=0$.

If $2s\le j<4s$, then $$b_{ij}=
\begin{cases}
-\kappa w_{(j+m)n+g(j)+1\ra (j+m+1)n}\otimes e_{jn+g(j)},\quad i=j-s;\\
0,\quad\text{otherwise.}\end{cases}$$

If $4s\le j<6s$, then $$b_{ij}=
\begin{cases}
\kappa w_{(j+m)n+g(j+s)\ra (j+m)n+5}\otimes e_{jn+g(j)+1},\quad i=j-s;\\
0,\quad\text{otherwise.}\end{cases}$$

If $6s\le j<8s$, then $$b_{ij}=
\begin{cases}
e_{(j+m+1)n+g(j+s)+1}\otimes w_{jn+5\ra (j+1)n+g(j+s)},\quad i=(j+1)_s+s(8-j_2);\\
0,\quad\text{otherwise.}\end{cases}$$

$(4)$ If $r_0=3$, then $\Omega^{3}(Y_t^{(7)})$ is described with
$(9s\times 8s)$-matrix with the following elements $b_{ij}${\rm:}

If $0\le j<s$, then $$b_{ij}=
\begin{cases}
\kappa w_{(j+m)n+g(j)+1\ra (j+m+1)n}\otimes e_{jn},\quad i=j;\\
-\kappa w_{(j+m)n+g(j+s)+1\ra (j+m+1)n}\otimes e_{jn},\quad i=j+s;\\
0,\quad\text{otherwise.}\end{cases}$$

If $s\le j<3s$, then $b_{ij}=0$.

If $3s\le j<5s$, then $$b_{ij}=
\begin{cases}
e_{(j+m+1)n+g(j+s)+1}\otimes w_{jn+g(j+s)\ra (j+1)n},\quad i=(j+1)_s+s(j_2-3);\\
f_1(j,4s)w_{(j+m)n+5\ra (j+m+1)n+g(j+s)+1}\otimes e_{jn+g(j+s)},\quad i=j-s;\\
0,\quad\text{otherwise.}\end{cases}$$

If $5s\le j<7s$, then $b_{ij}=0$.

If $7s\le j<8s$, then $$b_{ij}=
\begin{cases}
-\kappa w_{(j+m+1)n+g(j+s)+1\ra (j+m+1)n+5}\otimes w_{jn+5\ra (j+1)n},\quad i=(j+1)_s;\\
\kappa w_{(j+m+1)n+g(j)+1\ra (j+m+1)n+5}\otimes w_{jn+5\ra (j+1)n},\quad i=(j+1)_s+s;\\
0,\quad\text{otherwise.}\end{cases}$$

If $8s\le j<9s$, then $b_{ij}=0$.

$(5)$ If $r_0=4$, then $\Omega^{4}(Y_t^{(7)})$ is described with
$(8s\times 9s)$-matrix with the following elements $b_{ij}${\rm:}

If $0\le j<2s$, then $$b_{ij}=
\begin{cases}
f_1(j,s)w_{(j+m-1)n+5\ra (j+m)n+g(j+s)+1}\otimes e_{jn},\quad i=j-sj_2;\\
-f_1(j,s)w_{(j+m)n+g(j+s)\ra (j+m)n+g(j+s)+1}\otimes w_{jn\ra jn+g(j)},\quad i=j+2s;\\
f_1(j,s)e_{(j+m)n+g(j+s)+1}\otimes w_{jn\ra jn+g(j)+1},\quad i=j+6s;\\
0,\quad\text{otherwise.}\end{cases}$$

If $2s\le j<4s$, then $$b_{ij}=
\begin{cases}
\kappa w_{(j+m)n+g(j+s)\ra (j+m)n+5}\otimes e_{jn+g(j)},\quad i=j;\\
-\kappa w_{(j+m)n+g(j+s)+1\ra (j+m)n+5}\otimes w_{jn+g(j)\ra jn+g(j)+1},\quad i=j+4s;\\
0,\quad\text{otherwise.}\end{cases}$$

If $4s\le j<8s$, then $b_{ij}=0$.

$(6)$ If $r_0=5$, then $\Omega^{5}(Y_t^{(7)})$ is described with
$(6s\times 8s)$-matrix with the following elements $b_{ij}${\rm:}

If $0\le j<s$, then $b_{ij}=0$.

If $s\le j<3s$, then $$b_{ij}=
\begin{cases}
-f_1(j,2s)w_{(j+m+1)n+g(j)\ra (j+m+1)n+g(j)+1}\otimes w_{jn+g(j+s)\ra (j+1)n},\\\quad\quad\quad i=(j+1)_s+s(2-j_2);\\
-f_1(j,2s)w_{(j+m+1)n\ra (j+m+1)n+g(j)+1}\otimes e_{jn+g(j+s)},\quad i=j+s;\\
0,\quad\text{otherwise.}\end{cases}$$

If $3s\le j<5s$, then $b_{ij}=0$.

If $5s\le j<6s$, then $$b_{ij}=
\begin{cases}
\kappa e_{(j+m+2)n}\otimes w_{jn+5\ra (j+1)n+g(j+s)},\quad i=(j+1)_s+2s;\\
\kappa e_{(j+m+2)n}\otimes w_{jn+5\ra (j+1)n+g(j)},\quad i=(j+1)_s+3s;\\
0,\quad\text{otherwise.}\end{cases}$$

$(7)$ If $r_0=6$, then $\Omega^{6}(Y_t^{(7)})$ is described with
$(7s\times 9s)$-matrix with the following elements $b_{ij}${\rm:}

If $0\le j<s$, then $$b_{ij}=
\begin{cases}
\kappa w_{(j+m)n+g(j)\ra (j+m)n+5}\otimes w_{jn\ra jn+g(j)},\quad i=j+s;\\
-\kappa w_{(j+m)n+g(j+s)\ra (j+m)n+5}\otimes w_{jn\ra jn+g(j+s)},\quad i=j+2s;\\
0,\quad\text{otherwise.}\end{cases}$$

If $s\le j<3s$, then $$b_{ij}=
\begin{cases}
\kappa e_{(j+m+1)n}\otimes w_{jn+g(j+s)\ra (j+1)n},\quad i=(j+1)_s;\\
\kappa w_{(j+m)n+g(j+s)\ra (j+m+1)n}\otimes e_{jn+g(j+s)},\quad i=j;\\
0,\quad\text{otherwise.}\end{cases}$$

If $3s\le j<5s$, then $$b_{ij}=
\begin{cases}
-f_1(j,4s)e_{(j+m+1)n+g(j+s)}\otimes w_{jn+g(j+s)+1\ra (j+1)n+g(j+s)},\quad i=(j+1)_s+s(j_2-2);\\
0,\quad\text{otherwise.}\end{cases}$$

If $5s\le j<7s$, then $$b_{ij}=
\begin{cases}
f_1(j,6s)w_{(j+m+1)n+g(j)\ra (j+m+1)n+g(j)+1}\otimes w_{jn+5\ra (j+1)n+g(j)},\quad i=(j+1)_s+s(7-j_2);\\
0,\quad\text{otherwise.}\end{cases}$$

$(8)$ If $r_0=7$, then $\Omega^{7}(Y_t^{(7)})$ is described with
$(6s\times 8s)$-matrix with the following elements $b_{ij}${\rm:}

If $0\le j<s$, then $b_{ij}=0$.

If $s\le j<3s$, then $$b_{ij}=
\begin{cases}
f_1(j,2s)w_{(j+m)n+5\ra (j+m+1)n+g(j+s)}\otimes e_{jn+g(j+s)},\quad i=j+s;\\
f_1(j,2s)w_{(j+m+1)n\ra (j+m+1)n+g(j+s)}\otimes w_{jn+g(j+s)\ra jn+g(j+s)+1},\quad i=j+3s;\\
-f_1(j,2s)e_{(j+m+1)n+g(j+s)}\otimes w_{jn+g(j+s)\ra jn+5},\quad i=j+5s;\\
0,\quad\text{otherwise.}\end{cases}$$

If $3s\le j<6s$, then $b_{ij}=0$.

$(9)$ If $r_0=8$, then $\Omega^{8}(Y_t^{(7)})$ is described with
$(6s\times 6s)$-matrix with the following elements $b_{ij}${\rm:}

If $0\le j<s$, then $$b_{ij}=
\begin{cases}
\kappa w_{(j+m)n+g(j+s)\ra (j+m+1)n}\otimes w_{jn\ra jn+g(j)+1},\quad i=j+3s;\\
-\kappa w_{(j+m)n+g(j)\ra (j+m+1)n}\otimes w_{jn\ra jn+g(j+s)+1},\quad i=j+4s;\\
0,\quad\text{otherwise.}\end{cases}$$

If $s\le j<2s$, then $$b_{ij}=
\begin{cases}
w_{(j+m+1)n\ra (j+m+1)n+g(j+s)}\otimes w_{jn+g(j+s)\ra jn+5},\quad i=j+4s;\\
0,\quad\text{otherwise.}\end{cases}$$

If $2s\le j<3s$, then $$b_{ij}=
\begin{cases}
-w_{(j+m)n+5\ra (j+m+1)n+g(j+s)}\otimes w_{jn+g(j+s)\ra (j+1)n},\quad i=(j+1)_s;\\
w_{(j+m)n+g(j+s)+1\ra (j+m+1)n+g(j+s)}\otimes e_{jn+g(j+s)},\quad i=j;\\
0,\quad\text{otherwise.}\end{cases}$$

If $3s\le j<5s$, then $$b_{ij}=
\begin{cases}
-f_1(j,4s)e_{(j+m+1)n+g(j+s)+1}\otimes w_{jn+g(j+s)+1\ra (j+1)n+g(j+s)},\quad i=(j+1)_s+s(j_2-2);\\
0,\quad\text{otherwise.}\end{cases}$$

If $5s\le j<6s$, then $$b_{ij}=
\begin{cases}
\kappa w_{(j+m+1)n+g(j+s)+1\ra (j+m+1)n+5}\otimes w_{jn+5\ra (j+1)n+g(j+s)},\quad i=(j+1)_s+s;\\
-\kappa w_{(j+m+1)n+g(j)+1\ra (j+m+1)n+5}\otimes w_{jn+5\ra (j+1)n+g(j)},\quad i=(j+1)_s+2s;\\
0,\quad\text{otherwise.}\end{cases}$$

$(10)$ If $r_0=9$, then $\Omega^{9}(Y_t^{(7)})$ is described with
$(7s\times 7s)$-matrix with the following elements $b_{ij}${\rm:}

If $0\le j<s$, then $$b_{ij}=
\begin{cases}
-w_{(j+m+1)n\ra (j+m+1)n+g(j)}\otimes w_{jn\ra jn+g(j+s)},\quad i=j+2s;\\
0,\quad\text{otherwise.}\end{cases}$$

If $s\le j<2s$, then $$b_{ij}=
\begin{cases}
-w_{(j+m)n+5\ra (j+m+1)n+g(j)}\otimes e_{jn},\quad i=j-s;\\
0,\quad\text{otherwise.}\end{cases}$$

If $2s\le j<3s$, then $$b_{ij}=
\begin{cases}
e_{(j+m+1)n+g(j)+1}\otimes w_{jn+g(j)\ra jn+5},\quad i=j+3s;\\
0,\quad\text{otherwise.}\end{cases}$$

If $3s\le j<4s$, then $$b_{ij}=
\begin{cases}
w_{(j+m+1)n\ra (j+m+1)n+g(j)+1}\otimes e_{jn+g(j)},\quad i=j-s;\\
-e_{(j+m+1)n+g(j)+1}\otimes w_{jn+g(j)\ra jn+5},\quad i=j+3s;\\
0,\quad\text{otherwise.}\end{cases}$$

If $4s\le j<6s$, then $b_{ij}=0$.

If $6s\le j<7s$, then $$b_{ij}=
\begin{cases}
-\kappa w_{(j+m+1)n+5\ra (j+m+2)n}\otimes w_{jn+5\ra (j+1)n},\quad i=(j+1)_s;\\
0,\quad\text{otherwise.}\end{cases}$$

$(11)$ If $r_0=10$, then $\Omega^{10}(Y_t^{(7)})$ is described with
$(6s\times 6s)$-matrix with the following elements $b_{ij}${\rm:}

If $0\le j<s$, then $$b_{ij}=
\begin{cases}
-\kappa w_{(j+m)n\ra (j+m)n+5}\otimes e_{jn},\quad i=j;\\
\kappa w_{(j+m)n+g(j+s)\ra (j+m)n+5}\otimes w_{jn\ra jn+g(j)},\quad i=j+s;\\
\kappa w_{(j+m)n+g(j)\ra (j+m)n+5}\otimes w_{jn\ra jn+g(j+s)},\quad i=j+2s;\\
-\kappa w_{(j+m)n+g(j+s)+1\ra (j+m)n+5}\otimes w_{jn\ra jn+g(j)+1},\quad i=j+3s;\\
-\kappa w_{(j+m)n+g(j)+1\ra (j+m)n+5}\otimes w_{jn\ra jn+g(j+s)+1},\quad i=j+4s;\\
0,\quad\text{otherwise.}\end{cases}$$

If $s\le j<2s$, then $$b_{ij}=
\begin{cases}
w_{(j+m)n+5\ra (j+m+1)n+g(j+s)+1}\otimes w_{jn+g(j+s)\ra jn+5},\quad i=j+4s;\\
0,\quad\text{otherwise.}\end{cases}$$

If $2s\le j<3s$, then $$b_{ij}=
\begin{cases}
-w_{(j+m+1)n\ra (j+m+1)n+g(j+s)+1}\otimes w_{jn+g(j+s)\ra (j+1)n},\quad i=(j+1)_s;\\
w_{(j+m)n+5\ra (j+m+1)n+g(j+s)+1}\otimes w_{jn+g(j+s)\ra jn+5},\quad i=j+3s;\\
0,\quad\text{otherwise.}\end{cases}$$

If $3s\le j<4s$, then $b_{ij}=0$.

If $4s\le j<5s$, then $$b_{ij}=
\begin{cases}
w_{(j+m+1)n\ra (j+m+1)n+g(j)}\otimes w_{jn+g(j+s)+1\ra (j+1)n},\quad i=(j+1)_s;\\
0,\quad\text{otherwise.}\end{cases}$$

If $5s\le j<6s$, then $$b_{ij}=
\begin{cases}
\kappa w_{(j+m+1)n+g(j+s)\ra (j+m+2)n}\otimes w_{jn+5\ra (j+1)n+g(j)},\quad i=(j+1)_s+2s;\\
-\kappa w_{(j+m+1)n+g(j+s)+1\ra (j+m+2)n}\otimes w_{jn+5\ra (j+1)n+g(j)+1},\quad i=(j+1)_s+4s;\\
0,\quad\text{otherwise.}\end{cases}$$

\medskip
$({\rm II})$ Represent an arbitrary $t_0\in\N$ in the form
$t_0=11\ell_0+r_0$, where $0\le r_0\le 10.$ Then
$\Omega^{t_0}(Y_t^{(7)})$ is a $\Omega^{r_0}(Y_t^{(7)})$, whose left
components twisted by $\sigma^{\ell_0}$.
\end{pr}
\begin{pr}[Translates for the case 8]
$({\rm I})$ Let $r_0\in\N$, $r_0<11$. $r_0$-translates of the
elements $Y^{(8)}_t$ are described by the following way.

$(1)$ If $r_0=0$, then $\Omega^{0}(Y_t^{(8)})$ is described with
$(9s\times 6s)$-matrix with one nonzero element that is of the
following form{\rm:}
$$b_{3s+z_0(0, \ell_0, 2),6s+z_0(0, \ell_0, 2)}=w_{(j+m)n+g(j)+1\ra (j+m+1)n+g(j)+1}\otimes e_{jn+g(j)+1}.$$

$(2)$ If $r_0=1$, then $\Omega^{1}(Y_t^{(8)})$ is described with
$(8s\times 7s)$-matrix with one nonzero element that is of the
following form{\rm:}
\begin{multline*}
b_{(6s+z_1(-1, \ell_0,
2)+1)_s+2s+(\ell_0s)_{2s},6s+z_1(-1, \ell_0, 2)}=\\
-\kappa_0 w_{(j+m+1)n+g(j+s)+1\ra (j+m+2)n+g(j+s)+1}\otimes w_{jn+5\ra (j+1)n+g(j+s)}.
\end{multline*}

$(3)$ If $r_0=2$, then $\Omega^{2}(Y_t^{(8)})$ is described with
$(9s\times 6s)$-matrix with one nonzero element that is of the
following form{\rm:}
$$b_{(3s+z_0(-1, \ell_0, 2)+1)_s+s+(\ell_0s)_{2s},3s+z_0(-1, \ell_0, 2)}=e_{(j+m+1)n+g(j+s)+1}\otimes w_{jn+g(j+s)\ra (j+1)n+g(j+s)}.$$

$(4)$ If $r_0=3$, then $\Omega^{3}(Y_t^{(8)})$ is described with
$(8s\times 8s)$-matrix with one nonzero element that is of the
following form{\rm:}
$$b_{(2s+z_0(-1, \ell_0, 2)+1)_s+2s+(\ell_0s)_{2s},2s+z_0(-1, \ell_0, 2)}=\kappa\kappa_0 e_{(j+m+1)n+5}\otimes w_{jn+g(j)\ra (j+1)n+g(j)}.$$

$(5)$ If $r_0=4$, then $\Omega^{4}(Y_t^{(8)})$ is described with
$(6s\times 9s)$-matrix with one nonzero element that is of the
following form{\rm:}
\begin{multline*}
b_{(s+z_0(-1, \ell_0, 2)+1)_s+2s+(\ell_0s)_{2s},s+z_0(-1, \ell_0, 2)}=\\
-w_{(j+m+1)n+g(j)\ra (j+m+1)n+g(j)+1}\otimes w_{jn+g(j+s)\ra (j+1)n+g(j+s)}.
\end{multline*}

$(6)$ If $r_0=5$, then $\Omega^{5}(Y_t^{(8)})$ is described with
$(7s\times 8s)$-matrix with one nonzero element that is of the
following form{\rm:}
$$b_{(s+z_0(-1, \ell_0, 2)+1)_s+2s+(\ell_0s)_{2s},s+z_0(-1, \ell_0, 2)}=\kappa\kappa_0 e_{(j+m+2)n}\otimes w_{jn+g(j+s)\ra (j+1)n+g(j+s)}.$$

$(7)$ If $r_0=6$, then $\Omega^{6}(Y_t^{(8)})$ is described with
$(6s\times 9s)$-matrix with one nonzero element that is of the
following form{\rm:}
$$b_{(s+z_0(-1, \ell_0, 2)+1)_s+s+(\ell_0s)_{2s},s+z_0(-1, \ell_0, 2)}=\\
e_{(j+m+1)n+g(j+s)}\otimes w_{jn+g(j+s)\ra (j+1)n+g(j+s)}.$$

$(8)$ If $r_0=7$, then $\Omega^{7}(Y_t^{(8)})$ is described with
$(6s\times 8s)$-matrix with one nonzero element that is of the
following form{\rm:}
\begin{multline*}b_{(s+z_0(-1, \ell_0, 2)+1)_s+2s+(\ell_0s)_{2s},s+z_0(-1, \ell_0, 2)}=\\
-w_{(j+m+1)n+5\ra (j+m+2)n+g(j+s)}\otimes w_{jn+g(j+s)\ra (j+1)n+g(j+s)}.
\end{multline*}

$(9)$ If $r_0=8$, then $\Omega^{8}(Y_t^{(8)})$ is described with
$(7s\times 6s)$-matrix with one nonzero element that is of the
following form{\rm:}
$$b_{(2s+z_0(-1, \ell_0, 2)+1)_s+s+(\ell_0s)_{2s},2s+z_0(-1, \ell_0, 2)}=-e_{(j+m+1)n+g(j)+1}\otimes w_{jn+g(j)\ra (j+1)n+g(j)}.$$

$(10)$ If $r_0=9$, then $\Omega^{9}(Y_t^{(8)})$ is described with
$(6s\times 7s)$-matrix with one nonzero element that is of the
following form{\rm:}
$$b_{5s+z_0(-1, \ell_0, 2),s+z_0(-1, \ell_0, 2)}=-w_{(j+m+1)n+g(j+s)+1\ra (j+m+2)n+g(j+s)+1}\otimes w_{jn+g(j+s)\ra jn+5}.$$

$(11)$ If $r_0=10$, then $\Omega^{10}(Y_t^{(8)})$ is described with
$(8s\times 6s)$-matrix with one nonzero element that is of the
following form{\rm:}
$$b_{(s+z_0(-1, \ell_0, 2))_{2s}+3s,z_0(-1, \ell_0, 2)}=\kappa_0 w_{(j+m)n+g(j)+1\ra (j+m+1)n+g(j)+1}\otimes w_{jn\ra jn+g(j+s)+1}.$$

\medskip
$({\rm II})$ Represent an arbitrary $t_0\in\N$ in the form
$t_0=11\ell_0+r_0$, where $0\le r_0\le 10.$ Then
$\Omega^{t_0}(Y_t^{(8)})$ is a $\Omega^{r_0}(Y_t^{(8)})$, whose left
components twisted by $\sigma^{\ell_0}$.
\end{pr}
\begin{pr}[Translates for the case 9]
$({\rm I})$ Let $r_0\in\N$, $r_0<11$. $r_0$-translates of the
elements $Y^{(9)}_t$ are described by the following way.

$(1)$ If $r_0=0$, then $\Omega^{0}(Y_t^{(9)})$ is described with
$(9s\times 6s)$-matrix with the following elements $b_{ij}${\rm:}

If $0\le j<s$, then $b_{ij}=0$.

If $s\le j<2s$, then $$b_{ij}=
\begin{cases}
\kappa e_{(j+m)n}\otimes e_{jn},\quad i=j-s;\\
0,\quad\text{otherwise.}\end{cases}$$

If $2s\le j<4s$, then $$b_{ij}=
\begin{cases}
e_{(j+m)n+g(j)}\otimes e_{jn+g(j)},\quad i=j-s;\\
0,\quad\text{otherwise.}\end{cases}$$

If $4s\le j<6s$, then $b_{ij}=0$.

If $6s\le j<8s$, then $$b_{ij}=
\begin{cases}
e_{(j+m)n+g(j)+1}\otimes e_{jn+g(j)+1},\quad i=j-3s;\\
0,\quad\text{otherwise.}\end{cases}$$

If $8s\le j<9s$, then $$b_{ij}=
\begin{cases}
\kappa e_{(j+m)n+5}\otimes e_{jn+5},\quad i=j-3s;\\
0,\quad\text{otherwise.}\end{cases}$$

$(2)$ If $r_0=1$, then $\Omega^{1}(Y_t^{(9)})$ is described with
$(8s\times 7s)$-matrix with the following elements $b_{ij}${\rm:}

If $0\le j<2s$, then $$b_{ij}=
\begin{cases}
e_{(j+m)n+g(j+s)}\otimes e_{jn},\quad i=(j)_s+f_0(j,s)s;\\
0,\quad\text{otherwise.}\end{cases}$$

If $2s\le j<4s$, then $$b_{ij}=
\begin{cases}
-\kappa f_1(j,3s)w_{(j+m)n+g(j)+1\ra (j+m+1)n}\otimes e_{jn+g(j)},\quad i=j;\\
\kappa f_1(j,3s)w_{(j+m)n+5\ra (j+m+1)n}\otimes w_{jn+g(j)\ra jn+g(j)+1},\quad i=j+2s;\\
\kappa f_1(j,3s)e_{(j+m+1)n}\otimes w_{jn+g(j)\ra jn+5},\quad i=(j)_s+6s;\\
0,\quad\text{otherwise.}\end{cases}$$

If $4s\le j<6s$, then $$b_{ij}=
\begin{cases}
\kappa f_1(j,5s)e_{(j+m)n+5}\otimes e_{jn+g(j)+1},\quad i=j;\\
0,\quad\text{otherwise.}\end{cases}$$

If $6s\le j<8s$, then $$b_{ij}=
\begin{cases}
-f_1(j,7s)w_{(j+m+1)n+g(j+s)\ra (j+m+1)n+g(j+s)+1}\otimes w_{jn+5\ra (j+1)n},\\\quad\quad\quad i=(j+1)_s+f_0(j,7s)s;\\
f_1(j,7s)e_{(j+m+1)n+g(j+s)+1}\otimes w_{jn+5\ra (j+1)n+g(j+s)},\quad i=(j+1)_s+(f_0(j,7s)+2)s;\\
-f_1(j,7s)w_{(j+m+1)n\ra (j+m+1)n+g(j+s)+1}\otimes e_{jn+5},\quad i=(j)_s+6s;\\
0,\quad\text{otherwise.}\end{cases}$$

$(3)$ If $r_0=2$, then $\Omega^{2}(Y_t^{(9)})$ is described with
$(9s\times 6s)$-matrix with the following elements $b_{ij}${\rm:}

If $0\le j<s$, then $$b_{ij}=
\begin{cases}
-\kappa w_{(j+m-1)n+5\ra (j+m)n}\otimes e_{jn},\quad i=j;\\
0,\quad\text{otherwise.}\end{cases}$$

If $s\le j<3s$, then $$b_{ij}=
\begin{cases}
-e_{(j+m)n+g(j)}\otimes w_{jn+g(j+s)\ra jn+g(j+s)+1},\quad i=j+2s;\\
0,\quad\text{otherwise.}\end{cases}$$

If $3s\le j<5s$, then $$b_{ij}=
\begin{cases}
e_{(j+m)n+g(j+s)+1}\otimes e_{jn+g(j+s)},\quad i=j-2s;\\
0,\quad\text{otherwise.}\end{cases}$$

If $5s\le j<7s$, then $$b_{ij}=
\begin{cases}
-w_{(j+m)n+g(j)\ra (j+m)n+g(j)+1}\otimes e_{jn+g(j+s)+1},\quad i=j-2s;\\
0,\quad\text{otherwise.}\end{cases}$$

If $7s\le j<8s$, then $$b_{ij}=
\begin{cases}
-\kappa e_{(j+m)n+5}\otimes w_{jn+5\ra (j+1)n},\quad i=(j+1)_s;\\
0,\quad\text{otherwise.}\end{cases}$$

If $8s\le j<9s$, then $$b_{ij}=
\begin{cases}
-\kappa w_{(j+m)n+5\ra (j+m+1)n}\otimes w_{jn+5\ra (j+1)n},\quad i=(j+1)_s;\\
\kappa e_{(j+m+1)n}\otimes e_{jn+5},\quad i=j-3s;\\
0,\quad\text{otherwise.}\end{cases}$$

$(4)$ If $r_0=3$, then $\Omega^{3}(Y_t^{(9)})$ is described with
$(8s\times 8s)$-matrix with the following elements $b_{ij}${\rm:}

If $0\le j<2s$, then $$b_{ij}=
\begin{cases}
-f_1(j,s)e_{(j+m)n+g(j+s)+1}\otimes e_{jn},\quad i=(j)_s+f_0(j,s)s;\\
0,\quad\text{otherwise.}\end{cases}$$

If $2s\le j<4s$, then $$b_{ij}=
\begin{cases}
\kappa f_1(j,3s)e_{(j+m)n+5}\otimes e_{jn+g(j)},\quad i=j;\\
0,\quad\text{otherwise.}\end{cases}$$

If $4s\le j<6s$, then $$b_{ij}=
\begin{cases}
\kappa f_1(j,5s)e_{(j+m+1)n}\otimes e_{jn+g(j)+1},\quad i=j;\\
0,\quad\text{otherwise.}\end{cases}$$

If $6s\le j<8s$, then $$b_{ij}=
\begin{cases}
e_{(j+m+1)n+g(j+s)}\otimes e_{jn+5},\quad i=(j)_s+(f_0(j,7s)+6)s;\\
0,\quad\text{otherwise.}\end{cases}$$

$(5)$ If $r_0=4$, then $\Omega^{4}(Y_t^{(9)})$ is described with
$(6s\times 9s)$-matrix with the following elements $b_{ij}${\rm:}

If $0\le j<s$, then $$b_{ij}=
\begin{cases}
\kappa e_{(j+m-1)n+5}\otimes e_{jn},\quad i=j;\\
0,\quad\text{otherwise.}\end{cases}$$

If $s\le j<3s$, then $$b_{ij}=
\begin{cases}
-w_{(j+m)n+g(j)\ra (j+m)n+g(j)+1}\otimes e_{jn+g(j+s)},\quad i=j+s;\\
-e_{(j+m)n+g(j)+1}\otimes w_{jn+g(j+s)\ra jn+g(j+s)+1},\quad i=j+5s;\\
0,\quad\text{otherwise.}\end{cases}$$

If $3s\le j<5s$, then $$b_{ij}=
\begin{cases}
e_{(j+m)n+g(j+s)}\otimes e_{jn+g(j+s)+1},\quad i=j+s;\\
0,\quad\text{otherwise.}\end{cases}$$

If $5s\le j<6s$, then $$b_{ij}=
\begin{cases}
-\kappa e_{(j+m+1)n}\otimes w_{jn+5\ra (j+1)n},\quad i=(j+1)_s+s;\\
-\kappa w_{(j+m)n+5\ra (j+m+1)n}\otimes e_{jn+5},\quad i=j+3s;\\
0,\quad\text{otherwise.}\end{cases}$$

$(6)$ If $r_0=5$, then $\Omega^{5}(Y_t^{(9)})$ is described with
$(7s\times 8s)$-matrix with the following elements $b_{ij}${\rm:}

If $0\le j<s$, then $$b_{ij}=
\begin{cases}
-\kappa w_{(j+m)n+g(j)\ra (j+m)n+5}\otimes e_{jn},\quad i=j;\\
-\kappa w_{(j+m)n+g(j+s)\ra (j+m)n+5}\otimes e_{jn},\quad i=j+s;\\
-\kappa e_{(j+m)n+5}\otimes w_{jn\ra jn+g(j)+1},\quad i=j+4s;\\
-\kappa e_{(j+m)n+5}\otimes w_{jn\ra jn+g(j+s)+1},\quad i=j+5s;\\
0,\quad\text{otherwise.}\end{cases}$$

If $s\le j<3s$, then $$b_{ij}=
\begin{cases}
\kappa f_1(j,2s)e_{(j+m+1)n}\otimes e_{jn+g(j+s)},\quad i=j+s;\\
\kappa f_1(j,2s)w_{(j+m)n+5\ra (j+m+1)n}\otimes w_{jn+g(j+s)\ra jn+g(j+s)+1},\quad i=j+3s;\\
0,\quad\text{otherwise.}\end{cases}$$

If $3s\le j<5s$, then $$b_{ij}=
\begin{cases}
-e_{(j+m+1)n+g(j+s)}\otimes w_{jn+g(j+s)+1\ra (j+1)n},\quad i=(j+1)_s+(1-f_0(j,4s))s;\\
-w_{(j+m)n+5\ra (j+m+1)n+g(j+s)}\otimes e_{jn+g(j+s)+1},\quad i=j+s;\\
0,\quad\text{otherwise.}\end{cases}$$

If $5s\le j<7s$, then $$b_{ij}=
\begin{cases}
-w_{(j+m+1)n+g(j)\ra (j+m+1)n+g(j)+1}\otimes w_{jn+5\ra (j+1)n},\quad i=(j+1)_s+f_0(j,6s)s;\\
-f_1(j,6s)e_{(j+m+1)n+g(j)+1}\otimes e_{jn+5},\quad i=(j)_s+(6+f_0(j,6s))s;\\
0,\quad\text{otherwise.}\end{cases}$$

$(7)$ If $r_0=6$, then $\Omega^{6}(Y_t^{(9)})$ is described with
$(6s\times 9s)$-matrix with the following elements $b_{ij}${\rm:}

If $0\le j<s$, then $$b_{ij}=
\begin{cases}
\kappa e_{(j+m)n}\otimes e_{jn},\quad i=j;\\
0,\quad\text{otherwise.}\end{cases}$$

If $s\le j<3s$, then $$b_{ij}=
\begin{cases}
e_{(j+m)n+g(j+s)}\otimes e_{jn+g(j+s)},\quad i=j;\\
0,\quad\text{otherwise.}\end{cases}$$

If $3s\le j<5s$, then $$b_{ij}=
\begin{cases}
e_{(j+m)n+g(j+s)+1}\otimes e_{jn+g(j+s)+1},\quad i=j+2s;\\
0,\quad\text{otherwise.}\end{cases}$$

If $5s\le j<6s$, then $$b_{ij}=
\begin{cases}
\kappa e_{(j+m)n+5}\otimes e_{jn+5},\quad i=j+2s;\\
0,\quad\text{otherwise.}\end{cases}$$

$(8)$ If $r_0=7$, then $\Omega^{7}(Y_t^{(9)})$ is described with
$(6s\times 8s)$-matrix with the following elements $b_{ij}${\rm:}

If $0\le j<s$, then $$b_{ij}=
\begin{cases}
-\kappa w_{(j+m)n+g(j+s)+1\ra (j+m+1)n}\otimes e_{jn},\quad i=j+s;\\
\kappa w_{(j+m)n+5\ra (j+m+1)n}\otimes w_{jn\ra jn+g(j)},\quad i=j+2s;\\
-\kappa e_{(j+m+1)n}\otimes w_{jn\ra jn+g(j)+1},\quad i=j+4s;\\
-\kappa e_{(j+m+1)n}\otimes w_{jn\ra jn+g(j+s)+1},\quad i=j+5s;\\
0,\quad\text{otherwise.}\end{cases}$$

If $s\le j<3s$, then $$b_{ij}=
\begin{cases}
-w_{(j+m)n+5\ra (j+m+1)n+g(j+s)}\otimes e_{jn+g(j+s)},\quad i=j+s;\\
w_{(j+m+1)n\ra (j+m+1)n+g(j+s)}\otimes w_{jn+g(j+s)\ra jn+g(j+s)+1},\quad i=j+3s;\\
-e_{(j+m+1)n+g(j+s)}\otimes w_{jn+g(j+s)\ra jn+5},\quad i=j+5s;\\
0,\quad\text{otherwise.}\end{cases}$$

If $3s\le j<5s$, then $$b_{ij}=
\begin{cases}
-e_{(j+m+1)n+g(j+s)+1}\otimes w_{jn+g(j+s)+1\ra (j+1)n},\quad i=(j+1)_s+(1-f_0(j,4s))s;\\
-w_{(j+m+1)n\ra (j+m+1)n+g(j+s)+1}\otimes e_{jn+g(j+s)+1},\quad i=j+s;\\
w_{(j+m+1)n+g(j+s)\ra (j+m+1)n+g(j+s)+1}\otimes w_{jn+g(j+s)+1\ra jn+5},\quad i=j+3s;\\
0,\quad\text{otherwise.}\end{cases}$$

If $5s\le j<6s$, then $$b_{ij}=
\begin{cases}
\kappa w_{(j+m+1)n+g(j)+1\ra (j+m+1)n+5}\otimes w_{jn+5\ra (j+1)n},\quad i=(j+1)_s+s;\\
-\kappa e_{(j+m+1)n+5}\otimes w_{jn+5\ra (j+1)n+g(j+s)},\quad i=(j+1)_s+2s;\\
-\kappa w_{(j+m+1)n+g(j+s)\ra (j+m+1)n+5}\otimes e_{jn+5},\quad i=j+s;\\
-\kappa w_{(j+m+1)n+g(j)\ra (j+m+1)n+5}\otimes e_{jn+5},\quad i=j+2s;\\
0,\quad\text{otherwise.}\end{cases}$$

$(9)$ If $r_0=8$, then $\Omega^{8}(Y_t^{(9)})$ is described with
$(7s\times 6s)$-matrix with the following elements $b_{ij}${\rm:}

If $0\le j<s$, then $$b_{ij}=
\begin{cases}
e_{(j+m)n+g(j)}\otimes w_{jn\ra jn+g(j+s)+1},\quad i=j+4s;\\
0,\quad\text{otherwise.}\end{cases}$$

If $s\le j<2s$, then $$b_{ij}=
\begin{cases}
-w_{(j+m-1)n+5\ra (j+m)n+g(j)}\otimes e_{jn},\quad i=j-s;\\
e_{(j+m)n+g(j)}\otimes w_{jn\ra jn+g(j+s)+1},\quad i=j+2s;\\
0,\quad\text{otherwise.}\end{cases}$$

If $2s\le j<4s$, then $$b_{ij}=
\begin{cases}
-e_{(j+m)n+g(j)+1}\otimes e_{jn+g(j)},\quad i=j-s;\\
0,\quad\text{otherwise.}\end{cases}$$

If $4s\le j<5s$, then $$b_{ij}=
\begin{cases}
\kappa e_{(j+m)n+5}\otimes w_{jn+g(j)+1\ra (j+1)n},\quad i=(j+1)_s;\\
\kappa w_{(j+m)n+g(j+s)\ra (j+m)n+5}\otimes e_{jn+g(j)+1},\quad i=j-s;\\
0,\quad\text{otherwise.}\end{cases}$$

If $5s\le j<6s$, then $$b_{ij}=
\begin{cases}
\kappa w_{(j+m)n+g(j+s)\ra (j+m)n+5}\otimes e_{jn+g(j)+1},\quad i=j-s;\\
0,\quad\text{otherwise.}\end{cases}$$

If $6s\le j<7s$, then $$b_{ij}=
\begin{cases}
-\kappa e_{(j+m+1)n}\otimes e_{jn+5},\quad i=j-s;\\
0,\quad\text{otherwise.}\end{cases}$$

$(10)$ If $r_0=9$, then $\Omega^{9}(Y_t^{(9)})$ is described with
$(6s\times 7s)$-matrix with the following elements $b_{ij}${\rm:}

If $0\le j<s$, then $$b_{ij}=
\begin{cases}
-\kappa e_{(j+m)n+5}\otimes e_{jn},\quad i=j;\\
0,\quad\text{otherwise.}\end{cases}$$

If $s\le j<3s$, then $$b_{ij}=
\begin{cases}
-w_{(j+m+1)n\ra (j+m+1)n+g(j+s)+1}\otimes e_{jn+g(j+s)},\quad i=j;\\
-e_{(j+m+1)n+g(j+s)+1}\otimes w_{jn+g(j+s)\ra jn+5},\quad i=j+4s;\\
0,\quad\text{otherwise.}\end{cases}$$

If $3s\le j<5s$, then $$b_{ij}=
\begin{cases}
-e_{(j+m+1)n+g(j)}\otimes e_{jn+g(j+s)+1},\quad i=j;\\
0,\quad\text{otherwise.}\end{cases}$$

If $5s\le j<6s$, then $$b_{ij}=
\begin{cases}
-\kappa e_{(j+m+2)n}\otimes w_{jn+5\ra (j+1)n+g(j+s)},\quad i=(j+1)_s+s;\\
-\kappa w_{(j+m+1)n+g(j+s)+1\ra (j+m+2)n}\otimes e_{jn+5},\quad i=j;\\
0,\quad\text{otherwise.}\end{cases}$$

$(11)$ If $r_0=10$, then $\Omega^{10}(Y_t^{(9)})$ is described with
$(8s\times 6s)$-matrix with the following elements $b_{ij}${\rm:}

If $0\le j<2s$, then $$b_{ij}=
\begin{cases}
w_{(j+m)n\ra (j+m)n+g(j)+1}\otimes e_{jn},\quad i=(j)_s;\\
f_1(j,s)w_{(j+m)n+g(j)\ra (j+m)n+g(j)+1}\otimes w_{jn\ra jn+g(j+s)},\quad i=(j)_s+(1+f_0(j,s))s;\\
f_1(j,s)e_{(j+m)n+g(j)+1}\otimes w_{jn\ra jn+g(j+s)+1},\quad i=(j)_s+(3+f_0(j,s))s;\\
0,\quad\text{otherwise.}\end{cases}$$

If $2s\le j<4s$, then $$b_{ij}=
\begin{cases}
\kappa w_{(j+m)n+g(j+s)\ra (j+m)n+5}\otimes e_{jn+g(j)},\quad i=j-s;\\
\kappa w_{(j+m)n+g(j+s)+1\ra (j+m)n+5}\otimes w_{jn+g(j)\ra jn+g(j)+1},\quad i=j+s;\\
\kappa f_1(j,3s)e_{(j+m)n+5}\otimes w_{jn+g(j)\ra jn+5},\quad i=(j)_s+5s;\\
0,\quad\text{otherwise.}\end{cases}$$

If $4s\le j<5s$, then $$b_{ij}=
\begin{cases}
\kappa e_{(j+m+1)n}\otimes w_{jn+g(j)+1\ra (j+1)n},\quad i=(j+1)_s;\\
\kappa w_{(j+m)n+g(j+s)+1\ra (j+m+1)n}\otimes e_{jn+g(j)+1},\quad i=j-s;\\
\kappa w_{(j+m)n+5\ra (j+m+1)n}\otimes w_{jn+g(j)+1\ra jn+5},\quad i=j+s;\\
0,\quad\text{otherwise.}\end{cases}$$

If $5s\le j<6s$, then $$b_{ij}=
\begin{cases}
\kappa w_{(j+m)n+g(j+s)+1\ra (j+m+1)n}\otimes e_{jn+g(j)+1},\quad i=j-s;\\
0,\quad\text{otherwise.}\end{cases}$$

If $6s\le j<7s$, then $$b_{ij}=
\begin{cases}
w_{(j+m+1)n\ra (j+m+1)n+g(j)}\otimes w_{jn+5\ra (j+1)n},\quad i=(j+1)_s;\\
e_{(j+m+1)n+g(j)}\otimes w_{jn+5\ra (j+1)n+g(j+s)},\quad i=(j+1)_s+2s;\\
w_{(j+m)n+5\ra (j+m+1)n+g(j)}\otimes e_{jn+5},\quad i=j-s;\\
0,\quad\text{otherwise.}\end{cases}$$

If $7s\le j<8s$, then $$b_{ij}=
\begin{cases}
e_{(j+m+1)n+g(j)}\otimes w_{jn+5\ra (j+1)n+g(j+s)},\quad i=(j+1)_s+s;\\
0,\quad\text{otherwise.}\end{cases}$$

\medskip
$({\rm II})$ Represent an arbitrary $t_0\in\N$ in the form
$t_0=11\ell_0+r_0$, where $0\le r_0\le 10.$ Then
$\Omega^{t_0}(Y_t^{(9)})$ is a $\Omega^{r_0}(Y_t^{(9)})$, whose left
components twisted by $\sigma^{\ell_0}$, and coefficients multiplied
by $(-1)^{\ell_0}$.
\end{pr}
\begin{pr}[Translates for the case 11]
$({\rm I})$ Let $r_0\in\N$, $r_0<11$. $r_0$-translates of the
elements $Y^{(11)}_t$ are described by the following way.

$(1)$ If $r_0=0$, then $\Omega^{0}(Y_t^{(11)})$ is described with
$(8s\times 6s)$-matrix with the following elements $b_{ij}${\rm:}

If $0\le j<s$, then $$b_{ij}=
\begin{cases}
w_{(j+m)n\ra (j+m)n+g(j)}\otimes e_{jn},\quad i=j;\\
0,\quad\text{otherwise.}\end{cases}$$

If $s\le j<2s$, then $b_{ij}=0$.

If $2s\le j<3s$, then $$b_{ij}=
\begin{cases}
\kappa w_{(j+m)n+g(j)\ra (j+m+1)n}\otimes e_{jn+g(j)},\quad i=j-s;\\
0,\quad\text{otherwise.}\end{cases}$$

If $3s\le j<4s$, then $b_{ij}=0$.

If $4s\le j<5s$, then $$b_{ij}=
\begin{cases}
-\kappa w_{(j+m)n+g(j)+1\ra (j+m)n+5}\otimes e_{jn+g(j)+1},\quad i=j-s;\\
0,\quad\text{otherwise.}\end{cases}$$

If $5s\le j<6s$, then $b_{ij}=0$.

If $6s\le j<7s$, then $$b_{ij}=
\begin{cases}
w_{(j+m)n+5\ra (j+m+1)n+g(j)+1}\otimes e_{jn+5},\quad i=j-s;\\
0,\quad\text{otherwise.}\end{cases}$$

If $7s\le j<8s$, then $b_{ij}=0$.

$(2)$ If $r_0=1$, then $\Omega^{1}(Y_t^{(11)})$ is described with
$(9s\times 7s)$-matrix with the following elements $b_{ij}${\rm:}

If $0\le j<s$, then $$b_{ij}=
\begin{cases}
\kappa w_{(j+m)n+g(j)\ra (j+m+1)n}\otimes e_{jn},\quad i=j;\\
-\kappa w_{(j+m)n+g(j)+1\ra (j+m+1)n}\otimes w_{jn\ra jn+g(j)},\quad i=j+2s;\\
0,\quad\text{otherwise.}\end{cases}$$

If $s\le j<2s$, then $$b_{ij}=
\begin{cases}
w_{(j+m)n+g(j+s)+1\ra (j+m+1)n+g(j+s)}\otimes e_{jn+g(j+s)},\quad i=j+s;\\
w_{(j+m)n+5\ra (j+m+1)n+g(j+s)}\otimes w_{jn+g(j+s)\ra jn+g(j+s)+1},\quad i=j+3s;\\
w_{(j+m+1)n\ra (j+m+1)n+g(j+s)}\otimes w_{jn+g(j+s)\ra jn+5},\quad i=j+5s;\\
0,\quad\text{otherwise.}\end{cases}$$

If $2s\le j<4s$, then $b_{ij}=0$.

If $4s\le j<5s$, then $$b_{ij}=
\begin{cases}
w_{(j+m+1)n+g(j)\ra (j+m+1)n+g(j)+1}\otimes w_{jn+g(j+s)\ra (j+1)n},\quad i=(j+1)_s;\\
-w_{(j+m)n+5\ra (j+m+1)n+g(j)+1}\otimes w_{jn+g(j+s)\ra jn+g(j+s)+1},\quad i=j+s;\\
0,\quad\text{otherwise.}\end{cases}$$

If $5s\le j<6s$, then $$b_{ij}=
\begin{cases}
-w_{(j+m)n+5\ra (j+m+1)n+g(j+s)+1}\otimes e_{jn+g(j+s)+1},\quad i=j-s;\\
0,\quad\text{otherwise.}\end{cases}$$

If $6s\le j<7s$, then $b_{ij}=0$.

If $7s\le j<8s$, then $$b_{ij}=
\begin{cases}
-\kappa w_{(j+m+1)n+g(j+s)\ra (j+m+1)n+5}\otimes w_{jn+5\ra (j+1)n},\quad i=(j+1)_s;\\
-\kappa w_{(j+m+1)n+g(j+s)+1\ra (j+m+1)n+5}\otimes w_{jn+5\ra (j+1)n+g(j+s)},\quad i=(j+1)_s+2s;\\
\kappa w_{(j+m+1)n\ra (j+m+1)n+5}\otimes e_{jn+5},\quad i=j-s;\\
0,\quad\text{otherwise.}\end{cases}$$

If $8s\le j<9s$, then $b_{ij}=0$.

$(3)$ If $r_0=2$, then $\Omega^{2}(Y_t^{(11)})$ is described with
$(8s\times 6s)$-matrix with the following elements $b_{ij}${\rm:}

If $0\le j<s$, then $$b_{ij}=
\begin{cases}
-w_{(j+m-1)n+5\ra (j+m)n+g(j)+1}\otimes e_{jn},\quad i=j;\\
-e_{(j+m)n+g(j)+1}\otimes w_{jn\ra jn+g(j)},\quad i=j+s;\\
0,\quad\text{otherwise.}\end{cases}$$

If $s\le j<2s$, then $b_{ij}=0$.

If $2s\le j<3s$, then $$b_{ij}=
\begin{cases}
-\kappa w_{(j+m)n+g(j)+1\ra (j+m)n+5}\otimes e_{jn+g(j)},\quad i=j-s;\\
0,\quad\text{otherwise.}\end{cases}$$

If $3s\le j<4s$, then $$b_{ij}=
\begin{cases}
\kappa w_{(j+m)n+g(j+s)\ra (j+m)n+5}\otimes w_{jn+g(j)\ra jn+g(j)+1},\quad i=j+s;\\
0,\quad\text{otherwise.}\end{cases}$$

If $4s\le j<5s$, then $b_{ij}=0$.

If $5s\le j<6s$, then $$b_{ij}=
\begin{cases}
-\kappa w_{(j+m)n+g(j+s)\ra (j+m+1)n}\otimes e_{jn+g(j)+1},\quad i=j-s;\\
0,\quad\text{otherwise.}\end{cases}$$

If $6s\le j<7s$, then $$b_{ij}=
\begin{cases}
w_{(j+m+1)n\ra (j+m+1)n+g(j)}\otimes e_{jn+5},\quad i=j-s;\\
0,\quad\text{otherwise.}\end{cases}$$

If $7s\le j<8s$, then $b_{ij}=0$.

$(4)$ If $r_0=3$, then $\Omega^{3}(Y_t^{(11)})$ is described with
$(6s\times 8s)$-matrix with the following elements $b_{ij}${\rm:}

If $0\le j<s$, then $$b_{ij}=
\begin{cases}
-\kappa w_{(j+m)n+g(j)+1\ra (j+m)n+5}\otimes e_{jn},\quad i=j;\\
0,\quad\text{otherwise.}\end{cases}$$

If $s\le j<2s$, then $$b_{ij}=
\begin{cases}
-e_{(j+m+1)n+g(j+s)+1}\otimes w_{jn+g(j+s)\ra (j+1)n},\quad i=(j+1)_s;\\
-w_{(j+m)n+5\ra (j+m+1)n+g(j+s)+1}\otimes e_{jn+g(j+s)},\quad i=j+s;\\
w_{(j+m+1)n\ra (j+m+1)n+g(j+s)+1}\otimes w_{jn+g(j+s)\ra jn+g(j+s)+1},\quad i=j+3s;\\
0,\quad\text{otherwise.}\end{cases}$$

If $2s\le j<4s$, then $b_{ij}=0$.

If $4s\le j<5s$, then $$b_{ij}=
\begin{cases}
-w_{(j+m+1)n\ra (j+m+1)n+g(j)}\otimes e_{jn+g(j+s)+1},\quad i=j+s;\\
0,\quad\text{otherwise.}\end{cases}$$

If $5s\le j<6s$, then $$b_{ij}=
\begin{cases}
-\kappa w_{(j+m+1)n+g(j+s)+1\ra (j+m+2)n}\otimes w_{jn+5\ra (j+1)n},\quad i=(j+1)_s;\\
\kappa w_{(j+m+1)n+g(j+s)\ra (j+m+2)n}\otimes e_{jn+5},\quad i=j+s;\\
0,\quad\text{otherwise.}\end{cases}$$

$(5)$ If $r_0=4$, then $\Omega^{4}(Y_t^{(11)})$ is described with
$(7s\times 9s)$-matrix with the following elements $b_{ij}${\rm:}

If $0\le j<s$, then $$b_{ij}=
\begin{cases}
\kappa w_{(j+m-1)n+5\ra (j+m)n+5}\otimes e_{jn},\quad i=j;\\
-\kappa w_{(j+m)n+g(j)\ra (j+m)n+5}\otimes w_{jn\ra jn+g(j+s)},\quad i=j+3s;\\
-\kappa w_{(j+m)n+g(j)\ra (j+m)n+5}\otimes w_{jn\ra jn+g(j)+1},\quad i=j+4s;\\
-\kappa w_{(j+m)n+g(j)+1\ra (j+m)n+5}\otimes w_{jn\ra jn+g(j+s)+1},\quad i=j+7s;\\
0,\quad\text{otherwise.}\end{cases}$$

If $s\le j<2s$, then $$b_{ij}=
\begin{cases}
\kappa w_{(j+m)n+g(j+s)\ra (j+m+1)n}\otimes w_{jn+g(j+s)\ra jn+g(j+s)+1},\quad i=j+3s;\\
0,\quad\text{otherwise.}\end{cases}$$

If $2s\le j<3s$, then $$b_{ij}=
\begin{cases}
-\kappa w_{(j+m)n+g(j)\ra (j+m+1)n}\otimes e_{jn+g(j+s)},\quad i=j+s;\\
-\kappa w_{(j+m)n+g(j)+1\ra (j+m+1)n}\otimes w_{jn+g(j+s)\ra jn+g(j+s)+1},\quad i=j+5s;\\
0,\quad\text{otherwise.}\end{cases}$$

If $3s\le j<4s$, then $b_{ij}=0$.

If $4s\le j<5s$, then $$b_{ij}=
\begin{cases}
-w_{(j+m+1)n\ra (j+m+1)n+g(j)}\otimes w_{jn+g(j+s)+1\ra (j+1)n},\quad i=(j+1)_s+s;\\
-w_{(j+m)n+g(j)+1\ra (j+m+1)n+g(j)}\otimes e_{jn+g(j+s)+1},\quad i=j+3s;\\
0,\quad\text{otherwise.}\end{cases}$$

If $5s\le j<6s$, then $$b_{ij}=
\begin{cases}
w_{(j+m+1)n+g(j+s)\ra (j+m+1)n+g(j+s)+1}\otimes w_{jn+5\ra (j+1)n+g(j+s)+1},\quad i=(j+1)_s+4s;\\
-e_{(j+m+1)n+g(j+s)+1}\otimes w_{jn+5\ra (j+1)n+g(j)+1},\quad i=(j+1)_s+7s;\\
-w_{(j+m)n+5\ra (j+m+1)n+g(j+s)+1}\otimes e_{jn+5},\quad i=j+3s;\\
0,\quad\text{otherwise.}\end{cases}$$

If $6s\le j<7s$, then $b_{ij}=0$.

$(6)$ If $r_0=5$, then $\Omega^{5}(Y_t^{(11)})$ is described with
$(6s\times 8s)$-matrix with the following elements $b_{ij}${\rm:}

If $0\le j<2s$, then $b_{ij}=0$.

If $2s\le j<3s$, then $$b_{ij}=
\begin{cases}
-w_{(j+m+1)n\ra (j+m+1)n+g(j)}\otimes e_{jn+g(j+s)},\quad i=j+s;\\
0,\quad\text{otherwise.}\end{cases}$$

If $3s\le j<4s$, then $b_{ij}=0$.

If $4s\le j<5s$, then $$b_{ij}=
\begin{cases}
w_{(j+m+1)n+g(j)\ra (j+m+1)n+g(j)+1}\otimes w_{jn+g(j+s)+1\ra (j+1)n},\quad i=(j+1)_s;\\
w_{(j+m)n+5\ra (j+m+1)n+g(j)+1}\otimes e_{jn+g(j+s)+1},\quad i=j+s;\\
e_{(j+m+1)n+g(j)+1}\otimes w_{jn+g(j+s)+1\ra jn+5},\quad i=j+2s;\\
0,\quad\text{otherwise.}\end{cases}$$

If $5s\le j<6s$, then $$b_{ij}=
\begin{cases}
-\kappa w_{(j+m+1)n+g(j+s)+1\ra (j+m+1)n+5}\otimes e_{jn+5},\quad i=j+s;\\
0,\quad\text{otherwise.}\end{cases}$$

$(7)$ If $r_0=6$, then $\Omega^{6}(Y_t^{(11)})$ is described with
$(6s\times 9s)$-matrix with the following elements $b_{ij}${\rm:}

If $0\le j<s$, then $$b_{ij}=
\begin{cases}
-\kappa w_{(j+m)n+g(j)+1\ra (j+m+1)n}\otimes w_{jn\ra jn+g(j+s)},\quad i=j+4s;\\
-\kappa w_{(j+m)n+g(j)+1\ra (j+m+1)n}\otimes w_{jn\ra jn+g(j)+1},\quad i=j+5s;\\
0,\quad\text{otherwise.}\end{cases}$$

If $s\le j<2s$, then $b_{ij}=0$.

If $2s\le j<3s$, then $$b_{ij}=
\begin{cases}
w_{(j+m+1)n\ra (j+m+1)n+g(j)}\otimes w_{jn+g(j+s)\ra (j+1)n},\quad i=(j+1)_s;\\
w_{(j+m)n+g(j)+1\ra (j+m+1)n+g(j)}\otimes e_{jn+g(j+s)},\quad i=j+2s;\\
0,\quad\text{otherwise.}\end{cases}$$

If $3s\le j<4s$, then $b_{ij}=0$.

If $4s\le j<5s$, then $$b_{ij}=
\begin{cases}
-w_{(j+m+1)n\ra (j+m+1)n+g(j)+1}\otimes w_{jn+g(j+s)+1\ra (j+1)n},\quad i=(j+1)_s;\\
-w_{(j+m)n+5\ra (j+m+1)n+g(j)+1}\otimes w_{jn+g(j+s)+1\ra jn+5},\quad i=j+3s;\\
0,\quad\text{otherwise.}\end{cases}$$

If $5s\le j<6s$, then $$b_{ij}=
\begin{cases}
\kappa w_{(j+m+1)n+g(j)\ra (j+m+1)n+5}\otimes w_{jn+5\ra (j+1)n+g(j)},\quad i=(j+1)_s+2s;\\
-\kappa w_{(j+m+1)n+g(j+s)+1\ra (j+m+1)n+5}\otimes w_{jn+5\ra (j+1)n+g(j)},\quad i=(j+1)_s+4s;\\
-\kappa w_{(j+m+1)n\ra (j+m+1)n+5}\otimes e_{jn+5},\quad i=j+3s;\\
0,\quad\text{otherwise.}\end{cases}$$

$(8)$ If $r_0=7$, then $\Omega^{7}(Y_t^{(11)})$ is described with
$(7s\times 8s)$-matrix with the following elements $b_{ij}${\rm:}

If $0\le j<s$, then $b_{ij}=0$.

If $s\le j<2s$, then $$b_{ij}=
\begin{cases}
-w_{(j+m)n+5\ra (j+m+1)n+g(j+s)}\otimes w_{jn\ra jn+g(j)},\quad i=j+2s;\\
-w_{(j+m+1)n\ra (j+m+1)n+g(j+s)}\otimes w_{jn\ra jn+g(j+s)+1},\quad i=j+3s;\\
-w_{(j+m+1)n\ra (j+m+1)n+g(j+s)}\otimes w_{jn\ra jn+g(j)+1},\quad i=j+4s;\\
0,\quad\text{otherwise.}\end{cases}$$

If $2s\le j<3s$, then $b_{ij}=0$.

If $3s\le j<4s$, then $$b_{ij}=
\begin{cases}
w_{(j+m)n+5\ra (j+m+1)n+g(j+s)+1}\otimes e_{jn+g(j)},\quad i=j;\\
0,\quad\text{otherwise.}\end{cases}$$

If $4s\le j<5s$, then $$b_{ij}=
\begin{cases}
-\kappa w_{(j+m+1)n\ra (j+m+1)n+5}\otimes e_{jn+g(j)+1},\quad i=j;\\
0,\quad\text{otherwise.}\end{cases}$$

If $5s\le j<6s$, then $$b_{ij}=
\begin{cases}
-\kappa w_{(j+m+1)n+g(j+s)+1\ra (j+m+1)n+5}\otimes w_{jn+g(j)+1\ra (j+1)n},\quad i=(j+1)_s;\\
-\kappa w_{(j+m+1)n\ra (j+m+1)n+5}\otimes e_{jn+g(j)+1},\quad i=j;\\
-\kappa w_{(j+m+1)n+g(j)\ra (j+m+1)n+5}\otimes w_{jn+g(j)+1\ra jn+5},\quad i=j+2s;\\
0,\quad\text{otherwise.}\end{cases}$$

If $6s\le j<7s$, then $$b_{ij}=
\begin{cases}
\kappa w_{(j+m+1)n+g(j)+1\ra (j+m+2)n}\otimes w_{jn+5\ra (j+1)n},\quad i=(j+1)_s;\\
\kappa w_{(j+m+1)n+g(j)\ra (j+m+2)n}\otimes e_{jn+5},\quad i=j;\\
-\kappa w_{(j+m+1)n+g(j+s)\ra (j+m+2)n}\otimes e_{jn+5},\quad i=j+s;\\
0,\quad\text{otherwise.}\end{cases}$$

$(9)$ If $r_0=8$, then $\Omega^{8}(Y_t^{(11)})$ is described with
$(6s\times 6s)$-matrix with the following elements $b_{ij}${\rm:}

If $0\le j<s$, then $$b_{ij}=
\begin{cases}
\kappa e_{(j+m)n+5}\otimes w_{jn\ra (j+1)n},\quad i=(j+1)_s;\\
-\kappa w_{(j+m)n+g(j+s)+1\ra (j+m)n+5}\otimes w_{jn\ra jn+g(j+s)},\quad i=j+2s;\\
0,\quad\text{otherwise.}\end{cases}$$

If $s\le j<2s$, then $b_{ij}=0$.

If $2s\le j<3s$, then $$b_{ij}=
\begin{cases}
w_{(j+m)n+5\ra (j+m+1)n+g(j)+1}\otimes w_{jn+g(j+s)\ra (j+1)n},\quad i=(j+1)_s;\\
0,\quad\text{otherwise.}\end{cases}$$

If $3s\le j<4s$, then $$b_{ij}=
\begin{cases}
w_{(j+m)n+5\ra (j+m+1)n+g(j+s)}\otimes w_{jn+g(j+s)+1\ra (j+1)n},\quad i=(j+1)_s;\\
w_{(j+m+1)n\ra (j+m+1)n+g(j+s)}\otimes w_{jn+g(j+s)+1\ra jn+5},\quad i=j+2s;\\
0,\quad\text{otherwise.}\end{cases}$$

If $4s\le j<5s$, then $b_{ij}=0$.

If $5s\le j<6s$, then $$b_{ij}=
\begin{cases}
-\kappa w_{(j+m+1)n+g(j+s)+1\ra (j+m+2)n}\otimes w_{jn+5\ra (j+1)n+g(j+s)},\quad i=(j+1)_s+s;\\
-\kappa w_{(j+m+1)n+g(j)\ra (j+m+2)n}\otimes w_{jn+5\ra (j+1)n+g(j+s)+1},\quad i=(j+1)_s+3s;\\
0,\quad\text{otherwise.}\end{cases}$$

$(10)$ If $r_0=9$, then $\Omega^{9}(Y_t^{(11)})$ is described with
$(8s\times 7s)$-matrix with the following elements $b_{ij}${\rm:}

If $0\le j<s$, then $$b_{ij}=
\begin{cases}
w_{(j+m)n+5\ra (j+m+1)n+g(j+s)+1}\otimes e_{jn},\quad i=j;\\
0,\quad\text{otherwise.}\end{cases}$$

If $s\le j<2s$, then $$b_{ij}=
\begin{cases}
w_{(j+m+1)n\ra (j+m+1)n+g(j+s)+1}\otimes w_{jn\ra jn+g(j)},\quad i=j+s;\\
0,\quad\text{otherwise.}\end{cases}$$

If $2s\le j<3s$, then $$b_{ij}=
\begin{cases}
-\kappa w_{(j+m+1)n\ra (j+m+1)n+5}\otimes e_{jn+g(j)},\quad i=j-s;\\
0,\quad\text{otherwise.}\end{cases}$$

If $3s\le j<4s$, then $$b_{ij}=
\begin{cases}
-\kappa e_{(j+m+1)n+5}\otimes w_{jn+g(j)\ra (j+1)n},\quad i=(j+1)_s;\\
0,\quad\text{otherwise.}\end{cases}$$

If $4s\le j<5s$, then $$b_{ij}=
\begin{cases}
-\kappa w_{(j+m+1)n+5\ra (j+m+2)n}\otimes w_{jn+g(j)+1\ra (j+1)n},\quad i=(j+1)_s;\\
\kappa w_{(j+m+1)n+g(j+s)\ra (j+m+2)n}\otimes e_{jn+g(j)+1},\quad i=j-s;\\
0,\quad\text{otherwise.}\end{cases}$$

If $5s\le j<6s$, then $b_{ij}=0$.

If $6s\le j<7s$, then $$b_{ij}=
\begin{cases}
w_{(j+m+1)n+5\ra (j+m+2)n+g(j+s)}\otimes w_{jn+5\ra (j+1)n},\quad i=(j+1)_s;\\
0,\quad\text{otherwise.}\end{cases}$$

If $7s\le j<8s$, then $$b_{ij}=
\begin{cases}
-w_{(j+m+2)n\ra (j+m+2)n+g(j+s)}\otimes w_{jn+5\ra (j+1)n+g(j+s)},\quad i=(j+1)_s+s;\\
0,\quad\text{otherwise.}\end{cases}$$

$(11)$ If $r_0=10$, then $\Omega^{10}(Y_t^{(11)})$ is described with
$(9s\times 6s)$-matrix with the following elements $b_{ij}${\rm:}

If $0\le j<s$, then $$b_{ij}=
\begin{cases}
\kappa w_{(j+m)n\ra (j+m)n+5}\otimes e_{jn},\quad i=j;\\
\kappa w_{(j+m)n+g(j+s)\ra (j+m)n+5}\otimes w_{jn\ra jn+g(j)},\quad i=j+s;\\
-\kappa w_{(j+m)n+g(j)\ra (j+m)n+5}\otimes w_{jn\ra jn+g(j+s)},\quad i=j+2s;\\
-\kappa w_{(j+m)n+g(j+s)+1\ra (j+m)n+5}\otimes w_{jn\ra jn+g(j)+1},\quad i=j+3s;\\
\kappa w_{(j+m)n+g(j)+1\ra (j+m)n+5}\otimes w_{jn\ra jn+g(j+s)+1},\quad i=j+4s;\\
-\kappa e_{(j+m)n+5}\otimes w_{jn\ra jn+5},\quad i=j+5s;\\
0,\quad\text{otherwise.}\end{cases}$$

If $s\le j<2s$, then $$b_{ij}=
\begin{cases}
-\kappa e_{(j+m+1)n}\otimes w_{jn\ra (j+1)n},\quad i=(j+1)_s;\\
\kappa w_{(j+m)n+g(j)\ra (j+m+1)n}\otimes w_{jn\ra jn+g(j+s)},\quad i=j;\\
0,\quad\text{otherwise.}\end{cases}$$

If $2s\le j<3s$, then $$b_{ij}=
\begin{cases}
-w_{(j+m+1)n\ra (j+m+1)n+g(j)}\otimes w_{jn+g(j)\ra (j+1)n},\quad i=(j+1)_s;\\
0,\quad\text{otherwise.}\end{cases}$$

If $3s\le j<4s$, then $b_{ij}=0$.

If $4s\le j<5s$, then $$b_{ij}=
\begin{cases}
-w_{(j+m)n+g(j+s)+1\ra (j+m+1)n+g(j+s)}\otimes e_{jn+g(j)+1},\quad i=j-s;\\
-w_{(j+m)n+5\ra (j+m+1)n+g(j+s)}\otimes w_{jn+g(j)+1\ra jn+5},\quad i=j+s;\\
0,\quad\text{otherwise.}\end{cases}$$

If $5s\le j<6s$, then $b_{ij}=0$.

If $6s\le j<7s$, then $$b_{ij}=
\begin{cases}
w_{(j+m+1)n\ra (j+m+1)n+g(j)+1}\otimes w_{jn+g(j)+1\ra (j+1)n},\quad i=(j+1)_s;\\
0,\quad\text{otherwise.}\end{cases}$$

If $7s\le j<9s$, then $b_{ij}=0$.

\medskip
$({\rm II})$ Represent an arbitrary $t_0\in\N$ in the form
$t_0=11\ell_0+r_0$, where $0\le r_0\le 10.$ Then
$\Omega^{t_0}(Y_t^{(11)})$ is a $\Omega^{r_0}(Y_t^{(11)})$, whose
left components twisted by $\sigma^{\ell_0}$, and coefficients
multiplied by $(-1)^{\ell_0}$.
\end{pr}
\begin{pr}[Translates for the case 12]
$({\rm I})$ Let $r_0\in\N$, $r_0<11$. $r_0$-translates of the
elements $Y^{(12)}_t$ are described by the following way.

$(1)$ If $r_0=0$, then $\Omega^{0}(Y_t^{(12)})$ is described with
$(8s\times 6s)$-matrix with the following elements $b_{ij}${\rm:}

If $0\le j<s$, then $$b_{ij}=
\begin{cases}
-w_{(j+m)n\ra (j+m)n+g(j+s)}\otimes e_{jn},\quad i=j;\\
0,\quad\text{otherwise.}\end{cases}$$

If $s\le j<4s$, then $b_{ij}=0$.

If $4s\le j<5s$, then $$b_{ij}=
\begin{cases}
\kappa w_{(j+m)n+g(j)+1\ra (j+m)n+5}\otimes e_{jn+g(j)+1},\quad i=j-s;\\
0,\quad\text{otherwise.}\end{cases}$$

If $5s\le j<8s$, then $b_{ij}=0$.

$(2)$ If $r_0=1$, then $\Omega^{1}(Y_t^{(12)})$ is described with
$(9s\times 7s)$-matrix with the following elements $b_{ij}${\rm:}

If $0\le j<s$, then $$b_{ij}=
\begin{cases}
-\kappa w_{(j+m)n+g(j+s)\ra (j+m+1)n}\otimes e_{jn},\quad i=j+s;\\
-\kappa w_{(j+m)n+g(j+s)+1\ra (j+m+1)n}\otimes w_{jn\ra jn+g(j+s)},\quad i=j+3s;\\
\kappa w_{(j+m)n+5\ra (j+m+1)n}\otimes w_{jn\ra jn+g(j)+1},\quad i=j+4s;\\
-\kappa w_{(j+m)n+5\ra (j+m+1)n}\otimes w_{jn\ra jn+g(j+s)+1},\quad i=j+5s;\\
0,\quad\text{otherwise.}\end{cases}$$

If $s\le j<2s$, then $$b_{ij}=
\begin{cases}
e_{(j+m+1)n+g(j)}\otimes w_{jn+g(j+s)\ra (j+1)n},\quad i=(j+1)_s+s;\\
0,\quad\text{otherwise.}\end{cases}$$

If $2s\le j<3s$, then $b_{ij}=0$.

If $3s\le j<4s$, then $$b_{ij}=
\begin{cases}
w_{(j+m)n+g(j+s)+1\ra (j+m+1)n+g(j+s)+1}\otimes e_{jn+g(j+s)},\quad i=j-s;\\
0,\quad\text{otherwise.}\end{cases}$$

If $4s\le j<5s$, then $$b_{ij}=
\begin{cases}
-w_{(j+m)n+g(j+s)+1\ra (j+m+1)n+g(j+s)+1}\otimes e_{jn+g(j+s)},\quad i=j-s;\\
-w_{(j+m)n+5\ra (j+m+1)n+g(j+s)+1}\otimes w_{jn+g(j+s)\ra jn+g(j+s)+1},\quad i=j+s;\\
-w_{(j+m+1)n\ra (j+m+1)n+g(j+s)+1}\otimes w_{jn+g(j+s)\ra jn+5},\quad i=j+2s;\\
0,\quad\text{otherwise.}\end{cases}$$

If $5s\le j<7s$, then $b_{ij}=0$.

If $7s\le j<8s$, then $$b_{ij}=
\begin{cases}
\kappa w_{(j+m+1)n+g(j+s)\ra (j+m+1)n+5}\otimes w_{jn+5\ra (j+1)n},\quad i=(j+1)_s;\\
\kappa w_{(j+m+1)n+g(j+s)+1\ra (j+m+1)n+5}\otimes w_{jn+5\ra (j+1)n+g(j+s)},\quad i=(j+1)_s+2s;\\
0,\quad\text{otherwise.}\end{cases}$$

If $8s\le j<9s$, then $b_{ij}=0$.

$(3)$ If $r_0=2$, then $\Omega^{2}(Y_t^{(12)})$ is described with
$(8s\times 6s)$-matrix with the following elements $b_{ij}${\rm:}

If $0\le j<s$, then $$b_{ij}=
\begin{cases}
w_{(j+m)n+g(j+s)\ra (j+m)n+g(j+s)+1}\otimes w_{jn\ra jn+g(j)+1},\quad i=j+3s;\\
0,\quad\text{otherwise.}\end{cases}$$

If $s\le j<2s$, then $$b_{ij}=
\begin{cases}
-w_{(j+m-1)n+5\ra (j+m)n+g(j+s)+1}\otimes e_{jn},\quad i=j-s;\\
0,\quad\text{otherwise.}\end{cases}$$

If $2s\le j<3s$, then $$b_{ij}=
\begin{cases}
\kappa e_{(j+m)n+5}\otimes w_{jn+g(j)\ra (j+1)n},\quad i=(j+1)_s;\\
0,\quad\text{otherwise.}\end{cases}$$

If $3s\le j<4s$, then $$b_{ij}=
\begin{cases}
\kappa w_{(j+m)n+g(j+s)\ra (j+m)n+5}\otimes w_{jn+g(j)\ra jn+g(j)+1},\quad i=j+s;\\
0,\quad\text{otherwise.}\end{cases}$$

If $4s\le j<5s$, then $b_{ij}=0$.

If $5s\le j<6s$, then $$b_{ij}=
\begin{cases}
\kappa w_{(j+m)n+5\ra (j+m+1)n}\otimes w_{jn+g(j)+1\ra (j+1)n},\quad i=(j+1)_s;\\
0,\quad\text{otherwise.}\end{cases}$$

If $6s\le j<7s$, then $$b_{ij}=
\begin{cases}
w_{(j+m)n+5\ra (j+m+1)n+g(j+s)}\otimes w_{jn+5\ra (j+1)n},\quad i=(j+1)_s;\\
0,\quad\text{otherwise.}\end{cases}$$

If $7s\le j<8s$, then $b_{ij}=0$.

$(4)$ If $r_0=3$, then $\Omega^{3}(Y_t^{(12)})$ is described with
$(6s\times 8s)$-matrix with the following elements $b_{ij}${\rm:}

If $0\le j<s$, then $$b_{ij}=
\begin{cases}
-\kappa w_{(j+m)n+g(j+s)+1\ra (j+m)n+5}\otimes e_{jn},\quad i=j+s;\\
\kappa e_{(j+m)n+5}\otimes w_{jn\ra jn+g(j+s)},\quad i=j+3s;\\
0,\quad\text{otherwise.}\end{cases}$$

If $s\le j<6s$, then $b_{ij}=0$.

$(5)$ If $r_0=4$, then $\Omega^{4}(Y_t^{(12)})$ is described with
$(7s\times 9s)$-matrix with the following elements $b_{ij}${\rm:}

If $0\le j<s$, then $$b_{ij}=
\begin{cases}
-\kappa e_{(j+m)n+5}\otimes w_{jn\ra (j+1)n},\quad i=(j+1)_s;\\
-\kappa w_{(j+m)n+g(j)\ra (j+m)n+5}\otimes w_{jn\ra jn+g(j+s)},\quad i=j+3s;\\
\kappa w_{(j+m)n+g(j)+1\ra (j+m)n+5}\otimes w_{jn\ra jn+g(j+s)+1},\quad i=j+7s;\\
0,\quad\text{otherwise.}\end{cases}$$

If $s\le j<2s$, then $$b_{ij}=
\begin{cases}
\kappa w_{(j+m)n+g(j)\ra (j+m+1)n}\otimes e_{jn+g(j+s)},\quad i=j+s;\\
-\kappa w_{(j+m)n+g(j)+1\ra (j+m+1)n}\otimes w_{jn+g(j+s)\ra jn+g(j+s)+1},\quad i=j+5s;\\
0,\quad\text{otherwise.}\end{cases}$$

If $2s\le j<6s$, then $b_{ij}=0$.

If $6s\le j<7s$, then $$b_{ij}=
\begin{cases}
w_{(j+m+1)n+g(j)\ra (j+m+1)n+g(j)+1}\otimes w_{jn+5\ra (j+1)n+g(j+s)},\quad i=(j+1)_s+3s;\\
-e_{(j+m+1)n+g(j)+1}\otimes w_{jn+5\ra (j+1)n+g(j+s)+1},\quad i=(j+1)_s+7s;\\
0,\quad\text{otherwise.}\end{cases}$$

$(6)$ If $r_0=5$, then $\Omega^{5}(Y_t^{(12)})$ is described with
$(6s\times 8s)$-matrix with the following elements $b_{ij}${\rm:}

If $0\le j<s$, then $$b_{ij}=
\begin{cases}
\kappa e_{(j+m+1)n}\otimes w_{jn\ra jn+g(j)},\quad i=j+2s;\\
\kappa e_{(j+m+1)n}\otimes w_{jn\ra jn+g(j+s)},\quad i=j+3s;\\
0,\quad\text{otherwise.}\end{cases}$$

If $s\le j<6s$, then $b_{ij}=0$.

$(7)$ If $r_0=6$, then $\Omega^{6}(Y_t^{(12)})$ is described with
$(6s\times 9s)$-matrix with the following elements $b_{ij}${\rm:}

If $0\le j<s$, then $$b_{ij}=
\begin{cases}
-\kappa w_{(j+m)n\ra (j+m+1)n}\otimes e_{jn},\quad i=j;\\
0,\quad\text{otherwise.}\end{cases}$$

If $s\le j<3s$, then $$b_{ij}=
\begin{cases}
e_{(j+m+1)n+g(j+s)}\otimes w_{jn+g(j+s)\ra (j+1)n+g(j+s)},\quad i=(j+1)_s+j_2s;\\
0,\quad\text{otherwise.}\end{cases}$$

If $3s\le j<5s$, then $$b_{ij}=
\begin{cases}
w_{(j+m+1)n+g(j+s)\ra (j+m+1)n+g(j+s)+1}\otimes w_{jn+g(j+s)+1\ra (j+1)n+g(j+s)},\\\quad\quad\quad i=(j+1)_s+(j_2-2)s;\\
0,\quad\text{otherwise.}\end{cases}$$

If $5s\le j<6s$, then $$b_{ij}=
\begin{cases}
-\kappa w_{(j+m+1)n\ra (j+m+1)n+5}\otimes w_{jn+5\ra (j+1)n},\quad i=(j+1)_s;\\
\kappa w_{(j+m)n+5\ra (j+m+1)n+5}\otimes e_{jn+5},\quad i=j+2s;\\
0,\quad\text{otherwise.}\end{cases}$$

$(8)$ If $r_0=7$, then $\Omega^{7}(Y_t^{(12)})$ is described with
$(7s\times 8s)$-matrix with the following elements $b_{ij}${\rm:}

If $0\le j<4s$, then $b_{ij}=0$.

If $4s\le j<6s$, then $$b_{ij}=
\begin{cases}
\kappa w_{(j+m+1)n\ra (j+m+1)n+5}\otimes e_{jn+g(j)+1},\quad i=j;\\
-\kappa w_{(j+m+1)n+g(j)\ra (j+m+1)n+5}\otimes w_{jn+g(j)+1\ra jn+5},\quad i=j+2s;\\
0,\quad\text{otherwise.}\end{cases}$$

If $6s\le j<7s$, then $b_{ij}=0$.

$(9)$ If $r_0=8$, then $\Omega^{8}(Y_t^{(12)})$ is described with
$(6s\times 6s)$-matrix with the following elements $b_{ij}${\rm:}

If $0\le j<s$, then $$b_{ij}=
\begin{cases}
\kappa w_{(j+m)n+g(j+s)\ra (j+m)n+5}\otimes w_{jn\ra jn+g(j)+1},\quad i=j+3s;\\
-\kappa w_{(j+m)n+g(j)\ra (j+m)n+5}\otimes w_{jn\ra jn+g(j+s)+1},\quad i=j+4s;\\
0,\quad\text{otherwise.}\end{cases}$$

If $s\le j<5s$, then $b_{ij}=0$.

If $5s\le j<6s$, then $$b_{ij}=
\begin{cases}
\kappa w_{(j+m+1)n+g(j+s)+1\ra (j+m+2)n}\otimes w_{jn+5\ra (j+1)n+g(j+s)},\quad i=(j+1)_s+s;\\
0,\quad\text{otherwise.}\end{cases}$$

$(10)$ If $r_0=9$, then $\Omega^{9}(Y_t^{(12)})$ is described with
$(8s\times 7s)$-matrix with the following elements $b_{ij}${\rm:}

If $0\le j<4s$, then $b_{ij}=0$.

If $4s\le j<5s$, then $$b_{ij}=
\begin{cases}
-\kappa e_{(j+m+2)n}\otimes w_{jn+g(j)+1\ra (j+1)n+g(j)},\quad i=(j+1)_s+s;\\
0,\quad\text{otherwise.}\end{cases}$$

If $5s\le j<6s$, then $b_{ij}=0$.

If $6s\le j<7s$, then $$b_{ij}=
\begin{cases}
-w_{(j+m+1)n+g(j)+1\ra (j+m+2)n+g(j)}\otimes e_{jn+5},\quad i=j-s;\\
0,\quad\text{otherwise.}\end{cases}$$

If $7s\le j<8s$, then $b_{ij}=0$.

$(11)$ If $r_0=10$, then $\Omega^{10}(Y_t^{(12)})$ is described with
$(9s\times 6s)$-matrix with the following elements $b_{ij}${\rm:}

If $0\le j<2s$, then $b_{ij}=0$.

If $2s\le j<3s$, then $$b_{ij}=
\begin{cases}
e_{(j+m+1)n+g(j+s)}\otimes w_{jn+g(j)\ra (j+1)n+g(j)},\quad i=(j+1)_s+s;\\
0,\quad\text{otherwise.}\end{cases}$$

If $3s\le j<4s$, then $$b_{ij}=
\begin{cases}
w_{(j+m)n+g(j+s)\ra (j+m+1)n+g(j+s)}\otimes e_{jn+g(j)},\quad i=j-s;\\
-w_{(j+m)n+g(j+s)+1\ra (j+m+1)n+g(j+s)}\otimes w_{jn+g(j)\ra jn+g(j)+1},\quad i=j+s;\\
0,\quad\text{otherwise.}\end{cases}$$

If $4s\le j<5s$, then $$b_{ij}=
\begin{cases}
w_{(j+m+1)n\ra (j+m+1)n+g(j)}\otimes w_{jn+g(j)+1\ra (j+1)n},\quad i=(j+1)_s;\\
-w_{(j+m)n+5\ra (j+m+1)n+g(j)}\otimes w_{jn+g(j)+1\ra jn+5},\quad i=j+s;\\
0,\quad\text{otherwise.}\end{cases}$$

If $5s\le j<6s$, then $b_{ij}=0$.

If $6s\le j<7s$, then $$b_{ij}=
\begin{cases}
w_{(j+m+1)n+g(j+s)\ra (j+m+1)n+g(j+s)+1}\otimes w_{jn+g(j)+1\ra (j+1)n+g(j)},\quad i=(j+1)_s+s;\\
-e_{(j+m+1)n+g(j+s)+1}\otimes w_{jn+g(j)+1\ra (j+1)n+g(j)+1},\quad i=(j+1)_s+3s;\\
0,\quad\text{otherwise.}\end{cases}$$

If $7s\le j<8s$, then $$b_{ij}=
\begin{cases}
-w_{(j+m)n+g(j+s)+1\ra (j+m+1)n+g(j+s)+1}\otimes e_{jn+g(j)+1},\quad i=j-3s;\\
0,\quad\text{otherwise.}\end{cases}$$

If $8s\le j<9s$, then $b_{ij}=0$.

\medskip
$({\rm II})$ Represent an arbitrary $t_0\in\N$ in the form
$t_0=11\ell_0+r_0$, where $0\le r_0\le 10.$ Then
$\Omega^{t_0}(Y_t^{(12)})$ is a $\Omega^{r_0}(Y_t^{(12)})$, whose
left components twisted by $\sigma^{\ell_0}$.
\end{pr}
\begin{pr}[Translates for the case 13]
$({\rm I})$ Let $r_0\in\N$, $r_0<11$. $r_0$-translates of the
elements $Y^{(13)}_t$ are described by the following way.

$(1)$ If $r_0=0$, then $\Omega^{0}(Y_t^{(13)})$ is described with
$(9s\times 6s)$-matrix with the following elements $b_{ij}${\rm:}

If $0\le j<s$, then $b_{ij}=0$.

If $s\le j<3s$, then $$b_{ij}=
\begin{cases}
-f_1(j,2s)e_{(j+m)n+g(j+s)}\otimes e_{jn+g(j+s)},\quad i=j;\\
0,\quad\text{otherwise.}\end{cases}$$

If $3s\le j<5s$, then $b_{ij}=0$.

If $5s\le j<7s$, then $$b_{ij}=
\begin{cases}
-f_1(j,6s)e_{(j+m)n+g(j+s)+1}\otimes e_{jn+g(j+s)+1},\quad i=j-2s;\\
0,\quad\text{otherwise.}\end{cases}$$

If $7s\le j<8s$, then $b_{ij}=0$.

If $8s\le j<9s$, then $$b_{ij}=
\begin{cases}
-\kappa w_{(j+m)n+5\ra (j+m+1)n}\otimes e_{jn+5},\quad i=j-3s;\\
0,\quad\text{otherwise.}\end{cases}$$

$(2)$ If $r_0=1$, then $\Omega^{1}(Y_t^{(13)})$ is described with
$(8s\times 7s)$-matrix with the following elements $b_{ij}${\rm:}

If $0\le j<2s$, then $$b_{ij}=
\begin{cases}
f_1(j,s)e_{(j+m)n+g(j)+1}\otimes w_{jn\ra jn+g(j)},\quad i=j+2s;\\
0,\quad\text{otherwise.}\end{cases}$$

If $2s\le j<4s$, then $$b_{ij}=
\begin{cases}
-\kappa f_1(j,3s)w_{(j+m)n+g(j)+1\ra (j+m)n+5}\otimes e_{jn+g(j)},\quad i=j;\\
0,\quad\text{otherwise.}\end{cases}$$

If $4s\le j<6s$, then $$b_{ij}=
\begin{cases}
-\kappa f_1(j,5s)w_{(j+m)n+5\ra (j+m+1)n}\otimes e_{jn+g(j)+1},\quad i=j;\\
0,\quad\text{otherwise.}\end{cases}$$

If $6s\le j<8s$, then $$b_{ij}=
\begin{cases}
f_1(j,7s)e_{(j+m+1)n+g(j)}\otimes w_{jn+5\ra (j+1)n},\quad i=(j+1)_s+(j_2-6)s;\\
f_1(j,7s)w_{(j+m+1)n\ra (j+m+1)n+g(j)}\otimes e_{jn+5},\quad i=(j)_s+6s;\\
0,\quad\text{otherwise.}\end{cases}$$

$(3)$ If $r_0=2$, then $\Omega^{2}(Y_t^{(13)})$ is described with
$(6s\times 6s)$-matrix with the following elements $b_{ij}${\rm:}

If $0\le j<s$, then $b_{ij}=0$.

If $s\le j<3s$, then $$b_{ij}=
\begin{cases}
-f_1(j,2s)e_{(j+m)n+g(j+s)+1}\otimes e_{jn+g(j+s)},\quad i=j;\\
0,\quad\text{otherwise.}\end{cases}$$

If $3s\le j<5s$, then $$b_{ij}=
\begin{cases}
-f_1(j,4s)e_{(j+m)n+g(j)}\otimes e_{jn+g(j+s)+1},\quad i=j;\\
0,\quad\text{otherwise.}\end{cases}$$

If $5s\le j<6s$, then $$b_{ij}=
\begin{cases}
\kappa w_{(j+m)n+5\ra (j+m+1)n}\otimes w_{jn+5\ra (j+1)n},\quad i=(j+1)_s;\\
0,\quad\text{otherwise.}\end{cases}$$

$(4)$ If $r_0=3$, then $\Omega^{3}(Y_t^{(13)})$ is described with
$(7s\times 8s)$-matrix with the following elements $b_{ij}${\rm:}

If $0\le j<s$, then $$b_{ij}=
\begin{cases}
-\kappa w_{(j+m)n+g(j)+1\ra (j+m)n+5}\otimes e_{jn},\quad i=j;\\
\kappa w_{(j+m)n+g(j+s)+1\ra (j+m)n+5}\otimes e_{jn},\quad i=j+s;\\
0,\quad\text{otherwise.}\end{cases}$$

If $s\le j<3s$, then $$b_{ij}=
\begin{cases}
-\kappa f_1(j,2s)w_{(j+m)n+5\ra (j+m+1)n}\otimes e_{jn+g(j+s)},\quad i=j+s;\\
0,\quad\text{otherwise.}\end{cases}$$

If $3s\le j<5s$, then $$b_{ij}=
\begin{cases}
-f_1(j,4s)w_{(j+m+1)n\ra (j+m+1)n+g(j)}\otimes e_{jn+g(j+s)+1},\quad i=j+s;\\
-f_1(j,4s)e_{(j+m+1)n+g(j)}\otimes w_{jn+g(j+s)+1\ra jn+5},\quad i=j+2(1+f_0(j,4s))s;\\
0,\quad\text{otherwise.}\end{cases}$$

If $5s\le j<7s$, then $$b_{ij}=
\begin{cases}
e_{(j+m+1)n+g(j+s)+1}\otimes w_{jn+5\ra (j+1)n},\quad i=(j+1)_s+(j_2-5)s;\\
0,\quad\text{otherwise.}\end{cases}$$

$(5)$ If $r_0=4$, then $\Omega^{4}(Y_t^{(13)})$ is described with
$(6s\times 9s)$-matrix with the following elements $b_{ij}${\rm:}

If $0\le j<s$, then $$b_{ij}=
\begin{cases}
-\kappa w_{(j+m-1)n+5\ra (j+m)n}\otimes e_{jn},\quad i=j;\\
0,\quad\text{otherwise.}\end{cases}$$

If $s\le j<3s$, then $$b_{ij}=
\begin{cases}
-f_1(j,2s)e_{(j+m)n+g(j)}\otimes e_{jn+g(j+s)},\quad i=j+s;\\
0,\quad\text{otherwise.}\end{cases}$$

If $3s\le j<5s$, then $$b_{ij}=
\begin{cases}
-f_1(j,4s)e_{(j+m)n+g(j)+1}\otimes e_{jn+g(j+s)+1},\quad i=j+3s;\\
0,\quad\text{otherwise.}\end{cases}$$

If $5s\le j<6s$, then $b_{ij}=0$.

$(6)$ If $r_0=5$, then $\Omega^{5}(Y_t^{(13)})$ is described with
$(6s\times 8s)$-matrix with the following elements $b_{ij}${\rm:}

If $0\le j<s$, then $$b_{ij}=
\begin{cases}
\kappa e_{(j+m+1)n}\otimes w_{jn\ra jn+g(j)},\quad i=j+2s;\\
\kappa e_{(j+m+1)n}\otimes w_{jn\ra jn+g(j+s)},\quad i=j+3s;\\
0,\quad\text{otherwise.}\end{cases}$$

If $s\le j<3s$, then $$b_{ij}=
\begin{cases}
-f_1(j,2s)e_{(j+m+1)n+g(j)}\otimes w_{jn+g(j+s)\ra (j+1)n},\quad i=(j+1)_s+f_0(j,2s)s;\\
-f_1(j,2s)w_{(j+m+1)n\ra (j+m+1)n+g(j)}\otimes e_{jn+g(j+s)},\quad i=j+s;\\
0,\quad\text{otherwise.}\end{cases}$$

If $3s\le j<5s$, then $$b_{ij}=
\begin{cases}
f_1(j,4s)w_{(j+m)n+5\ra (j+m+1)n+g(j)+1}\otimes e_{jn+g(j+s)+1},\quad i=j+s;\\
e_{(j+m+1)n+g(j)+1}\otimes w_{jn+g(j+s)+1\ra jn+5},\quad i=j+2(1+f_0(j,4s))s;\\
0,\quad\text{otherwise.}\end{cases}$$

If $5s\le j<6s$, then $$b_{ij}=
\begin{cases}
-\kappa w_{(j+m+1)n+g(j+s)+1\ra (j+m+1)n+5}\otimes e_{jn+5},\quad i=j+s;\\
\kappa w_{(j+m+1)n+g(j)+1\ra (j+m+1)n+5}\otimes e_{jn+5},\quad i=j+2s;\\
0,\quad\text{otherwise.}\end{cases}$$

$(7)$ If $r_0=6$, then $\Omega^{6}(Y_t^{(13)})$ is described with
$(7s\times 9s)$-matrix with the following elements $b_{ij}${\rm:}

If $0\le j<2s$, then $$b_{ij}=
\begin{cases}
-f_1(j,s)e_{(j+m)n+g(j+s)}\otimes w_{jn\ra jn+g(j+s)},\quad i=j+2f_0(j,s)s;\\
0,\quad\text{otherwise.}\end{cases}$$

If $2s\le j<4s$, then $$b_{ij}=
\begin{cases}
f_1(j,3s)e_{(j+m)n+g(j+s)+1}\otimes e_{jn+g(j)},\quad i=j+s;\\
0,\quad\text{otherwise.}\end{cases}$$

If $4s\le j<6s$, then $$b_{ij}=
\begin{cases}
\kappa w_{(j+m)n+g(j)+1\ra (j+m)n+5}\otimes e_{jn+g(j)+1},\quad i=j+s;\\
0,\quad\text{otherwise.}\end{cases}$$

If $6s\le j<7s$, then $$b_{ij}=
\begin{cases}
-\kappa e_{(j+m+1)n}\otimes w_{jn+5\ra (j+1)n},\quad i=(j+1)_s;\\
\kappa w_{(j+m)n+5\ra (j+m+1)n}\otimes e_{jn+5},\quad i=j+s;\\
0,\quad\text{otherwise.}\end{cases}$$

$(8)$ If $r_0=7$, then $\Omega^{7}(Y_t^{(13)})$ is described with
$(6s\times 8s)$-matrix with the following elements $b_{ij}${\rm:}

If $0\le j<s$, then $$b_{ij}=
\begin{cases}
\kappa w_{(j+m)n+g(j)+1\ra (j+m)n+5}\otimes e_{jn},\quad i=j;\\
-\kappa w_{(j+m)n+g(j+s)+1\ra (j+m)n+5}\otimes e_{jn},\quad i=j+s;\\
0,\quad\text{otherwise.}\end{cases}$$

If $s\le j<3s$, then $$b_{ij}=
\begin{cases}
-f_1(j,2s)e_{(j+m+1)n+g(j)+1}\otimes w_{jn+g(j+s)\ra (j+1)n},\quad i=(j+1)_s+f_0(j,2s)s;\\
f_1(j,2s)w_{(j+m)n+5\ra (j+m+1)n+g(j)+1}\otimes e_{jn+g(j+s)},\quad i=j+s;\\
0,\quad\text{otherwise.}\end{cases}$$

If $3s\le j<5s$, then $$b_{ij}=
\begin{cases}
-f_1(j,4s)w_{(j+m+1)n\ra (j+m+1)n+g(j+s)}\otimes e_{jn+g(j+s)+1},\quad i=j+s;\\
f_1(j,4s)e_{(j+m+1)n+g(j+s)}\otimes w_{jn+g(j+s)+1\ra jn+5},\quad i=j+3s;\\
0,\quad\text{otherwise.}\end{cases}$$

If $5s\le j<6s$, then $$b_{ij}=
\begin{cases}
\kappa w_{(j+m+1)n+g(j)+1\ra (j+m+2)n}\otimes w_{jn+5\ra (j+1)n},\quad i=(j+1)_s+s;\\
-\kappa e_{(j+m+2)n}\otimes w_{jn+5\ra (j+1)n+g(j+s)+1},\quad i=(j+1)_s+4s;\\
-\kappa e_{(j+m+2)n}\otimes w_{jn+5\ra (j+1)n+g(j)+1},\quad i=(j+1)_s+5s;\\
0,\quad\text{otherwise.}\end{cases}$$

$(9)$ If $r_0=8$, then $\Omega^{8}(Y_t^{(13)})$ is described with
$(8s\times 6s)$-matrix with the following elements $b_{ij}${\rm:}

If $0\le j<2s$, then $$b_{ij}=
\begin{cases}
-f_1(j,s)w_{(j+m-1)n+5\ra (j+m)n+g(j+s)+1}\otimes e_{jn},\quad i=(j)_s;\\
-e_{(j+m)n+g(j+s)+1}\otimes w_{jn\ra jn+g(j+s)},\quad i=(j)_s+(1+f_0(j,s))s;\\
0,\quad\text{otherwise.}\end{cases}$$

If $2s\le j<4s$, then $$b_{ij}=
\begin{cases}
\kappa w_{(j+m)n+g(j)+1\ra (j+m)n+5}\otimes e_{jn+g(j)},\quad i=j-s;\\
0,\quad\text{otherwise.}\end{cases}$$

If $4s\le j<5s$, then $$b_{ij}=
\begin{cases}
-\kappa w_{(j+m)n+5\ra (j+m+1)n}\otimes w_{jn+g(j)+1\ra (j+1)n},\quad i=(j+1)_s;\\
\kappa w_{(j+m)n+g(j+s)\ra (j+m+1)n}\otimes e_{jn+g(j)+1},\quad i=j-s;\\
-\kappa e_{(j+m+1)n}\otimes w_{jn+g(j)+1\ra jn+5},\quad i=j+s;\\
0,\quad\text{otherwise.}\end{cases}$$

If $5s\le j<6s$, then $$b_{ij}=
\begin{cases}
\kappa w_{(j+m)n+g(j+s)\ra (j+m+1)n}\otimes e_{jn+g(j)+1},\quad i=j-s;\\
-\kappa e_{(j+m+1)n}\otimes w_{jn+g(j)+1\ra jn+5},\quad i=j;\\
0,\quad\text{otherwise.}\end{cases}$$

If $6s\le j<8s$, then $$b_{ij}=
\begin{cases}
f_1(j,7s)e_{(j+m+1)n+g(j+s)}\otimes w_{jn+5\ra (j+1)n+g(j)+1},\quad i=(j+1)_s+(j_2-3)s;\\
0,\quad\text{otherwise.}\end{cases}$$

$(10)$ If $r_0=9$, then $\Omega^{9}(Y_t^{(13)})$ is described with
$(9s\times 7s)$-matrix with the following elements $b_{ij}${\rm:}

If $0\le j<s$, then $b_{ij}=0$.

If $s\le j<2s$, then $$b_{ij}=
\begin{cases}
\kappa w_{(j+m)n+5\ra (j+m+1)n}\otimes e_{jn},\quad i=j-s;\\
0,\quad\text{otherwise.}\end{cases}$$

If $2s\le j<4s$, then $$b_{ij}=
\begin{cases}
-f_1(j,3s)w_{(j+m+1)n\ra (j+m+1)n+g(j)}\otimes e_{jn+g(j)},\quad i=j-s;\\
0,\quad\text{otherwise.}\end{cases}$$

If $4s\le j<6s$, then $$b_{ij}=
\begin{cases}
-f_1(j,5s)e_{(j+m+1)n+g(j+s)}\otimes e_{jn+g(j)+1},\quad i=j-s;\\
0,\quad\text{otherwise.}\end{cases}$$

If $6s\le j<8s$, then $$b_{ij}=
\begin{cases}
f_1(j,7s)e_{(j+m+1)n+g(j)+1}\otimes w_{jn+g(j)+1\ra jn+5},\quad i=j-s;\\
0,\quad\text{otherwise.}\end{cases}$$

If $8s\le j<9s$, then $$b_{ij}=
\begin{cases}
-\kappa e_{(j+m+1)n+5}\otimes w_{jn+5\ra (j+1)n},\quad i=(j+1)_s;\\
0,\quad\text{otherwise.}\end{cases}$$

$(11)$ If $r_0=10$, then $\Omega^{10}(Y_t^{(13)})$ is described with
$(8s\times 6s)$-matrix with the following elements $b_{ij}${\rm:}

If $0\le j<2s$, then $$b_{ij}=
\begin{cases}
w_{(j+m)n\ra (j+m)n+g(j+s)}\otimes e_{jn},\quad i=(j)_s;\\
-f_1(j,s)e_{(j+m)n+g(j+s)}\otimes w_{jn\ra jn+g(j)},\quad i=j+s;\\
0,\quad\text{otherwise.}\end{cases}$$

If $2s\le j<4s$, then $$b_{ij}=
\begin{cases}
\kappa w_{(j+m)n+g(j+s)\ra (j+m+1)n}\otimes e_{jn+g(j)},\quad i=j-s;\\
-\kappa w_{(j+m)n+g(j+s)+1\ra (j+m+1)n}\otimes w_{jn+g(j)\ra jn+g(j)+1},\quad i=j+s;\\
-\kappa f_1(j,3s)w_{(j+m)n+5\ra (j+m+1)n}\otimes w_{jn+g(j)\ra jn+5},\quad i=(j)_s+5s;\\
0,\quad\text{otherwise.}\end{cases}$$

If $4s\le j<6s$, then $$b_{ij}=
\begin{cases}
\kappa w_{(j+m)n+g(j+s)+1\ra (j+m)n+5}\otimes e_{jn+g(j)+1},\quad i=j-s;\\
\kappa f_1(j,5s)e_{(j+m)n+5}\otimes w_{jn+g(j)+1\ra jn+5},\quad i=(j)_s+5s;\\
0,\quad\text{otherwise.}\end{cases}$$

If $6s\le j<8s$, then $$b_{ij}=
\begin{cases}
f_1(j,7s)w_{(j+m+1)n\ra (j+m+1)n+g(j+s)+1}\otimes w_{jn+5\ra (j+1)n},\quad i=(j+1)_s;\\
-w_{(j+m+1)n+g(j+s)\ra (j+m+1)n+g(j+s)+1}\otimes w_{jn+5\ra (j+1)n+g(j)},\quad i=(j+1)_s+(j_2-5)s;\\
e_{(j+m+1)n+g(j+s)+1}\otimes w_{jn+5\ra (j+1)n+g(j)+1},\quad i=(j+1)_s+(j_2-3)s;\\
0,\quad\text{otherwise.}\end{cases}$$

\medskip
$({\rm II})$ Represent an arbitrary $t_0\in\N$ in the form
$t_0=11\ell_0+r_0$, where $0\le r_0\le 10.$ Then
$\Omega^{t_0}(Y_t^{(13)})$ is a $\Omega^{r_0}(Y_t^{(13)})$, whose
left components twisted by $\sigma^{\ell_0}$.
\end{pr}
\begin{pr}[Translates for the case 14]
$({\rm I})$ Let $r_0\in\N$, $r_0<11$. $r_0$-translates of the
elements $Y^{(14)}_t$ are described by the following way.

$(1)$ If $r_0=0$, then $\Omega^{0}(Y_t^{(14)})$ is described with
$(9s\times 6s)$-matrix with one nonzero element that is of the
following form{\rm:}
$$b_{(-3\ell_0)_s,(-3\ell_0)_s}=\kappa w_{(j+m)n\ra (j+m+1)n}\otimes e_{jn}.$$

$(2)$ If $r_0=1$, then $\Omega^{1}(Y_t^{(14)})$ is described with
$(8s\times 7s)$-matrix with one nonzero element that is of the
following form{\rm:}
$$b_{(-3\ell_0)_s,5s+(-1-3\ell_0)_s}=\kappa w_{(j+m+1)n+g(j+s)\ra (j+m+2)n}\otimes w_{jn+g(j)+1\ra (j+1)n}.$$

$(3)$ If $r_0=2$, then $\Omega^{2}(Y_t^{(14)})$ is described with
$(6s\times 6s)$-matrix with one nonzero element that is of the
following form{\rm:}
$$b_{4s+(-1-3\ell_0)_s,4s+(-1-3\ell_0)_s}=w_{(j+m)n+g(j)\ra (j+m+1)n+g(j)}\otimes e_{jn+g(j+s)+1}.$$

$(4)$ If $r_0=3$, then $\Omega^{3}(Y_t^{(14)})$ is described with
$(7s\times 8s)$-matrix with one nonzero element that is of the
following form{\rm:}
$$b_{5s+(-1-3\ell_0)_s,4s+(-2-3\ell_0)_s}=-w_{(j+m+2)n\ra (j+m+2)n+g(j)}\otimes w_{jn+g(j+s)+1\ra (j+1)n+g(j+s)+1}.$$

$(5)$ If $r_0=4$, then $\Omega^{4}(Y_t^{(14)})$ is described with
$(6s\times 9s)$-matrix with one nonzero element that is of the
following form{\rm:}
$$b_{7s+(-1-3\ell_0)_s,4s+(-2-3\ell_0)_s}=-e_{(j+m+1)n+g(j)+1}\otimes w_{jn+g(j+s)+1\ra (j+1)n+g(j+s)+1}.$$

$(6)$ If $r_0=5$, then $\Omega^{5}(Y_t^{(14)})$ is described with
$(6s\times 8s)$-matrix with one nonzero element that is of the
following form{\rm:}
$$b_{5s+(-1-3\ell_0)_s,4s+(-2-3\ell_0)_s}=w_{(j+m+1)n+5\ra (j+m+2)n+g(j)+1}\otimes w_{jn+g(j+s)+1\ra (j+1)n+g(j+s)+1}.$$

$(7)$ If $r_0=6$, then $\Omega^{6}(Y_t^{(14)})$ is described with
$(7s\times 9s)$-matrix with one nonzero element that is of the
following form{\rm:}
$$b_{6s+(-1-3\ell_0)_s,5s+(-2-3\ell_0)_s}=-\kappa w_{(j+m+1)n+g(j)+1\ra (j+m+1)n+5}\otimes w_{jn+g(j)+1\ra (j+1)n+g(j)+1}.$$

$(8)$ If $r_0=7$, then $\Omega^{7}(Y_t^{(14)})$ is described with
$(6s\times 8s)$-matrix with one nonzero element that is of the
following form{\rm:}
$$b_{5s+(-1-3\ell_0)_s,4s+(-2-3\ell_0)_s}=-w_{(j+m+2)n\ra (j+m+2)n+g(j+s)}\otimes w_{jn+g(j+s)+1\ra (j+1)n+g(j+s)+1}.$$

$(9)$ If $r_0=8$, then $\Omega^{8}(Y_t^{(14)})$ is described with
$(8s\times 6s)$-matrix with one nonzero element that is of the
following form{\rm:}
$$b_{5s+(-2-3\ell_0)_s,5s+(-2-3\ell_0)_s}=\kappa w_{(j+m+1)n\ra (j+m+2)n}\otimes w_{jn+g(j)+1\ra jn+5}.$$

$(10)$ If $r_0=9$, then $\Omega^{9}(Y_t^{(14)})$ is described with
$(9s\times 7s)$-matrix with one nonzero element that is of the
following form{\rm:}
$$b_{s+(-2-3\ell_0)_s,s+(-2-3\ell_0)_s}=-\kappa w_{(j+m+1)n\ra (j+m+2)n}\otimes w_{jn\ra jn+g(j+s)}.$$

$(11)$ If $r_0=10$, then $\Omega^{10}(Y_t^{(14)})$ is described with
$(8s\times 6s)$-matrix with one nonzero element that is of the
following form{\rm:}
$$b_{(2s+(-3-3\ell_0)_s+1)_s,2s+(-3-3\ell_0)_s}=\kappa w_{(j+m+1)n\ra (j+m+2)n}\otimes w_{jn+g(j)\ra (j+1)n}.$$

\medskip
$({\rm II})$ Represent an arbitrary $t_0\in\N$ in the form
$t_0=11\ell_0+r_0$, where $0\le r_0\le 10.$ Then
$\Omega^{t_0}(Y_t^{(14)})$ is a $\Omega^{r_0}(Y_t^{(14)})$, whose
left components twisted by $\sigma^{\ell_0}$.
\end{pr}
\begin{pr}[Translates for the case 15]
$({\rm I})$ Let $r_0\in\N$, $r_0<11$. $r_0$-translates of the
elements $Y^{(15)}_t$ are described by the following way.

$(1)$ If $r_0=0$, then $\Omega^{0}(Y_t^{(15)})$ is described with
$(9s\times 6s)$-matrix with the following elements $b_{ij}${\rm:}

If $0\le j<s$, then $$b_{ij}=
\begin{cases}
e_{(j+m)n}\otimes e_{jn},\quad i=j;\\
0,\quad\text{otherwise.}\end{cases}$$

If $s\le j<3s$, then $b_{ij}=0$.

If $3s\le j<5s$, then $$b_{ij}=
\begin{cases}
w_{(j+m)n+g(j+s)\ra (j+m)n+g(j+s)+1}\otimes e_{jn+g(j+s)},\quad i=j-2s;\\
0,\quad\text{otherwise.}\end{cases}$$

If $5s\le j<7s$, then $b_{ij}=0$.

If $7s\le j<8s$, then $$b_{ij}=
\begin{cases}
e_{(j+m)n+5}\otimes e_{jn+5},\quad i=j-2s;\\
0,\quad\text{otherwise.}\end{cases}$$

If $8s\le j<9s$, then $b_{ij}=0$.

$(2)$ If $r_0=1$, then $\Omega^{1}(Y_t^{(15)})$ is described with
$(8s\times 7s)$-matrix with the following elements $b_{ij}${\rm:}

If $0\le j<2s$, then $$b_{ij}=
\begin{cases}
w_{(j+m)n+g(j+s)\ra (j+m)n+g(j+s)+1}\otimes e_{jn},\quad i=(j)_s+f_0(j,s)s;\\
0,\quad\text{otherwise.}\end{cases}$$

If $2s\le j<4s$, then $$b_{ij}=
\begin{cases}
w_{(j+m)n+g(j)+1\ra (j+m)n+5}\otimes e_{jn+g(j)},\quad i=j;\\
e_{(j+m)n+5}\otimes w_{jn+g(j)\ra jn+g(j)+1},\quad i=j+2s;\\
0,\quad\text{otherwise.}\end{cases}$$

If $4s\le j<5s$, then $b_{ij}=0$.

If $5s\le j<6s$, then $$b_{ij}=
\begin{cases}
e_{(j+m+1)n}\otimes w_{jn+g(j)+1\ra jn+5},\quad i=j+s;\\
0,\quad\text{otherwise.}\end{cases}$$

If $6s\le j<7s$, then $$b_{ij}=
\begin{cases}
w_{(j+m+1)n\ra (j+m+1)n+g(j+s)}\otimes e_{jn+5},\quad i=j;\\
0,\quad\text{otherwise.}\end{cases}$$

If $7s\le j<8s$, then $b_{ij}=0$.

$(3)$ If $r_0=2$, then $\Omega^{2}(Y_t^{(15)})$ is described with
$(6s\times 6s)$-matrix with the following elements $b_{ij}${\rm:}

If $0\le j<s$, then $$b_{ij}=
\begin{cases}
e_{(j+m-1)n+5}\otimes e_{jn},\quad i=j;\\
0,\quad\text{otherwise.}\end{cases}$$

If $s\le j<3s$, then $$b_{ij}=
\begin{cases}
w_{(j+m)n+g(j)\ra (j+m)n+g(j)+1}\otimes w_{jn+g(j+s)\ra jn+g(j+s)+1},\quad i=j+2s;\\
0,\quad\text{otherwise.}\end{cases}$$

If $3s\le j<5s$, then $b_{ij}=0$.

If $5s\le j<6s$, then $$b_{ij}=
\begin{cases}
e_{(j+m+1)n}\otimes e_{jn+5},\quad i=j;\\
0,\quad\text{otherwise.}\end{cases}$$

$(4)$ If $r_0=3$, then $\Omega^{3}(Y_t^{(15)})$ is described with
$(7s\times 8s)$-matrix with the following elements $b_{ij}${\rm:}

If $0\le j<s$, then $$b_{ij}=
\begin{cases}
w_{(j+m)n+g(j)+1\ra (j+m)n+5}\otimes e_{jn},\quad i=j;\\
e_{(j+m)n+5}\otimes w_{jn\ra jn+g(j)},\quad i=j+2s;\\
0,\quad\text{otherwise.}\end{cases}$$

If $s\le j<3s$, then $$b_{ij}=
\begin{cases}
e_{(j+m+1)n}\otimes w_{jn+g(j+s)\ra jn+g(j+s)+1},\quad i=j+3s;\\
0,\quad\text{otherwise.}\end{cases}$$

If $3s\le j<5s$, then $$b_{ij}=
\begin{cases}
w_{(j+m+1)n\ra (j+m+1)n+g(j+s)}\otimes e_{jn+g(j+s)+1},\quad i=j+s;\\
0,\quad\text{otherwise.}\end{cases}$$

If $5s\le j<7s$, then $$b_{ij}=
\begin{cases}
w_{(j+m+1)n+g(j)\ra (j+m+1)n+g(j)+1}\otimes e_{jn+5},\quad i=j+2f_0(j,6s)s;\\
0,\quad\text{otherwise.}\end{cases}$$

$(5)$ If $r_0=4$, then $\Omega^{4}(Y_t^{(15)})$ is described with
$(6s\times 9s)$-matrix with the following elements $b_{ij}${\rm:}

If $0\le j<s$, then $$b_{ij}=
\begin{cases}
e_{(j+m)n}\otimes e_{jn},\quad i=j+s;\\
0,\quad\text{otherwise.}\end{cases}$$

If $s\le j<3s$, then $b_{ij}=0$.

If $3s\le j<5s$, then $$b_{ij}=
\begin{cases}
w_{(j+m)n+g(j+s)\ra (j+m)n+g(j+s)+1}\otimes e_{jn+g(j+s)+1},\quad i=j+s;\\
0,\quad\text{otherwise.}\end{cases}$$

If $5s\le j<6s$, then $$b_{ij}=
\begin{cases}
e_{(j+m)n+5}\otimes e_{jn+5},\quad i=j+3s;\\
0,\quad\text{otherwise.}\end{cases}$$

$(6)$ If $r_0=5$, then $\Omega^{5}(Y_t^{(15)})$ is described with
$(6s\times 8s)$-matrix with the following elements $b_{ij}${\rm:}

If $0\le j<s$, then $$b_{ij}=
\begin{cases}
w_{(j+m)n+g(j+s)\ra (j+m+1)n}\otimes e_{jn},\quad i=j+s;\\
e_{(j+m+1)n}\otimes w_{jn\ra jn+g(j)},\quad i=j+2s;\\
w_{(j+m)n+5\ra (j+m+1)n}\otimes w_{jn\ra jn+g(j)+1},\quad i=j+4s;\\
0,\quad\text{otherwise.}\end{cases}$$

If $s\le j<3s$, then $$b_{ij}=
\begin{cases}
e_{(j+m+1)n+g(j+s)}\otimes w_{jn+g(j+s)\ra (j+1)n},\quad i=(j+1)_s+(j_2-1)s;\\
w_{(j+m)n+5\ra (j+m+1)n+g(j+s)}\otimes w_{jn+g(j+s)\ra jn+g(j+s)+1},\quad i=j+3s;\\
0,\quad\text{otherwise.}\end{cases}$$

If $3s\le j<5s$, then $$b_{ij}=
\begin{cases}
w_{(j+m+1)n+g(j+s)\ra (j+m+1)n+g(j+s)+1}\otimes w_{jn+g(j+s)+1\ra (j+1)n},\quad i=(j+1)_s+(j_2-3)s;\\
w_{(j+m)n+5\ra (j+m+1)n+g(j+s)+1}\otimes e_{jn+g(j+s)+1},\quad i=j+s;\\
0,\quad\text{otherwise.}\end{cases}$$

If $5s\le j<6s$, then $$b_{ij}=
\begin{cases}
w_{(j+m+1)n+g(j+s)\ra (j+m+1)n+5}\otimes w_{jn+5\ra (j+1)n},\quad i=(j+1)_s;\\
e_{(j+m+1)n+5}\otimes w_{jn+5\ra (j+1)n+g(j)+1},\quad i=(j+1)_s+5s;\\
w_{(j+m+1)n+g(j+s)+1\ra (j+m+1)n+5}\otimes e_{jn+5},\quad i=j+s;\\
0,\quad\text{otherwise.}\end{cases}$$

$(7)$ If $r_0=6$, then $\Omega^{6}(Y_t^{(15)})$ is described with
$(7s\times 9s)$-matrix with the following elements $b_{ij}${\rm:}

If $0\le j<s$, then $$b_{ij}=
\begin{cases}
e_{(j+m)n+g(j)}\otimes w_{jn\ra jn+g(j)},\quad i=j+s;\\
0,\quad\text{otherwise.}\end{cases}$$

If $s\le j<2s$, then $$b_{ij}=
\begin{cases}
w_{(j+m)n\ra (j+m)n+g(j)}\otimes e_{jn},\quad i=j-s;\\
e_{(j+m)n+g(j)}\otimes w_{jn\ra jn+g(j)},\quad i=j+s;\\
0,\quad\text{otherwise.}\end{cases}$$

If $2s\le j<4s$, then $b_{ij}=0$.

If $4s\le j<5s$, then $$b_{ij}=
\begin{cases}
w_{(j+m)n+g(j)+1\ra (j+m)n+5}\otimes e_{jn+g(j)+1},\quad i=j+s;\\
0,\quad\text{otherwise.}\end{cases}$$

If $5s\le j<6s$, then $$b_{ij}=
\begin{cases}
w_{(j+m)n+g(j)+1\ra (j+m)n+5}\otimes e_{jn+g(j)+1},\quad i=j+s;\\
e_{(j+m)n+5}\otimes w_{jn+g(j)+1\ra jn+5},\quad i=j+2s;\\
0,\quad\text{otherwise.}\end{cases}$$

If $6s\le j<7s$, then $$b_{ij}=
\begin{cases}
e_{(j+m+1)n}\otimes w_{jn+5\ra (j+1)n},\quad i=(j+1)_s;\\
e_{(j+m+1)n}\otimes e_{jn+5},\quad i=j+2s;\\
0,\quad\text{otherwise.}\end{cases}$$

$(8)$ If $r_0=7$, then $\Omega^{7}(Y_t^{(15)})$ is described with
$(6s\times 8s)$-matrix with the following elements $b_{ij}${\rm:}

If $0\le j<s$, then $$b_{ij}=
\begin{cases}
w_{(j+m)n+g(j+s)+1\ra (j+m)n+5}\otimes e_{jn},\quad i=j+s;\\
e_{(j+m)n+5}\otimes w_{jn\ra jn+g(j)},\quad i=j+2s;\\
0,\quad\text{otherwise.}\end{cases}$$

If $s\le j<3s$, then $$b_{ij}=
\begin{cases}
e_{(j+m+1)n+g(j+s)+1}\otimes w_{jn+g(j+s)\ra (j+1)n},\quad i=(j+1)_s+(j_2-1)s;\\
w_{(j+m+1)n\ra (j+m+1)n+g(j+s)+1}\otimes w_{jn+g(j+s)\ra jn+g(j+s)+1},\quad i=j+3s;\\
0,\quad\text{otherwise.}\end{cases}$$

If $3s\le j<5s$, then $$b_{ij}=
\begin{cases}
w_{(j+m+1)n\ra (j+m+1)n+g(j)}\otimes e_{jn+g(j+s)+1},\quad i=j+s;\\
0,\quad\text{otherwise.}\end{cases}$$

If $5s\le j<6s$, then $$b_{ij}=
\begin{cases}
e_{(j+m+2)n}\otimes w_{jn+5\ra (j+1)n+g(j+s)+1},\quad i=(j+1)_s+4s;\\
w_{(j+m+1)n+g(j+s)\ra (j+m+2)n}\otimes e_{jn+5},\quad i=j+s;\\
0,\quad\text{otherwise.}\end{cases}$$

$(9)$ If $r_0=8$, then $\Omega^{8}(Y_t^{(15)})$ is described with
$(8s\times 6s)$-matrix with the following elements $b_{ij}${\rm:}

If $0\le j<s$, then $$b_{ij}=
\begin{cases}
e_{(j+m)n+g(j)+1}\otimes w_{jn\ra jn+g(j)},\quad i=j+s;\\
w_{(j+m)n+g(j)\ra (j+m)n+g(j)+1}\otimes w_{jn\ra jn+g(j+s)+1},\quad i=j+4s;\\
0,\quad\text{otherwise.}\end{cases}$$

If $s\le j<2s$, then $$b_{ij}=
\begin{cases}
w_{(j+m-1)n+5\ra (j+m)n+g(j)+1}\otimes e_{jn},\quad i=j-s;\\
e_{(j+m)n+g(j)+1}\otimes w_{jn\ra jn+g(j)},\quad i=j+s;\\
w_{(j+m)n+g(j)\ra (j+m)n+g(j)+1}\otimes w_{jn\ra jn+g(j+s)+1},\quad i=j+2s;\\
0,\quad\text{otherwise.}\end{cases}$$

If $2s\le j<3s$, then $$b_{ij}=
\begin{cases}
e_{(j+m)n+5}\otimes w_{jn+g(j)\ra (j+1)n},\quad i=(j+1)_s;\\
0,\quad\text{otherwise.}\end{cases}$$

If $3s\le j<4s$, then $b_{ij}=0$.

If $4s\le j<5s$, then $$b_{ij}=
\begin{cases}
w_{(j+m)n+g(j+s)\ra (j+m+1)n}\otimes e_{jn+g(j)+1},\quad i=j-s;\\
e_{(j+m+1)n}\otimes w_{jn+g(j)+1\ra jn+5},\quad i=j+s;\\
0,\quad\text{otherwise.}\end{cases}$$

If $5s\le j<6s$, then $$b_{ij}=
\begin{cases}
w_{(j+m)n+g(j+s)\ra (j+m+1)n}\otimes e_{jn+g(j)+1},\quad i=j-s;\\
0,\quad\text{otherwise.}\end{cases}$$

If $6s\le j<7s$, then $$b_{ij}=
\begin{cases}
w_{(j+m+1)n\ra (j+m+1)n+g(j)}\otimes e_{jn+5},\quad i=j-s;\\
0,\quad\text{otherwise.}\end{cases}$$

If $7s\le j<8s$, then $b_{ij}=0$.

$(10)$ If $r_0=9$, then $\Omega^{9}(Y_t^{(15)})$ is described with
$(9s\times 7s)$-matrix with the following elements $b_{ij}${\rm:}

If $0\le j<s$, then $$b_{ij}=
\begin{cases}
e_{(j+m)n+5}\otimes e_{jn},\quad i=j;\\
0,\quad\text{otherwise.}\end{cases}$$

If $s\le j<2s$, then $$b_{ij}=
\begin{cases}
e_{(j+m+1)n}\otimes w_{jn\ra jn+g(j+s)},\quad i=j;\\
0,\quad\text{otherwise.}\end{cases}$$

If $2s\le j<4s$, then $$b_{ij}=
\begin{cases}
w_{(j+m+1)n\ra (j+m+1)n+g(j+s)}\otimes e_{jn+g(j)},\quad i=j-s;\\
0,\quad\text{otherwise.}\end{cases}$$

If $4s\le j<6s$, then $b_{ij}=0$.

If $6s\le j<8s$, then $$b_{ij}=
\begin{cases}
w_{(j+m+1)n+g(j+s)\ra (j+m+1)n+g(j+s)+1}\otimes e_{jn+g(j)+1},\quad i=j-3s;\\
0,\quad\text{otherwise.}\end{cases}$$

If $8s\le j<9s$, then $$b_{ij}=
\begin{cases}
w_{(j+m+1)n+g(j)+1\ra (j+m+1)n+5}\otimes e_{jn+5},\quad i=j-3s;\\
0,\quad\text{otherwise.}\end{cases}$$

$(11)$ If $r_0=10$, then $\Omega^{10}(Y_t^{(15)})$ is described with
$(8s\times 6s)$-matrix with the following elements $b_{ij}${\rm:}

If $0\le j<s$, then $$b_{ij}=
\begin{cases}
w_{(j+m)n\ra (j+m)n+g(j)}\otimes e_{jn},\quad i=j;\\
0,\quad\text{otherwise.}\end{cases}$$

If $s\le j<2s$, then $b_{ij}=0$.

If $2s\le j<3s$, then $$b_{ij}=
\begin{cases}
e_{(j+m+1)n}\otimes w_{jn+g(j)\ra (j+1)n},\quad i=(j+1)_s;\\
w_{(j+m)n+g(j+s)\ra (j+m+1)n}\otimes e_{jn+g(j)},\quad i=j-s;\\
0,\quad\text{otherwise.}\end{cases}$$

If $3s\le j<4s$, then $$b_{ij}=
\begin{cases}
w_{(j+m)n+g(j+s)\ra (j+m+1)n}\otimes e_{jn+g(j)},\quad i=j-s;\\
0,\quad\text{otherwise.}\end{cases}$$

If $4s\le j<5s$, then $$b_{ij}=
\begin{cases}
w_{(j+m)n+g(j+s)+1\ra (j+m)n+5}\otimes e_{jn+g(j)+1},\quad i=j-s;\\
e_{(j+m)n+5}\otimes w_{jn+g(j)+1\ra jn+5},\quad i=j+s;\\
0,\quad\text{otherwise.}\end{cases}$$

If $5s\le j<6s$, then $$b_{ij}=
\begin{cases}
w_{(j+m)n+g(j+s)+1\ra (j+m)n+5}\otimes e_{jn+g(j)+1},\quad i=j-s;\\
0,\quad\text{otherwise.}\end{cases}$$

If $6s\le j<7s$, then $$b_{ij}=
\begin{cases}
w_{(j+m+1)n+g(j)\ra (j+m+1)n+g(j)+1}\otimes w_{jn+5\ra (j+1)n+g(j+s)},\quad i=(j+1)_s+2s;\\
w_{(j+m)n+5\ra (j+m+1)n+g(j)+1}\otimes e_{jn+5},\quad i=j-s;\\
0,\quad\text{otherwise.}\end{cases}$$

If $7s\le j<8s$, then $$b_{ij}=
\begin{cases}
w_{(j+m+1)n+g(j)\ra (j+m+1)n+g(j)+1}\otimes w_{jn+5\ra (j+1)n+g(j+s)},\quad i=(j+1)_s+s;\\
0,\quad\text{otherwise.}\end{cases}$$

\medskip
$({\rm II})$ Represent an arbitrary $t_0\in\N$ in the form
$t_0=11\ell_0+r_0$, where $0\le r_0\le 10.$ Then
$\Omega^{t_0}(Y_t^{(15)})$ is a $\Omega^{r_0}(Y_t^{(15)})$, whose
left components twisted by $\sigma^{\ell_0}$.
\end{pr}
\begin{pr}[Translates for the case 16]
$({\rm I})$ Let $r_0\in\N$, $r_0<11$. $r_0$-translates of the
elements $Y^{(16)}_t$ are described by the following way.

$(1)$ If $r_0=0$, then $\Omega^{0}(Y_t^{(16)})$ is described with
$(8s\times 6s)$-matrix with the following elements $b_{ij}${\rm:}

If $0\le j<(-4\ell_0)_s$, then $b_{ij}=0$.

If $j=(-4\ell_0)_s$, then $$b_{ij}=
\begin{cases}
w_{(j+m)n\ra (j+m)n+g(j)+1}\otimes e_{jn},\quad i=j;\\
0,\quad\text{otherwise.}\end{cases}$$

If $(-4\ell_0)_s<j<2s+(-4\ell_0)_s$, then $b_{ij}=0$.

If $j=2s+(-4\ell_0)_s$, then $$b_{ij}=
\begin{cases}
-\kappa w_{(j+m)n+g(j)\ra (j+m)n+5}\otimes e_{jn+g(j)},\quad i=j-s;\\
0,\quad\text{otherwise.}\end{cases}$$

If $2s+(-4\ell_0)_s<j<8s$, then $b_{ij}=0$.

$(2)$ If $r_0=1$, then $\Omega^{1}(Y_t^{(16)})$ is described with
$(6s\times 7s)$-matrix with the following elements $b_{ij}${\rm:}

If $0\le j<(-4\ell_0)_s$, then $b_{ij}=0$.

If $j=(-4\ell_0)_s$, then $$b_{ij}=
\begin{cases}
\kappa w_{(j+m)n+g(j)\ra (j+m)n+5}\otimes e_{jn},\quad i=j;\\
0,\quad\text{otherwise.}\end{cases}$$

If $(-4\ell_0)_s<j<s+(-1-4\ell_0)_s$, then $b_{ij}=0$.

If $j=s+(-1-4\ell_0)_s$, then $$b_{ij}=
\begin{cases}
-w_{(j+m+1)n+g(j+s)\ra (j+m+1)n+g(j+s)+1}\otimes w_{jn+g(j+s)\ra (j+1)n},\quad i=(j+1)_s;\\
-e_{(j+m+1)n+g(j+s)+1}\otimes w_{jn+g(j+s)\ra (j+1)n+g(j+s)},\quad i=(j+1)_s+2s;\\
0,\quad\text{otherwise.}\end{cases}$$

If $s+(-1-4\ell_0)_s<j<5s+(-1-4\ell_0)_s$, then $b_{ij}=0$.

If $j=5s+(-1-4\ell_0)_s$, then $$b_{ij}=
\begin{cases}
-\kappa w_{(j+m+1)n+g(j+s)\ra (j+m+2)n}\otimes w_{jn+5\ra (j+1)n},\quad i=(j+1)_s;\\
-\kappa w_{(j+m+1)n\ra (j+m+2)n}\otimes e_{jn+5},\quad i=j+s;\\
0,\quad\text{otherwise.}\end{cases}$$

If $5s+(-1-4\ell_0)_s<j<6s$, then $b_{ij}=0$.

$(3)$ If $r_0=2$, then $\Omega^{2}(Y_t^{(16)})$ is described with
$(7s\times 6s)$-matrix with the following elements $b_{ij}${\rm:}

If $0\le j<s+(-1-4\ell_0)_s$, then $b_{ij}=0$.

If $j=s+(-1-4\ell_0)_s$, then $$b_{ij}=
\begin{cases}
-\kappa w_{(j+m)n+g(j+s)+1\ra (j+m+1)n}\otimes e_{jn+g(j+s)},\quad i=j;\\
0,\quad\text{otherwise.}\end{cases}$$

If $s+(-1-4\ell_0)_s<j<2s+(-1-4\ell_0)_s$, then $b_{ij}=0$.

If $j=2s+(-1-4\ell_0)_s$, then $$b_{ij}=
\begin{cases}
-\kappa w_{(j+m)n+g(j)\ra (j+m+1)n}\otimes w_{jn+g(j+s)\ra jn+g(j+s)+1},\quad i=j+2s;\\
0,\quad\text{otherwise.}\end{cases}$$

If $2s+(-1-4\ell_0)_s<j<4s+(-1-4\ell_0)_s$, then $b_{ij}=0$.

If $j=4s+(-1-4\ell_0)_s$, then $$b_{ij}=
\begin{cases}
w_{(j+m)n+g(j)\ra (j+m+1)n+g(j)}\otimes e_{jn+g(j+s)+1},\quad i=j;\\
0,\quad\text{otherwise.}\end{cases}$$

If $4s+(-1-4\ell_0)_s<j<7s$, then $b_{ij}=0$.

$(4)$ If $r_0=3$, then $\Omega^{3}(Y_t^{(16)})$ is described with
$(6s\times 8s)$-matrix with one nonzero element that is of the
following form{\rm:}
$$b_{(-1-4\ell_0)_s,(-1-4\ell_0)_s}=\kappa w_{(j+m)n+g(j)+1\ra (j+m+1)n}\otimes e_{jn}.$$

$(5)$ If $r_0=4$, then $\Omega^{4}(Y_t^{(16)})$ is described with
$(6s\times 9s)$-matrix with the following elements $b_{ij}${\rm:}

If $0\le j<(-2-4\ell_0)_s$, then $b_{ij}=0$.

If $j=(-2-4\ell_0)_s$, then $$b_{ij}=
\begin{cases}
-\kappa e_{(j+m+1)n}\otimes w_{jn\ra (j+1)n},\quad i=(j+1)_s+s;\\
0,\quad\text{otherwise.}\end{cases}$$

If $(-2-4\ell_0)_s<j<2s+(-2-4\ell_0)_s$, then $b_{ij}=0$.

If $j=2s+(-2-4\ell_0)_s$, then $$b_{ij}=
\begin{cases}
-w_{(j+m+1)n\ra (j+m+1)n+g(j)}\otimes w_{jn+g(j+s)\ra (j+1)n},\quad i=(j+1)_s+s;\\
0,\quad\text{otherwise.}\end{cases}$$

If $2s+(-2-4\ell_0)_s<j<4s+(-2-4\ell_0)_s$, then $b_{ij}=0$.

If $j=4s+(-2-4\ell_0)_s$, then $$b_{ij}=
\begin{cases}
w_{(j+m)n+g(j)+1\ra (j+m+1)n+g(j)+1}\otimes e_{jn+g(j+s)+1},\quad i=j+3s;\\
0,\quad\text{otherwise.}\end{cases}$$

If $4s+(-2-4\ell_0)_s<j<6s$, then $b_{ij}=0$.

$(6)$ If $r_0=5$, then $\Omega^{5}(Y_t^{(16)})$ is described with
$(7s\times 8s)$-matrix with the following elements $b_{ij}${\rm:}

If $0\le j<(-2-4\ell_0)_s$, then $b_{ij}=0$.

If $j=(-2-4\ell_0)_s$, then $$b_{ij}=
\begin{cases}
e_{(j+m+1)n+g(j+s)}\otimes w_{jn\ra (j+1)n},\quad i=(j+1)_s+s;\\
0,\quad\text{otherwise.}\end{cases}$$

If $(-2-4\ell_0)_s<j<3s+(-2-4\ell_0)_s$, then $b_{ij}=0$.

If $j=3s+(-2-4\ell_0)_s$, then $$b_{ij}=
\begin{cases}
w_{(j+m+1)n\ra (j+m+1)n+g(j+s)+1}\otimes e_{jn+g(j)},\quad i=j;\\
0,\quad\text{otherwise.}\end{cases}$$

If $3s+(-2-4\ell_0)_s<j<7s$, then $b_{ij}=0$.

$(7)$ If $r_0=6$, then $\Omega^{6}(Y_t^{(16)})$ is described with
$(6s\times 9s)$-matrix with the following elements $b_{ij}${\rm:}

If $0\le j<(-2-4\ell_0)_s$, then $b_{ij}=0$.

If $j=(-2-4\ell_0)_s$, then $$b_{ij}=
\begin{cases}
\kappa w_{(j+m)n+g(j+s)\ra (j+m)n+5}\otimes w_{jn\ra jn+g(j+s)},\quad i=j+2s;\\
0,\quad\text{otherwise.}\end{cases}$$

If $(-2-4\ell_0)_s<j<2s+(-3-4\ell_0)_s$, then $b_{ij}=0$.

If $j=2s+(-3-4\ell_0)_s$, then $$b_{ij}=
\begin{cases}
w_{(j+m+1)n\ra (j+m+1)n+g(j)+1}\otimes w_{jn+g(j+s)\ra (j+1)n},\quad i=(j+1)_s;\\
0,\quad\text{otherwise.}\end{cases}$$

If $2s+(-3-4\ell_0)_s<j<6s$, then $b_{ij}=0$.

$(8)$ If $r_0=7$, then $\Omega^{7}(Y_t^{(16)})$ is described with
$(8s\times 8s)$-matrix with the following elements $b_{ij}${\rm:}

If $0\le j<s+(-3-4\ell_0)_s$, then $b_{ij}=0$.

If $j=s+(-3-4\ell_0)_s$, then $$b_{ij}=
\begin{cases}
e_{(j+m+1)n+g(j+s)+1}\otimes w_{jn\ra (j+1)n},\quad i=(j+1)_s;\\
0,\quad\text{otherwise.}\end{cases}$$

If $s+(-3-4\ell_0)_s<j<3s+(-3-4\ell_0)_s$, then $b_{ij}=0$.

If $j=3s+(-3-4\ell_0)_s$, then $$b_{ij}=
\begin{cases}
-\kappa w_{(j+m+1)n+g(j)+1\ra (j+m+1)n+5}\otimes w_{jn+g(j)\ra (j+1)n},\quad i=(j+1)_s+s;\\
0,\quad\text{otherwise.}\end{cases}$$

If $3s+(-3-4\ell_0)_s<j<6s+(-3-4\ell_0)_s$, then $b_{ij}=0$.

If $j=6s+(-3-4\ell_0)_s$, then $$b_{ij}=
\begin{cases}
-w_{(j+m+1)n+g(j+s)\ra (j+m+2)n+g(j+s)}\otimes e_{jn+5},\quad i=j+s;\\
0,\quad\text{otherwise.}\end{cases}$$

If $6s+(-3-4\ell_0)_s<j<8s$, then $b_{ij}=0$.

$(9)$ If $r_0=8$, then $\Omega^{8}(Y_t^{(16)})$ is described with
$(9s\times 6s)$-matrix with the following elements $b_{ij}${\rm:}

If $0\le j<(-3-4\ell_0)_s$, then $b_{ij}=0$.

If $j=(-3-4\ell_0)_s$, then $$b_{ij}=
\begin{cases}
\kappa e_{(j+m)n+5}\otimes w_{jn\ra (j+1)n},\quad i=(j+1)_s;\\
0,\quad\text{otherwise.}\end{cases}$$

If $(-3-4\ell_0)_s<j<3s+(-3-4\ell_0)_s$, then $b_{ij}=0$.

If $j=3s+(-3-4\ell_0)_s$, then $$b_{ij}=
\begin{cases}
w_{(j+m)n+g(j)+1\ra (j+m+1)n+g(j)}\otimes e_{jn+g(j)},\quad i=j-s;\\
0,\quad\text{otherwise.}\end{cases}$$

If $3s+(-3-4\ell_0)_s<j<4s+(-3-4\ell_0)_s$, then $b_{ij}=0$.

If $j=4s+(-3-4\ell_0)_s$, then $$b_{ij}=
\begin{cases}
-w_{(j+m)n+g(j+s)\ra (j+m+1)n+g(j+s)}\otimes e_{jn+g(j)+1},\quad i=j-s;\\
0,\quad\text{otherwise.}\end{cases}$$

If $4s+(-3-4\ell_0)_s<j<9s$, then $b_{ij}=0$.

$(10)$ If $r_0=9$, then $\Omega^{9}(Y_t^{(16)})$ is described with
$(8s\times 7s)$-matrix with the following elements $b_{ij}${\rm:}

If $0\le j<(-3-4\ell_0)_s$, then $b_{ij}=0$.

If $j=(-3-4\ell_0)_s$, then $$b_{ij}=
\begin{cases}
-w_{(j+m)n+5\ra (j+m+1)n+g(j+s)}\otimes e_{jn},\quad i=j;\\
0,\quad\text{otherwise.}\end{cases}$$

If $(-3-4\ell_0)_s<j<6s+(-4-4\ell_0)_s$, then $b_{ij}=0$.

If $j=6s+(-4-4\ell_0)_s$, then $$b_{ij}=
\begin{cases}
-w_{(j+m+1)n+5\ra (j+m+2)n+g(j+s)+1}\otimes w_{jn+5\ra (j+1)n},\quad i=(j+1)_s;\\
0,\quad\text{otherwise.}\end{cases}$$

If $6s+(-4-4\ell_0)_s<j<8s$, then $b_{ij}=0$.

$(11)$ If $r_0=10$, then $\Omega^{10}(Y_t^{(16)})$ is described with
$(9s\times 6s)$-matrix with the following elements $b_{ij}${\rm:}

If $0\le j<s+(-4-4\ell_0)_s$, then $b_{ij}=0$.

If $j=s+(-4-4\ell_0)_s$, then $$b_{ij}=
\begin{cases}
w_{(j+m+1)n\ra (j+m+1)n+g(j)}\otimes w_{jn+g(j+s)\ra (j+1)n},\quad i=(j+1)_s;\\
-e_{(j+m+1)n+g(j)}\otimes w_{jn+g(j+s)\ra (j+1)n+g(j+s)},\quad i=(j+1)_s+s;\\
0,\quad\text{otherwise.}\end{cases}$$

If $s+(-4-4\ell_0)_s<j<5s+(-4-4\ell_0)_s$, then $b_{ij}=0$.

If $j=5s+(-4-4\ell_0)_s$, then $$b_{ij}=
\begin{cases}
-w_{(j+m)n+g(j)+1\ra (j+m+1)n+g(j)+1}\otimes e_{jn+g(j+s)+1},\quad i=j-2s;\\
-w_{(j+m)n+5\ra (j+m+1)n+g(j)+1}\otimes w_{jn+g(j+s)+1\ra jn+5},\quad i=j;\\
0,\quad\text{otherwise.}\end{cases}$$

If $5s+(-4-4\ell_0)_s<j<9s$, then $b_{ij}=0$.

\medskip
$({\rm II})$ Represent an arbitrary $t_0\in\N$ in the form
$t_0=11\ell_0+r_0$, where $0\le r_0\le 10.$ Then
$\Omega^{t_0}(Y_t^{(16)})$ is a $\Omega^{r_0}(Y_t^{(16)})$, whose
left components twisted by $\sigma^{\ell_0}$, and coefficients
multiplied by $(-1)^{\ell_0}$.
\end{pr}
\begin{pr}[Translates for the case 18]
$({\rm I})$ Let $r_0\in\N$, $r_0<11$. $r_0$-translates of the
elements $Y^{(18)}_t$ are described by the following way.

$(1)$ If $r_0=0$, then $\Omega^{0}(Y_t^{(18)})$ is described with
$(6s\times 6s)$-matrix with the following elements $b_{ij}${\rm:}

If $0\le j<s$, then $b_{ij}=0$.

If $s\le j<3s$, then $$b_{ij}=
\begin{cases}
w_{(j+m)n+g(j+s)\ra (j+m)n+g(j+s)+1}\otimes e_{jn+g(j+s)},\quad i=j;\\
0,\quad\text{otherwise.}\end{cases}$$

If $3s\le j<5s$, then $b_{ij}=0$.

If $5s\le j<6s$, then $$b_{ij}=
\begin{cases}
-\kappa w_{(j+m)n+5\ra (j+m+1)n}\otimes e_{jn+5},\quad i=j;\\
0,\quad\text{otherwise.}\end{cases}$$

$(2)$ If $r_0=1$, then $\Omega^{1}(Y_t^{(18)})$ is described with
$(7s\times 7s)$-matrix with the following elements $b_{ij}${\rm:}

If $0\le j<s$, then $$b_{ij}=
\begin{cases}
\kappa w_{(j+m)n+g(j)\ra (j+m)n+5}\otimes e_{jn},\quad i=j;\\
-\kappa w_{(j+m)n+g(j+s)\ra (j+m)n+5}\otimes e_{jn},\quad i=j+s;\\
0,\quad\text{otherwise.}\end{cases}$$

If $s\le j<3s$, then $$b_{ij}=
\begin{cases}
\kappa w_{(j+m)n+g(j+s)+1\ra (j+m+1)n}\otimes e_{jn+g(j+s)},\quad i=j+s;\\
\kappa w_{(j+m)n+5\ra (j+m+1)n}\otimes w_{jn+g(j+s)\ra jn+g(j+s)+1},\quad i=j+3s;\\
0,\quad\text{otherwise.}\end{cases}$$

If $3s\le j<5s$, then $$b_{ij}=
\begin{cases}
-w_{(j+m)n+5\ra (j+m+1)n+g(j)}\otimes e_{jn+g(j+s)+1},\quad i=j+s;\\
0,\quad\text{otherwise.}\end{cases}$$

If $5s\le j<7s$, then $$b_{ij}=
\begin{cases}
-w_{(j+m+1)n+g(j+s)\ra (j+m+1)n+g(j+s)+1}\otimes w_{jn+5\ra (j+1)n},\quad i=(j+1)_s+(j_2-5)s;\\
-w_{(j+m+1)n\ra (j+m+1)n+g(j+s)+1}\otimes e_{jn+5},\quad i=j+(6-j_2)s;\\
0,\quad\text{otherwise.}\end{cases}$$

$(3)$ If $r_0=2$, then $\Omega^{2}(Y_t^{(18)})$ is described with
$(6s\times 6s)$-matrix with the following elements $b_{ij}${\rm:}

If $0\le j<s$, then $$b_{ij}=
\begin{cases}
\kappa w_{(j+m-1)n+5\ra (j+m)n}\otimes e_{jn},\quad i=j;\\
0,\quad\text{otherwise.}\end{cases}$$

If $s\le j<3s$, then $b_{ij}=0$.

If $3s\le j<5s$, then $$b_{ij}=
\begin{cases}
-w_{(j+m)n+g(j)\ra (j+m)n+g(j)+1}\otimes e_{jn+g(j+s)+1},\quad i=j;\\
0,\quad\text{otherwise.}\end{cases}$$

If $5s\le j<6s$, then $b_{ij}=0$.

$(4)$ If $r_0=3$, then $\Omega^{3}(Y_t^{(18)})$ is described with
$(6s\times 8s)$-matrix with the following elements $b_{ij}${\rm:}

If $0\le j<s$, then $$b_{ij}=
\begin{cases}
\kappa w_{(j+m)n+g(j)+1\ra (j+m+1)n}\otimes e_{jn},\quad i=j;\\
\kappa w_{(j+m)n+5\ra (j+m+1)n}\otimes w_{jn\ra jn+g(j)},\quad i=j+2s;\\
0,\quad\text{otherwise.}\end{cases}$$

If $s\le j<3s$, then $$b_{ij}=
\begin{cases}
-w_{(j+m)n+5\ra (j+m+1)n+g(j)}\otimes e_{jn+g(j+s)},\quad i=j+s;\\
0,\quad\text{otherwise.}\end{cases}$$

If $3s\le j<5s$, then $$b_{ij}=
\begin{cases}
-f_1(j,4s)e_{(j+m+1)n+g(j)+1}\otimes w_{jn+g(j+s)+1\ra (j+1)n},\quad i=(j+1)_s+(4-j_2)s;\\
0,\quad\text{otherwise.}\end{cases}$$

If $5s\le j<6s$, then $$b_{ij}=
\begin{cases}
-\kappa w_{(j+m+1)n+g(j+s)+1\ra (j+m+1)n+5}\otimes w_{jn+5\ra (j+1)n},\quad i=(j+1)_s;\\
-\kappa e_{(j+m+1)n+5}\otimes w_{jn+5\ra (j+1)n+g(j+s)},\quad i=(j+1)_s+2s;\\
0,\quad\text{otherwise.}\end{cases}$$

$(5)$ If $r_0=4$, then $\Omega^{4}(Y_t^{(18)})$ is described with
$(7s\times 9s)$-matrix with the following elements $b_{ij}${\rm:}

If $0\le j<s$, then $b_{ij}=0$.

If $s\le j<2s$, then $$b_{ij}=
\begin{cases}
w_{(j+m-1)n+5\ra (j+m)n+g(j+s)}\otimes e_{jn},\quad i=j-s;\\
0,\quad\text{otherwise.}\end{cases}$$

If $2s\le j<4s$, then $$b_{ij}=
\begin{cases}
w_{(j+m)n+g(j+s)\ra (j+m)n+g(j+s)+1}\otimes e_{jn+g(j)},\quad i=j;\\
-e_{(j+m)n+g(j+s)+1}\otimes w_{jn+g(j)\ra jn+g(j)+1},\quad i=j+4s;\\
0,\quad\text{otherwise.}\end{cases}$$

If $4s\le j<5s$, then $$b_{ij}=
\begin{cases}
\kappa e_{(j+m)n+5}\otimes w_{jn+g(j)+1\ra (j+1)n},\quad i=(j+1)_s;\\
0,\quad\text{otherwise.}\end{cases}$$

If $5s\le j<7s$, then $b_{ij}=0$.

$(6)$ If $r_0=5$, then $\Omega^{5}(Y_t^{(18)})$ is described with
$(6s\times 8s)$-matrix with the following elements $b_{ij}${\rm:}

If $0\le j<s$, then $$b_{ij}=
\begin{cases}
-\kappa w_{(j+m)n+g(j)\ra (j+m)n+5}\otimes e_{jn},\quad i=j;\\
-\kappa w_{(j+m)n+g(j+s)\ra (j+m)n+5}\otimes e_{jn},\quad i=j+s;\\
\kappa e_{(j+m)n+5}\otimes w_{jn\ra jn+g(j)+1},\quad i=j+4s;\\
\kappa e_{(j+m)n+5}\otimes w_{jn\ra jn+g(j+s)+1},\quad i=j+5s;\\
0,\quad\text{otherwise.}\end{cases}$$

If $s\le j<3s$, then $$b_{ij}=
\begin{cases}
-w_{(j+m+1)n+g(j)\ra (j+m+1)n+g(j)+1}\otimes w_{jn+g(j+s)\ra (j+1)n},\quad i=(j+1)_s+(2-j_2)s;\\
-w_{(j+m+1)n\ra (j+m+1)n+g(j)+1}\otimes e_{jn+g(j+s)},\quad i=j+s;\\
0,\quad\text{otherwise.}\end{cases}$$

If $3s\le j<5s$, then $b_{ij}=0$.

If $5s\le j<6s$, then $$b_{ij}=
\begin{cases}
-\kappa e_{(j+m+2)n}\otimes w_{jn+5\ra (j+1)n+g(j+s)},\quad i=(j+1)_s+2s;\\
0,\quad\text{otherwise.}\end{cases}$$

$(7)$ If $r_0=6$, then $\Omega^{6}(Y_t^{(18)})$ is described with
$(8s\times 9s)$-matrix with the following elements $b_{ij}${\rm:}

If $0\le j<2s$, then $$b_{ij}=
\begin{cases}
f_1(j,s)w_{(j+m)n\ra (j+m)n+g(j+s)+1}\otimes e_{jn},\quad i=(j)_s;\\
-f_1(j,s)w_{(j+m)n+g(j+s)\ra (j+m)n+g(j+s)+1}\otimes w_{jn\ra jn+g(j+s)},\quad i=(j)_s+(2-j_2)s;\\
0,\quad\text{otherwise.}\end{cases}$$

If $2s\le j<4s$, then $$b_{ij}=
\begin{cases}
-\kappa f_1(j,3s)w_{(j+m)n+g(j)\ra (j+m)n+5}\otimes e_{jn+g(j)},\quad i=j-s;\\
0,\quad\text{otherwise.}\end{cases}$$

If $4s\le j<5s$, then $$b_{ij}=
\begin{cases}
\kappa e_{(j+m+1)n}\otimes w_{jn+g(j)+1\ra (j+1)n},\quad i=(j+1)_s;\\
0,\quad\text{otherwise.}\end{cases}$$

If $5s\le j<6s$, then $b_{ij}=0$.

If $6s\le j<7s$, then $$b_{ij}=
\begin{cases}
w_{(j+m+1)n\ra (j+m+1)n+g(j+s)}\otimes w_{jn+5\ra (j+1)n},\quad i=(j+1)_s;\\
-e_{(j+m+1)n+g(j+s)}\otimes w_{jn+5\ra (j+1)n+g(j+s)},\quad i=(j+1)_s+2s;\\
0,\quad\text{otherwise.}\end{cases}$$

If $7s\le j<8s$, then $$b_{ij}=
\begin{cases}
-e_{(j+m+1)n+g(j+s)}\otimes w_{jn+5\ra (j+1)n+g(j+s)},\quad i=(j+1)_s+s;\\
0,\quad\text{otherwise.}\end{cases}$$

$(8)$ If $r_0=7$, then $\Omega^{7}(Y_t^{(18)})$ is described with
$(9s\times 8s)$-matrix with the following elements $b_{ij}${\rm:}

If $0\le j<s$, then $b_{ij}=0$.

If $s\le j<2s$, then $$b_{ij}=
\begin{cases}
-\kappa w_{(j+m)n+g(j)+1\ra (j+m+1)n}\otimes e_{jn},\quad i=j;\\
\kappa w_{(j+m)n+5\ra (j+m+1)n}\otimes w_{jn\ra jn+g(j+s)},\quad i=j+s;\\
\kappa e_{(j+m+1)n}\otimes w_{jn\ra jn+g(j+s)+1},\quad i=j+3s;\\
\kappa e_{(j+m+1)n}\otimes w_{jn\ra jn+g(j)+1},\quad i=j+4s;\\
0,\quad\text{otherwise.}\end{cases}$$

If $2s\le j<4s$, then $$b_{ij}=
\begin{cases}
w_{(j+m)n+5\ra (j+m+1)n+g(j)}\otimes e_{jn+g(j)},\quad i=j;\\
w_{(j+m+1)n\ra (j+m+1)n+g(j)}\otimes w_{jn+g(j)\ra jn+g(j)+1},\quad i=j+2s;\\
-e_{(j+m+1)n+g(j)}\otimes w_{jn+g(j)\ra jn+5},\quad i=j+4s;\\
0,\quad\text{otherwise.}\end{cases}$$

If $4s\le j<9s$, then $b_{ij}=0$.

$(9)$ If $r_0=8$, then $\Omega^{8}(Y_t^{(18)})$ is described with
$(8s\times 6s)$-matrix with the following elements $b_{ij}${\rm:}

If $0\le j<s$, then $$b_{ij}=
\begin{cases}
w_{(j+m-1)n+5\ra (j+m)n+g(j+s)}\otimes e_{jn},\quad i=j;\\
e_{(j+m)n+g(j+s)}\otimes w_{jn\ra jn+g(j)+1},\quad i=j+3s;\\
0,\quad\text{otherwise.}\end{cases}$$

If $s\le j<2s$, then $$b_{ij}=
\begin{cases}
e_{(j+m)n+g(j+s)}\otimes w_{jn\ra jn+g(j)+1},\quad i=j+3s;\\
0,\quad\text{otherwise.}\end{cases}$$

If $2s\le j<3s$, then $$b_{ij}=
\begin{cases}
-\kappa w_{(j+m)n+5\ra (j+m+1)n}\otimes w_{jn+g(j)\ra (j+1)n},\quad i=(j+1)_s;\\
-\kappa w_{(j+m)n+g(j)+1\ra (j+m+1)n}\otimes e_{jn+g(j)},\quad i=j-s;\\
0,\quad\text{otherwise.}\end{cases}$$

If $3s\le j<4s$, then $$b_{ij}=
\begin{cases}
\kappa w_{(j+m)n+g(j)+1\ra (j+m+1)n}\otimes e_{jn+g(j)},\quad i=j-s;\\
0,\quad\text{otherwise.}\end{cases}$$

If $4s\le j<6s$, then $b_{ij}=0$.

If $6s\le j<8s$, then $$b_{ij}=
\begin{cases}
-f_1(j,7s)e_{(j+m+1)n+g(j+s)+1}\otimes w_{jn+5\ra (j+1)n+g(j+s)},\quad i=(j+1)_s+(8-j_2)s;\\
0,\quad\text{otherwise.}\end{cases}$$

$(10)$ If $r_0=9$, then $\Omega^{9}(Y_t^{(18)})$ is described with
$(9s\times 7s)$-matrix with the following elements $b_{ij}${\rm:}

If $0\le j<s$, then $$b_{ij}=
\begin{cases}
\kappa w_{(j+m)n+5\ra (j+m+1)n}\otimes e_{jn},\quad i=j;\\
0,\quad\text{otherwise.}\end{cases}$$

If $s\le j<3s$, then $b_{ij}=0$.

If $3s\le j<5s$, then $$b_{ij}=
\begin{cases}
w_{(j+m+1)n\ra (j+m+1)n+g(j+s)+1}\otimes e_{jn+g(j+s)},\quad i=j-2s;\\
-e_{(j+m+1)n+g(j+s)+1}\otimes w_{jn+g(j+s)\ra jn+5},\quad i=j+2s;\\
0,\quad\text{otherwise.}\end{cases}$$

If $5s\le j<7s$, then $b_{ij}=0$.

If $7s\le j<8s$, then $$b_{ij}=
\begin{cases}
-\kappa e_{(j+m+1)n+5}\otimes w_{jn+5\ra (j+1)n},\quad i=(j+1)_s;\\
0,\quad\text{otherwise.}\end{cases}$$

If $8s\le j<9s$, then $b_{ij}=0$.

$(11)$ If $r_0=10$, then $\Omega^{10}(Y_t^{(18)})$ is described with
$(8s\times 6s)$-matrix with the following elements $b_{ij}${\rm:}

If $0\le j<2s$, then $$b_{ij}=
\begin{cases}
f_1(j,s)w_{(j+m)n\ra (j+m)n+g(j+s)+1}\otimes e_{jn},\quad i=(j)_s;\\
-w_{(j+m)n+g(j+s)\ra (j+m)n+g(j+s)+1}\otimes w_{jn\ra jn+g(j)},\quad i=j+s;\\
e_{(j+m)n+g(j+s)+1}\otimes w_{jn\ra jn+g(j)+1},\quad i=j+3s;\\
0,\quad\text{otherwise.}\end{cases}$$

If $2s\le j<4s$, then $$b_{ij}=
\begin{cases}
\kappa f_1(j,3s)w_{(j+m)n+g(j+s)\ra (j+m)n+5}\otimes e_{jn+g(j)},\quad i=j-s;\\
-\kappa f_1(j,3s)w_{(j+m)n+g(j+s)+1\ra (j+m)n+5}\otimes w_{jn+g(j)\ra jn+g(j)+1},\quad i=j+s;\\
-\kappa e_{(j+m)n+5}\otimes w_{jn+g(j)\ra jn+5},\quad i=(j)_s+5s;\\
0,\quad\text{otherwise.}\end{cases}$$

If $4s\le j<8s$, then $b_{ij}=0$.

\medskip
$({\rm II})$ Represent an arbitrary $t_0\in\N$ in the form
$t_0=11\ell_0+r_0$, where $0\le r_0\le 10.$ Then
$\Omega^{t_0}(Y_t^{(18)})$ is a $\Omega^{r_0}(Y_t^{(18)})$, whose
left components twisted by $\sigma^{\ell_0}$.
\end{pr}
\begin{pr}[Translates for the case 20]
$({\rm I})$ Let $r_0\in\N$, $r_0<11$. $r_0$-translates of the
elements $Y^{(20)}_t$ are described by the following way.

$(1)$ If $r_0=0$, then $\Omega^{0}(Y_t^{(20)})$ is described with
$(7s\times 6s)$-matrix with the following elements $b_{ij}${\rm:}

If $0\le j<s$, then $$b_{ij}=
\begin{cases}
\kappa w_{(j+m)n\ra (j+m)n+5}\otimes e_{jn},\quad i=j;\\
0,\quad\text{otherwise.}\end{cases}$$

If $s\le j<2s$, then $$b_{ij}=
\begin{cases}
\kappa w_{(j+m)n+g(j+s)\ra (j+m+1)n}\otimes e_{jn+g(j+s)},\quad i=j;\\
0,\quad\text{otherwise.}\end{cases}$$

If $2s\le j<3s$, then $b_{ij}=0$.

If $3s\le j<4s$, then $$b_{ij}=
\begin{cases}
-w_{(j+m)n+g(j+s)+1\ra (j+m+1)n+g(j+s)}\otimes e_{jn+g(j+s)+1},\quad i=j;\\
0,\quad\text{otherwise.}\end{cases}$$

If $4s\le j<6s$, then $b_{ij}=0$.

If $6s\le j<7s$, then $$b_{ij}=
\begin{cases}
-w_{(j+m)n+5\ra (j+m+1)n+g(j)+1}\otimes e_{jn+5},\quad i=j-s;\\
0,\quad\text{otherwise.}\end{cases}$$

$(2)$ If $r_0=1$, then $\Omega^{1}(Y_t^{(20)})$ is described with
$(6s\times 7s)$-matrix with the following elements $b_{ij}${\rm:}

If $0\le j<s$, then $$b_{ij}=
\begin{cases}
-\kappa w_{(j+m)n+g(j)\ra (j+m+1)n}\otimes e_{jn},\quad i=j;\\
0,\quad\text{otherwise.}\end{cases}$$

If $s\le j<2s$, then $$b_{ij}=
\begin{cases}
w_{(j+m)n+g(j+s)+1\ra (j+m+1)n+g(j+s)}\otimes e_{jn+g(j+s)},\quad i=j+s;\\
0,\quad\text{otherwise.}\end{cases}$$

If $2s\le j<3s$, then $b_{ij}=0$.

If $3s\le j<4s$, then $$b_{ij}=
\begin{cases}
w_{(j+m)n+5\ra (j+m+1)n+g(j+s)+1}\otimes e_{jn+g(j+s)+1},\quad i=j+s;\\
0,\quad\text{otherwise.}\end{cases}$$

If $4s\le j<5s$, then $b_{ij}=0$.

If $5s\le j<6s$, then $$b_{ij}=
\begin{cases}
-\kappa w_{(j+m+1)n\ra (j+m+1)n+5}\otimes e_{jn+5},\quad i=j+s;\\
0,\quad\text{otherwise.}\end{cases}$$

$(3)$ If $r_0=2$, then $\Omega^{2}(Y_t^{(20)})$ is described with
$(6s\times 6s)$-matrix with the following elements $b_{ij}${\rm:}

If $0\le j<s$, then $$b_{ij}=
\begin{cases}
\kappa w_{(j+m)n+g(j)+1\ra (j+m+1)n}\otimes w_{jn\ra jn+g(j)},\quad i=j+s;\\
0,\quad\text{otherwise.}\end{cases}$$

If $s\le j<2s$, then $$b_{ij}=
\begin{cases}
-w_{(j+m)n+5\ra (j+m+1)n+g(j+s)}\otimes w_{jn+g(j+s)\ra (j+1)n},\quad i=(j+1)_s;\\
0,\quad\text{otherwise.}\end{cases}$$

If $2s\le j<3s$, then $b_{ij}=0$.

If $3s\le j<4s$, then $$b_{ij}=
\begin{cases}
-w_{(j+m)n+5\ra (j+m+1)n+g(j+s)+1}\otimes w_{jn+g(j+s)+1\ra (j+1)n},\quad i=(j+1)_s;\\
0,\quad\text{otherwise.}\end{cases}$$

If $4s\le j<5s$, then $b_{ij}=0$.

If $5s\le j<6s$, then $$b_{ij}=
\begin{cases}
\kappa w_{(j+m+1)n+g(j+s)+1\ra (j+m+1)n+5}\otimes w_{jn+5\ra (j+1)n+g(j+s)},\quad i=(j+1)_s+s;\\
0,\quad\text{otherwise.}\end{cases}$$

$(4)$ If $r_0=3$, then $\Omega^{3}(Y_t^{(20)})$ is described with
$(7s\times 8s)$-matrix with the following elements $b_{ij}${\rm:}

If $0\le j<s$, then $$b_{ij}=
\begin{cases}
-w_{(j+m)n+5\ra (j+m+1)n+g(j)}\otimes w_{jn\ra jn+g(j)},\quad i=j+2s;\\
0,\quad\text{otherwise.}\end{cases}$$

If $s\le j<4s$, then $b_{ij}=0$.

If $4s\le j<5s$, then $$b_{ij}=
\begin{cases}
\kappa w_{(j+m+1)n+g(j)+1\ra (j+m+1)n+5}\otimes w_{jn+g(j)+1\ra (j+1)n},\quad i=(j+1)_s;\\
0,\quad\text{otherwise.}\end{cases}$$

If $5s\le j<7s$, then $b_{ij}=0$.

$(5)$ If $r_0=4$, then $\Omega^{4}(Y_t^{(20)})$ is described with
$(6s\times 9s)$-matrix with the following elements $b_{ij}${\rm:}

If $0\le j<s$, then $$b_{ij}=
\begin{cases}
\kappa w_{(j+m-1)n+5\ra (j+m)n+5}\otimes e_{jn},\quad i=j;\\
-\kappa w_{(j+m)n+g(j)\ra (j+m)n+5}\otimes w_{jn\ra jn+g(j+s)},\quad i=j+3s;\\
\kappa w_{(j+m)n+g(j)+1\ra (j+m)n+5}\otimes w_{jn\ra jn+g(j+s)+1},\quad i=j+7s;\\
0,\quad\text{otherwise.}\end{cases}$$

If $s\le j<2s$, then $$b_{ij}=
\begin{cases}
-w_{(j+m)n+5\ra (j+m+1)n+g(j+s)+1}\otimes w_{jn+g(j+s)\ra (j+1)n},\quad i=(j+1)_s;\\
0,\quad\text{otherwise.}\end{cases}$$

If $2s\le j<6s$, then $b_{ij}=0$.

$(6)$ If $r_0=5$, then $\Omega^{5}(Y_t^{(20)})$ is described with
$(8s\times 8s)$-matrix with the following elements $b_{ij}${\rm:}

If $0\le j<s$, then $$b_{ij}=
\begin{cases}
w_{(j+m+1)n+g(j)\ra (j+m+1)n+g(j)+1}\otimes w_{jn\ra (j+1)n},\quad i=(j+1)_s;\\
w_{(j+m+1)n\ra (j+m+1)n+g(j)+1}\otimes w_{jn\ra jn+g(j+s)},\quad i=j+3s;\\
0,\quad\text{otherwise.}\end{cases}$$

If $s\le j<3s$, then $b_{ij}=0$.

If $3s\le j<4s$, then $$b_{ij}=
\begin{cases}
\kappa w_{(j+m+1)n+g(j+s)\ra (j+m+1)n+5}\otimes w_{jn+g(j)\ra (j+1)n},\quad i=(j+1)_s;\\
\kappa w_{(j+m+1)n\ra (j+m+1)n+5}\otimes e_{jn+g(j)},\quad i=j;\\
0,\quad\text{otherwise.}\end{cases}$$

If $4s\le j<6s$, then $b_{ij}=0$.

If $6s\le j<7s$, then $$b_{ij}=
\begin{cases}
-w_{(j+m+2)n\ra (j+m+2)n+g(j)}\otimes w_{jn+5\ra (j+1)n+g(j+s)},\quad i=(j+1)_s+3s;\\
0,\quad\text{otherwise.}\end{cases}$$

If $7s\le j<8s$, then $b_{ij}=0$.

$(7)$ If $r_0=6$, then $\Omega^{6}(Y_t^{(20)})$ is described with
$(9s\times 9s)$-matrix with the following elements $b_{ij}${\rm:}

If $0\le j<s$, then $b_{ij}=0$.

If $s\le j<2s$, then $$b_{ij}=
\begin{cases}
-\kappa w_{(j+m)n\ra (j+m+1)n}\otimes e_{jn},\quad i=j-s;\\
\kappa w_{(j+m)n+g(j+s)\ra (j+m+1)n}\otimes w_{jn\ra jn+g(j+s)},\quad i=j;\\
0,\quad\text{otherwise.}\end{cases}$$

If $2s\le j<3s$, then $b_{ij}=0$.

If $3s\le j<4s$, then $$b_{ij}=
\begin{cases}
-w_{(j+m+1)n\ra (j+m+1)n+g(j+s)}\otimes w_{jn+g(j)\ra (j+1)n},\quad i=(j+1)_s;\\
0,\quad\text{otherwise.}\end{cases}$$

If $4s\le j<9s$, then $b_{ij}=0$.

$(8)$ If $r_0=7$, then $\Omega^{7}(Y_t^{(20)})$ is described with
$(8s\times 8s)$-matrix with the following elements $b_{ij}${\rm:}

If $0\le j<s$, then $$b_{ij}=
\begin{cases}
w_{(j+m)n+5\ra (j+m+1)n+g(j)}\otimes w_{jn\ra jn+g(j)},\quad i=j+2s;\\
w_{(j+m+1)n\ra (j+m+1)n+g(j)}\otimes w_{jn\ra jn+g(j)+1},\quad i=j+4s;\\
w_{(j+m+1)n\ra (j+m+1)n+g(j)}\otimes w_{jn\ra jn+g(j+s)+1},\quad i=j+5s;\\
0,\quad\text{otherwise.}\end{cases}$$

If $s\le j<2s$, then $b_{ij}=0$.

If $2s\le j<3s$, then $$b_{ij}=
\begin{cases}
\kappa w_{(j+m+1)n+g(j)+1\ra (j+m+2)n}\otimes w_{jn+g(j)\ra (j+1)n},\quad i=(j+1)_s;\\
0,\quad\text{otherwise.}\end{cases}$$

If $3s\le j<8s$, then $b_{ij}=0$.

$(9)$ If $r_0=8$, then $\Omega^{8}(Y_t^{(20)})$ is described with
$(9s\times 6s)$-matrix with the following elements $b_{ij}${\rm:}

If $0\le j<s$, then $$b_{ij}=
\begin{cases}
-\kappa w_{(j+m)n+g(j)+1\ra (j+m+1)n}\otimes w_{jn\ra jn+g(j)},\quad i=j+s;\\
-\kappa w_{(j+m)n+g(j)\ra (j+m+1)n}\otimes w_{jn\ra jn+g(j+s)+1},\quad i=j+4s;\\
0,\quad\text{otherwise.}\end{cases}$$

If $s\le j<2s$, then $$b_{ij}=
\begin{cases}
-w_{(j+m)n+5\ra (j+m+1)n+g(j+s)}\otimes w_{jn+g(j+s)\ra (j+1)n},\quad i=(j+1)_s;\\
0,\quad\text{otherwise.}\end{cases}$$

If $2s\le j<7s$, then $b_{ij}=0$.

If $7s\le j<8s$, then $$b_{ij}=
\begin{cases}
\kappa w_{(j+m+1)n\ra (j+m+1)n+5}\otimes e_{jn+5},\quad i=j-2s;\\
0,\quad\text{otherwise.}\end{cases}$$

If $8s\le j<9s$, then $b_{ij}=0$.

$(10)$ If $r_0=9$, then $\Omega^{9}(Y_t^{(20)})$ is described with
$(8s\times 7s)$-matrix with the following elements $b_{ij}${\rm:}

If $0\le j<s$, then $$b_{ij}=
\begin{cases}
-w_{(j+m)n+5\ra (j+m+1)n+g(j)+1}\otimes e_{jn},\quad i=j;\\
0,\quad\text{otherwise.}\end{cases}$$

If $s\le j<2s$, then $b_{ij}=0$.

If $2s\le j<3s$, then $$b_{ij}=
\begin{cases}
\kappa w_{(j+m+1)n\ra (j+m+1)n+5}\otimes e_{jn+g(j)},\quad i=j-s;\\
0,\quad\text{otherwise.}\end{cases}$$

If $3s\le j<4s$, then $$b_{ij}=
\begin{cases}
\kappa w_{(j+m+1)n\ra (j+m+1)n+5}\otimes e_{jn+g(j)},\quad i=j-s;\\
0,\quad\text{otherwise.}\end{cases}$$

If $4s\le j<5s$, then $b_{ij}=0$.

If $5s\le j<6s$, then $$b_{ij}=
\begin{cases}
\kappa w_{(j+m+1)n+g(j+s)\ra (j+m+2)n}\otimes e_{jn+g(j)+1},\quad i=j-s;\\
0,\quad\text{otherwise.}\end{cases}$$

If $6s\le j<7s$, then $$b_{ij}=
\begin{cases}
w_{(j+m+1)n+5\ra (j+m+2)n+g(j)}\otimes w_{jn+5\ra (j+1)n},\quad i=(j+1)_s;\\
0,\quad\text{otherwise.}\end{cases}$$

If $7s\le j<8s$, then $b_{ij}=0$.

$(11)$ If $r_0=10$, then $\Omega^{10}(Y_t^{(20)})$ is described with
$(6s\times 6s)$-matrix with the following elements $b_{ij}${\rm:}

If $0\le j<s$, then $$b_{ij}=
\begin{cases}
\kappa w_{(j+m)n\ra (j+m)n+5}\otimes e_{jn},\quad i=j;\\
0,\quad\text{otherwise.}\end{cases}$$

If $s\le j<4s$, then $b_{ij}=0$.

If $4s\le j<5s$, then $$b_{ij}=
\begin{cases}
-w_{(j+m)n+g(j)+1\ra (j+m+1)n+g(j)}\otimes e_{jn+g(j+s)+1},\quad i=j;\\
w_{(j+m)n+5\ra (j+m+1)n+g(j)}\otimes w_{jn+g(j+s)+1\ra jn+5},\quad i=j+s;\\
0,\quad\text{otherwise.}\end{cases}$$

If $5s\le j<6s$, then $$b_{ij}=
\begin{cases}
-\kappa w_{(j+m+1)n+g(j)\ra (j+m+2)n}\otimes w_{jn+5\ra (j+1)n+g(j+s)},\quad i=(j+1)_s+s;\\
\kappa w_{(j+m+1)n+g(j)+1\ra (j+m+2)n}\otimes w_{jn+5\ra (j+1)n+g(j+s)+1},\quad i=(j+1)_s+3s;\\
-\kappa w_{(j+m+1)n+g(j+s)+1\ra (j+m+2)n}\otimes w_{jn+5\ra (j+1)n+g(j)+1},\quad i=(j+1)_s+4s;\\
\kappa w_{(j+m+1)n+5\ra (j+m+2)n}\otimes w_{jn+5\ra (j+1)n+5},\quad i=(j+1)_s+5s;\\
0,\quad\text{otherwise.}\end{cases}$$

\medskip
$({\rm II})$ Represent an arbitrary $t_0\in\N$ in the form
$t_0=11\ell_0+r_0$, where $0\le r_0\le 10.$ Then
$\Omega^{t_0}(Y_t^{(20)})$ is a $\Omega^{r_0}(Y_t^{(20)})$, whose
left components twisted by $\sigma^{\ell_0}$.
\end{pr}
\begin{pr}[Translates for the case 22]
$({\rm I})$ Let $r_0\in\N$, $r_0<11$. $r_0$-translates of the
elements $Y^{(22)}_t$ are described by the following way.

$(1)$ If $r_0=0$, then $\Omega^{0}(Y_t^{(22)})$ is described with
$(6s\times 6s)$-matrix with the following elements $b_{ij}${\rm:}

If $0\le j<s$, then $b_{ij}=0$.

If $s\le j<3s$, then $$b_{ij}=
\begin{cases}
f_1(j,2s)e_{(j+m)n+g(j+s)}\otimes e_{jn+g(j+s)},\quad i=j;\\
0,\quad\text{otherwise.}\end{cases}$$

If $3s\le j<5s$, then $$b_{ij}=
\begin{cases}
f_1(j,4s)e_{(j+m)n+g(j+s)+1}\otimes e_{jn+g(j+s)+1},\quad i=j;\\
0,\quad\text{otherwise.}\end{cases}$$

If $5s\le j<6s$, then $b_{ij}=0$.

$(2)$ If $r_0=1$, then $\Omega^{1}(Y_t^{(22)})$ is described with
$(6s\times 7s)$-matrix with the following elements $b_{ij}${\rm:}

If $0\le j<s$, then $$b_{ij}=
\begin{cases}
\kappa w_{(j+m)n+g(j)\ra (j+m+1)n}\otimes e_{jn},\quad i=j;\\
-\kappa w_{(j+m)n+g(j+s)\ra (j+m+1)n}\otimes e_{jn},\quad i=j+s;\\
-\kappa w_{(j+m)n+g(j)+1\ra (j+m+1)n}\otimes w_{jn\ra jn+g(j)},\quad i=j+2s;\\
\kappa w_{(j+m)n+g(j+s)+1\ra (j+m+1)n}\otimes w_{jn\ra jn+g(j+s)},\quad i=j+3s;\\
0,\quad\text{otherwise.}\end{cases}$$

If $s\le j<3s$, then $$b_{ij}=
\begin{cases}
f_1(j,2s)e_{(j+m+1)n+g(j+s)}\otimes w_{jn+g(j+s)\ra (j+1)n},\quad i=(j+1)_s+(j_2-1)s;\\
-f_1(j,2s)w_{(j+m)n+g(j+s)+1\ra (j+m+1)n+g(j+s)}\otimes e_{jn+g(j+s)},\quad i=j+s;\\
f_1(j,2s)w_{(j+m)n+5\ra (j+m+1)n+g(j+s)}\otimes w_{jn+g(j+s)\ra jn+g(j+s)+1},\quad i=j+3s;\\
f_1(j,2s)w_{(j+m+1)n\ra (j+m+1)n+g(j+s)}\otimes w_{jn+g(j+s)\ra jn+5},\quad i=(j)_s+6s;\\
0,\quad\text{otherwise.}\end{cases}$$

If $3s\le j<5s$, then $$b_{ij}=
\begin{cases}
-f_1(j,4s)w_{(j+m+1)n+g(j+s)\ra (j+m+1)n+g(j+s)+1}\otimes w_{jn+g(j+s)+1\ra (j+1)n},\\\quad\quad\quad i=(j+1)_s+(j_2-3)s;\\
f_1(j,4s)e_{(j+m+1)n+g(j+s)+1}\otimes w_{jn+g(j+s)+1\ra (j+1)n+g(j+s)},\quad i=(j+1)_s+(j_2-1)s;\\
-f_1(j,4s)w_{(j+m)n+5\ra (j+m+1)n+g(j+s)+1}\otimes e_{jn+g(j+s)+1},\quad i=j+s;\\
-f_1(j,4s)w_{(j+m+1)n\ra (j+m+1)n+g(j+s)+1}\otimes w_{jn+g(j+s)+1\ra jn+5},\quad i=(j)_s+6s;\\
0,\quad\text{otherwise.}\end{cases}$$

If $5s\le j<6s$, then $$b_{ij}=
\begin{cases}
\kappa w_{(j+m+1)n+g(j+s)\ra (j+m+1)n+5}\otimes w_{jn+5\ra (j+1)n},\quad i=(j+1)_s;\\
-\kappa w_{(j+m+1)n+g(j)\ra (j+m+1)n+5}\otimes w_{jn+5\ra (j+1)n},\quad i=(j+1)_s+s;\\
-\kappa w_{(j+m+1)n+g(j+s)+1\ra (j+m+1)n+5}\otimes w_{jn+5\ra (j+1)n+g(j+s)},\quad i=(j+1)_s+2s;\\
\kappa w_{(j+m+1)n+g(j)+1\ra (j+m+1)n+5}\otimes w_{jn+5\ra (j+1)n+g(j)},\quad i=(j+1)_s+3s;\\
0,\quad\text{otherwise.}\end{cases}$$

$(3)$ If $r_0=2$, then $\Omega^{2}(Y_t^{(22)})$ is described with
$(7s\times 6s)$-matrix with the following elements $b_{ij}${\rm:}

If $0\le j<2s$, then $$b_{ij}=
\begin{cases}
-w_{(j+m-1)n+5\ra (j+m)n+g(j)}\otimes e_{jn},\quad i=(j)_s;\\
-f_1(j,s)e_{(j+m)n+g(j)}\otimes w_{jn\ra jn+g(j+s)+1},\quad i=(j)_s+(4-j_2)s;\\
0,\quad\text{otherwise.}\end{cases}$$

If $2s\le j<4s$, then $$b_{ij}=
\begin{cases}
-f_1(j,3s)e_{(j+m)n+g(j)+1}\otimes e_{jn+g(j)},\quad i=j-s;\\
0,\quad\text{otherwise.}\end{cases}$$

If $4s\le j<6s$, then $$b_{ij}=
\begin{cases}
\kappa f_1(j,5s)w_{(j+m)n+g(j+s)\ra (j+m)n+5}\otimes e_{jn+g(j)+1},\quad i=j-s;\\
0,\quad\text{otherwise.}\end{cases}$$

If $6s\le j<7s$, then $b_{ij}=0$.

$(4)$ If $r_0=3$, then $\Omega^{3}(Y_t^{(22)})$ is described with
$(6s\times 8s)$-matrix with the following elements $b_{ij}${\rm:}

If $0\le j<s$, then $$b_{ij}=
\begin{cases}
-\kappa w_{(j+m)n+g(j)+1\ra (j+m)n+5}\otimes e_{jn},\quad i=j;\\
\kappa w_{(j+m)n+g(j+s)+1\ra (j+m)n+5}\otimes e_{jn},\quad i=j+s;\\
0,\quad\text{otherwise.}\end{cases}$$

If $s\le j<3s$, then $$b_{ij}=
\begin{cases}
e_{(j+m+1)n+g(j+s)+1}\otimes w_{jn+g(j+s)\ra (j+1)n},\quad i=(j+1)_s+(j_2-1)s;\\
-f_1(j,2s)w_{(j+m)n+5\ra (j+m+1)n+g(j+s)+1}\otimes e_{jn+g(j+s)},\quad i=j+s;\\
0,\quad\text{otherwise.}\end{cases}$$

If $3s\le j<5s$, then $$b_{ij}=
\begin{cases}
-f_1(j,4s)w_{(j+m+1)n\ra (j+m+1)n+g(j)}\otimes e_{jn+g(j+s)+1},\quad i=j+s;\\
-f_1(j,4s)e_{(j+m+1)n+g(j)}\otimes w_{jn+g(j+s)+1\ra jn+5},\quad i=(j)_s+(10-j_2)s;\\
0,\quad\text{otherwise.}\end{cases}$$

If $5s\le j<6s$, then $$b_{ij}=
\begin{cases}
-\kappa w_{(j+m+1)n+g(j+s)+1\ra (j+m+2)n}\otimes w_{jn+5\ra (j+1)n},\quad i=(j+1)_s;\\
\kappa w_{(j+m+1)n+g(j+s)\ra (j+m+2)n}\otimes e_{jn+5},\quad i=j+s;\\
-\kappa w_{(j+m+1)n+g(j)\ra (j+m+2)n}\otimes e_{jn+5},\quad i=j+2s;\\
0,\quad\text{otherwise.}\end{cases}$$

$(5)$ If $r_0=4$, then $\Omega^{4}(Y_t^{(22)})$ is described with
$(8s\times 9s)$-matrix with the following elements $b_{ij}${\rm:}

If $0\le j<2s$, then $$b_{ij}=
\begin{cases}
w_{(j+m-1)n+5\ra (j+m)n+g(j)+1}\otimes e_{jn},\quad i=(j)_s;\\
-w_{(j+m)n+g(j)\ra (j+m)n+g(j)+1}\otimes w_{jn\ra jn+g(j+s)},\quad i=(j)_s+(3-j_2)s;\\
-e_{(j+m)n+g(j)+1}\otimes w_{jn\ra jn+g(j+s)+1},\quad i=(j)_s+(7-j_2)s;\\
0,\quad\text{otherwise.}\end{cases}$$

If $2s\le j<4s$, then $$b_{ij}=
\begin{cases}
\kappa f_1(j,3s)w_{(j+m)n+g(j+s)\ra (j+m)n+5}\otimes e_{jn+g(j)},\quad i=j;\\
\kappa f_1(j,3s)w_{(j+m)n+g(j+s)+1\ra (j+m)n+5}\otimes w_{jn+g(j)\ra jn+g(j)+1},\quad i=j+4s;\\
0,\quad\text{otherwise.}\end{cases}$$

If $4s\le j<6s$, then $$b_{ij}=
\begin{cases}
-\kappa f_1(j,5s)w_{(j+m)n+g(j)\ra (j+m+1)n}\otimes e_{jn+g(j)+1},\quad i=j;\\
-\kappa f_1(j,5s)w_{(j+m)n+g(j+s)+1\ra (j+m+1)n}\otimes e_{jn+g(j)+1},\quad i=j+2s;\\
0,\quad\text{otherwise.}\end{cases}$$

If $6s\le j<7s$, then $$b_{ij}=
\begin{cases}
w_{(j+m)n+5\ra (j+m+1)n+g(j)}\otimes w_{jn+5\ra (j+1)n},\quad i=(j+1)_s;\\
-e_{(j+m+1)n+g(j)}\otimes w_{jn+5\ra (j+1)n+g(j+s)},\quad i=(j+1)_s+3s;\\
-w_{(j+m)n+5\ra (j+m+1)n+g(j)}\otimes e_{jn+5},\quad i=j+2s;\\
0,\quad\text{otherwise.}\end{cases}$$

If $7s\le j<8s$, then $$b_{ij}=
\begin{cases}
e_{(j+m+1)n+g(j)}\otimes w_{jn+5\ra (j+1)n+g(j+s)},\quad i=(j+1)_s+2s;\\
-w_{(j+m)n+5\ra (j+m+1)n+g(j)}\otimes e_{jn+5},\quad i=j+s;\\
0,\quad\text{otherwise.}\end{cases}$$

$(6)$ If $r_0=5$, then $\Omega^{5}(Y_t^{(22)})$ is described with
$(9s\times 8s)$-matrix with the following elements $b_{ij}${\rm:}

If $0\le j<s$, then $b_{ij}=0$.

If $s\le j<2s$, then $$b_{ij}=
\begin{cases}
\kappa e_{(j+m+1)n}\otimes w_{jn\ra jn+g(j+s)},\quad i=j+s;\\
\kappa e_{(j+m+1)n}\otimes w_{jn\ra jn+g(j)},\quad i=j+2s;\\
-\kappa w_{(j+m)n+5\ra (j+m+1)n}\otimes w_{jn\ra jn+g(j+s)+1},\quad i=j+3s;\\
-\kappa w_{(j+m)n+5\ra (j+m+1)n}\otimes w_{jn\ra jn+g(j)+1},\quad i=j+4s;\\
0,\quad\text{otherwise.}\end{cases}$$

If $2s\le j<4s$, then $$b_{ij}=
\begin{cases}
-f_1(j,3s)e_{(j+m+1)n+g(j+s)}\otimes w_{jn+g(j)\ra (j+1)n},\quad i=(j+1)_s+(3-j_2)s;\\
-f_1(j,3s)w_{(j+m+1)n\ra (j+m+1)n+g(j+s)}\otimes e_{jn+g(j)},\quad i=j;\\
f_1(j,3s)w_{(j+m)n+5\ra (j+m+1)n+g(j+s)}\otimes w_{jn+g(j)\ra jn+g(j)+1},\quad i=j+2s;\\
0,\quad\text{otherwise.}\end{cases}$$

If $4s\le j<6s$, then $$b_{ij}=
\begin{cases}
f_1(j,5s)w_{(j+m)n+5\ra (j+m+1)n+g(j)}\otimes e_{jn+g(j)+1},\quad i=j;\\
0,\quad\text{otherwise.}\end{cases}$$

If $6s\le j<8s$, then $$b_{ij}=
\begin{cases}
-f_1(j,7s)w_{(j+m)n+5\ra (j+m+1)n+g(j+s)+1}\otimes e_{jn+g(j)+1},\quad i=j-2s;\\
e_{(j+m+1)n+g(j+s)+1}\otimes w_{jn+g(j)+1\ra jn+5},\quad i=(j)_s+(13-j_2)s;\\
0,\quad\text{otherwise.}\end{cases}$$

If $8s\le j<9s$, then $$b_{ij}=
\begin{cases}
-\kappa w_{(j+m+1)n+g(j)+1\ra (j+m+1)n+5}\otimes e_{jn+5},\quad i=j-2s;\\
\kappa w_{(j+m+1)n+g(j+s)+1\ra (j+m+1)n+5}\otimes e_{jn+5},\quad i=j-s;\\
0,\quad\text{otherwise.}\end{cases}$$

$(7)$ If $r_0=6$, then $\Omega^{6}(Y_t^{(22)})$ is described with
$(8s\times 9s)$-matrix with the following elements $b_{ij}${\rm:}

If $0\le j<2s$, then $$b_{ij}=
\begin{cases}
f_1(j,s)e_{(j+m)n+g(j)}\otimes w_{jn\ra jn+g(j)},\quad i=j+s;\\
0,\quad\text{otherwise.}\end{cases}$$

If $2s\le j<3s$, then $$b_{ij}=
\begin{cases}
-\kappa e_{(j+m+1)n}\otimes w_{jn+g(j)\ra (j+1)n},\quad i=(j+1)_s;\\
\kappa w_{(j+m)n+g(j)\ra (j+m+1)n}\otimes e_{jn+g(j)},\quad i=j-s;\\
-\kappa w_{(j+m)n+5\ra (j+m+1)n}\otimes w_{jn+g(j)\ra jn+5},\quad i=j+5s;\\
\kappa e_{(j+m+1)n}\otimes w_{jn+g(j)\ra jn+5},\quad i=j+6s;\\
0,\quad\text{otherwise.}\end{cases}$$

If $3s\le j<4s$, then $$b_{ij}=
\begin{cases}
-\kappa w_{(j+m)n+g(j)\ra (j+m+1)n}\otimes e_{jn+g(j)},\quad i=j-s;\\
\kappa w_{(j+m)n+g(j)+1\ra (j+m+1)n}\otimes w_{jn+g(j)\ra jn+g(j)+1},\quad i=j+3s;\\
-\kappa w_{(j+m)n+5\ra (j+m+1)n}\otimes w_{jn+g(j)\ra jn+5},\quad i=j+4s;\\
-\kappa e_{(j+m+1)n}\otimes w_{jn+g(j)\ra jn+5},\quad i=j+5s;\\
0,\quad\text{otherwise.}\end{cases}$$

If $4s\le j<6s$, then $$b_{ij}=
\begin{cases}
\kappa f_1(j,5s)w_{(j+m)n+g(j)+1\ra (j+m)n+5}\otimes e_{jn+g(j)+1},\quad i=j+s;\\
0,\quad\text{otherwise.}\end{cases}$$

If $6s\le j<7s$, then $$b_{ij}=
\begin{cases}
-e_{(j+m+1)n+g(j)+1}\otimes w_{jn+5\ra (j+1)n+g(j)+1},\quad i=(j+1)_s+5s;\\
-w_{(j+m)n+5\ra (j+m+1)n+g(j)+1}\otimes e_{jn+5},\quad i=j+s;\\
-w_{(j+m+1)n\ra (j+m+1)n+g(j)+1}\otimes e_{jn+5},\quad i=j+2s;\\
0,\quad\text{otherwise.}\end{cases}$$

If $7s\le j<8s$, then $$b_{ij}=
\begin{cases}
e_{(j+m+1)n+g(j)+1}\otimes w_{jn+5\ra (j+1)n+g(j+s)},\quad i=(j+1)_s+3s;\\
e_{(j+m+1)n+g(j)+1}\otimes w_{jn+5\ra (j+1)n+g(j)+1},\quad i=(j+1)_s+6s;\\
w_{(j+m)n+5\ra (j+m+1)n+g(j)+1}\otimes e_{jn+5},\quad i=j;\\
0,\quad\text{otherwise.}\end{cases}$$

$(8)$ If $r_0=7$, then $\Omega^{7}(Y_t^{(22)})$ is described with
$(9s\times 8s)$-matrix with the following elements $b_{ij}${\rm:}

If $0\le j<s$, then $$b_{ij}=
\begin{cases}
-\kappa e_{(j+m+1)n}\otimes w_{jn\ra jn+g(j)+1},\quad i=j+4s;\\
-\kappa e_{(j+m+1)n}\otimes w_{jn\ra jn+g(j+s)+1},\quad i=j+5s;\\
0,\quad\text{otherwise.}\end{cases}$$

If $s\le j<2s$, then $$b_{ij}=
\begin{cases}
-w_{(j+m)n+5\ra (j+m+1)n+g(j+s)}\otimes e_{jn+g(j+s)},\quad i=j+s;\\
-w_{(j+m+1)n\ra (j+m+1)n+g(j+s)}\otimes w_{jn+g(j+s)\ra jn+g(j+s)+1},\quad i=j+3s;\\
-e_{(j+m+1)n+g(j+s)}\otimes w_{jn+g(j+s)\ra jn+5},\quad i=j+5s;\\
0,\quad\text{otherwise.}\end{cases}$$

If $2s\le j<3s$, then $$b_{ij}=
\begin{cases}
w_{(j+m)n+5\ra (j+m+1)n+g(j+s)}\otimes e_{jn+g(j+s)},\quad i=j+s;\\
e_{(j+m+1)n+g(j+s)}\otimes w_{jn+g(j+s)\ra jn+5},\quad i=j+5s;\\
0,\quad\text{otherwise.}\end{cases}$$

If $3s\le j<4s$, then $$b_{ij}=
\begin{cases}
e_{(j+m+1)n+g(j)+1}\otimes w_{jn+g(j+s)\ra (j+1)n},\quad i=(j+1)_s+s;\\
w_{(j+m+1)n\ra (j+m+1)n+g(j)+1}\otimes w_{jn+g(j+s)\ra jn+g(j+s)+1},\quad i=j+s;\\
0,\quad\text{otherwise.}\end{cases}$$

If $4s\le j<5s$, then $b_{ij}=0$.

If $5s\le j<7s$, then $$b_{ij}=
\begin{cases}
-e_{(j+m+1)n+g(j+s)+1}\otimes w_{jn+g(j+s)+1\ra (j+1)n},\quad i=(j+1)_s+(j_2-5)s;\\
-f_1(j,6s)w_{(j+m+1)n\ra (j+m+1)n+g(j+s)+1}\otimes e_{jn+g(j+s)+1},\quad i=j-s;\\
-w_{(j+m+1)n+g(j+s)\ra (j+m+1)n+g(j+s)+1}\otimes w_{jn+g(j+s)+1\ra jn+5},\quad i=j+s;\\
0,\quad\text{otherwise.}\end{cases}$$

If $7s\le j<8s$, then $$b_{ij}=
\begin{cases}
-\kappa w_{(j+m+1)n+g(j)+1\ra (j+m+1)n+5}\otimes w_{jn+5\ra (j+1)n},\quad i=(j+1)_s+s;\\
-\kappa w_{(j+m+1)n+g(j+s)\ra (j+m+1)n+5}\otimes e_{jn+5},\quad i=j-s;\\
0,\quad\text{otherwise.}\end{cases}$$

If $8s\le j<9s$, then $$b_{ij}=
\begin{cases}
\kappa e_{(j+m+2)n}\otimes w_{jn+5\ra (j+1)n+g(j)+1},\quad i=(j+1)_s+4s;\\
-\kappa e_{(j+m+2)n}\otimes w_{jn+5\ra (j+1)n+g(j+s)+1},\quad i=(j+1)_s+5s;\\
\kappa w_{(j+m+1)n+g(j+s)\ra (j+m+2)n}\otimes e_{jn+5},\quad i=j-s;\\
0,\quad\text{otherwise.}\end{cases}$$

$(9)$ If $r_0=8$, then $\Omega^{8}(Y_t^{(22)})$ is described with
$(8s\times 6s)$-matrix with the following elements $b_{ij}${\rm:}

If $0\le j<s$, then $$b_{ij}=
\begin{cases}
e_{(j+m)n+g(j)+1}\otimes w_{jn\ra jn+g(j)},\quad i=j+s;\\
w_{(j+m)n+g(j)\ra (j+m)n+g(j)+1}\otimes w_{jn\ra jn+g(j+s)+1},\quad i=j+4s;\\
0,\quad\text{otherwise.}\end{cases}$$

If $s\le j<2s$, then $$b_{ij}=
\begin{cases}
-e_{(j+m)n+g(j)+1}\otimes w_{jn\ra jn+g(j)},\quad i=j+s;\\
0,\quad\text{otherwise.}\end{cases}$$

If $2s\le j<3s$, then $$b_{ij}=
\begin{cases}
\kappa w_{(j+m)n+g(j)+1\ra (j+m)n+5}\otimes e_{jn+g(j)},\quad i=j-s;\\
0,\quad\text{otherwise.}\end{cases}$$

If $3s\le j<4s$, then $$b_{ij}=
\begin{cases}
-\kappa w_{(j+m)n+g(j)+1\ra (j+m)n+5}\otimes e_{jn+g(j)},\quad i=j-s;\\
-\kappa w_{(j+m)n+g(j+s)\ra (j+m)n+5}\otimes w_{jn+g(j)\ra jn+g(j)+1},\quad i=j+s;\\
0,\quad\text{otherwise.}\end{cases}$$

If $4s\le j<5s$, then $$b_{ij}=
\begin{cases}
\kappa w_{(j+m)n+5\ra (j+m+1)n}\otimes w_{jn+g(j)+1\ra (j+1)n},\quad i=(j+1)_s;\\
\kappa w_{(j+m)n+g(j+s)\ra (j+m+1)n}\otimes e_{jn+g(j)+1},\quad i=j-s;\\
\kappa e_{(j+m+1)n}\otimes w_{jn+g(j)+1\ra jn+5},\quad i=j+s;\\
0,\quad\text{otherwise.}\end{cases}$$

If $5s\le j<6s$, then $$b_{ij}=
\begin{cases}
-\kappa w_{(j+m)n+g(j+s)\ra (j+m+1)n}\otimes e_{jn+g(j)+1},\quad i=j-s;\\
-\kappa e_{(j+m+1)n}\otimes w_{jn+g(j)+1\ra jn+5},\quad i=j;\\
0,\quad\text{otherwise.}\end{cases}$$

If $6s\le j<7s$, then $$b_{ij}=
\begin{cases}
-e_{(j+m+1)n+g(j)}\otimes w_{jn+5\ra (j+1)n+g(j+s)+1},\quad i=(j+1)_s+4s;\\
w_{(j+m+1)n\ra (j+m+1)n+g(j)}\otimes e_{jn+5},\quad i=j-s;\\
0,\quad\text{otherwise.}\end{cases}$$

If $7s\le j<8s$, then $$b_{ij}=
\begin{cases}
-w_{(j+m)n+5\ra (j+m+1)n+g(j)}\otimes w_{jn+5\ra (j+1)n},\quad i=(j+1)_s;\\
e_{(j+m+1)n+g(j)}\otimes w_{jn+5\ra (j+1)n+g(j+s)+1},\quad i=(j+1)_s+3s;\\
-w_{(j+m+1)n\ra (j+m+1)n+g(j)}\otimes e_{jn+5},\quad i=j-2s;\\
0,\quad\text{otherwise.}\end{cases}$$

$(10)$ If $r_0=9$, then $\Omega^{9}(Y_t^{(22)})$ is described with
$(6s\times 7s)$-matrix with the following elements $b_{ij}${\rm:}

If $0\le j<s$, then $b_{ij}=0$.

If $s\le j<3s$, then $$b_{ij}=
\begin{cases}
-f_1(j,2s)w_{(j+m+1)n\ra (j+m+1)n+g(j+s)+1}\otimes e_{jn+g(j+s)},\quad i=j;\\
-f_1(j,2s)e_{(j+m+1)n+g(j+s)+1}\otimes w_{jn+g(j+s)\ra jn+5},\quad i=j+4s;\\
0,\quad\text{otherwise.}\end{cases}$$

If $3s\le j<5s$, then $$b_{ij}=
\begin{cases}
-f_1(j,4s)e_{(j+m+1)n+g(j)}\otimes e_{jn+g(j+s)+1},\quad i=j;\\
0,\quad\text{otherwise.}\end{cases}$$

If $5s\le j<6s$, then $$b_{ij}=
\begin{cases}
\kappa w_{(j+m+1)n+5\ra (j+m+2)n}\otimes w_{jn+5\ra (j+1)n},\quad i=(j+1)_s;\\
0,\quad\text{otherwise.}\end{cases}$$

$(11)$ If $r_0=10$, then $\Omega^{10}(Y_t^{(22)})$ is described with
$(7s\times 6s)$-matrix with the following elements $b_{ij}${\rm:}

If $0\le j<s$, then $$b_{ij}=
\begin{cases}
-\kappa w_{(j+m)n+g(j+s)\ra (j+m)n+5}\otimes w_{jn\ra jn+g(j)},\quad i=j+s;\\
-\kappa w_{(j+m)n+g(j)\ra (j+m)n+5}\otimes w_{jn\ra jn+g(j+s)},\quad i=j+2s;\\
-\kappa w_{(j+m)n+g(j+s)+1\ra (j+m)n+5}\otimes w_{jn\ra jn+g(j)+1},\quad i=j+3s;\\
-\kappa w_{(j+m)n+g(j)+1\ra (j+m)n+5}\otimes w_{jn\ra jn+g(j+s)+1},\quad i=j+4s;\\
0,\quad\text{otherwise.}\end{cases}$$

If $s\le j<3s$, then $$b_{ij}=
\begin{cases}
\kappa f_1(j,2s)w_{(j+m)n+g(j)\ra (j+m+1)n}\otimes e_{jn+g(j+s)},\quad i=j;\\
\kappa f_1(j,2s)w_{(j+m)n+g(j)+1\ra (j+m+1)n}\otimes w_{jn+g(j+s)\ra jn+g(j+s)+1},\quad i=j+2s;\\
\kappa w_{(j+m)n+5\ra (j+m+1)n}\otimes w_{jn+g(j+s)\ra jn+5},\quad i=(j)_s+5s;\\
0,\quad\text{otherwise.}\end{cases}$$

If $3s\le j<5s$, then $$b_{ij}=
\begin{cases}
-w_{(j+m+1)n\ra (j+m+1)n+g(j)}\otimes w_{jn+g(j+s)+1\ra (j+1)n},\quad i=(j+1)_s;\\
f_1(j,4s)e_{(j+m+1)n+g(j)}\otimes w_{jn+g(j+s)+1\ra (j+1)n+g(j+s)},\quad i=(j+1)_s+(j_2-2)s;\\
-f_1(j,4s)w_{(j+m)n+g(j)+1\ra (j+m+1)n+g(j)}\otimes e_{jn+g(j+s)+1},\quad i=j;\\
-w_{(j+m)n+5\ra (j+m+1)n+g(j)}\otimes w_{jn+g(j+s)+1\ra jn+5},\quad i=(j)_s+5s;\\
0,\quad\text{otherwise.}\end{cases}$$

If $5s\le j<7s$, then $$b_{ij}=
\begin{cases}
-w_{(j+m+1)n\ra (j+m+1)n+g(j+s)+1}\otimes w_{jn+5\ra (j+1)n},\quad i=(j+1)_s;\\
-f_1(j,6s)w_{(j+m+1)n+g(j+s)\ra (j+m+1)n+g(j+s)+1}\otimes w_{jn+5\ra (j+1)n+g(j)},\\\quad\quad\quad i=(j+1)_s+(7-j_2)s;\\
-f_1(j,6s)e_{(j+m+1)n+g(j+s)+1}\otimes w_{jn+5\ra (j+1)n+g(j)+1},\quad i=(j+1)_s+(9-j_2)s;\\
0,\quad\text{otherwise.}\end{cases}$$

\medskip
$({\rm II})$ Represent an arbitrary $t_0\in\N$ in the form
$t_0=11\ell_0+r_0$, where $0\le r_0\le 10.$ Then
$\Omega^{t_0}(Y_t^{(22)})$ is a $\Omega^{r_0}(Y_t^{(22)})$, whose
left components twisted by $\sigma^{\ell_0}$, and coefficients
multiplied by $(-1)^{\ell_0}$.
\end{pr}
\begin{pr}[Translates for the case 23]
$({\rm I})$ Let $r_0\in\N$, $r_0<11$. $r_0$-translates of the
elements $Y^{(23)}_t$ are described by the following way.

$(1)$ If $r_0=0$, then $\Omega^{0}(Y_t^{(23)})$ is described with
$(6s\times 6s)$-matrix with one nonzero element that is of the
following form{\rm:}
$$b_{3s+z_0(0,\ell_0,n),3s+z_0(0,\ell_0,n)}=w_{(j+m)n+g(j+s)+1\ra (j+m+1)n+g(j+s)+1}\otimes e_{jn+g(j+s)+1}.$$

$(2)$ If $r_0=1$, then $\Omega^{1}(Y_t^{(23)})$ is described with
$(7s\times 7s)$-matrix with one nonzero element that is of the
following form{\rm:}
$$b_{2s+z_0(0,\ell_0,n),2s+z_0(0,\ell_0,n)}=w_{(j+m)n+g(j)+1\ra (j+m+1)n+g(j)+1}\otimes e_{jn+g(j)}.$$

$(3)$ If $r_0=2$, then $\Omega^{2}(Y_t^{(23)})$ is described with
$(6s\times 6s)$-matrix with one nonzero element that is of the
following form{\rm:}
$$b_{s+z_0(-1,\ell_0,n),s+z_0(-1,\ell_0,n)}=w_{(j+m)n+g(j+s)+1\ra (j+m+1)n+g(j+s)+1}\otimes e_{jn+g(j+s)}.$$

$(4)$ If $r_0=3$, then $\Omega^{3}(Y_t^{(23)})$ is described with
$(8s\times 8s)$-matrix with one nonzero element that is of the
following form{\rm:}
$$b_{z_0(-1,\ell_0,n),z_0(-1,\ell_0,n)}=w_{(j+m)n+g(j)+1\ra (j+m+1)n+g(j)+1}\otimes e_{jn}.$$

$(5)$ If $r_0=4$, then $\Omega^{4}(Y_t^{(23)})$ is described with
$(9s\times 9s)$-matrix with one nonzero element that is of the
following form{\rm:}
$$b_{6s+z_1(-2,\ell_0,n),6s+z_1(-2,\ell_0,n)}=w_{(j+m)n+g(j+s)+1\ra (j+m+1)n+g(j+s)+1}\otimes e_{jn+g(j)+1}.$$

$(6)$ If $r_0=5$, then $\Omega^{5}(Y_t^{(23)})$ is described with
$(8s\times 8s)$-matrix with one nonzero element that is of the
following form{\rm:}
$$b_{6s+z_0(-3,\ell_0,n),6s+z_0(-3,\ell_0,n)}=w_{(j+m+1)n+g(j)+1\ra (j+m+2)n+g(j)+1}\otimes e_{jn+5}.$$

$(7)$ If $r_0=6$, then $\Omega^{6}(Y_t^{(23)})$ is described with
$(9s\times 9s)$-matrix with the following elements $b_{ij}${\rm:}

If $0\le j<3s+z_1(-3,\ell_0,n)$, then $b_{ij}=0$.

If $j=3s+z_1(-3,\ell_0,n)$, then $$b_{ij}=
\begin{cases}
w_{(j+m)n+g(j)+1\ra (j+m+1)n+g(j)+1}\otimes e_{jn+g(j+s)},\quad i=j;\\
0,\quad\text{otherwise.}\end{cases}$$

If $3s+z_1(-3,\ell_0,n)<j<5s+z_0(-3,\ell_0,n)$, then $b_{ij}=0$.

If $j=5s+z_0(-3,\ell_0,n)$, then $$b_{ij}=
\begin{cases}
w_{(j+m)n+g(j+s)+1\ra (j+m+1)n+g(j+s)+1}\otimes e_{jn+g(j+s)+1},\quad i=j;\\
0,\quad\text{otherwise.}\end{cases}$$

If $5s+z_0(-3,\ell_0,n)<j<9s$, then $b_{ij}=0$.

$(8)$ If $r_0=7$, then $\Omega^{7}(Y_t^{(23)})$ is described with
$(8s\times 8s)$-matrix with one nonzero element that is of the
following form{\rm:}
$$b_{z_0(-3,\ell_0,n),z_0(-3,\ell_0,n)}=w_{(j+m)n+g(j)+1\ra (j+m+1)n+g(j)+1}\otimes e_{jn}.$$

$(9)$ If $r_0=8$, then $\Omega^{8}(Y_t^{(23)})$ is described with
$(6s\times 6s)$-matrix with one nonzero element that is of the
following form{\rm:}
$$b_{s+z_0(-4,\ell_0,n),s+z_0(-4,\ell_0,n)}=w_{(j+m)n+g(j+s)+1\ra (j+m+1)n+g(j+s)+1}\otimes e_{jn+g(j+s)}.$$

$(10)$ If $r_0=9$, then $\Omega^{9}(Y_t^{(23)})$ is described with
$(7s\times 7s)$-matrix with one nonzero element that is of the
following form{\rm:}
$$b_{5s+z_0(-5,\ell_0,n),5s+z_0(-5,\ell_0,n)}=w_{(j+m+1)n+g(j+s)+1\ra (j+m+2)n+g(j+s)+1}\otimes e_{jn+5}.$$

$(11)$ If $r_0=10$, then $\Omega^{10}(Y_t^{(23)})$ is described with
$(6s\times 6s)$-matrix with one nonzero element that is of the
following form{\rm:}
$$b_{3s+z_1(-5,\ell_0,n),3s+z_1(-5,\ell_0,n)}=w_{(j+m)n+g(j)+1\ra (j+m+1)n+g(j)+1}\otimes e_{jn+g(j+s)+1}.$$

\medskip
$({\rm II})$ Represent an arbitrary $t_0\in\N$ in the form
$t_0=11\ell_0+r_0$, where $0\le r_0\le 10.$ Then
$\Omega^{t_0}(Y_t^{(23)})$ is a $\Omega^{r_0}(Y_t^{(23)})$, whose
left components twisted by $\sigma^{\ell_0}$, and coefficients
multiplied by $(-1)^{\ell_0}$.
\end{pr}
\begin{pr}[Translates for the case 24]
$({\rm I})$ Let $r_0\in\N$, $r_0<11$. $r_0$-translates of the
elements $Y^{(24)}_t$ are described by the following way.

$(1)$ If $r_0=0$, then $\Omega^{0}(Y_t^{(24)})$ is described with
$(6s\times 6s)$-matrix with one nonzero element that is of the
following form{\rm:}
$$b_{0,0}=w_{(j+m)n\ra (j+m+1)n}\otimes e_{jn}.$$

$(2)$ If $r_0=1$, then $\Omega^{1}(Y_t^{(24)})$ is described with
$(7s\times 7s)$-matrix with one nonzero element that is of the
following form{\rm:}
$$b_{6s,6s}=w_{(j+m)n\ra (j+m+1)n}\otimes e_{jn+5}.$$

$(3)$ If $r_0=2$, then $\Omega^{2}(Y_t^{(24)})$ is described with
$(6s\times 6s)$-matrix with one nonzero element that is of the
following form{\rm:}
$$b_{5s,5s}=w_{(j+m)n\ra (j+m+1)n}\otimes e_{jn+5}.$$

$(4)$ If $r_0=3$, then $\Omega^{3}(Y_t^{(24)})$ is described with
$(8s\times 8s)$-matrix with the following elements $b_{ij}${\rm:}

If $0\le j<4s$, then $b_{ij}=0$.

If $j=4s$, then $$b_{ij}=
\begin{cases}
w_{(j+m)n\ra (j+m+1)n}\otimes e_{jn+g(j)+1},\quad i=j;\\
0,\quad\text{otherwise.}\end{cases}$$

If $4s<j<5s$, then $b_{ij}=0$.

If $j=5s$, then $$b_{ij}=
\begin{cases}
w_{(j+m)n\ra (j+m+1)n}\otimes e_{jn+g(j)+1},\quad i=j;\\
0,\quad\text{otherwise.}\end{cases}$$

If $5s<j<8s$, then $b_{ij}=0$.

$(5)$ If $r_0=4$, then $\Omega^{4}(Y_t^{(24)})$ is described with
$(9s\times 9s)$-matrix with one nonzero element that is of the
following form{\rm:}
$$b_{s,s}=w_{(j+m)n\ra (j+m+1)n}\otimes e_{jn}.$$

$(6)$ If $r_0=5$, then $\Omega^{5}(Y_t^{(24)})$ is described with
$(8s\times 8s)$-matrix with the following elements $b_{ij}${\rm:}

If $0\le j<2s$, then $b_{ij}=0$.

If $j=2s$, then $$b_{ij}=
\begin{cases}
w_{(j+m)n\ra (j+m+1)n}\otimes e_{jn+g(j)},\quad i=j;\\
0,\quad\text{otherwise.}\end{cases}$$

If $2s<j<3s$, then $b_{ij}=0$.

If $j=3s$, then $$b_{ij}=
\begin{cases}
w_{(j+m)n\ra (j+m+1)n}\otimes e_{jn+g(j)},\quad i=j;\\
0,\quad\text{otherwise.}\end{cases}$$

If $3s<j<8s$, then $b_{ij}=0$.

$(7)$ If $r_0=6$, then $\Omega^{6}(Y_t^{(24)})$ is described with
$(9s\times 9s)$-matrix with the following elements $b_{ij}${\rm:}

If $j=0$, then $$b_{ij}=
\begin{cases}
w_{(j+m)n\ra (j+m+1)n}\otimes e_{jn},\quad i=j;\\
0,\quad\text{otherwise.}\end{cases}$$

If $0<j<8s$, then $b_{ij}=0$.

If $j=8s$, then $$b_{ij}=
\begin{cases}
w_{(j+m)n\ra (j+m+1)n}\otimes e_{jn+5},\quad i=j;\\
0,\quad\text{otherwise.}\end{cases}$$

If $8s<j<9s$, then $b_{ij}=0$.

$(8)$ If $r_0=7$, then $\Omega^{7}(Y_t^{(24)})$ is described with
$(8s\times 8s)$-matrix with the following elements $b_{ij}${\rm:}

If $0\le j<4s$, then $b_{ij}=0$.

If $j=4s$, then $$b_{ij}=
\begin{cases}
w_{(j+m)n\ra (j+m+1)n}\otimes e_{jn+g(j)+1},\quad i=j;\\
0,\quad\text{otherwise.}\end{cases}$$

If $4s<j<5s$, then $b_{ij}=0$.

If $j=5s$, then $$b_{ij}=
\begin{cases}
w_{(j+m)n\ra (j+m+1)n}\otimes e_{jn+g(j)+1},\quad i=j;\\
0,\quad\text{otherwise.}\end{cases}$$

If $5s<j<8s$, then $b_{ij}=0$.

$(9)$ If $r_0=8$, then $\Omega^{8}(Y_t^{(24)})$ is described with
$(6s\times 6s)$-matrix with one nonzero element that is of the
following form{\rm:}
$$b_{5s,5s}=w_{(j+m)n\ra (j+m+1)n}\otimes e_{jn+5}.$$

$(10)$ If $r_0=9$, then $\Omega^{9}(Y_t^{(24)})$ is described with
$(7s\times 7s)$-matrix with the following elements $b_{ij}${\rm:}

If $0\le j<s$, then $b_{ij}=0$.

If $j=s$, then $$b_{ij}=
\begin{cases}
w_{(j+m)n\ra (j+m+1)n}\otimes e_{jn+g(j+s)},\quad i=j;\\
0,\quad\text{otherwise.}\end{cases}$$

If $s<j<2s$, then $b_{ij}=0$.

If $j=2s$, then $$b_{ij}=
\begin{cases}
w_{(j+m)n\ra (j+m+1)n}\otimes e_{jn+g(j+s)},\quad i=j;\\
0,\quad\text{otherwise.}\end{cases}$$

If $2s<j<7s$, then $b_{ij}=0$.

$(11)$ If $r_0=10$, then $\Omega^{10}(Y_t^{(24)})$ is described with
$(6s\times 6s)$-matrix with one nonzero element that is of the
following form{\rm:}
$$b_{0,0}=w_{(j+m)n\ra (j+m+1)n}\otimes e_{jn}.$$

\medskip
$({\rm II})$ Represent an arbitrary $t_0\in\N$ in the form
$t_0=11\ell_0+r_0$, where $0\le r_0\le 10.$ Then
$\Omega^{t_0}(Y_t^{(24)})$ is a $\Omega^{r_0}(Y_t^{(24)})$, whose
left components twisted by $\sigma^{\ell_0}$, and coefficients
multiplied by $(-1)^{\ell_0}$.
\end{pr}
\begin{pr}[Translates for the case 25]
$({\rm I})$ Let $r_0\in\N$, $r_0<11$. $r_0$-translates of the
elements $Y^{(25)}_t$ are described by the following way.

$(1)$ If $r_0=0$, then $\Omega^{0}(Y_t^{(25)})$ is described with
$(6s\times 6s)$-matrix with one nonzero element that is of the
following form{\rm:}
$$b_{s+z_0(0,\ell_0,n),s+z_0(0,\ell_0,n)}=w_{(j+m)n+g(j+s)\ra (j+m+1)n+g(j+s)}\otimes e_{jn+g(j+s)}.$$

$(2)$ If $r_0=1$, then $\Omega^{1}(Y_t^{(25)})$ is described with
$(7s\times 7s)$-matrix with one nonzero element that is of the
following form{\rm:}
$$b_{z_0(0,\ell_0,n),z_0(0,\ell_0,n)}=w_{(j+m)n+g(j)\ra (j+m+1)n+g(j)}\otimes e_{jn}.$$

$(3)$ If $r_0=2$, then $\Omega^{2}(Y_t^{(25)})$ is described with
$(6s\times 6s)$-matrix with one nonzero element that is of the
following form{\rm:}
$$b_{3s+z_1(0,\ell_0,n),3s+z_1(0,\ell_0,n)}=w_{(j+m)n+g(j)\ra (j+m+1)n+g(j)}\otimes e_{jn+g(j+s)+1}.$$

$(4)$ If $r_0=3$, then $\Omega^{3}(Y_t^{(25)})$ is described with
$(8s\times 8s)$-matrix with one nonzero element that is of the
following form{\rm:}
$$b_{6s+z_0(0,\ell_0,n),6s+z_0(0,\ell_0,n)}=w_{(j+m)n+g(j)\ra (j+m+1)n+g(j)}\otimes e_{jn+5}.$$

$(5)$ If $r_0=4$, then $\Omega^{4}(Y_t^{(25)})$ is described with
$(9s\times 9s)$-matrix with the following elements $b_{ij}${\rm:}

If $0\le j<2s+z_1(0,\ell_0,n)$, then $b_{ij}=0$.

If $j=2s+z_1(0,\ell_0,n)$, then $$b_{ij}=
\begin{cases}
w_{(j+m)n+g(j+s)\ra (j+m+1)n+g(j+s)}\otimes e_{jn+g(j)},\quad i=j;\\
0,\quad\text{otherwise.}\end{cases}$$

If $2s+z_1(0,\ell_0,n)<j<4s+z_0(0,\ell_0,n)$, then $b_{ij}=0$.

If $j=4s+z_0(0,\ell_0,n)$, then $$b_{ij}=
\begin{cases}
w_{(j+m)n+g(j)\ra (j+m+1)n+g(j)}\otimes e_{jn+g(j)+1},\quad i=j;\\
0,\quad\text{otherwise.}\end{cases}$$

If $4s+z_0(0,\ell_0,n)<j<9s$, then $b_{ij}=0$.

$(6)$ If $r_0=5$, then $\Omega^{5}(Y_t^{(25)})$ is described with
$(8s\times 8s)$-matrix with one nonzero element that is of the
following form{\rm:}
$$b_{z_0(0,\ell_0,n),z_0(0,\ell_0,n)}=w_{(j+m)n+g(j)\ra (j+m+1)n+g(j)}\otimes e_{jn}.$$

$(7)$ If $r_0=6$, then $\Omega^{6}(Y_t^{(25)})$ is described with
$(9s\times 9s)$-matrix with one nonzero element that is of the
following form{\rm:}
$$b_{s+z_0(0,\ell_0,n),s+z_0(0,\ell_0,n)}=w_{(j+m)n+g(j+s)\ra (j+m+1)n+g(j+s)}\otimes e_{jn+g(j+s)}.$$

$(8)$ If $r_0=7$, then $\Omega^{7}(Y_t^{(25)})$ is described with
$(8s\times 8s)$-matrix with one nonzero element that is of the
following form{\rm:}
$$b_{6s+z_0(0,\ell_0,n),6s+z_0(0,\ell_0,n)}=w_{(j+m)n+g(j)\ra (j+m+1)n+g(j)}\otimes e_{jn+5}.$$

$(9)$ If $r_0=8$, then $\Omega^{8}(Y_t^{(25)})$ is described with
$(6s\times 6s)$-matrix with one nonzero element that is of the
following form{\rm:}
$$b_{3s+z_1(0,\ell_0,n),3s+z_1(0,\ell_0,n)}=w_{(j+m)n+g(j)\ra (j+m+1)n+g(j)}\otimes e_{jn+g(j+s)+1}.$$

$(10)$ If $r_0=9$, then $\Omega^{9}(Y_t^{(25)})$ is described with
$(7s\times 7s)$-matrix with one nonzero element that is of the
following form{\rm:}
$$b_{3s+z_1(0,\ell_0,n),3s+z_1(0,\ell_0,n)}=w_{(j+m)n+g(j)\ra (j+m+1)n+g(j)}\otimes e_{jn+g(j+s)+1}.$$

$(11)$ If $r_0=10$, then $\Omega^{10}(Y_t^{(25)})$ is described with
$(6s\times 6s)$-matrix with one nonzero element that is of the
following form{\rm:}
$$b_{s+z_1(0,\ell_0,n),s+z_1(0,\ell_0,n)}=w_{(j+m)n+g(j)\ra (j+m+1)n+g(j)}\otimes e_{jn+g(j+s)}.$$

\medskip
$({\rm II})$ Represent an arbitrary $t_0\in\N$ in the form
$t_0=11\ell_0+r_0$, where $0\le r_0\le 10.$ Then
$\Omega^{t_0}(Y_t^{(25)})$ is a $\Omega^{r_0}(Y_t^{(25)})$, whose
left components twisted by $\sigma^{\ell_0}$, and coefficients
multiplied by $(-1)^{\ell_0}$.
\end{pr}
\begin{pr}[Translates for the case 26]
$({\rm I})$ Let $r_0\in\N$, $r_0<11$. $r_0$-translates of the
elements $Y^{(26)}_t$ are described by the following way.

$(1)$ If $r_0=0$, then $\Omega^{0}(Y_t^{(26)})$ is described with
$(6s\times 6s)$-matrix with one nonzero element that is of the
following form{\rm:}
$$b_{s+z_1(0,\ell_0,n),s+z_1(0,\ell_0,n)}=-w_{(j+m)n+g(j+s)\ra (j+m+1)n+g(j+s)}\otimes e_{jn+g(j+s)}.$$

$(2)$ If $r_0=1$, then $\Omega^{1}(Y_t^{(26)})$ is described with
$(7s\times 7s)$-matrix with one nonzero element that is of the
following form{\rm:}
$$b_{z_1(0,\ell_0,n),z_1(0,\ell_0,n)}=-w_{(j+m)n+g(j)\ra (j+m+1)n+g(j)}\otimes e_{jn}.$$

$(3)$ If $r_0=2$, then $\Omega^{2}(Y_t^{(26)})$ is described with
$(6s\times 6s)$-matrix with one nonzero element that is of the
following form{\rm:}
$$b_{3s+z_0(0,\ell_0,n),3s+z_0(0,\ell_0,n)}=-w_{(j+m)n+g(j)\ra (j+m+1)n+g(j)}\otimes e_{jn+g(j+s)+1}.$$

$(4)$ If $r_0=3$, then $\Omega^{3}(Y_t^{(26)})$ is described with
$(8s\times 8s)$-matrix with one nonzero element that is of the
following form{\rm:}
$$b_{6s+z_1(0,\ell_0,n),6s+z_1(0,\ell_0,n)}=-w_{(j+m)n+g(j)\ra (j+m+1)n+g(j)}\otimes e_{jn+5}.$$

$(5)$ If $r_0=4$, then $\Omega^{4}(Y_t^{(26)})$ is described with
$(9s\times 9s)$-matrix with the following elements $b_{ij}${\rm:}

If $0\le j<2s+z_0(0,\ell_0,n)$, then $b_{ij}=0$.

If $j=2s+z_0(0,\ell_0,n)$, then $$b_{ij}=
\begin{cases}
-w_{(j+m)n+g(j+s)\ra (j+m+1)n+g(j+s)}\otimes e_{jn+g(j)},\quad i=j;\\
0,\quad\text{otherwise.}\end{cases}$$

If $2s+z_0(0,\ell_0,n)<j<4s+z_1(0,\ell_0,n)$, then $b_{ij}=0$.

If $j=4s+z_1(0,\ell_0,n)$, then $$b_{ij}=
\begin{cases}
-w_{(j+m)n+g(j)\ra (j+m+1)n+g(j)}\otimes e_{jn+g(j)+1},\quad i=j;\\
0,\quad\text{otherwise.}\end{cases}$$

If $4s+z_1(0,\ell_0,n)<j<9s$, then $b_{ij}=0$.

$(6)$ If $r_0=5$, then $\Omega^{5}(Y_t^{(26)})$ is described with
$(8s\times 8s)$-matrix with one nonzero element that is of the
following form{\rm:}
$$b_{z_1(0,\ell_0,n),z_1(0,\ell_0,n)}=-w_{(j+m)n+g(j)\ra (j+m+1)n+g(j)}\otimes e_{jn}.$$

$(7)$ If $r_0=6$, then $\Omega^{6}(Y_t^{(26)})$ is described with
$(9s\times 9s)$-matrix with one nonzero element that is of the
following form{\rm:}
$$b_{s+z_1(0,\ell_0,n),s+z_1(0,\ell_0,n)}=-w_{(j+m)n+g(j+s)\ra (j+m+1)n+g(j+s)}\otimes e_{jn+g(j+s)}.$$

$(8)$ If $r_0=7$, then $\Omega^{7}(Y_t^{(26)})$ is described with
$(8s\times 8s)$-matrix with one nonzero element that is of the
following form{\rm:}
$$b_{6s+z_1(0,\ell_0,n),6s+z_1(0,\ell_0,n)}=-w_{(j+m)n+g(j)\ra (j+m+1)n+g(j)}\otimes e_{jn+5}.$$

$(9)$ If $r_0=8$, then $\Omega^{8}(Y_t^{(26)})$ is described with
$(6s\times 6s)$-matrix with one nonzero element that is of the
following form{\rm:}
$$b_{3s+z_0(0,\ell_0,n),3s+z_0(0,\ell_0,n)}=-w_{(j+m)n+g(j)\ra (j+m+1)n+g(j)}\otimes e_{jn+g(j+s)+1}.$$

$(10)$ If $r_0=9$, then $\Omega^{9}(Y_t^{(26)})$ is described with
$(7s\times 7s)$-matrix with one nonzero element that is of the
following form{\rm:}
$$b_{3s+z_0(0,\ell_0,n),3s+z_0(0,\ell_0,n)}=-w_{(j+m)n+g(j)\ra (j+m+1)n+g(j)}\otimes e_{jn+g(j+s)+1}.$$

$(11)$ If $r_0=10$, then $\Omega^{10}(Y_t^{(26)})$ is described with
$(6s\times 6s)$-matrix with one nonzero element that is of the
following form{\rm:}
$$b_{s+z_0(0,\ell_0,n),s+z_0(0,\ell_0,n)}=-w_{(j+m)n+g(j)\ra (j+m+1)n+g(j)}\otimes e_{jn+g(j+s)}.$$

\medskip
$({\rm II})$ Represent an arbitrary $t_0\in\N$ in the form
$t_0=11\ell_0+r_0$, where $0\le r_0\le 10.$ Then
$\Omega^{t_0}(Y_t^{(26)})$ is a $\Omega^{r_0}(Y_t^{(26)})$, whose
left components twisted by $\sigma^{\ell_0}$, and coefficients
multiplied by $(-1)^{\ell_0}$.
\end{pr}
\begin{pr}[Translates for the case 27]
$({\rm I})$ Let $r_0\in\N$, $r_0<11$. $r_0$-translates of the
elements $Y^{(27)}_t$ are described by the following way.

$(1)$ If $r_0=0$, then $\Omega^{0}(Y_t^{(27)})$ is described with
$(6s\times 6s)$-matrix with one nonzero element that is of the
following form{\rm:}
$$b_{3s+z_1(0,\ell_0,n),3s+z_1(0,\ell_0,n)}=-w_{(j+m)n+g(j+s)+1\ra (j+m+1)n+g(j+s)+1}\otimes e_{jn+g(j+s)+1}.$$

$(2)$ If $r_0=1$, then $\Omega^{1}(Y_t^{(27)})$ is described with
$(7s\times 7s)$-matrix with one nonzero element that is of the
following form{\rm:}
$$b_{2s+z_1(0,\ell_0,n),2s+z_1(0,\ell_0,n)}=-w_{(j+m)n+g(j)+1\ra (j+m+1)n+g(j)+1}\otimes e_{jn+g(j)}.$$

$(3)$ If $r_0=2$, then $\Omega^{2}(Y_t^{(27)})$ is described with
$(6s\times 6s)$-matrix with one nonzero element that is of the
following form{\rm:}
$$b_{s+z_1(0,\ell_0,n),s+z_1(0,\ell_0,n)}=-w_{(j+m)n+g(j+s)+1\ra (j+m+1)n+g(j+s)+1}\otimes e_{jn+g(j+s)}.$$

$(4)$ If $r_0=3$, then $\Omega^{3}(Y_t^{(27)})$ is described with
$(8s\times 8s)$-matrix with one nonzero element that is of the
following form{\rm:}
$$b_{z_1(0,\ell_0,n),z_1(0,\ell_0,n)}=-w_{(j+m)n+g(j)+1\ra (j+m+1)n+g(j)+1}\otimes e_{jn}.$$

$(5)$ If $r_0=4$, then $\Omega^{4}(Y_t^{(27)})$ is described with
$(9s\times 9s)$-matrix with one nonzero element that is of the
following form{\rm:}
$$b_{6s+z_0(0,\ell_0,n),6s+z_0(0,\ell_0,n)}=-w_{(j+m)n+g(j+s)+1\ra (j+m+1)n+g(j+s)+1}\otimes e_{jn+g(j)+1}.$$

$(6)$ If $r_0=5$, then $\Omega^{5}(Y_t^{(27)})$ is described with
$(8s\times 8s)$-matrix with one nonzero element that is of the
following form{\rm:}
$$b_{6s+z_1(0,\ell_0,n),6s+z_1(0,\ell_0,n)}=-w_{(j+m)n+g(j)+1\ra (j+m+1)n+g(j)+1}\otimes e_{jn+5}.$$

$(7)$ If $r_0=6$, then $\Omega^{6}(Y_t^{(27)})$ is described with
$(9s\times 9s)$-matrix with the following elements $b_{ij}${\rm:}

If $0\le j<3s+z_0(0,\ell_0,n)$, then $b_{ij}=0$.

If $j=3s+z_0(0,\ell_0,n)$, then $$b_{ij}=
\begin{cases}
-w_{(j+m)n+g(j)+1\ra (j+m+1)n+g(j)+1}\otimes e_{jn+g(j+s)},\quad i=j;\\
0,\quad\text{otherwise.}\end{cases}$$

If $3s+z_0(0,\ell_0,n)<j<5s+z_1(0,\ell_0,n)$, then $b_{ij}=0$.

If $j=5s+z_1(0,\ell_0,n)$, then $$b_{ij}=
\begin{cases}
-w_{(j+m)n+g(j+s)+1\ra (j+m+1)n+g(j+s)+1}\otimes e_{jn+g(j+s)+1},\quad i=j;\\
0,\quad\text{otherwise.}\end{cases}$$

If $5s+z_1(0,\ell_0,n)<j<9s$, then $b_{ij}=0$.

$(8)$ If $r_0=7$, then $\Omega^{7}(Y_t^{(27)})$ is described with
$(8s\times 8s)$-matrix with one nonzero element that is of the
following form{\rm:}
$$b_{z_1(0,\ell_0,n),z_1(0,\ell_0,n)}=-w_{(j+m)n+g(j)+1\ra (j+m+1)n+g(j)+1}\otimes e_{jn}.$$

$(9)$ If $r_0=8$, then $\Omega^{8}(Y_t^{(27)})$ is described with
$(6s\times 6s)$-matrix with one nonzero element that is of the
following form{\rm:}
$$b_{s+z_1(0,\ell_0,n),s+z_1(0,\ell_0,n)}=-w_{(j+m)n+g(j+s)+1\ra (j+m+1)n+g(j+s)+1}\otimes e_{jn+g(j+s)}.$$

$(10)$ If $r_0=9$, then $\Omega^{9}(Y_t^{(27)})$ is described with
$(7s\times 7s)$-matrix with one nonzero element that is of the
following form{\rm:}
$$b_{5s+z_1(0,\ell_0,n),5s+z_1(0,\ell_0,n)}=-w_{(j+m)n+g(j+s)+1\ra (j+m+1)n+g(j+s)+1}\otimes e_{jn+5}.$$

$(11)$ If $r_0=10$, then $\Omega^{10}(Y_t^{(27)})$ is described with
$(6s\times 6s)$-matrix with one nonzero element that is of the
following form{\rm:}
$$b_{3s+z_0(0,\ell_0,n),3s+z_0(0,\ell_0,n)}=-w_{(j+m)n+g(j)+1\ra (j+m+1)n+g(j)+1}\otimes e_{jn+g(j+s)+1}.$$

\medskip
$({\rm II})$ Represent an arbitrary $t_0\in\N$ in the form
$t_0=11\ell_0+r_0$, where $0\le r_0\le 10.$ Then
$\Omega^{t_0}(Y_t^{(27)})$ is a $\Omega^{r_0}(Y_t^{(27)})$, whose
left components twisted by $\sigma^{\ell_0}$, and coefficients
multiplied by $(-1)^{\ell_0}$.
\end{pr}
\begin{pr}[Translates for the case 28]
$({\rm I})$ Let $r_0\in\N$, $r_0<11$. $r_0$-translates of the
elements $Y^{(28)}_t$ are described by the following way.

$(1)$ If $r_0=0$, then $\Omega^{0}(Y_t^{(28)})$ is described with
$(6s\times 6s)$-matrix with one nonzero element that is of the
following form{\rm:}
$$b_{5s,5s}=w_{(j+m)n+5\ra (j+m+1)n+5}\otimes e_{jn+5}.$$

$(2)$ If $r_0=1$, then $\Omega^{1}(Y_t^{(28)})$ is described with
$(7s\times 7s)$-matrix with the following elements $b_{ij}${\rm:}

If $0\le j<4s$, then $b_{ij}=0$.

If $j=4s$, then $$b_{ij}=
\begin{cases}
w_{(j+m)n+5\ra (j+m+1)n+5}\otimes e_{jn+g(j)+1},\quad i=j;\\
0,\quad\text{otherwise.}\end{cases}$$

If $4s<j<5s$, then $b_{ij}=0$.

If $j=5s$, then $$b_{ij}=
\begin{cases}
w_{(j+m)n+5\ra (j+m+1)n+5}\otimes e_{jn+g(j)+1},\quad i=j;\\
0,\quad\text{otherwise.}\end{cases}$$

If $5s<j<7s$, then $b_{ij}=0$.

$(3)$ If $r_0=2$, then $\Omega^{2}(Y_t^{(28)})$ is described with
$(6s\times 6s)$-matrix with one nonzero element that is of the
following form{\rm:}
$$b_{0,0}=w_{(j+m)n+5\ra (j+m+1)n+5}\otimes e_{jn}.$$

$(4)$ If $r_0=3$, then $\Omega^{3}(Y_t^{(28)})$ is described with
$(8s\times 8s)$-matrix with the following elements $b_{ij}${\rm:}

If $0\le j<2s$, then $b_{ij}=0$.

If $j=2s$, then $$b_{ij}=
\begin{cases}
w_{(j+m)n+5\ra (j+m+1)n+5}\otimes e_{jn+g(j)},\quad i=j;\\
0,\quad\text{otherwise.}\end{cases}$$

If $2s<j<3s$, then $b_{ij}=0$.

If $j=3s$, then $$b_{ij}=
\begin{cases}
w_{(j+m)n+5\ra (j+m+1)n+5}\otimes e_{jn+g(j)},\quad i=j;\\
0,\quad\text{otherwise.}\end{cases}$$

If $3s<j<8s$, then $b_{ij}=0$.

$(5)$ If $r_0=4$, then $\Omega^{4}(Y_t^{(28)})$ is described with
$(9s\times 9s)$-matrix with the following elements $b_{ij}${\rm:}

If $j=0$, then $$b_{ij}=
\begin{cases}
w_{(j+m)n+5\ra (j+m+1)n+5}\otimes e_{jn},\quad i=j;\\
0,\quad\text{otherwise.}\end{cases}$$

If $0<j<8s$, then $b_{ij}=0$.

If $j=8s$, then $$b_{ij}=
\begin{cases}
w_{(j+m)n+5\ra (j+m+1)n+5}\otimes e_{jn+5},\quad i=j;\\
0,\quad\text{otherwise.}\end{cases}$$

If $8s<j<9s$, then $b_{ij}=0$.

$(6)$ If $r_0=5$, then $\Omega^{5}(Y_t^{(28)})$ is described with
$(8s\times 8s)$-matrix with the following elements $b_{ij}${\rm:}

If $0\le j<4s$, then $b_{ij}=0$.

If $j=4s$, then $$b_{ij}=
\begin{cases}
w_{(j+m)n+5\ra (j+m+1)n+5}\otimes e_{jn+g(j)+1},\quad i=j;\\
0,\quad\text{otherwise.}\end{cases}$$

If $4s<j<5s$, then $b_{ij}=0$.

If $j=5s$, then $$b_{ij}=
\begin{cases}
w_{(j+m)n+5\ra (j+m+1)n+5}\otimes e_{jn+g(j)+1},\quad i=j;\\
0,\quad\text{otherwise.}\end{cases}$$

If $5s<j<8s$, then $b_{ij}=0$.

$(7)$ If $r_0=6$, then $\Omega^{6}(Y_t^{(28)})$ is described with
$(9s\times 9s)$-matrix with one nonzero element that is of the
following form{\rm:}
$$b_{7s,7s}=w_{(j+m)n+5\ra (j+m+1)n+5}\otimes e_{jn+5}.$$

$(8)$ If $r_0=7$, then $\Omega^{7}(Y_t^{(28)})$ is described with
$(8s\times 8s)$-matrix with the following elements $b_{ij}${\rm:}

If $0\le j<2s$, then $b_{ij}=0$.

If $j=2s$, then $$b_{ij}=
\begin{cases}
w_{(j+m)n+5\ra (j+m+1)n+5}\otimes e_{jn+g(j)},\quad i=j;\\
0,\quad\text{otherwise.}\end{cases}$$

If $2s<j<3s$, then $b_{ij}=0$.

If $j=3s$, then $$b_{ij}=
\begin{cases}
w_{(j+m)n+5\ra (j+m+1)n+5}\otimes e_{jn+g(j)},\quad i=j;\\
0,\quad\text{otherwise.}\end{cases}$$

If $3s<j<8s$, then $b_{ij}=0$.

$(9)$ If $r_0=8$, then $\Omega^{8}(Y_t^{(28)})$ is described with
$(6s\times 6s)$-matrix with one nonzero element that is of the
following form{\rm:}
$$b_{0,0}=w_{(j+m)n+5\ra (j+m+1)n+5}\otimes e_{jn}.$$

$(10)$ If $r_0=9$, then $\Omega^{9}(Y_t^{(28)})$ is described with
$(7s\times 7s)$-matrix with one nonzero element that is of the
following form{\rm:}
$$b_{0,0}=w_{(j+m)n+5\ra (j+m+1)n+5}\otimes e_{jn}.$$

$(11)$ If $r_0=10$, then $\Omega^{10}(Y_t^{(28)})$ is described with
$(6s\times 6s)$-matrix with one nonzero element that is of the
following form{\rm:}
$$b_{5s,5s}=w_{(j+m)n+5\ra (j+m+1)n+5}\otimes e_{jn+5}.$$

\medskip
$({\rm II})$ Represent an arbitrary $t_0\in\N$ in the form
$t_0=11\ell_0+r_0$, where $0\le r_0\le 10.$ Then
$\Omega^{t_0}(Y_t^{(28)})$ is a $\Omega^{r_0}(Y_t^{(28)})$, whose
left components twisted by $\sigma^{\ell_0}$, and coefficients
multiplied by $(-1)^{\ell_0}$.
\end{pr}

\begin{proof}
We prove the proposition \ref{shifts_3}, the other propositions are
proved similarly. Show that the following squares are commutative:
$$\begin{CD} Q_{t+t_0} @>{d_{t+t_0-1}}>>Q_{t+t_0-1} \\ @V{f_{t_0}}VV @VV{f_{t_0-1}}V\\ Q_{t_0} @>{d_{t_0-1}}>> Q_{t_0-1}.\end{CD}$$
We shall describe the matrixes of products $d_{t_0-1}f_{t_0}$ and
see that they coincide with $f_{t_0-1}d_{t+t_0-1}$.

$(1)$ If $t_0<11$ and $r_0=0$, then the product $d_{r_0-1}f_{r_0}$
is an $(7s\times 6s)$-matrix $C=(c_{ij})$ with the following
elements $c_{ij}$:

If $0\le j<6s+(-3\ell_0)_s$, then $c_{ij}=0$.

If $j=6s+(-3\ell_0)_s$, then $$c_{ij}=
\begin{cases}
-\kappa w_{(j+m)n\ra (j+m+1)n}\otimes w_{jn+5\ra (j+1)n},\quad i=(j+1)_s;\\
-\kappa w_{(j+m)n+g(j)\ra (j+m+1)n}\otimes w_{jn+5\ra (j+1)n+g(j)},\quad i=(j+1)_s+s;\\
-\kappa w_{(j+m)n+g(j+s)\ra (j+m+1)n}\otimes w_{jn+5\ra (j+1)n+g(j+s)},\quad i=(j+1)_s+2s;\\
-\kappa w_{(j+m)n+g(j)+1\ra (j+m+1)n}\otimes w_{jn+5\ra (j+1)n+g(j)+1},\quad i=(j+1)_s+3s;\\
-\kappa w_{(j+m)n+g(j+s)+1\ra (j+m+1)n}\otimes w_{jn+5\ra (j+1)n+g(j+s)+1},\quad i=(j+1)_s+4s;\\
-\kappa w_{(j+m)n+5\ra (j+m+1)n}\otimes w_{jn+5\ra (j+1)n+5},\quad i=(j+1)_s+5s;\\
0,\quad\text{otherwise.}\end{cases}$$

If $6s+(-3\ell_0)_s<j<7s$, then $c_{ij}=0$.

$(2)$ If $t_0<11$ and $r_0=1$, then the product $d_{r_0-1}f_{r_0}$
is an $(6s\times 6s)$-matrix $C=(c_{ij})$ with the following
elements $c_{ij}$:

If $0\le j<s+(-3\ell_0)_s$, then $c_{ij}=0$.

If $j=s+(-3\ell_0)_s$, then $$c_{ij}=
\begin{cases}
w_{(j+m)n+5\ra (j+m+1)n+g(j+s)+1}\otimes w_{jn+g(j+s)\ra jn+5},\quad i=j+4s;\\
0,\quad\text{otherwise.}\end{cases}$$

If $s+(-3\ell_0)_s<j<2s+(-3\ell_0)_s$, then $c_{ij}=0$.

If $j=2s+(-3\ell_0)_s$, then $$c_{ij}=
\begin{cases}
w_{(j+m)n+5\ra (j+m+1)n+g(j+s)+1}\otimes w_{jn+g(j+s)\ra jn+5},\quad i=j+3s;\\
0,\quad\text{otherwise.}\end{cases}$$

If $2s+(-3\ell_0)_s<j<3s+(-3\ell_0)_s$, then $c_{ij}=0$.

If $j=3s+(-3\ell_0)_s$, then $$c_{ij}=
\begin{cases}
w_{(j+m)n+5\ra (j+m+1)n+g(j)}\otimes w_{jn+g(j+s)+1\ra jn+5},\quad i=j+2s;\\
0,\quad\text{otherwise.}\end{cases}$$

If $3s+(-3\ell_0)_s<j<4s+(-3\ell_0)_s$, then $c_{ij}=0$.

If $j=4s+(-3\ell_0)_s$, then $$c_{ij}=
\begin{cases}
w_{(j+m)n+5\ra (j+m+1)n+g(j)}\otimes w_{jn+g(j+s)+1\ra jn+5},\quad i=j+s;\\
0,\quad\text{otherwise.}\end{cases}$$

If $4s+(-3\ell_0)_s<j<5s+(-1-3\ell_0)_s$, then $c_{ij}=0$.

If $j=5s+(-1-3\ell_0)_s$, then $$c_{ij}=
\begin{cases}
\kappa w_{(j+m+1)n+5\ra (j+m+2)n}\otimes w_{jn+5\ra (j+1)n+5},\quad i=(j+1)_s+5s;\\
0,\quad\text{otherwise.}\end{cases}$$

If $5s+(-1-3\ell_0)_s<j<6s$, then $c_{ij}=0$.

$(3)$ If $t_0<11$ and $r_0=2$, then the product $d_{r_0-1}f_{r_0}$
is an $(8s\times 7s)$-matrix $C=(c_{ij})$ with the following
elements $c_{ij}$:

If $0\le j<(-3\ell_0)_s$, then $c_{ij}=0$.

If $j=(-3\ell_0)_s$, then $$c_{ij}=
\begin{cases}
w_{(j+m-1)n+g(j)+1\ra (j+m)n+g(j)+1}\otimes w_{jn\ra jn+g(j)},\quad i=j+2s;\\
w_{(j+m-1)n+5\ra (j+m)n+g(j)+1}\otimes w_{jn\ra jn+g(j)+1},\quad i=j+4s;\\
-w_{(j+m-1)n+5\ra (j+m)n+g(j)+1}\otimes w_{jn\ra jn+g(j+s)+1},\quad i=j+5s;\\
0,\quad\text{otherwise.}\end{cases}$$

If $(-3\ell_0)_s<j<s+(-3\ell_0)_s$, then $c_{ij}=0$.

If $j=s+(-3\ell_0)_s$, then $$c_{ij}=
\begin{cases}
-w_{(j+m-1)n+g(j)+1\ra (j+m)n+g(j)+1}\otimes w_{jn\ra jn+g(j)},\quad i=j+2s;\\
w_{(j+m-1)n+5\ra (j+m)n+g(j)+1}\otimes w_{jn\ra jn+g(j+s)+1},\quad i=j+3s;\\
-w_{(j+m-1)n+5\ra (j+m)n+g(j)+1}\otimes w_{jn\ra jn+g(j)+1},\quad i=j+4s;\\
0,\quad\text{otherwise.}\end{cases}$$

If $s+(-3\ell_0)_s<j<6s+(-1-3\ell_0)_s$, then $c_{ij}=0$.

If $j=6s+(-1-3\ell_0)_s$, then $$c_{ij}=
\begin{cases}
w_{(j+m)n+5\ra (j+m+1)n+g(j)}\otimes w_{jn+5\ra (j+1)n+g(j+s)+1},\quad i=(j+1)_s+5s;\\
w_{(j+m+1)n\ra (j+m+1)n+g(j)}\otimes w_{jn+5\ra (j+1)n+5},\quad i=(j+1)_s+6s;\\
0,\quad\text{otherwise.}\end{cases}$$

If $6s+(-1-3\ell_0)_s<j<7s+(-1-3\ell_0)_s$, then $c_{ij}=0$.

If $j=7s+(-1-3\ell_0)_s$, then $$c_{ij}=
\begin{cases}
w_{(j+m)n+5\ra (j+m+1)n+g(j)}\otimes w_{jn+5\ra (j+1)n+g(j+s)+1},\quad i=(j+1)_s+4s;\\
w_{(j+m+1)n\ra (j+m+1)n+g(j)}\otimes w_{jn+5\ra (j+1)n+5},\quad i=(j+1)_s+6s;\\
0,\quad\text{otherwise.}\end{cases}$$

If $7s+(-1-3\ell_0)_s<j<8s$, then $c_{ij}=0$.

$(4)$ If $t_0<11$ and $r_0=3$, then the product $d_{r_0-1}f_{r_0}$
is an $(9s\times 6s)$-matrix $C=(c_{ij})$ with the following
elements $c_{ij}$:

If $0\le j<6s+(-1-3\ell_0)_s$, then $c_{ij}=0$.

If $j=6s+(-1-3\ell_0)_s$, then $$c_{ij}=
\begin{cases}
-w_{(j+m)n+5\ra (j+m+1)n+g(j+s)+1}\otimes w_{jn+g(j)+1\ra (j+1)n},\quad i=(j+1)_s;\\
w_{(j+m+1)n+g(j+s)\ra (j+m+1)n+g(j+s)+1}\otimes w_{jn+g(j)+1\ra (j+1)n+g(j)+1},\quad i=(j+1)_s+3s;\\
0,\quad\text{otherwise.}\end{cases}$$

If $6s+(-1-3\ell_0)_s<j<7s+(-1-3\ell_0)_s$, then $c_{ij}=0$.

If $j=7s+(-1-3\ell_0)_s$, then $$c_{ij}=
\begin{cases}
w_{(j+m)n+5\ra (j+m+1)n+g(j+s)+1}\otimes w_{jn+g(j)+1\ra (j+1)n},\quad i=(j+1)_s;\\
w_{(j+m+1)n+g(j+s)\ra (j+m+1)n+g(j+s)+1}\otimes w_{jn+g(j)+1\ra (j+1)n+g(j)+1},\quad i=(j+1)_s+4s;\\
0,\quad\text{otherwise.}\end{cases}$$

If $7s+(-1-3\ell_0)_s<j<8s+(-1-3\ell_0)_s$, then $c_{ij}=0$.

If $j=8s+(-1-3\ell_0)_s$, then $$c_{ij}=
\begin{cases}
\kappa w_{(j+m)n+5\ra (j+m+1)n+5}\otimes w_{jn+5\ra (j+1)n},\quad i=(j+1)_s;\\
-\kappa w_{(j+m+1)n+g(j+s)\ra (j+m+1)n+5}\otimes w_{jn+5\ra (j+1)n+g(j)+1},\quad i=(j+1)_s+3s;\\
\kappa w_{(j+m+1)n+g(j)\ra (j+m+1)n+5}\otimes w_{jn+5\ra (j+1)n+g(j+s)+1},\quad i=(j+1)_s+4s;\\
0,\quad\text{otherwise.}\end{cases}$$

If $8s+(-1-3\ell_0)_s<j<9s$, then $c_{ij}=0$.

$(5)$ If $t_0<11$ and $r_0=4$, then the product $d_{r_0-1}f_{r_0}$
is an $(8s\times 8s)$-matrix $C=(c_{ij})$ with the following
elements $c_{ij}$:

If $0\le j<2s+(-1-3\ell_0)_s$, then $c_{ij}=0$.

If $j=2s+(-1-3\ell_0)_s$, then $$c_{ij}=
\begin{cases}
\kappa w_{(j+m)n+g(j+s)+1\ra (j+m+1)n}\otimes w_{jn+g(j)\ra (j+1)n},\quad i=(j+1)_s+s;\\
0,\quad\text{otherwise.}\end{cases}$$

If $2s+(-1-3\ell_0)_s<j<3s+(-1-3\ell_0)_s$, then $c_{ij}=0$.

If $j=3s+(-1-3\ell_0)_s$, then $$c_{ij}=
\begin{cases}
-\kappa w_{(j+m)n+g(j+s)+1\ra (j+m+1)n}\otimes w_{jn+g(j)\ra (j+1)n},\quad i=(j+1)_s;\\
0,\quad\text{otherwise.}\end{cases}$$

If $3s+(-1-3\ell_0)_s<j<4s+(-1-3\ell_0)_s$, then $c_{ij}=0$.

If $j=4s+(-1-3\ell_0)_s$, then $$c_{ij}=
\begin{cases}
\kappa w_{(j+m)n+g(j+s)+1\ra (j+m)n+5}\otimes w_{jn+g(j)+1\ra (j+1)n},\quad i=(j+1)_s+s;\\
\kappa w_{(j+m)n+g(j)\ra (j+m)n+5}\otimes w_{jn+g(j)+1\ra jn+5},\quad i=j+2s;\\
-\kappa w_{(j+m)n+g(j+s)\ra (j+m)n+5}\otimes w_{jn+g(j)+1\ra jn+5},\quad i=j+3s;\\
0,\quad\text{otherwise.}\end{cases}$$

If $4s+(-1-3\ell_0)_s<j<5s+(-1-3\ell_0)_s$, then $c_{ij}=0$.

If $j=5s+(-1-3\ell_0)_s$, then $$c_{ij}=
\begin{cases}
-\kappa w_{(j+m)n+g(j+s)+1\ra (j+m)n+5}\otimes w_{jn+g(j)+1\ra (j+1)n},\quad i=(j+1)_s;\\
-\kappa w_{(j+m)n+g(j+s)\ra (j+m)n+5}\otimes w_{jn+g(j)+1\ra jn+5},\quad i=j+s;\\
\kappa w_{(j+m)n+g(j)\ra (j+m)n+5}\otimes w_{jn+g(j)+1\ra jn+5},\quad i=j+2s;\\
0,\quad\text{otherwise.}\end{cases}$$

If $5s+(-1-3\ell_0)_s<j<8s$, then $c_{ij}=0$.

$(6)$ If $t_0<11$ and $r_0=5$, then the product $d_{r_0-1}f_{r_0}$
is an $(9s\times 9s)$-matrix $C=(c_{ij})$ with the following
elements $c_{ij}$:

If $0\le j<s+(-1-3\ell_0)_s$, then $c_{ij}=0$.

If $j=s+(-1-3\ell_0)_s$, then $$c_{ij}=
\begin{cases}
-w_{(j+m)n+5\ra (j+m+1)n+g(j+s)}\otimes w_{jn+g(j+s)\ra (j+1)n},\quad i=(j+1)_s;\\
w_{(j+m+1)n\ra (j+m+1)n+g(j+s)}\otimes w_{jn+g(j+s)\ra (j+1)n},\quad i=(j+1)_s+s;\\
0,\quad\text{otherwise.}\end{cases}$$

If $s+(-1-3\ell_0)_s<j<2s+(-1-3\ell_0)_s$, then $c_{ij}=0$.

If $j=2s+(-1-3\ell_0)_s$, then $$c_{ij}=
\begin{cases}
-w_{(j+m+1)n\ra (j+m+1)n+g(j+s)}\otimes w_{jn+g(j+s)\ra (j+1)n},\quad i=(j+1)_s+s;\\
0,\quad\text{otherwise.}\end{cases}$$

If $2s+(-1-3\ell_0)_s<j<3s+(-1-3\ell_0)_s$, then $c_{ij}=0$.

If $j=3s+(-1-3\ell_0)_s$, then $$c_{ij}=
\begin{cases}
w_{(j+m+1)n\ra (j+m+1)n+g(j)+1}\otimes w_{jn+g(j+s)\ra (j+1)n},\quad i=(j+1)_s+s;\\
-w_{(j+m)n+5\ra (j+m+1)n+g(j)+1}\otimes w_{jn+g(j+s)\ra jn+5},\quad i=j+5s;\\
0,\quad\text{otherwise.}\end{cases}$$

If $3s+(-1-3\ell_0)_s<j<4s+(-1-3\ell_0)_s$, then $c_{ij}=0$.

If $j=4s+(-1-3\ell_0)_s$, then $$c_{ij}=
\begin{cases}
-w_{(j+m+1)n\ra (j+m+1)n+g(j)+1}\otimes w_{jn+g(j+s)\ra (j+1)n},\quad i=(j+1)_s+s;\\
w_{(j+m)n+5\ra (j+m+1)n+g(j)+1}\otimes w_{jn+g(j+s)\ra jn+5},\quad i=j+4s;\\
0,\quad\text{otherwise.}\end{cases}$$

If $4s+(-1-3\ell_0)_s<j<5s+(-1-3\ell_0)_s$, then $c_{ij}=0$.

If $j=5s+(-1-3\ell_0)_s$, then $$c_{ij}=
\begin{cases}
-w_{(j+m)n+5\ra (j+m+1)n+g(j+s)+1}\otimes w_{jn+g(j+s)+1\ra (j+1)n},\quad i=(j+1)_s;\\
w_{(j+m)n+5\ra (j+m+1)n+g(j+s)+1}\otimes w_{jn+g(j+s)+1\ra jn+5},\quad i=j+3s;\\
0,\quad\text{otherwise.}\end{cases}$$

If $5s+(-1-3\ell_0)_s<j<6s+(-1-3\ell_0)_s$, then $c_{ij}=0$.

If $j=6s+(-1-3\ell_0)_s$, then $$c_{ij}=
\begin{cases}
-w_{(j+m)n+5\ra (j+m+1)n+g(j+s)+1}\otimes w_{jn+g(j+s)+1\ra jn+5},\quad i=j+2s;\\
0,\quad\text{otherwise.}\end{cases}$$

If $6s+(-1-3\ell_0)_s<j<9s$, then $c_{ij}=0$.

$(7)$ If $t_0<11$ and $r_0=6$, then the product $d_{r_0-1}f_{r_0}$
is an $(8s\times 8s)$-matrix $C=(c_{ij})$ with the following
elements $c_{ij}$:

If $0\le j<(-1-3\ell_0)_s$, then $c_{ij}=0$.

If $j=(-1-3\ell_0)_s$, then $$c_{ij}=
\begin{cases}
w_{(j+m)n+g(j)\ra (j+m)n+g(j)+1}\otimes w_{jn\ra (j+1)n},\quad i=(j+1)_s;\\
w_{(j+m)n\ra (j+m)n+g(j)+1}\otimes w_{jn\ra jn+g(j+s)},\quad i=j+3s;\\
-w_{(j+m-1)n+5\ra (j+m)n+g(j)+1}\otimes w_{jn\ra jn+g(j)+1},\quad i=j+4s;\\
-w_{(j+m-1)n+5\ra (j+m)n+g(j)+1}\otimes w_{jn\ra jn+g(j+s)+1},\quad i=j+5s;\\
0,\quad\text{otherwise.}\end{cases}$$

If $(-1-3\ell_0)_s<j<s+(-1-3\ell_0)_s$, then $c_{ij}=0$.

If $j=s+(-1-3\ell_0)_s$, then $$c_{ij}=
\begin{cases}
w_{(j+m)n+g(j)\ra (j+m)n+g(j)+1}\otimes w_{jn\ra (j+1)n},\quad i=(j+1)_s+s;\\
w_{(j+m)n\ra (j+m)n+g(j)+1}\otimes w_{jn\ra jn+g(j+s)},\quad i=j+s;\\
-w_{(j+m-1)n+5\ra (j+m)n+g(j)+1}\otimes w_{jn\ra jn+g(j+s)+1},\quad i=j+3s;\\
-w_{(j+m-1)n+5\ra (j+m)n+g(j)+1}\otimes w_{jn\ra jn+g(j)+1},\quad i=j+4s;\\
0,\quad\text{otherwise.}\end{cases}$$

If $s+(-1-3\ell_0)_s<j<2s+(-1-3\ell_0)_s$, then $c_{ij}=0$.

If $j=2s+(-1-3\ell_0)_s$, then $$c_{ij}=
\begin{cases}
-\kappa w_{(j+m)n+g(j)\ra (j+m)n+5}\otimes w_{jn+g(j)\ra (j+1)n},\quad i=(j+1)_s;\\
\kappa w_{(j+m)n\ra (j+m)n+5}\otimes e_{jn+g(j)},\quad i=j;\\
0,\quad\text{otherwise.}\end{cases}$$

If $2s+(-1-3\ell_0)_s<j<3s+(-1-3\ell_0)_s$, then $c_{ij}=0$.

If $j=3s+(-1-3\ell_0)_s$, then $$c_{ij}=
\begin{cases}
-\kappa w_{(j+m)n+g(j)\ra (j+m)n+5}\otimes w_{jn+g(j)\ra (j+1)n},\quad i=(j+1)_s+s;\\
\kappa w_{(j+m)n\ra (j+m)n+5}\otimes e_{jn+g(j)},\quad i=j;\\
0,\quad\text{otherwise.}\end{cases}$$

If $3s+(-1-3\ell_0)_s<j<8s$, then $c_{ij}=0$.

$(8)$ If $t_0<11$ and $r_0=7$, then the product $d_{r_0-1}f_{r_0}$
is an $(6s\times 9s)$-matrix $C=(c_{ij})$ with the following
elements $c_{ij}$:

If $0\le j<s+(-2-3\ell_0)_s$, then $c_{ij}=0$.

If $j=s+(-2-3\ell_0)_s$, then $$c_{ij}=
\begin{cases}
-w_{(j+m+1)n\ra (j+m+1)n+g(j+s)+1}\otimes w_{jn+g(j+s)\ra (j+1)n},\quad i=(j+1)_s;\\
w_{(j+m+1)n+g(j+s)\ra (j+m+1)n+g(j+s)+1}\otimes w_{jn+g(j+s)\ra (j+1)n+g(j+s)},\quad i=(j+1)_s+s;\\
0,\quad\text{otherwise.}\end{cases}$$

If $s+(-2-3\ell_0)_s<j<2s+(-2-3\ell_0)_s$, then $c_{ij}=0$.

If $j=2s+(-2-3\ell_0)_s$, then $$c_{ij}=
\begin{cases}
-w_{(j+m+1)n\ra (j+m+1)n+g(j+s)+1}\otimes w_{jn+g(j+s)\ra (j+1)n},\quad i=(j+1)_s;\\
w_{(j+m+1)n+g(j+s)\ra (j+m+1)n+g(j+s)+1}\otimes w_{jn+g(j+s)\ra (j+1)n+g(j+s)},\quad i=(j+1)_s+2s;\\
0,\quad\text{otherwise.}\end{cases}$$

If $2s+(-2-3\ell_0)_s<j<5s+(-2-3\ell_0)_s$, then $c_{ij}=0$.

If $j=5s+(-2-3\ell_0)_s$, then $$c_{ij}=
\begin{cases}
\kappa w_{(j+m+1)n\ra (j+m+2)n}\otimes w_{jn+5\ra (j+1)n},\quad i=(j+1)_s;\\
-\kappa w_{(j+m+1)n+g(j+s)\ra (j+m+2)n}\otimes w_{jn+5\ra (j+1)n+g(j+s)},\quad i=(j+1)_s+s;\\
-\kappa w_{(j+m+1)n+g(j)\ra (j+m+2)n}\otimes w_{jn+5\ra (j+1)n+g(j)},\quad i=(j+1)_s+2s;\\
0,\quad\text{otherwise.}\end{cases}$$

If $5s+(-2-3\ell_0)_s<j<6s$, then $c_{ij}=0$.

$(9)$ If $t_0<11$ and $r_0=8$, then the product $d_{r_0-1}f_{r_0}$
is an $(7s\times 8s)$-matrix $C=(c_{ij})$ with the following
elements $c_{ij}$:

If $0\le j<(-2-3\ell_0)_s$, then $c_{ij}=0$.

If $j=(-2-3\ell_0)_s$, then $$c_{ij}=
\begin{cases}
\kappa w_{(j+m)n+g(j)+1\ra (j+m)n+5}\otimes w_{jn\ra (j+1)n},\quad i=(j+1)_s;\\
-\kappa w_{(j+m)n+g(j+s)+1\ra (j+m)n+5}\otimes w_{jn\ra (j+1)n},\quad i=(j+1)_s+s;\\
0,\quad\text{otherwise.}\end{cases}$$

If $(-2-3\ell_0)_s<j<s+(-2-3\ell_0)_s$, then $c_{ij}=0$.

If $j=s+(-2-3\ell_0)_s$, then $$c_{ij}=
\begin{cases}
-\kappa w_{(j+m)n+g(j+s)+1\ra (j+m+1)n}\otimes w_{jn+g(j+s)\ra (j+1)n},\quad i=(j+1)_s;\\
-\kappa w_{(j+m)n+g(j+s)\ra (j+m+1)n}\otimes w_{jn+g(j+s)\ra jn+5},\quad i=j+5s;\\
-\kappa w_{(j+m)n+g(j)\ra (j+m+1)n}\otimes w_{jn+g(j+s)\ra jn+5},\quad i=j+6s;\\
0,\quad\text{otherwise.}\end{cases}$$

If $s+(-2-3\ell_0)_s<j<2s+(-2-3\ell_0)_s$, then $c_{ij}=0$.

If $j=2s+(-2-3\ell_0)_s$, then $$c_{ij}=
\begin{cases}
-\kappa w_{(j+m)n+g(j+s)+1\ra (j+m+1)n}\otimes w_{jn+g(j+s)\ra (j+1)n},\quad i=(j+1)_s+s;\\
-\kappa w_{(j+m)n+g(j)\ra (j+m+1)n}\otimes w_{jn+g(j+s)\ra jn+5},\quad i=j+4s;\\
-\kappa w_{(j+m)n+g(j+s)\ra (j+m+1)n}\otimes w_{jn+g(j+s)\ra jn+5},\quad i=j+5s;\\
0,\quad\text{otherwise.}\end{cases}$$

If $2s+(-2-3\ell_0)_s<j<3s+(-2-3\ell_0)_s$, then $c_{ij}=0$.

If $j=3s+(-2-3\ell_0)_s$, then $$c_{ij}=
\begin{cases}
w_{(j+m)n+g(j)\ra (j+m+1)n+g(j)}\otimes w_{jn+g(j+s)+1\ra jn+5},\quad i=j+4s;\\
0,\quad\text{otherwise.}\end{cases}$$

If $3s+(-2-3\ell_0)_s<j<4s+(-2-3\ell_0)_s$, then $c_{ij}=0$.

If $j=4s+(-2-3\ell_0)_s$, then $$c_{ij}=
\begin{cases}
w_{(j+m)n+g(j)\ra (j+m+1)n+g(j)}\otimes w_{jn+g(j+s)+1\ra jn+5},\quad i=j+2s;\\
0,\quad\text{otherwise.}\end{cases}$$

If $4s+(-2-3\ell_0)_s<j<7s$, then $c_{ij}=0$.

$(10)$ If $t_0<11$ and $r_0=9$, then the product $d_{r_0-1}f_{r_0}$
is an $(6s\times 6s)$-matrix $C=(c_{ij})$ with the following
elements $c_{ij}$:

If $0\le j<(-2-3\ell_0)_s$, then $c_{ij}=0$.

If $j=(-2-3\ell_0)_s$, then $$c_{ij}=
\begin{cases}
\kappa w_{(j+m)n+5\ra (j+m+1)n}\otimes w_{jn\ra (j+1)n},\quad i=(j+1)_s;\\
\kappa w_{(j+m)n+g(j)+1\ra (j+m+1)n}\otimes w_{jn\ra jn+g(j)},\quad i=j+s;\\
-\kappa w_{(j+m)n+g(j+s)\ra (j+m+1)n}\otimes w_{jn\ra jn+g(j)+1},\quad i=j+3s;\\
\kappa e_{(j+m+1)n}\otimes w_{jn\ra jn+5},\quad i=j+5s;\\
0,\quad\text{otherwise.}\end{cases}$$

If $(-2-3\ell_0)_s<j<2s+(-2-3\ell_0)_s$, then $c_{ij}=0$.

If $j=2s+(-2-3\ell_0)_s$, then $$c_{ij}=
\begin{cases}
w_{(j+m)n+g(j)\ra (j+m+1)n+g(j)}\otimes w_{jn+g(j+s)\ra jn+g(j+s)+1},\quad i=j+2s;\\
-w_{(j+m+1)n\ra (j+m+1)n+g(j)}\otimes w_{jn+g(j+s)\ra jn+5},\quad i=j+3s;\\
0,\quad\text{otherwise.}\end{cases}$$

If $2s+(-2-3\ell_0)_s<j<6s$, then $c_{ij}=0$.

$(11)$ If $t_0<11$ and $r_0=10$, then the product $d_{r_0-1}f_{r_0}$
is an $(6s\times 7s)$-matrix $C=(c_{ij})$ with the following
elements $c_{ij}$:

If $0\le j<(-2-3\ell_0)_s$, then $c_{ij}=0$.

If $j=(-2-3\ell_0)_s$, then $$c_{ij}=
\begin{cases}
\kappa w_{(j+m)n\ra (j+m+1)n}\otimes w_{jn\ra jn+g(j)},\quad i=j+s;\\
-\kappa w_{(j+m)n\ra (j+m+1)n}\otimes w_{jn\ra jn+g(j+s)},\quad i=j+2s;\\
0,\quad\text{otherwise.}\end{cases}$$

If $(-2-3\ell_0)_s<j<s+(-3-3\ell_0)_s$, then $c_{ij}=0$.

If $j=s+(-3-3\ell_0)_s$, then $$c_{ij}=
\begin{cases}
w_{(j+m+1)n\ra (j+m+1)n+g(j)}\otimes w_{jn+g(j+s)\ra (j+1)n+g(j+s)},\quad i=(j+1)_s+s;\\
0,\quad\text{otherwise.}\end{cases}$$

If $s+(-3-3\ell_0)_s<j<2s+(-3-3\ell_0)_s$, then $c_{ij}=0$.

If $j=2s+(-3-3\ell_0)_s$, then $$c_{ij}=
\begin{cases}
w_{(j+m+1)n\ra (j+m+1)n+g(j)}\otimes w_{jn+g(j+s)\ra (j+1)n+g(j+s)},\quad i=(j+1)_s+2s;\\
0,\quad\text{otherwise.}\end{cases}$$

If $2s+(-3-3\ell_0)_s<j<3s+(-3-3\ell_0)_s$, then $c_{ij}=0$.

If $j=3s+(-3-3\ell_0)_s$, then $$c_{ij}=
\begin{cases}
-w_{(j+m+1)n\ra (j+m+1)n+g(j)+1}\otimes w_{jn+g(j+s)+1\ra (j+1)n+g(j+s)},\quad i=(j+1)_s+s;\\
0,\quad\text{otherwise.}\end{cases}$$

If $3s+(-3-3\ell_0)_s<j<4s+(-3-3\ell_0)_s$, then $c_{ij}=0$.

If $j=4s+(-3-3\ell_0)_s$, then $$c_{ij}=
\begin{cases}
-w_{(j+m+1)n\ra (j+m+1)n+g(j)+1}\otimes w_{jn+g(j+s)+1\ra (j+1)n+g(j+s)},\quad i=(j+1)_s+2s;\\
0,\quad\text{otherwise.}\end{cases}$$

If $4s+(-3-3\ell_0)_s<j<5s+(-3-3\ell_0)_s$, then $c_{ij}=0$.

If $j=5s+(-3-3\ell_0)_s$, then $$c_{ij}=
\begin{cases}
-\kappa w_{(j+m+1)n\ra (j+m+1)n+5}\otimes w_{jn+5\ra (j+1)n+g(j+s)},\quad i=(j+1)_s+s;\\
\kappa w_{(j+m+1)n\ra (j+m+1)n+5}\otimes w_{jn+5\ra (j+1)n+g(j)},\quad i=(j+1)_s+2s;\\
0,\quad\text{otherwise.}\end{cases}$$

If $5s+(-3-3\ell_0)_s<j<6s$, then $c_{ij}=0$.

$(12)$ If $t_0\ge 11$, then the product matrix $d_{t_0-1}f_{t_0}$ is
a $d_{r_0-1}f_{r_0}$, whose left components twisted by
$\sigma^{\ell_0}$, and coefficients multiplied by $(-1)^{\ell_0}$.

The described matrixes of products $d_{t_0-1}f_{t_0}$ coincide with
$f_{t_0-1}d_{t+t_0-1}$.
\end{proof}

\section{Multiplications in $\HH^*(R)$}

From the descriptions of elements $Y^{(i)}_t$ and
its$\Omega$-translates we can find multiplications of the elements
using the formula \eqref{mult_formula}.

To find the multiplications we need the description of $\sigma^t$
for an arbitrary $t\in\N$. From the description of an automorphism
$\sigma$ we have:
\begin{align*}
\sigma^{2t}(\a_i)&=\begin{cases} (-1)^t\a_{i+6tn},\quad i\equiv
0,2(3);\\\a_{i+6tn},\quad i\equiv 1(3),
\end{cases}
\sigma^{2t}(\b_i)=\begin{cases} (-1)^t\b_{i+6tn},\quad i\equiv
0,2(3);\\\b_{i+6tn},\quad i\equiv 1(3),
\end{cases}\\
\sigma^{2t+1}(\a_i)&=\begin{cases} (-1)^{t+1}\b_{i+3(2t+1)n},\quad
i\equiv 0(3);\\-\b_{i+3(2t+1)n},\quad i\equiv 1(3);\\
(-1)^t\b_{i+3(2t+1)n},\quad i\equiv 2(3),
\end{cases}
\sigma^{2t+1}(\b_i)=\begin{cases} (-1)^t\a_{i+3(2t+1)n},\quad
i\equiv 0(3);\\-\a_{i+3(2t+1)n},\quad i\equiv 1(3);\\
(-1)^{t+1}\a_{i+3(2t+1)n},\quad i\equiv 2(3),
\end{cases}\\
\sigma^t(\g_i)&=(-1)^t\g_{i+tn}.
\end{align*}

We will find a multiplication of elements of the types 4 and 3 for
$s>1$.

Consider two arbitrary elements $Y_{t_4}^{(4)}$ and $Y_{t_3}^{(3)}$.
For its degrees $t_4$ and $t_3$ we have:
\begin{align*}
t_4&=11\ell_4+1,\text{ }\ell_4 n\equiv 0(s),\text{ }\ell_4\ndiv 2;\\
t_3&=11\ell_3+1,\text{ }\ell_3 n\equiv 0(s),\text{ }\ell_3\div 2.
\end{align*}
Let $t=t_4+t_3$; this is the degree of an element
$Y_{t_4}^{(4)}Y_{t_3}^{(3)}$. Then $t=11(\ell_4+\ell_3)+2$. Group of
the degree $t$ has type (5).

Denote by $B=(b_{ij})$ translate matrix of an element
$Y_{t_4}^{(4)}$ by degree $t_3$. This matrix is of the form.

If $0\le j<s$, then
\begin{align*}
b_{ij}&=
\begin{cases}
\kappa \sigma^{\ell_3}(w_{(j+m_4)n+g(j+s)\ra (j+m_4)n+5})\otimes e_{jn},\quad i=j+s;\\
\kappa \sigma^{\ell_3}(w_{(j+m_4)n+g(j+s)+1\ra (j+m_4)n+5})\otimes w_{jn\ra jn+g(j+s)},\quad i=j+3s;\\
0,\quad\text{otherwise;}\end{cases}\\
&=\begin{cases}
\kappa \sigma^{\ell_3}(w_{(j+m_4)n+3\ra (j+m_4)n+5})\otimes e_{jn},\quad i=j+s;\\
\kappa \sigma^{\ell_3}(w_{(j+m_4)n+4\ra (j+m_4)n+5})\otimes w_{jn\ra jn+3},\quad i=j+3s;\\
0,\quad\text{otherwise;}\end{cases}\\
&=\begin{cases}
\kappa \sigma^{\ell_3}(\b_{3(j+m_4)+2}\b_{3(j+m_4)+1})\otimes e_{jn},\quad i=j+s;\\
\kappa \sigma^{\ell_3}(\b_{3(j+m_4)+2})\otimes \b_{3j},\quad i=j+3s;\\
0,\quad\text{otherwise;}\end{cases}\\
&=\begin{cases}
(-1)^{\frac{\ell_4-1}{2}}(-1)^{\frac{\ell_3}{2}} \b_{3(j+\ell_3n)+2}\b_{3(j+\ell_3n)+1}\otimes e_{jn},\quad i=j+s;\\
(-1)^{\frac{\ell_4-1}{2}}(-1)^{\frac{\ell_3}{2}} \b_{3(j+\ell_3n)+2}\otimes \b_{3j},\quad i=j+3s;\\
0,\quad\text{otherwise;}\end{cases}\\
&=\begin{cases}
(-1)^{\frac{\ell_4+\ell_3-1}{2}}\b_{3j+2}\b_{3j+1}\otimes e_{jn},\quad i=j+s;\\
(-1)^{\frac{\ell_4+\ell_3-1}{2}}\b_{3j+2}\otimes \b_{3j},\quad i=j+3s;\\
0,\quad\text{otherwise.}\end{cases}
\end{align*}

If $s\le j<2s$, then
\begin{align*}b_{ij}&=
\begin{cases}
\sigma^{\ell_3}(w_{(j+m_4+1)n+g(j)\ra (j+m_4+1)n+g(j)+1})\otimes w_{jn+g(j+s)\ra (j+1)n},\quad i=(j+1)_s+s;\\
0,\quad\text{otherwise;}\end{cases}\\
&=\begin{cases}\sigma^{\ell_3}(w_{(j+m_4+1)n+3\ra (j+m_4+1)n+4})\otimes w_{jn+1\ra (j+1)n},\quad i=(j+1)_s+s;\\
0,\quad\text{otherwise;}\end{cases}\\
\end{align*}
\begin{align*}&=\begin{cases}
\sigma^{\ell_3}(\b_{3(j+m_4+1)+1})\otimes \g_j\a_{3j+2}\a_{3j+1},\quad i=(j+1)_s+s;\\
0,\quad\text{otherwise;}\end{cases}\\
&=\begin{cases}
\b_{3(j+1)+1}\otimes \g_j\a_{3j+2}\a_{3j+1},\quad i=(j+1)_s+s;\\
0,\quad\text{otherwise.}\end{cases}
\end{align*}

If $2s\le j<4s$, then $b_{ij}=0$.

If $4s\le j<5s$, then
\begin{align*}b_{ij}&=
\begin{cases}
\sigma^{\ell_3}(w_{(j+m_4)n+5\ra (j+m_4+1)n+g(j+s)})\otimes e_{jn+g(j+s)+1},\quad i=j+s;\\
\sigma^{\ell_3}(w_{(j+m_4+1)n\ra (j+m_4+1)n+g(j+s)})\otimes w_{jn+g(j+s)+1\ra jn+5},\quad i=j+2s;\\
0,\quad\text{otherwise;}\end{cases}\\
&=\begin{cases}
\sigma^{\ell_3}(\b_{3(j+m_4+1)}\g_{j+m})\otimes e_{jn+4},\quad i=j+s;\\
\sigma^{\ell_3}(\b_{3(j+m_4+1)})\otimes \b_{3j+2},\quad i=j+2s;\\
0,\quad\text{otherwise;}\end{cases}\\
&=\begin{cases}
(-1)^{\frac{\ell_3}{2}}\b_{3(j+1)}\g_j\otimes e_{jn+4},\quad i=j+s;\\
(-1)^{\frac{\ell_3}{2}}\b_{3(j+1)}\otimes \b_{3j+2},\quad i=j+2s;\\
0,\quad\text{otherwise.}\end{cases}
\end{align*}

If $5s\le j<6s$, then
\begin{align*}b_{ij}&=
\begin{cases}
\kappa \sigma^{\ell_3}(w_{(j+m_4+1)n\ra (j+m_4+2)n})\otimes e_{jn+5},\quad i=j+s;\\
0,\quad\text{otherwise;}\end{cases}\\
&=\begin{cases}
\kappa \sigma^{\ell_3}(\g_{j+m_4+1}\a_{3(j+m_4+1)+2}\a_{3(j+m_4+1)+1}\a_{3(j+m_4+1)})\otimes e_{jn+5},\quad i=j+s;\\
0,\quad\text{otherwise;}\end{cases}\\
&=\begin{cases}
(-1)^{\frac{\ell_4-1}{2}}\g_{j+1}\a_{3(j+1)+2}\a_{3(j+1)+1}\a_{3(j+1)}\otimes e_{jn+5},\quad i=j+s;\\
0,\quad\text{otherwise.}\end{cases}
\end{align*}

Multiply this matrix by an element $Y_{t_3}^{(3)}$, which is
$(7s\times 6s)$-matrix with single nonzero element
$y_{5s,6s}=(-1)^{\frac{\ell_3}{2}}\g_j\otimes e_{jn+5}.$

Multiplication is the matrix $C=(c_{ij})$ with single nonzero
element $c_{5s,4s}=\b_{3(j+1)}\g_j\otimes \b_{3j+2}.$ We must show
that this element coincide with $Y^{(5)}_t$ for degree of type (5).
Element $Y^{(5)}_t$ is matrix with single nonzero element
$y_{3s,3s}=\a_{3(j+1)}\g_j\a_{3j+2}\otimes e_{jn+2}.$ Thus we must
show that
$Y^{(5)}_t-C=\a_{3(j+1)}\g_j\a_{3j+2}-\b_{3(j+1)}\g_j\b_{3j+2}$ is
in $\Im\delta^{t-1}$.

Consider the matrix $D=(d)_{ij}$, which is obtained from
differential matrix $d_{t-1}$ by replacing elements which turns into
$0$ in $\Im\delta^{t-1}$, by $0$. Elements of this matrix are of the
form.

If $0\le j<s$, then
\begin{align*}(d)_{ij}&=\begin{cases}
\sigma^{\ell_3+\ell_4}(w_{jn+1+j_1+2f(j_1,2)\ra jn+5})\otimes w_{jn\ra jn+j_1},\\\quad\quad\quad  i=j+2sj_1,\text{ }0\le j_1< 3;\\
-\sigma^{\ell_3+\ell_4}(w_{jn+3+j_1\ra jn+5})\otimes w_{jn\ra jn+j_1+2(1-f(j_1,0))},\\\quad\quad\quad  i=j+2sj_1+s,\text{ }0\le j_1< 3;\\
0\quad\text{otherwise;}\end{cases}\\
&=\begin{cases}
(-1)^{\frac{\ell_3+\ell_4+1}{2}}\b_{3j+2}\b_{3j+1}\otimes e_{jn},\quad i=j;\\
e_{jn+5}\otimes \a_{3j+1}\a_{3j},\quad i=j+4s;\\
(-1)^{\frac{\ell_3+\ell_4+1}{2}}\a_{3j+2}\a_{3j+1}\otimes e_{jn},\quad i=j+s;\\
-e_{jn+5}\otimes \b_{3j+1}\b_{3j},\quad i=j+5s;\\
0\quad\text{otherwise.}\end{cases}
\end{align*}

If $s\le j<3s$, then $(d)_{ij}=0$.

If $3s\le j<5s$, then $$(d)_{ij}=\begin{cases}
(-1)^{\frac{\ell_3+\ell_4+1}{2}}\a_{3(j+1)}\g_j\otimes e_{jn+1},\quad i=j+s,\text{ }j<4s;\\
(-1)^{\frac{\ell_3+\ell_4-1}{2}}\b_{3(j+1)}\g_j\otimes e_{jn+4},\quad i=j+s,\text{ }j\ge 4s;\\
(-1)^{\frac{\ell_3+\ell_4-1}{2}}\a_{3(j+1)}\otimes \a_{3j+1},\quad i=(j)_s+6s,\text{ }j<4s;\\
(-1)^{\frac{\ell_3+\ell_4+1}{2}}\b_{3(j+1)}\otimes \b_{3j+1},\quad i=(j)_s+6s,\text{ }j\ge 4s;\\
e_{(j+1)n+g(j)}\otimes w_{jn+g(j+s)+1\ra}^{(2)},\quad i=(j+1)_s+sf_0(j,4s);\\
0\quad\text{otherwise.}\end{cases}$$

If $5s\le j<6s$, then $(d)_{ij}=0$.

$6s$-th row of image matrix multiplied by $(-1)^{\frac{\ell_3+\ell_4-1}{2}}$, coincide with
$Y^{(5)}_t-C$, hence, $Y^{(5)}_t-Y_{t_4}^{(4)}Y_{t_3}^{(3)}\in\Im\delta^{t-1}$, i. e.
$Y^{(5)}_t-Y_{t_4}^{(4)}Y_{t_3}^{(3)}=0$ in the cohomology ring.


Multiplications of other elements, except $Y^{(5)}$, $Y^{(10)}$, $Y^{(17)}$, $Y^{(19)}$ and
$Y^{(21)}$, are similarly considered. To get the whole picture we should prove the following lemma.

\begin{lem}$\text{ }$

$($a$)$ Let $Y^{(5)}$ be an arbitrary element from generators of the
corresponding type. Then there are elements $Y^{(3)}$ and $Y^{(4)}$
such as $Y^{(5)}=Y^{(3)}Y^{(4)}$.

$($b$)$ Let $Y^{(10)}$ be an arbitrary element from generators of
the corresponding type. Then there are elements $Y^{(3)}$ and
$Y^{(6)}$ such as $Y^{(10)}=Y^{(3)}Y^{(6)}$.

$($c$)$ Let $Y^{(17)}$ be an arbitrary element from generators of
the corresponding type. Then there are elements $Y^{(3)}$ and
$Y^{(15)}$ such as $Y^{(17)}=Y^{(3)}Y^{(15)}$.

$($d$)$ Let $Y^{(19)}$ be an arbitrary element from generators of
the corresponding type. Then there are elements $Y^{(3)}$ and
$Y^{(18)}$ such as $Y^{(19)}=Y^{(3)}Y^{(18)}$.

$($e$)$ Let $Y^{(21)}$ be an arbitrary element from generators of
the corresponding type. Then there are elements $Y^{(3)}$ and
$Y^{(20)}$ such as $Y^{(21)}=Y^{(3)}Y^{(20)}$.

\end{lem}

\begin{proof}
The degree 1 has type 3, for all $s$. It only remains to use the
relations for type (3).
\end{proof}

\end{document}